\documentclass{article}

\usepackage{longtable,amsmath, amsthm, amsfonts,stmaryrd, amssymb,graphicx,enumitem, wrapfig, scalerel, tikz,   rotating, comment,     float, hyperref,ifthen,lmodern, textcomp,subcaption,wrapfig}
\usepackage[dvipsnames]{xcolor}

\usetikzlibrary{arrows}
\usetikzlibrary{decorations.markings}
\usetikzlibrary{decorations}
\usetikzlibrary{patterns}
\usetikzlibrary{arrows}
\usetikzlibrary{arrows.meta}
\usetikzlibrary{intersections}
\usetikzlibrary{calc}
\usetikzlibrary{cd}
\usetikzlibrary{braids}

\newtheorem{prop}{Proposition}[section]
\newtheorem{thh}[prop]{Theorem}
\newtheorem{lemma}[prop]{Lemma}

\theoremstyle{definition}
\newtheorem{deff}[prop]{Definition}
\newtheorem*{claim}{Claim}

\newtheorem{example}[prop]{Example}

\newcommand{\inlinepic}[1]{\ensuremath{\vcenter{\hbox{\begin{tikzpicture}#1\end{tikzpicture}}}}}
\newcommand{\imagesfolder}{Figures/}

\tikzset{->-/.style={decoration={
  markings,
  mark=at position 0.5 with {\arrow{>}}},postaction={decorate}}}
\tikzset{-<-/.style={decoration={
  markings,
  mark=at position 0.5 with {\arrow{<}}},postaction={decorate}}}
\tikzset{->>-/.style={decoration={
  markings,
  mark=at position 0.5 with {\arrow{>>}}},postaction={decorate}}}
\tikzset{-<<-/.style={decoration={
  markings,
  mark=at position 0.5 with {\arrow{<<}}},postaction={decorate}}}

\newcommand{\BSphere}[1]{
\begin{scope}
  \draw[draw=Blue  ,  fill = Cyan,fill opacity=0.2] (0,0) circle (1cm);
  \draw[draw=Blue, dotted] (1,0) arc (0:180:1cm and 0.5cm) ;
  \draw [draw=Blue](1,0) arc (0:-180:1cm and 0.5cm) ;
  \ifthenelse{\equal{#1}{1}}{
      \fill[Blue] (0,0.7) circle (2pt);
      }{}
        
  \ifthenelse{\equal{#1}{2}}{
      \fill[Blue] (0.1,0.7) circle (2pt);
      \fill[Blue] (-0.1,0.7) circle (2pt);
      }{}

\end{scope}}
\newcommand{\RSphere}{
  \begin{scope}
    \draw[draw=Maroon, thick,fill=Red,fill opacity=0.6](0,0) circle (1cm);
    \draw[draw=Maroon, thick, dotted] (1,0) arc (0:180:1cm and 0.5cm);
    \draw[draw=Maroon, thick](1,0)  arc (0:-180:1cm and 0.5cm);
  \end{scope}
}

\newcommand{\BCap}[2]{

\draw[draw=Blue, fill=Cyan,fill opacity=0.2]

  (1,0) arc (0:180:1cm)     
  -- (-1,0) arc  (-180:0 :1cm and 0.5cm)
  -- cycle ; 
\draw[draw=Blue, dotted] (1,0) arc (0:180:1cm and 0.5cm) ;
\ifthenelse{\equal{#1}{1}}{
    \fill[Blue] (0,0.7) circle (2pt);
    }{}
        
\ifthenelse{\equal{#1}{2}}{
    \fill[Blue] (0.1,0.7) circle (2pt);
    \fill[Blue] (-0.1,0.7) circle (2pt);
    }{}
\ifthenelse{\equal{#2}{+}}{
\draw[draw=Blue,-<-]  (-1,0) arc  (-180:0 :1cm and 0.5cm);

}{}
\ifthenelse{\equal{#2}{-}}{
\draw[draw=Blue,->-]  (-1,0) arc  (-180:0 :1cm and 0.5cm);

}{}
}

\newcommand{\RCap}[1]{

\draw[draw=Maroon, thick, fill=Red,fill opacity=0.6]

  (1,0) arc (0:180:1cm)     
  -- (-1,0) arc  (-180:0 :1cm and 0.5cm)
  -- cycle ; 
\draw[draw=Maroon, thick, dotted] (1,0) arc (0:180:1cm and 0.5cm) ;

\ifthenelse{\equal{#1}{+}}{
\draw[draw=Maroon,-<-, thick]  (-1,0) arc  (-180:0 :1cm and 0.5cm);

}{}
\ifthenelse{\equal{#1}{-}}{
\draw[draw=Maroon,->-, thick]  (-1,0) arc  (-180:0 :1cm and 0.5cm);

}{}
}

\newcommand{\BRCap}[4]{

\draw[draw=Blue, fill=Cyan,fill opacity=0.2]

  (1,0.5) arc (0:180:1cm)     
  -- (-1,0) arc  (-180:0 :1cm and 0.5cm)
  -- cycle ; 
\draw[draw=Blue, dotted] (1,0) arc (0:180:1cm and 0.5cm) ;
\fill[Red, fill opacity= 0.60] (0,0.5) ellipse (1cm and 0.5cm);
\draw[draw=Maroon, thick,dotted] (1,0.5) arc (0:180:1cm and 0.5cm) ;
\draw[draw=Maroon, thick]  (-1,0.5) arc  (-180:0 :1cm and 0.5cm);

\ifthenelse{\equal{#1}{1}}{
    \fill[Blue] (0,1.2) circle (2pt);
    }{}
        
\ifthenelse{\equal{#1}{2}}{
    \fill[Blue] (0.1,1.2) circle (2pt);
    \fill[Blue] (-0.1,1.2) circle (2pt);
    }{}

\ifthenelse{\equal{#2}{+}}{
\draw[draw=Blue,-<-]  (-1,0) arc  (-180:0 :1cm and 0.5cm);
\draw[draw=Maroon,thick,->-]  (-1,0.5) arc  (-180:0 :1cm and 0.5cm);

}{}
\ifthenelse{\equal{#2}{-}}{
\draw[draw=Blue,->-]  (-1,0) arc  (-180:0 :1cm and 0.5cm);
\draw[draw=Maroon,thick,-<-]  (-1,0.5) arc  (-180:0 :1cm and 0.5cm);

}{}

\ifthenelse{\equal{#3}{1}}{
    \fill[Maroon] (0,0.5) circle (2pt);
    }{}

\ifthenelse{\equal{#4}{1}}{
    \fill[Blue] (0,-0.2) circle (2pt);
    }{}

}

\newcommand{\BCup}[2]{

\draw[draw=Blue, fill=Cyan,fill opacity=0.2]

  (1,0) arc (0:180:1cm and 0.5cm)     
  -- (-1,0) arc  (-180:0 :1cm)
  -- cycle ; 
\draw[draw=Blue] (-1,0) arc (-180:0:1cm and 0.5cm) ;
\ifthenelse{\equal{#1}{1}}{
    \fill[Blue] (0,-0.7) circle (2pt);
    }{}
        
\ifthenelse{\equal{#1}{2}}{
    \fill[Blue] (0.1,-0.7) circle (2pt);
    \fill[Blue] (-0.1,-0.7) circle (2pt);
    }{}
\ifthenelse{\equal{#2}{+}}{
\draw[draw=Blue,->-] (-1,0) arc (-180:0:1cm and 0.5cm) ;

}{}
\ifthenelse{\equal{#2}{-}}{
\draw[draw=Blue,-<-]  (1,0) arc  (0:180 :1cm and 0.5cm);

}{}
}

\newcommand{\RCup}[1]{

\draw[draw=Maroon,thick, fill=Red,fill opacity=0.6]

  (1,0) arc (0:180:1cm and 0.5cm)     
  -- (-1,0) arc  (-180:0 :1cm)
  -- cycle ; 
\draw[draw=Maroon,thick] (-1,0) arc (-180:0:1cm and 0.5cm) ;

\ifthenelse{\equal{#1}{+}}{
\draw[draw=Maroon,->-,thick]  (1,0) arc  (0:180:1cm and 0.5cm);
\draw[draw=Maroon,->-,thick] (-1,0) arc (-180:0:1cm and 0.5cm) ;

}{}
\ifthenelse{\equal{#1}{-}}{
\draw[draw=Maroon,-<-,thick]  (1,0) arc  (0:180 :1cm and 0.5cm);
\draw[draw=Maroon,-<-,thick] (-1,0) arc (-180:0:1cm and 0.5cm) ;

}{}
}
\newcommand{\BRCup}[4]{

\draw[draw=Blue, fill=Cyan,fill opacity=0.2]

  (1,0) arc (0:180:1cm and 0.5cm)     
  -- (-1,-0.5) arc  (-180:0 :1cm)
  -- cycle ; 
\draw[draw=Blue] (-1,0) arc (-180:0:1cm and 0.5cm) ;

\fill[Red, fill opacity= 0.60] (0,-0.5) ellipse (1cm and 0.5cm);
\draw[draw=Maroon, thick,dotted] (1,-0.5) arc (0:180:1cm and 0.5cm) ;
\draw[draw=Maroon, thick]  (-1,-0.5) arc  (-180:0 :1cm and 0.5cm);

\ifthenelse{\equal{#1}{1}}{
    \fill[Blue] (0,-1.2) circle (2pt);
    }{}
        
\ifthenelse{\equal{#1}{2}}{
    \fill[Blue] (0.1,-1.2) circle (2pt);
    \fill[Blue] (-0.1,-1.2) circle (2pt);
    }{}

\ifthenelse{\equal{#2}{+}}{
\draw[draw=Blue,->-]  (1,0) arc  (0:180:1cm and 0.5cm);
\draw[draw=Blue,->-] (-1,0) arc (-180:0:1cm and 0.5cm) ;
\draw[draw=Maroon, thick,->-] (-1,-0.5) arc (-180:0:1cm and 0.5cm) ;

}{}
\ifthenelse{\equal{#2}{-}}{
\draw[draw=Blue,-<-]  (1,0) arc  (0:180 :1cm and 0.5cm);
\draw[draw=Blue,-<-] (-1,0) arc (-180:0:1cm and 0.5cm) ;
\draw[draw=Maroon, thick,-<-] (-1,-0.5) arc (-180:0:1cm and 0.5cm) ;
}{}

\ifthenelse{\equal{#3}{1}}{
    \fill[Maroon] (0,-0.6) circle (2pt);
    }{}

\ifthenelse{\equal{#4}{1}}{
    \fill[Blue] (0,0.1) circle (2pt);
    }{}

}

\newcommand{\RDot}{
    \fill[Maroon] (0,0) circle (2pt);
    }
\newcommand{\BDot}{
    \fill[Blue] (0,0) circle (2pt);
    }
\newcommand{\Med}{
    \fill[Red, fill opacity=0.6] (0,0) ellipse (1cm and 0.5cm);
    \draw[Maroon,thick,dotted](1,0) arc (0:180:1cm and 0.5cm);
    \draw[Maroon,thick,](-1,0) arc (-180:0:1cm and 0.5cm);

}

\newcommand{\RDisk}[1]{
\begin{scope}
    \draw[draw=Maroon, thick,fill=Red,fill opacity= 0.6] (0,0) ellipse (1cm and 0.5cm);
    \ifthenelse{\equal{#1}{+}}{
    \draw[draw=Maroon, thick,->-]  (0,0) ellipse (1cm and 0.5cm);
    }{
    \draw[draw=Maroon,thick,-<-]  (0,0) ellipse (1cm and 0.5cm);
    }
\end{scope}}

\newcommand{\BDisk}[1]{
\begin{scope}
    \draw[draw=Blue,fill=Cyan,fill opacity= 0.2] (0,0) ellipse (1cm and 0.5cm);
    \ifthenelse{\equal{#1}{+}}{
    \draw[draw=Blue,->-]  (0,0) ellipse (1cm and 0.5cm);
    }{
    \draw[draw=Blue,-<-]  (0,0) ellipse (1cm and 0.5cm);
    }
\end{scope}}

\newcommand{\BCyl}[3]{
    \begin{scope}
        \draw[draw=Blue,fill=Cyan,fill opacity=0.2] (1,2) arc (0:180: 1 cm and 0.5 cm) -- (-1,-2) arc (-180:0: 1 cm and 0.5 cm) -- cycle;
        \draw[draw=Blue] (-1,2) arc (-180: 0: 1 cm and 0.5 cm);
        \draw[draw=Blue,dotted] (1,-2) arc (0:180: 1 cm and 0.5 cm);
        \ifthenelse{\equal{#1}{1}}{
        \fill[Blue] (0,0) circle (2pt);
        }{}
        \ifthenelse{\equal{#1}{2}}{
        \fill[Blue] (0,-0.5) circle (2pt);
        \fill[Blue] (0,0.5) circle (2pt);
        }{}    
        \ifthenelse{\equal{#2}{+}}{
        \draw[draw=Blue,->-] (-1,2)  arc (-180:0:1cm and 0.5 cm);
        }{
        \draw[draw=Blue,-<-] (-1,2)  arc (-180:0:1cm and 0.5 cm);
        }
        \ifthenelse{\equal{#3}{+}}{
        \draw[draw=Blue,->-] (-1,-2) arc (-180:0: 1 cm and 0.5 cm);    
        }{
        \draw[draw=Blue,-<-] (-1,-2) arc (-180:0: 1 cm and 0.5 cm);    
        }
    \end{scope}

}

\newcommand{\BId}[3]{
\begin{scope}
    \draw[draw=Blue, fill=Cyan, fill opacity=0.2]  
    (-1.5,0.5)--(1.5,3.5)--(1.5,-0.5)--(-1.5,-3.5)--cycle;
    \ifthenelse{\equal{#1}{1}}{
    \fill[Blue] (0,0) circle (2pt);
    }{}
    \ifthenelse{\equal{#1}{2}}{
    \fill[Blue] (0,-0.5) circle (2pt);
    \fill[Blue] (0,0.5) circle (2pt);
    }{}   
    \ifthenelse{\equal{#2}{+}}{
    \draw[draw=Blue,->-] (-1.5,0.5)--(1.5,3.5);
    
    }{}
    \ifthenelse{\equal{#2}{-}}{
    \draw[draw=Blue,-<-] (-1.5,0.5)--(1.5,3.5);
   
    }{}
    \ifthenelse{\equal{#3}{+}}{
    \draw[draw=Blue,->-] (1.5,-0.5)--(-1.5,-3.5);
    
    }{}
    \ifthenelse{\equal{#3}{-}}{

    \draw[draw=Blue,-<-] (1.5,-0.5)--(-1.5,-3.5);
   
    }{}
\end{scope}

}

\newcommand{\RCyl}[3]{
    \begin{scope}
    \draw[draw=Maroon, thick,fill=Red,fill opacity=0.6] (1,2) arc (0:180: 1 cm and 0.5 cm) -- (-1,-2) arc (-180:0: 1 cm and 0.5 cm) -- cycle;
    \draw[draw=Maroon, thick] (-1,2) arc (-180: 0: 1 cm and 0.5 cm);
    \draw[draw=Maroon, thick,dotted] (1,-2) arc (0:180: 1 cm and 0.5 cm);
    \ifthenelse{\equal{#1}{1}}{
    \fill[Maroon, thick] (0,0) circle (2pt);
    }{}

    \ifthenelse{\equal{#2}{+}}{
    \draw[draw=Maroon,thick,->-] (-1,2)  arc (-180:0:1cm and 0.5 cm);
    
    }{

    \draw[draw=Maroon, thick,-<-] (-1,2)  arc (-180:0:1cm and 0.5 cm);
    
    }
    \ifthenelse{\equal{#3}{+}}{
    \draw[draw=Maroon, thick,->-] (-1,-2) arc (-180:0: 1 cm and 0.5 cm);    
    }{
    \draw[draw=Maroon, thick,-<-] (-1,-2) arc (-180:0: 1 cm and 0.5 cm);    
    }{}
    \end{scope}

}

\newcommand{\RId}[4]{
\begin{scope}
    \draw[draw=Maroon, thick, fill=Red, fill opacity=0.6]  
    (-1.5,0.5)--(1.5,3.5)--(1.5,-0.5)--(-1.5,-3.5)--cycle;
    \ifthenelse{\equal{#1}{1}}{
    \fill[Maroon, thick] (0,0) circle (2pt);
    }{}
    \ifthenelse{\equal{#4}{1}}{
    \draw[draw=Maroon, thick,fill=Maroon] (-2pt,2pt)--(2pt,2pt)--(2pt,-2pt)--(-2pt,-2pt)--cycle;
    }{}   
    \ifthenelse{\equal{#2}{+}}{
    \draw[draw=Maroon, thick,->-] (-1.5,0.5)--(1.5,3.5);

    }{}
    \ifthenelse{\equal{#2}{-}}{

    \draw[draw=Maroon, thick,-<-] (-1.5,0.5)--(1.5,3.5);

    }{}
    \ifthenelse{\equal{#3}{+}}{
    \draw[draw=Maroon, thick,->-] (1.5,-0.5)--(-1.5,-3.5);

    }{}
    \ifthenelse{\equal{#3}{-}}{

    \draw[draw=Maroon, thick,-<-] (1.5,-0.5)--(-1.5,-3.5);

    }{}
\end{scope}

}

\newcommand{\Pimple}[1]{

\begin{scope}
    \fill[Cyan, fill opacity=0.2, even odd rule]
    (0,0) ellipse (1.5cm and 0.75 cm)
    (0,0) ellipse (1cm and 0.5 cm);

    \draw[draw=Blue] (0,0) ellipse (1.5cm and 0.75 cm);    
    \fill[Cyan!20] (1,0) arc (0:180:1cm) -- (-1,0) arc (-180:0: 1cm and 0.5 cm);
    \fill[Red, fill opacity=0.6]    

    (0,0) ellipse (1cm and 0.5 cm);
    \draw[draw=Maroon, thick] (-1,0) arc (-180:0: 1cm and 0.5 cm);
    \draw[draw=Blue,dotted](0,0) ellipse (1.5cm and 0.75 cm);    
    \draw[draw=Maroon, thick,dotted] (1,0) arc (0:180: 1cm and 0.5 cm);   
    \draw[draw=Blue] (1,0) arc (0:180:1cm);
    \ifthenelse{\equal{#1}{+}}{
    \draw[draw=Maroon, thick,->-] (-1,0) arc (-180:0: 1cm and 0.5 cm);   
    \draw[draw=Blue,->-] (-1.5,0) arc (-180:0: 1.5cm and 0.75 cm);
    }{
    \draw[draw=Maroon, thick,-<-] (-1,0) arc (-180:0: 1cm and 0.5 cm);   
    \draw[draw=Blue,-<-] (-1.5,0) arc (-180:0: 1.5cm and 0.75 cm);
    }
\end{scope}

}

\newcommand{\Saturn}[1]{
\begin{scope}
    \fill[Red,fill opacity=0.6, even odd rule] (0,0) ellipse (1.5cm and 0.75 cm) (0,0) ellipse (1cm and 0.5cm);
    \fill[Cyan,fill opacity=0.2] (0,0) circle (1cm);
    \draw[draw=Maroon, thick] (-1,0) arc (-180:0:1cm and 0.5 cm);

    \draw[draw=Maroon, thick,dotted] (1,0) arc (0:180:1cm and 0.5 cm);    
    \begin{scope}[even odd rule]
        \clip (0,0) circle (1.6cm) (1,0) arc (0:180:1cm);
        \draw[Maroon,  thick ](0,0) ellipse (1.5cm and 0.75cm);
    \end{scope}
    \draw[draw=Maroon,  thick] (1.5,0) arc (0:180:1.5cm and 0.75cm);
    \draw[draw=Blue] (1,0) arc (0:180:1cm);
    \begin{scope}[even odd rule]
        \clip (0,0) circle (1.6cm) (0,0) ellipse (1.5cm and 0.75cm);      
        \draw[draw=Blue](-1,0) arc (-180:0:1cm);
    \end{scope}

    \ifthenelse{\equal{#1}{+}}{
    \fill[Blue] (0,-0.9) circle (2pt);
    \draw[draw=Maroon, thick,->-] (-1,0) arc (-180:0:1cm and 0.5 cm);

    }{}
    \ifthenelse{\equal{#1}{-}}{
    \fill[Blue] (0,0.9) circle (2pt);
    \draw[draw=Maroon, thick,->-] (-1,0) arc (-180:0:1cm and 0.5 cm);
    
    }{}

\end{scope}
}
\newcommand{\BigonT}[1]{
    \begin{scope}

        \path[name path=EllipseTop] (0,2) ellipse [x radius=1cm, y radius=0.5cm];
        \path[name path=EllipseBottom] (0,-2) ellipse [x radius=1cm, y radius=0.5cm];
        \path[name path=Id](-1.5,0.5)--(1.5,3.5)--(1.5,-0.5)--(-1.5,-3.5)--cycle;
        \path [name intersections={of=EllipseTop and Id, by={A,B}}];
        \path [name intersections={of=EllipseBottom and Id, by={C,D}}];
        \fill[Cyan,fill opacity=0.2] (1,2) arc (0:180: 1 cm and 0.5 cm) -- (-1,-2) arc (-180:0: 1 cm and 0.5 cm) -- cycle;
        \draw[draw=Maroon,thick,fill=Red,fill opacity=0.6] (-1.5,0.5)--(B)--(D)--(-1.5,-3.5)--cycle;

        \begin{scope} [even odd rule] 
            \clip (0,0) circle (6cm) (1,2) arc (0:180: 1 cm and 0.5 cm) -- (-1,-2) arc (-180:0: 1 cm and 0.5 cm) -- cycle;
            \draw[draw=Maroon,thick,fill=Red,fill opacity=0.6] (A)--(1.5,3.5)--(1.5,-0.5)--(C)--cycle;
        \end{scope}

        \begin{scope}[even odd rule]
            \draw[draw=Blue](0,2) ellipse [x radius=1cm, y radius=0.5cm];
            \clip (0,0) circle (6cm)(-1.5,0.5)--(B)--(D)--(-1.5,-3.5)--cycle;
            \draw[draw=Blue](1,2) arc (0:180: 1 cm and 0.5 cm) -- (-1,-2) arc (-180:0: 1 cm and 0.5 cm) -- cycle;
        \end{scope}

        \begin{scope}
            \draw[draw=Blue,dotted](1,-2) arc (0:180:1cm and 0.5 cm);
            \clip (0,0) circle (6cm)(-1.5,0.5)--(B)--(D)--(-1.5,-3.5)--cycle;
            \draw[draw=Blue,dotted](1,2) arc (0:180: 1 cm and 0.5 cm) -- (-1,-2) arc (-180:0: 1 cm and 0.5 cm) -- cycle;
        \end{scope}

        \ifthenelse{\equal{#1}{+}}{\draw[draw=Maroon,thick,->-](B)--(D);}{}
        \ifthenelse{\equal{#1}{-}}{\draw[draw=Maroon,thick,-<-](B)--(D);}{}    
    
    \end{scope}
}

\newcommand{\BigonCupsT}[1]{

\begin{scope}
    \path[name path=EllipseTop] (0,2) ellipse [x radius=1cm, y radius=0.5cm];
    \path[name path=EllipseBottom] (0,-2) ellipse [x radius=1cm, y radius=0.5cm];
    \path[name path=Id](-1.5,0.5)--(1.5,3.5)--(1.5,-0.5)--(-1.5,-3.5)--cycle;
    
    \path [name intersections={of=EllipseTop and Id, by={A,B}}];
    \path [name intersections={of=EllipseBottom and Id, by={C,D}}];

          \fill[Cyan,fill opacity=0.2] (1,2) arc (0:180:1cm and 0.5cm)      -- (-1,2) arc  (-180:0 :1cm)  -- cycle ; 
          \fill[Cyan,fill opacity=0.2]
          (1,-2) arc (0:180:1cm)     
          -- (-1,-2) arc  (-180:0 :1cm and 0.5cm)
          -- cycle ;

          \begin{scope}[even odd rule]
                \clip (0,0) circle (6cm) 
                (0,2) ellipse(0.5cm and 1cm);
                \clip (0,0) circle (6cm) 
                (0,-2) ellipse(0.5cm and 1cm);
                \clip (0,0) circle (6cm) 
                (0,1) arc(-90:0:1cm)--(1,2) arc(0:90:1cm and 0.5cm);
                \clip (0,0) circle (6cm) 
                (1,-2) arc(0:90:1cm)--(0,-2) ;       
                
                \draw[draw=Maroon, thick,fill=Red,fill opacity=0.6](-1.5,0.5)--(1.5,3.5)--(1.5,-0.5)--(-1.5,-3.5)--cycle;
 
          \end{scope}
          \begin{scope}
                \clip (-1.5,0.5)--(0,2)--(0,-2)--(-1.5,-3.5);
                \draw[Maroon, thick] (0,2) ellipse (0.5cm and 1cm);
                \draw[Maroon, thick] (0,-2) ellipse (0.5cm and 1cm);                
           \end{scope}
           \draw[draw=Blue] (0,2) ellipse(1cm and 0.5cm);
           \draw[draw=Blue] (0,1) arc (-90:0:1cm);
           \draw[draw=Blue](1,-2) arc(0:90:1cm);
           \begin{scope}
                \clip (-1.5,0.5)--(0,2)--(0,4)--(-1.5,1);
                \draw[draw=Blue] (-1,2)arc (-180:-90:1cm);

           \end{scope}
           \begin{scope}
                \clip(-1.5,0.5)--(0,2)--(0,-2)--(-1.5,-3.5);
                \draw[draw=Blue,dotted] (-1,2)arc (-180:-90:1cm);

           \end{scope}           
           \begin{scope}
                \clip (0,-2) ellipse ( 0.5cm and 1cm) (0,-4)--(2,-4)--(2,-2)--(0,-2);
                \draw[draw=Blue] (-1,-2) arc(-180:0:1cm and 0.5cm);
           \end{scope}
           \draw[draw=Blue,dotted] (0,-1) arc(90:180:1cm);
           \draw[draw=Blue,dotted] (0,-2) ellipse (1cm and 0.5cm);

\end{scope}

}
\newcommand{\BigonId}{
\begin{scope}
    \fill[Cyan,fill opacity= 0.2]  
    (1,2) arc (0:180: 1cm and 0.5cm) -- (-1,-2) arc (-180:0: 1cm and 0.5cm)--cycle;
    \draw[draw=Blue] (0,2) ellipse ( 1cm and 0.5cm);
    \draw[draw=Maroon,thick,->-] (1,-2) --(1,2);
    \draw[draw=Blue,dotted] (1,-2) arc (0:180: 1cm and 0.5cm);
    \draw[draw=Blue] (-1,-2) arc (-180:0: 1cm and 0.5cm);
    \draw[draw=Maroon,thick,fill=Red, fill opacity=0.6] (-1,2)--(-2,2)--(-2,-2)--(-1,-2)--cycle;
    \draw[draw=Maroon,thick,fill=Red, fill opacity=0.6] (1,2)--(2,2)--(2,-2)--(1,-2);
\end{scope}
}
\newcommand{\BigonCups}[1]{
    \begin{scope}
        \fill[Cyan,fill opacity= 0.2]  
        (1,2) arc (0:180: 1cm and 0.5cm) -- (-1,2) arc (-180:0: 1cm )--cycle;
        \fill[Cyan,fill opacity= 0.2]  
        (1,-2) arc (0:180: 1cm ) -- (-1,-2) arc (-180:0: 1cm and 0.5cm)--cycle;
        \draw[draw=Maroon,fill=Red,fill opacity=0.6,thick]
        (-1,2)arc (-180:0:1cm)--(2,2)--(2,-2)--(1,-2) arc (0:180:1cm)--(-2,-2)--(-2,2)--cycle;
        \draw[Blue] (0,2) ellipse (1cm and 0.5cm);
        \draw[Blue] (-1,-2) arc (-180:0:1cm and 0.5cm);
        \draw[Blue,dotted] (1,-2) arc(0:180:1cm and 0.5cm);
        \draw[Maroon,thick,->-] (-1,2) arc (-180:0:1cm);
        \ifthenelse{\equal{#1}{+}}{
            \fill[Blue] (0,-2) circle (2pt);
        }{
            \fill[Blue] (0,2) circle (2pt);
        }

    \end{scope}
}

\newcommand{\dOne}{
    \begin{scope}
        \path[name path=Ellipse] (0.5,-2) ellipse [x radius=1cm, y radius=0.5cm];
        \path[name path=Id](-1.5,0.5)--(1.5,3.5)--(1.5,-0.5)--(-1.5,-3.5)--cycle;
        \path[name path=Segment] (0,-2)--(2,-2.7);
        \path[name intersections={of=Ellipse and Id, by={A,B}}];
        \path[name intersections={of=Segment and Ellipse, by={C,D}}];
        \path[name path= Arc] (2,2) to[out=-90,in=0] (0,-0.5);
        \path[name intersections={of=Arc and Id, by={E}}];

        \begin{scope}[even odd rule]
            \clip (-2,-4) rectangle (4,4) (0,-1.5) ellipse (0.5cm and 1cm);
            \fill[Cyan, fill opacity=0.2](-1.5,0.5)--(1.5,3.5)--(1.5,-0.5)--(-1.5,-3.5)--cycle;
            \clip (-2,-4) rectangle (4,4) (E)--(A)--(1.6,-0.6);            
            \draw[draw=Blue] (-1.5,0.5)--(1.5,3.5)--(1.5,-0.5)--(-1.5,-3.5)--cycle;
        \end{scope}
        \begin{scope}
            \clip (E)--(A)--(1.6,-0.6);            
            \draw[draw=Blue,dotted] (-1.5,0.5)--(1.5,3.5)--(1.5,-0.5)--(-1.5,-3.5)--cycle;
            
        \end{scope}
        \draw[draw=Blue,fill=Cyan,fill opacity=0.2,-<-] (3,2) ellipse (1cm and 0.5cm);
        \draw[Blue,->-]  (-1.5,0.5)--(1.5,3.5);
        \draw[draw=Blue] (2,2) to[out=-90,in=0] (0,-0.5);
        \draw[draw=Blue] (4,2) to[in=30, out=-90](C);
        \draw[draw=Maroon,thick,dotted,-<-] (A)--(B);
        \begin{scope}[even odd rule]
            \clip (-1.5,0.5)--(1.5,3.5)--(1.5,-0.5)--(-1.5,-3.5)--cycle;
            \fill[Red,fill opacity= 0.6] (0,-1.5) ellipse (0.5cm and 1cm);
            \draw[draw=Maroon,thick,dotted] (0,-2.5) arc (-90:90:0.5cm and 1cm);
            \clip (0,-1.5) ellipse (0.6cm and 1.1cm) (0,-2.5) arc (-90:90:0.6cm and 1.1cm);
            \draw[draw=Maroon,thick](0,-1.5) ellipse (0.5cm and 1cm);

        \end{scope}
        \begin{scope}
            \clip (-1.5,-3.5)--(1.5,-0.5)--(3.5,0.5)--(3.5,-3.5)--cycle;
            \draw[draw=Blue,dotted] (0.5,-2) ellipse (1cm and 0.5cm);
        \end{scope}
        \begin{scope}
            \clip(B)--(C)--(0.5,-3);
            \draw[draw=Blue] (0.5,-2) ellipse (1cm and 0.5cm);
        \end{scope}
        \begin{scope}[even odd rule]
            \clip (-2,-4) rectangle (4,4) (-1.5,0.5)--(1.5,3.5)--(1.5,-0.5)--(-1.5,-3.5)--cycle;
            \fill[Cyan,fill opacity=0.2](0,-0.5) to[out=0,in=-90] (2,2) arc [start angle=-180,end angle=0,x radius=1cm,y radius=0.5cm] to[out=-90,in=30] (C)-- (B);
        \end{scope}
        \begin{scope}
            \clip (B)--(C)--(0.5,-3);
            \fill[Cyan,fill opacity=0.2] (0.5,-2) ellipse [x radius=1cm, y radius=0.5cm];
        \end{scope}
    \end{scope}

}

\newcommand{\dTwo}{
    \begin{scope}
        \path[name path=Ellipse] (0.4,2) ellipse [x radius=1cm, y radius=0.5cm];
        \path[name path=Id](-1.5,0.5)--(1.5,3.5)--(1.5,-0.5)--(-1.5,-3.5)--cycle;
        \path[name intersections={of=Ellipse and Id, by={A,B}}];
        \begin{scope}[even odd rule]
            \clip (-2,-4) rectangle (4,4) (0,2.4) ellipse [x radius=0.5cm, y radius=1cm];
            \draw[draw=Blue,fill=Cyan,fill opacity=0.2](-1.5,0.5)--(1.5,3.5)--(1.5,-0.5)--(-1.5,-3.5)--cycle;
        \end{scope}
        \draw[Maroon,thick,-<-](A)--(B);
        \draw[Blue,->-] (A)--(1.5,3.5);
        \draw[Blue,->-] (-1.5,-3.5)--(1.5,-0.5);
        \begin{scope}
            \clip (-1.5,0.5)--(1.5,3.5)--(1.5,-0.5)--(-1.5,-3.5)--cycle;
            \fill[Red,fill opacity=0.6] (0,2.4) ellipse [x radius=0.5cm, y radius=1cm];
            \clip (0.4,2) ellipse [x radius=1cm, y radius=0.5cm];
            \draw[Maroon,thick] (0,2.4) ellipse [x radius=0.5cm, y radius=1cm];
        \end{scope}
        \begin{scope}
            \clip (-1.5,0.5)--(1.5,3.5)--(1.5,-0.5)--(-1.5,-3.5)--cycle;
            \draw[draw=Blue,postaction={decorate}, decoration={markings, mark=at position 0.15 with {\arrow{<}}}] (0.4,2) ellipse [x radius=1cm, y radius=0.5cm];
            \draw[draw=Blue](B) to[out=-60, in=180] (0.4,1) to[out=0,in=-90] (1.4,2);
        \end{scope}
        \begin{scope}
            \clip (-0.6,2) arc(-180:0:1cm and 0.5cm)--(-1.5,-3.5)--cycle;
            \draw[Maroon,thick,dotted] (0,2.4) ellipse [x radius=0.5cm, y radius=1cm];

        \end{scope}

    \end{scope}
}
\newcommand{\dThree}{
    \begin{scope}
        \fill[Cyan,fill opacity=0.2] (-1.5,0.5)--(1.5,3.5)--(1.5,-0.5)--(-1.5,-3.5)--cycle;
        \fill[Red,fill opacity=0.6,] (0,0) ellipse (0.5cm and 1cm);
        \begin{scope}
            \clip  (0,2) rectangle (-2,-2); 
            \draw[Maroon,thick] (0,0) ellipse (0.5cm and 1cm);
        \end{scope}
        \draw[Maroon,thick,dotted] (0,-1) arc (-90:90:0.5cm and 1cm);
        \draw[draw=Blue,-<-] (3,2) ellipse (1cm and 0.5cm);
        \draw[draw=Blue] (2,2) to[out=-90,in=0] (0,1);
        \draw[draw=Blue] (4,2) to[in=0, out=-90](0,-1);
        \draw[draw=Blue,->-] (-1.5,0.5)--(1.5,3.5);
        \draw[draw=Blue,->] (-1.5,-3.5)--(0,-2);
        \begin{scope}[even odd rule]
            \clip (-2,4) rectangle (4,-4) (-1.5,0.5)--(1.5,3.5)--(1.5,-0.5)--(-1.5,-3.5)--cycle;
        \fill[Cyan,fill opacity=0.2] (0,-1) to[out=0,in=-90] (4,2) arc [start angle=0,end angle=180,x radius=1cm,y radius=0.5cm] to[out=-90,in=0] (0,1);
        \end{scope}
        \begin{scope}[even odd rule]
            \clip (-2,4) rectangle (4,-4) (0,-1) to[out=0,in=-90] (4,2) arc [start angle=0,end angle=180,x radius=1cm,y radius=0.5cm] to[out=-90,in=0] (0,1);
            \draw[draw=Blue]  (-1.5,0.5)--(1.5,3.5)--(1.5,-0.5)--(-1.5,-3.5)--cycle;
        \end{scope}
        \begin{scope}
            \clip  (0,-1) to[out=0,in=-90] (4,2) arc [start angle=0,end angle=180,x radius=1cm,y radius=0.5cm] to[out=-90,in=0] (0,1);
            \draw[draw=Blue,dotted]  (-1.5,0.5)--(1.5,3.5)--(1.5,-0.5)--(-1.5,-3.5)--cycle;
        \end{scope}
        \draw[gray,dashed] (-2,0)--(4,0);
    \end{scope}
}

\newcommand{\SaddleDown}{

\coordinate (o) at (-135:0);
\coordinate (op) at (45:0.6);
\coordinate (a) at ($(-135:1)+ (-0.2,0.3)$);
\coordinate (b) at ($(-135:1)+ (0.5,0)$);
\coordinate (c) at ($(45:1)+ (-0.1,0.3)$);
\coordinate (d) at ($(45:1)+ (0.5,0)$);
\coordinate (A) at ($(a)+(0,2)$);
\coordinate (B) at ($(b)+(0,2)$);
\coordinate (C) at ($(c)+(0,2)$);
\coordinate (D) at ($(d)+(0,2)$);

\node[gray] at ($(a)+(-1,0)$) {s(c)=1};
\node[gray] at ($(A)+(-1,0.4)$) {s(c)=0};
\draw[<-,gray] ($(a)+(-1,0.4)$) -- ($(A)+(-1,0)$);

\draw[Blue] (a)--(A);
\draw[Blue] (b)--(B);
\draw[Blue] (d)--(D);
\draw[Blue,bend left] (B)to(D);
\draw[Blue,bend right,looseness=0.7] (A)to(C);
\draw[Blue, bend right](b)to(o);
\draw[Blue, bend right](op)to(d);
\draw[Maroon,thick] (o)--(op);
\draw[Maroon,thick,fill=Red,fill opacity=0.6,looseness=2.5] (o) to[in=90,out=90](op);
\begin{scope}
    \clip (B)to[bend left](D)--(d)--(b)--cycle;
    \draw[Blue,dotted] (C)--(c);
\end{scope}
\begin{scope}
    \clip (B)to[bend left](D)--(C)--(A)--cycle;
    \draw[Blue] (C)--(c);
\end{scope}
\draw[Blue,dotted,bend left](op)to(c);
\begin{scope}
    \clip (b)--(d)--(D)--(B)--(b);
    \draw[Blue,dotted,bend left] (a)to(o);
\end{scope}
\begin{scope}
    \clip (b)--(a)--(A)--(B)--(b);
    \draw[Blue,bend left] (a)to(o);
\end{scope}

\fill[Cyan,fill opacity=0.2] (A)to[bend right,looseness=0.7](C)--(c)to[bend right](op)to[in=90,out=90,looseness=2.5](o)to[bend right](a)--cycle;
\fill[Cyan,fill opacity=0.2] (B)to[bend left](D)--(d)to[bend left](op)to[in=90,out=90,looseness=2.5](o)to[bend left](b)--cycle;

}

\newcommand{\saddleDown}{

\coordinate (o) at (-135:0);
\coordinate (op) at (45:0.6);
\coordinate (a) at ($(-135:1)+ (-0.2,0.3)$);
\coordinate (b) at ($(-135:1)+ (0.5,0)$);
\coordinate (c) at ($(45:1)+ (-0.1,0.3)$);
\coordinate (d) at ($(45:1)+ (0.5,0)$);
\coordinate (A) at ($(a)+(0,2)$);
\coordinate (B) at ($(b)+(0,2)$);
\coordinate (C) at ($(c)+(0,2)$);
\coordinate (D) at ($(d)+(0,2)$);

\draw[Blue] (a)--(A);
\draw[Blue] (b)--(B);
\draw[Blue] (d)--(D);
\draw[Blue,bend left] (B)to(D);
\draw[Blue,bend right,looseness=0.7] (A)to(C);
\draw[Blue, bend right](b)to(o);
\draw[Blue, bend right](op)to(d);
\draw[Maroon,thick] (o)--(op);
\draw[Maroon,thick,fill=Red,fill opacity=0.6,looseness=2.5] (o) to[in=90,out=90](op);
\begin{scope}
    \clip (B)to[bend left](D)--(d)--(b)--cycle;
    \draw[Blue,dotted] (C)--(c);
\end{scope}
\begin{scope}
    \clip (B)to[bend left](D)--(C)--(A)--cycle;
    \draw[Blue] (C)--(c);
\end{scope}
\draw[Blue,dotted,bend left](op)to(c);
\begin{scope}
    \clip (b)--(d)--(D)--(B)--(b);
    \draw[Blue,dotted,bend left] (a)to(o);
\end{scope}
\begin{scope}
    \clip (b)--(a)--(A)--(B)--(b);
    \draw[Blue,bend left] (a)to(o);
\end{scope}

\fill[Cyan,fill opacity=0.2] (A)to[bend right,looseness=0.7](C)--(c)to[bend right](op)to[in=90,out=90,looseness=2.5](o)to[bend right](a)--cycle;
\fill[Cyan,fill opacity=0.2] (B)to[bend left](D)--(d)to[bend left](op)to[in=90,out=90,looseness=2.5](o)to[bend left](b)--cycle;

}

\newcommand{\SaddleUp}{

\coordinate (o) at (-135:0);
\coordinate (op) at (45:0.6);
\coordinate (a) at ($(-135:1)+ (-0.2,0.3)$);
\coordinate (b) at ($(-135:1)+ (0.5,0)$);
\coordinate (c) at ($(45:1)+ (-0.1,0.3)$);
\coordinate (d) at ($(45:1)+ (0.5,0)$);
\coordinate (A) at ($(a)+(0,-2)$);
\coordinate (B) at ($(b)+(0,-2)$);
\coordinate (C) at ($(c)+(0,-2)$);
\coordinate (D) at ($(d)+(0,-2)$);

\node[gray] at ($(a)+(-1,0.4)$) {s(c)=-1};
\node[gray] at ($(A)+(-1,0)$) {s(c)=0};
\draw[->,gray] ($(a)+(-1,0)$) -- ($(A)+(-1,0.4)$);

\draw[Blue] (a)--(A);
\draw[Blue] (b)--(B);
\draw[Blue] (d)--(D);
\draw[Blue,bend left] (B)to(D);
\draw[Blue, bend right](b)to(o);
\draw[Blue, bend right](op)to(d);
\draw[Maroon,thick] (o)--(op);
\draw[Maroon,thick,fill=Red,fill opacity=0.6,looseness=2.5] (o) to[in=-90,out=-90](op);
\draw[Blue,bend left] (op)to(c);
\draw[Blue,bend left] (a) to (o);

\fill[Cyan,fill opacity=0.2] (A)to[bend right,looseness=0.7](C)--(c)to[bend right](op)to[in=-90,out=-90,looseness=2.5](o)to[bend right](a)--cycle;
\fill[Cyan,fill opacity=0.2] (B)to[bend left](D)--(d)to[bend left](op)to[in=-90,out=-90,looseness=2.5](o)to[bend left](b)--cycle;

\begin{scope}
    \clip (B)to[bend left](D)--(d)to[bend left](op)to[in=-90,out=-90,looseness=2.5](o)to[bend left](b)--cycle;
    \draw[Blue,bend right,looseness=0.7,dotted] (A)to(C);
    \draw[Blue,dotted] (c)--(C);
\end{scope}
\begin{scope}
    \clip (op) to[bend right] (d)--(c)--cycle;
    \draw[Blue] (c)--(C);
\end{scope}
\begin{scope}
    \clip (a)--(A)--(B)--(b)--cycle;
    \draw[Blue,bend right,looseness=0.7] (A)to(C);
\end{scope}
}

\newcommand{\saddleUp}{

\coordinate (o) at (-135:0);
\coordinate (op) at (45:0.6);
\coordinate (a) at ($(-135:1)+ (-0.2,0.3)$);
\coordinate (b) at ($(-135:1)+ (0.5,0)$);
\coordinate (c) at ($(45:1)+ (-0.1,0.3)$);
\coordinate (d) at ($(45:1)+ (0.5,0)$);
\coordinate (A) at ($(a)+(0,-2)$);
\coordinate (B) at ($(b)+(0,-2)$);
\coordinate (C) at ($(c)+(0,-2)$);
\coordinate (D) at ($(d)+(0,-2)$);

\draw[Blue] (a)--(A);
\draw[Blue] (b)--(B);
\draw[Blue] (d)--(D);
\draw[Blue,bend left] (B)to(D);
\draw[Blue, bend right](b)to(o);
\draw[Blue, bend right](op)to(d);
\draw[Maroon,thick] (o)--(op);
\draw[Maroon,thick,fill=Red,fill opacity=0.6,looseness=2.5] (o) to[in=-90,out=-90](op);
\draw[Blue,bend left] (op)to(c);
\draw[Blue,bend left] (a) to (o);

\fill[Cyan,fill opacity=0.2] (A)to[bend right,looseness=0.7](C)--(c)to[bend right](op)to[in=-90,out=-90,looseness=2.5](o)to[bend right](a)--cycle;
\fill[Cyan,fill opacity=0.2] (B)to[bend left](D)--(d)to[bend left](op)to[in=-90,out=-90,looseness=2.5](o)to[bend left](b)--cycle;

\begin{scope}
    \clip (B)to[bend left](D)--(d)to[bend left](op)to[in=-90,out=-90,looseness=2.5](o)to[bend left](b)--cycle;
    \draw[Blue,bend right,looseness=0.7,dotted] (A)to(C);
    \draw[Blue,dotted] (c)--(C);
\end{scope}
\begin{scope}
    \clip (op) to[bend right] (d)--(c)--cycle;
    \draw[Blue] (c)--(C);
\end{scope}
\begin{scope}
    \clip (a)--(A)--(B)--(b)--cycle;
    \draw[Blue,bend right,looseness=0.7] (A)to(C);
\end{scope}
}

\newcommand{\ROnePlusCommutativity}{
    \begin{scope}
        \path[name path=Ellipse] (0.5,-2) ellipse [x radius=1cm, y radius=0.5cm];
        \path[name path=Id](-1.5,0.5)--(1.5,3.5)--(1.5,-0.5)--(-1.5,-3.5)--cycle;
        \path[name path=Segment] (0,-2)--(2,-2.7);
        \path[name intersections={of=Ellipse and Id, by={A,B}}];
        \path[name intersections={of=Segment and Ellipse, by={C,D}}];
        \path[name path= Arc] (2,1) to[out=-90,in=0] (0,-0.5);
        \path[name intersections={of=Arc and Id, by={E}}];

        \begin{scope}[even odd rule]
            \clip (-2,-4) rectangle (4,4) (0,-1.5) ellipse (0.5cm and 1cm);
            \fill[Cyan, fill opacity=0.2](-1.5,0.5)--(1.5,3.5)--(1.5,-0.5)--(-1.5,-3.5)--cycle;
            \clip (-2,-4) rectangle (4,4) (E)--(A)--(1.6,-0.6);            
            \draw[draw=Blue] (-1.5,0.5)--(1.5,3.5)--(1.5,-0.5)--(-1.5,-3.5)--cycle;
        \end{scope}
        \begin{scope}
            \clip (E)--(A)--(1.6,-0.6);            
            \draw[draw=Blue,dotted] (-1.5,0.5)--(1.5,3.5)--(1.5,-0.5)--(-1.5,-3.5)--cycle;
            
        \end{scope}
        \fill[Cyan,fill opacity=0.2] (3,1) circle (1cm);
        \draw[Maroon,thick,dotted, fill=Red,fill opacity=0.6](3,1) ellipse (1cm and 0.5cm);
        \draw[Maroon,thick,->-](2,1) arc(-180:0:1cm and 0.5cm);
        \draw[Blue,->-]  (-1.5,0.5)--(1.5,3.5);
        \draw[Blue,->-]  (-1.5,-3.5)--(A);

        \draw[draw=Blue] (2,1) to[out=-90,in=0] (0,-0.5);
        \draw[draw=Blue] (4,1) to[in=30, out=-90](C);
        \draw[draw=Maroon,thick,dotted] (A)--(B);
        \begin{scope}[even odd rule]
            \clip (-1.5,0.5)--(1.5,3.5)--(1.5,-0.5)--(-1.5,-3.5)--cycle;
            \fill[Red,fill opacity= 0.6] (0,-1.5) ellipse (0.5cm and 1cm);
            \draw[draw=Maroon,thick,dotted] (0,-2.5) arc (-90:90:0.5cm and 1cm);
            \clip (0,-1.5) ellipse (0.6cm and 1.1cm) (0,-2.5) arc (-90:90:0.6cm and 1.1cm);
            \draw[draw=Maroon,thick](0,-1.5) ellipse (0.5cm and 1cm);

        \end{scope}
        \begin{scope}
            \clip (-1.5,-3.5)--(1.5,-0.5)--(3.5,0.5)--(3.5,-3.5)--cycle;
            \draw[draw=Blue,dotted] (0.5,-2) ellipse (1cm and 0.5cm);
        \end{scope}
        \begin{scope}
            \clip(B)--(C)--(0.5,-3);
            \draw[draw=Blue] (0.5,-2) ellipse (1cm and 0.5cm);
        \end{scope}
        \begin{scope}[even odd rule]
            \clip (-2,-4) rectangle (4,4) (-1.5,0.5)--(1.5,3.5)--(1.5,-0.5)--(-1.5,-3.5)--cycle;
            \fill[Cyan,fill opacity=0.2](0,-0.5) to[out=0,in=-90] (2,1) arc [start angle=-180,end angle=0,x radius=1cm,y radius=1cm] to[out=-90,in=30] (C)-- (B);
        \end{scope}
        \begin{scope}
            \clip (B)--(C)--(0.5,-3);
            \fill[Cyan,fill opacity=0.2] (0.5,-2) ellipse [x radius=1cm, y radius=0.5cm];
        \end{scope}
        \draw[Blue] (2,1) arc(180:0:1cm);
    \end{scope}

}

\newcommand{\dFour}{
    \begin{scope}
        \path[name path=Ellipse] (0.4,-2) ellipse [x radius=1cm, y radius=0.5cm];
        \path[name path=Id](-1.5,0.5)--(1.5,3.5)--(1.5,-0.5)--(-1.5,-3.5)--cycle;
        \path[name intersections={of=Ellipse and Id, by={A,B}}];
        \path[name path=Arc](1.4,-2)arc(0:180:1cm);
        \path[name intersections={of= Arc and Id,by={C}}];
        \path[name path=Segment](0,-2)--(0,2);
        \path[name intersections={of= Arc and Segment,by={D}}];

        \begin{scope}
            \clip(0,0) rectangle (1.7,-3);
            \draw[Blue](1.4,-2)arc(0:180:1cm);
        \end{scope}
        \draw[Blue](C)--(1.5,-0.5)--(1.5,3.5);
        \draw[Blue,->-] (-1.5,0.5)--(1.5,3.5);
        \draw[Blue,-<-] (B)--(-1.5,-3.5);
        \draw[Blue](-1.5,-3.5)--(-1.5,0.5);
        \draw[Maroon,thick] (B)to[in=-170,out=135](D);
        \begin{scope}
            \clip (1.5,-0.5)--(-1.5,-3.5)--(1.5,-3.5)--cycle;
            \draw[Blue] (1.4,-2) arc(0:-180:1cm and 0.5cm);
        \end{scope}
        \draw[Maroon,thick,dotted](A)--(B);
        \draw[Maroon,thick,dotted](A)to[in=45,out=90](D);
        \begin{scope}
            \clip (0,-2)--(2,-2)--(1.5,-0.5)--cycle;
            \draw[Blue,dotted](1.4,-2) arc(0:180:1cm and 0.5cm);
        \end{scope}
        \draw[Blue,dotted](A)--(C);
        \fill[Red,fill opacity=0.6](A)--(B)to[in=-170,out=135](D)to[out=45,in=90](A);
        \begin{scope}[even odd rule]
            \clip (-1.5,0.5)--(1.5,3.5)--(1.5,-0.5)--(-1.5,-3.5)--cycle (A)--(B)to[in=-170,out=135](D)to[out=45,in=90](A);
            \fill[Cyan,fill opacity=0.2] (-1.5,0.5)--(1.5,3.5)--(1.5,-0.5)--(-1.5,-3.5)--cycle;
        \end{scope}
        \begin{scope}
            \clip (-1.5,-3.5)--(2,-3.5)--(1.5,-0.5)--cycle;
            \fill[Cyan,fill opacity=0.2](1.4,-2) arc(0:-180:1cm and 0.5cm);
            \fill[Cyan,fill opacity=0.2] (1.4,-2)arc(0:180:1cm);
        \end{scope}
    \end{scope}
}

\newcommand{\dFive}{
   \begin{scope}
        \path[name path=Ellipse] (0.4,2) ellipse [x radius=1cm, y radius=0.5cm];
        \path[name path=Id](-1.5,0.5)--(1.5,3.5)--(1.5,-0.5)--(-1.5,-3.5)--cycle;
        \path[name intersections={of=Ellipse and Id, by={A,B}}];
        \path[name path=ArcOne](1.4,2)to[in=90,out=-70](4,-2);
        \path[name path=ArcTwo](2,-2)to[out=90,in=-90](B);
        \path[name path=Segment] (-1.5,-3.5)--(1.5,-0.5)--(1.5,0);
        \path[name intersections={of=ArcOne and Id, by={C}}];
        \path[name intersections={of=ArcTwo and Segment, by={D}}];
        \begin{scope}[even odd rule]
            \clip (-2,-4) rectangle (4,4) (0,2.4) ellipse [x radius=0.5cm, y radius=1cm];
            \fill[Cyan,fill opacity=0.2](-1.5,0.5)--(1.5,3.5)--(1.5,-0.5)--(-1.5,-3.5)--cycle;
        \end{scope}
        \draw[Maroon,thick,-<-](A)--(B);
        \draw[Blue,->-] (A)--(1.5,3.5);
        \draw[Blue,->-] (-1.5,-3.5)--(D);
        \draw[Blue,dotted] (D)--(1.5,-0.5)--(C);
        \draw[Blue] (B)--(-1.5,0.5)--(-1.5,-3.5);
        \draw[Blue] (1.5,3.5)--(C);
        \draw[Blue,dotted](D)--(1.5,-0.5);
        \begin{scope}
            \clip (-1.5,0.5)--(1.5,3.5)--(1.5,-0.5)--(-1.5,-3.5)--cycle;
            \fill[Red,fill opacity=0.6] (0,2.4) ellipse [x radius=0.5cm, y radius=1cm];
            \clip (0.4,2) ellipse [x radius=1cm, y radius=0.5cm];
            \draw[Maroon,thick] (0,2.4) ellipse [x radius=0.5cm, y radius=1cm];
        \end{scope}
     
        \begin{scope}
            \clip (-0.6,2) arc(-180:0:1cm and 0.5cm)--(-1.5,-3.5)--cycle;
            \draw[Maroon,thick,dotted] (0,2.4) ellipse [x radius=0.5cm, y radius=1cm];
        \end{scope}
        \begin{scope}
            \clip (-1.5,0.5)--(1.5,3.5)--(1.5,-0.5)--(-1.5,-3.5)--cycle;
            \draw[Blue] (0.4,2) ellipse [x radius=1cm, y radius=0.5cm];
        \end{scope}
        \draw[Blue,dotted](4,-2) arc (0:180:1cm and 0.5cm);
        \draw[Blue,->-](4,-2) arc (0:-180:1cm and 0.5cm);
        \draw[Blue](1.4,2)to[in=90,out=-70](4,-2);
        \draw[Blue](B)to[in=90,out=-90](2,-2);
        \begin{scope}
            \clip (1.5,3.5)--(4.5,3.5)--(4.5,-3.5)--(-1.5,-3.5)--(1.5,-0.5)--cycle;
            \fill[Cyan, fill opacity=0.2](1.4,2)to[in=90,out=-70](4,-2) arc (0:-180:1cm and 0.5cm)--(2,-2)to[out=90,in=-90](B);

        \end{scope}
    \end{scope}
}

\newcommand{\dSix}{
   \begin{scope}
        \path[name path=Ellipse] (0.4,2) ellipse [x radius=1cm, y radius=0.5cm];
        \path[name path=Id](-1.5,0.5)--(1.5,3.5)--(1.5,-0.5)--(-1.5,-3.5)--cycle;
        \path[name intersections={of=Ellipse and Id, by={A,B}}];
        \path[name path=ArcOne](1.4,2)to[in=90,out=-70](4,-2);
        \path[name path=ArcTwo](2,-2)to[out=90,in=-90](B);
        \path[name path=Segment] (1.5,3.5)--(1.5,-1.5);
        \path[name intersections={of=ArcOne and Segment, by={C}}];
        \path[name intersections={of=ArcTwo and Segment, by={D}}];
        \begin{scope}[even odd rule]
            \clip (-2,-5) rectangle (4,4) (0,2.4) ellipse [x radius=0.5cm, y radius=1cm];
            \fill[Cyan,fill opacity=0.2](-1.5,0.5)--(1.5,3.5)--(1.5,-1.5)--(-1.5,-4.5)--cycle;
        \end{scope}
        \draw[Maroon,thick,-<-](A)--(B);
        \draw[Blue,->-] (A)--(1.5,3.5);
        \draw[Blue,->-] (-1.5,-4.5)--(1.5,-1.5);
        \draw[Blue,dotted] (D)--(C);
        \draw[Blue] (B)--(-1.5,0.5)--(-1.5,-4.5);
        \draw[Blue](D)--(1.5,-1.5);
        \draw[Blue] (1.5,3.5)--(C);
        \begin{scope}
            \clip (-1.5,0.5)--(1.5,3.5)--(1.5,-0.5)--(-1.5,-3.5)--cycle;
            \fill[Red,fill opacity=0.6] (0,2.4) ellipse [x radius=0.5cm, y radius=1cm];
            \clip (0.4,2) ellipse [x radius=1cm, y radius=0.5cm];
            \draw[Maroon,thick] (0,2.4) ellipse [x radius=0.5cm, y radius=1cm];
        \end{scope}
     
        \begin{scope}
            \clip (-0.6,2) arc(-180:0:1cm and 0.5cm)--(-1.5,-3.5)--cycle;
            \draw[Maroon,thick,dotted] (0,2.4) ellipse [x radius=0.5cm, y radius=1cm];
        \end{scope}
        \begin{scope}
            \clip (-1.5,0.5)--(1.5,3.5)--(1.5,-0.5)--(-1.5,-3.5)--cycle;
            \draw[Blue] (0.4,2) ellipse [x radius=1cm, y radius=0.5cm];
        \end{scope}
        \draw[Maroon,thick,dotted](4,-2) arc (0:180:1cm and 0.5cm);
        \draw[Maroon,thick,-<-](4,-2) arc (0:-180:1cm and 0.5cm);
        \fill[Red,fill opacity=0.6] (3,-2) ellipse (1cm and 0.5cm);
        \draw[Blue](1.4,2)to[in=90,out=-70](4,-2);
        \draw[Blue](B)to[in=90,out=-90](2,-2);
        \begin{scope}
            \clip (1.5,3.5)--(4.5,3.5)--(4.5,-4.5)--(-1.5,-4.5)--(1.5,-1.5)--cycle;
            \fill[Cyan, fill opacity=0.2](1.4,2)to[in=90,out=-70](4,-2) arc (0:-180:1cm)--(2,-2)to[out=90,in=-90](B);

        \end{scope}
        \draw[Blue](4,-2) arc (0:-180:1cm);
    \end{scope}

}

\newcommand{\dSeven}{

    \begin{scope}
        \fill[Cyan,fill opacity=0.2] (-1.5,0.5)--(1.5,3.5)--(1.5,-0.5)--(-1.5,-3.5)--cycle;
        \fill[Red,fill opacity=0.6,] (0,0) ellipse (0.5cm and 1cm);
        \begin{scope}
            \clip  (0,2) rectangle (-2,-2); 
            \draw[Maroon,thick] (0,0) ellipse (0.5cm and 1cm);
        \end{scope}
        \draw[Maroon,thick,dotted] (0,-1) arc (-90:90:0.5cm and 1cm);
        \draw[draw=Blue,dotted] (4,-2) arc (0:180:1cm and 0.5cm);
        \draw[draw=Blue,->-] (4,-2) arc (0:-180:1cm and 0.5cm);
        \draw[draw=Blue] (4,-2) to[out=90,in=0] (0,1);
        \draw[draw=Blue] (2,-2) to[in=0, out=90](0,-1);
        \draw[draw=Blue,->-] (-1.5,0.5)--(1.5,3.5);
        \draw[draw=Blue,->] (-1.5,-3.5)--(0,-2);
        \begin{scope}[even odd rule]
            \clip (-2,4) rectangle (4,-4) (-1.5,0.5)--(1.5,3.5)--(1.5,-0.5)--(-1.5,-3.5)--cycle;
        \fill[Cyan,fill opacity=0.2] (0,1) to[out=0,in=90] (4,-2) arc [start angle=0,end angle=-180,x radius=1cm,y radius=0.5cm] to[out=90,in=0] (0,-1);
        \end{scope}
        \begin{scope}[even odd rule]
            \clip (-2,4) rectangle (4,-4) (0,1) to[out=0,in=90] (4,-2) arc [start angle=0,end angle=180,x radius=1cm,y radius=0.5cm] to[out=90,in=0] (0,-1);
            \draw[draw=Blue]  (-1.5,0.5)--(1.5,3.5)--(1.5,-0.5)--(-1.5,-3.5)--cycle;
        \end{scope}
        \begin{scope}
            \clip  (0,1) to[out=0,in=90] (4,-2) arc [start angle=0,end angle=180,x radius=1cm,y radius=0.5cm] to[out=90,in=0] (0,-1);
            \draw[draw=Blue,dotted]  (-1.5,0.5)--(1.5,3.5)--(1.5,-0.5)--(-1.5,-3.5)--cycle;
        \end{scope}
        \draw[gray,dashed] (-2,0)--(4,0);
    \end{scope}

}
\newcommand{\Migration}[1]{
    \begin{scope}
        \draw[draw=Blue,fill=Cyan,fill opacity= 0.2]  
        (0,1)--(1,1.5)--(1,-0.5)--(0,-1)--(-1,-2)--(-1,0)--cycle;
        \begin{scope}[even odd rule]
            \clip (-1.6,1.1) rectangle (0.1,-1.1) (0,1)--(1,1.5)--(1,-0.5)--(0,-1)--(-1,-2)--(-1,0)--cycle;
            \draw[draw=Maroon,thick,fill=Red,fill opacity=0.6] (-1.5,1)-- (0,1)--(0,-1)--(-1.5,-1)--cycle;
        \end{scope}
        \draw[Maroon,thick,dotted] (0,-1)--(-1,-1);
        \draw[Maroon,thick] (0,1)--(0,-1);
        \ifthenelse{\equal{#1}{B1}}{\fill[Blue](0.5,0.25) circle (2pt);}{}
        \ifthenelse{\equal{#1}{B2}}{\fill[Blue](-0.5,-0.5) circle (2pt);}{}
        \ifthenelse{\equal{#1}{R}}{\fill[Maroon](-1,0.5) circle (2pt);}{}

    \end{scope}
}

\newcommand{\TubeTwo}{
    \begin{scope}[xshift=-2cm]
    \BId{}{+}{}
    \end{scope}
    \begin{scope}[xshift=2cm]
    \BId{}{-}{}
    \end{scope}
    \draw[Maroon,thick,->-] (2,0) ellipse (0.5cm and 1cm);
    \draw[Maroon,thick,dotted] (-2,-1) arc (-90:90:0.5cm and 1cm);
    \draw[Maroon,thick] (-2,1) arc (90:270:0.5cm and 1cm);  
    \draw[Maroon,thick] (1,1) -- (-2,1);
    \draw[Maroon,thick] (0.5,-1) -- (-2,-1);
    \draw[Maroon,thick,dotted] (2,-1) --(0.5,-1);
    \draw[Maroon,thick,dotted] (2,1) --(1,1);
    \fill[Red,fill opacity=0.6]  (-2,1) arc (90:270:0.5cm and 1cm)--(2,-1) arc (-90:-270: 0.5cm and 1cm)--cycle;
    \fill[Cyan,fill opacity=0.2]  (-2,0) ellipse (0.5cm and 1cm);

}

\newcommand{\TubeThree}{
    \begin{scope}[xshift=-2cm,even odd rule]
        \clip     (-1.9,0.9)--(1.9,3.9)--(1.9,-0.9)--(-1.9,-3.9)--cycle 
        (0,1) arc (90:270:0.5cm and 1cm)--(4,-1) arc (-90:-270: 0.5cm and 1cm)--cycle;
        \RId{}{-}{}{}
    \end{scope}
    \begin{scope}[xshift=2cm,even odd rule]
        \clip     (-1.9,0.9)--(1.9,3.9)--(1.9,-0.9)--(-1.9,-3.9)--cycle 
        (0,0) ellipse (0.5cm and 1cm);
        \BId{}{-}{}
    \end{scope}
    \draw[Maroon,thick,->-] (2,0) ellipse (0.5cm and 1cm);
    \draw[Maroon,thick,dotted] (-2,-1) arc (-90:90:0.5cm and 1cm);
    \draw[Maroon,thick,-<-] (-2,1) arc (90:270:0.5cm and 1cm);  
    \draw[Blue] (1,1) -- (-2,1);
    \draw[Blue] (0.5,-1) -- (-2,-1);
    \draw[Blue,dotted] (2,-1) --(0.5,-1);
    \draw[Blue,dotted] (2,1) --(1,1);
    \fill[Cyan,fill opacity=0.2]  (-2,1) arc (90:270:0.5cm and 1cm)--(2,-1) arc (-90:-270: 0.5cm and 1cm)--cycle;
    \fill[Red,fill opacity=0.6]  (2,0) ellipse (0.5cm and 1cm);

}

\newcommand{\TubeFour}{
    \begin{scope}[xshift=-2cm,even odd rule]
        \clip     (-1.9,0.9)--(1.9,3.9)--(1.9,-0.9)--(-1.9,-3.9)--cycle 
        (0,1) arc (90:270:0.5cm and 1cm)--(4,-1) arc (-90:-270: 0.5cm and 1cm)--cycle;
        \RId{}{+}{}{}
    \end{scope}
    \begin{scope}[xshift=2cm,even odd rule]
        \clip     (-1.9,0.9)--(1.9,3.9)--(1.9,-0.9)--(-1.9,-3.9)--cycle 
        (0,0) ellipse (0.5cm and 1cm);
        \BId{}{+}{}
    \end{scope}
    \draw[Maroon,thick,-<-] (2,0) ellipse (0.5cm and 1cm);
    \draw[Maroon,thick,dotted] (-2,-1) arc (-90:90:0.5cm and 1cm);
    \draw[Maroon,thick,->-] (-2,1) arc (90:270:0.5cm and 1cm);  
    \draw[Blue] (1,1) -- (-2,1);
    \draw[Blue] (0.5,-1) -- (-2,-1);
    \draw[Blue,dotted] (2,-1) --(0.5,-1);
    \draw[Blue,dotted] (2,1) --(1,1);
    \fill[Cyan,fill opacity=0.2]  (-2,1) arc (90:270:0.5cm and 1cm)--(2,-1) arc (-90:-270: 0.5cm and 1cm)--cycle;
    \fill[Red,fill opacity=0.6]  (2,0) ellipse (0.5cm and 1cm);

}

\newcommand{\IOne}{
    \begin{scope}
        \begin{scope}[xshift=-2cm,even odd rule]
            \clip (-1.6,0.6)--(1.6,3.6)--(1.6,-0.6)--(-1.6,-3.6)--cycle
            (-1.5,-1.5) ellipse (0.5cm and 1cm);
            \clip (-1.6,0.6)--(1.6,3.6)--(1.6,-0.6)--(-1.6,-3.6)--cycle
            (1.5,1.5) ellipse (0.5cm and 1cm);
            \RId{}{}{}{}
        \end{scope}
        \draw[Maroon,thick](-0.5,3.5)--(-3.5,0.5);
        \draw[Blue](-3.5,-2.5)--(-3.5,-0.5);
        \draw[Blue,dotted](-0.5,2.5)--(-0.5,0.5);

        \begin{scope}[xshift=2cm,even odd rule]
            \clip (-1.6,0.6)--(1.6,3.6)--(1.6,-0.6)--(-1.6,-3.6)--cycle
            (-1.5,-1.5) ellipse (0.5cm and 1cm);
            \clip (-1.6,0.6)--(1.6,3.6)--(1.6,-0.6)--(-1.6,-3.6)--cycle
            (1.5,1.5) ellipse (0.5cm and 1cm);
            \BId{}{}{}            
        \end{scope}
        \draw[Maroon,thick](0.5,-2.5)--(0.5,-0.5);
        \draw[Maroon,thick,](3.5,2.5)--(3.5,0.5);
        \begin{scope}[xshift=-2cm]
            \clip (-1.5,0.5)--(1.5,3.5)--(1.5,-0.5)--(-1.5,-3.5)--cycle;
            \draw[Blue] (-1.5,-1.5) ellipse (0.5cm and 1cm);
            \clip (-1.5,0.5)--(1.5,3.5)--(1.5,-0.5)--(-1.5,-3.5)--cycle;
            \draw[Blue, fill=Cyan,fill opacity=0.2] (1.5,1.5) ellipse (0.5cm and 1cm);

        \end{scope}
        \begin{scope}[xshift=2cm]
            \clip (-1.5,0.5)--(1.5,3.5)--(1.5,-0.5)--(-1.5,-3.5)--cycle;
            \draw[Maroon,thick,fill=Red,fill opacity=0.6] (-1.5,-1.5) ellipse (0.5cm and 1cm);
            \clip (-1.5,0.5)--(1.5,3.5)--(1.5,-0.5)--(-1.5,-3.5)--cycle;
            \draw[Maroon,thick,fill=Red,fill opacity=0.6] (1.5,1.5) ellipse (0.5cm and 1cm);

        \end{scope}
        \draw[Blue](-3.5,-2.5)--(0.5,-2.5);
        \draw[Blue](-3.5,-0.5)--(0.5,-0.5);
        \draw[Blue](-0.5,2.5)--(2.5,2.5);
        \draw[Blue](-0.5,0.5)--(0.5,0.5);
        \draw[Blue,dotted](2.5,2.5)--(3.5,2.5);
        \draw[Blue,dotted](0.5,0.5)--(3.5,0.5);
        \fill[Cyan,fill opacity=0.2] (-3.5,-2.5)--(-3.5,-0.5)--(0.5,-0.5)--(0.5,-2.5)--cycle;
        \begin{scope}[even odd rule]
            \clip (-2,4)rectangle (4,-4) (3.5,1.5)ellipse (0.5cm and 1cm);
            \fill[Cyan,fill opacity=0.2] (-0.5,2.5)--(3.5,2.5)--(3.5,0.5)--(-0.5,0.5)--cycle;

        \end{scope}

    \end{scope}
}

\newcommand{\ITwo}{
    \begin{scope}
        \begin{scope}[xshift=-2cm,even odd rule]
            \clip (-1.6,0.6)--(1.6,3.6)--(1.6,-0.6)--(-1.6,-3.6)--cycle
            (-1.5,-1.5) ellipse (0.5cm and 1cm);
            \clip (-1.6,0.6)--(1.6,3.6)--(1.6,-0.6)--(-1.6,-3.6)--cycle
            (1.5,1.5) ellipse (0.5cm and 1cm);
            \BId{}{}{}
        \end{scope}
        \draw[Blue,->-](-0.5,3.5)--(-3.5,0.5);

        \draw[Blue](-3.5,-2.5)--(-3.5,-0.5);
        \draw[Blue,dotted](-0.5,2.5)--(-0.5,0.5);

        \begin{scope}[xshift=2cm,even odd rule]
            \clip (-1.6,0.6)--(1.6,3.6)--(1.6,-0.6)--(-1.6,-3.6)--cycle
            (-1.5,-1.5) ellipse (0.5cm and 1cm);
            \clip (-1.6,0.6)--(1.6,3.6)--(1.6,-0.6)--(-1.6,-3.6)--cycle
            (1.5,1.5) ellipse (0.5cm and 1cm);
            \BId{}{}{}            
        \end{scope}
        \draw[Blue](0.5,-2.5)--(0.5,-0.5);
        \draw[Blue](3.5,2.5)--(3.5,0.5);
        \begin{scope}[xshift=-2cm]
            \clip (-1.5,0.5)--(1.5,3.5)--(1.5,-0.5)--(-1.5,-3.5)--cycle;
            \draw[Maroon,thick] (-1.5,-1.5) ellipse (0.5cm and 1cm);
            \fill[Cyan,fill opacity=0.2] (-1.5,-1.5) ellipse (0.5cm and 1cm);
            \clip (-1.5,0.5)--(1.5,3.5)--(1.5,-0.5)--(-1.5,-3.5)--cycle;
            \draw[Maroon,thick,fill=Red,fill opacity=0.6] (1.5,1.5) ellipse (0.5cm and 1cm);
            \fill[Cyan,fill opacity=0.2] (1.5,1.5) ellipse (0.5cm and 1cm);

        \end{scope}
        \begin{scope}[xshift=2cm]
            \clip (-1.5,0.5)--(1.5,3.5)--(1.5,-0.5)--(-1.5,-3.5)--cycle;
            \draw[Maroon,thick,fill=Red,fill opacity=0.6] (-1.5,-1.5) ellipse (0.5cm and 1cm);
            \fill[Cyan,fill opacity=0.2] (-1.5,-1.5) ellipse (0.5cm and 1cm);
            \clip (-1.5,0.5)--(1.5,3.5)--(1.5,-0.5)--(-1.5,-3.5)--cycle;
            \draw[Maroon,thick] (1.5,1.5) ellipse (0.5cm and 1cm);
            \fill[Cyan,fill opacity=0.2] (1.5,1.5) ellipse (0.5cm and 1cm);

        \end{scope}
        \draw[Maroon,thick](-3.5,-2.5)--(0.5,-2.5);
        \draw[Maroon,thick](-3.5,-0.5)--(0.5,-0.5);
        \draw[Maroon,thick](-0.5,2.5)--(2.5,2.5);
        \draw[Maroon,thick](-0.5,0.5)--(0.5,0.5);
        \draw[Maroon,thick,dotted](2.5,2.5)--(3.5,2.5);
        \draw[Maroon,thick,dotted](0.5,0.5)--(3.5,0.5);
        \begin{scope}[even odd rule]
            \clip (-2,4)rectangle (4,-4) (3.5,1.5)ellipse (0.5cm and 1cm);
            \fill[Red,fill opacity=0.6] (-0.5,2.5)--(3.5,2.5)--(3.5,0.5)--(-0.5,0.5)--cycle;
        \end{scope}
        \begin{scope}[even odd rule]
            \clip (-4,4)rectangle (4,-4) (-3.5,-1.5)ellipse (0.5cm and 1cm);
            \fill[Red,fill opacity=0.6] (-3.5,-2.5)--(-3.5,-0.5)--(0.5,-0.5)--(0.5,-2.5)--cycle;

        \end{scope}

    \end{scope}
}

\newcommand{\IdOne}{
\begin{scope}
    \fill[Cyan,fill opacity=0.2](-3.5,-0.5)--(-0.5,2.5)--(-0.5,0.5)--(-3.5,-2.5)--cycle;
    \draw[Maroon,thick,fill=Red,fill opacity=0.6](0.5,-0.5)--(3.5,2.5)--(3.5,0.5)--(0.5,-2.5)--cycle;
    \draw[Maroon,thick,fill=Red,fill opacity=0.6] (-3.5,-0.5)--(-3.5,0.5)--(-0.5,3.5)--(-0.5,2.5);
    \fill[Red,fill opacity=0.6] (-3.5,-2.5)--(-3.5,-3.5)--(-0.5,-0.5)--(-0.5,0.5);
    \draw[Blue,fill=Cyan,fill opacity=0.2] (0.5,-0.5)--(0.5,0.5)--(3.5,3.5)--(3.5,2.5);
    \draw[Blue,fill=Cyan,fill opacity=0.2] (0.5,-2.5)--(0.5,-3.5)--(3.5,-0.5)--(3.5,0.5);
    \fill[Cyan,fill opacity=0.2] (-3.5,-0.5)--(-0.5,2.5)--(3.5,2.5)--(0.5,-0.5);
    \fill[Cyan,fill opacity=0.2] (-3.5,-2.5)--(-0.5,0.5)--(3.5,0.5)--(0.5,-2.5);
    \draw[Blue](-3.5,-2.5)--(0.5,-2.5);
    \draw[Blue](-3.5,-0.5)--(0.5,-0.5);
    \draw[Blue](-0.5,2.5)--(2.5,2.5);
    \draw[Blue,dotted](2.5,2.5)--(3.5,2.5);
    \draw[Blue,dotted](-0.5,0.5)--(3.5,0.5);
    \draw[Blue](-0.5,2.5)--(-3.5,-0.5)--(-3.5,-2.5)--(-1.5,-0.5);
    \draw[Blue,dotted](-1.5,-0.5)--(-0.5,0.5)--(-0.5,2.5);
    \draw[Maroon,thick](-3.5,-2.5)--(-3.5,-3.5)--(-2.5,-2.5);
    \draw[Maroon,thick,dotted](-2.5,-2.5)--(-0.5,-0.5)--(-0.5,0.5);
    \draw[Maroon,thick](-0.5,3.5)--(-3.5,0.5);

\end{scope}
}

\newcommand{\IdTwo}{
\begin{scope}
    \fill[Cyan,fill opacity=0.2](-3.5,-0.5)--(-0.5,2.5)--(-0.5,0.5)--(-3.5,-2.5)--cycle;
    \fill[Cyan,fill opacity=0.2](0.5,-0.5)--(3.5,2.5)--(3.5,0.5)--(0.5,-2.5)--cycle;
    \fill[Cyan,fill opacity=0.2] (-3.5,-0.5)--(-3.5,0.5)--(-0.5,3.5)--(-0.5,2.5);
    \fill[Cyan,fill opacity=0.2] (-3.5,-2.5)--(-3.5,-3.5)--(-0.5,-0.5)--(-0.5,0.5);
    \fill[Cyan,fill opacity=0.2] (0.5,-0.5)--(0.5,0.5)--(3.5,3.5)--(3.5,2.5);
    \fill[Cyan,fill opacity=0.2] (0.5,-2.5)--(0.5,-3.5)--(3.5,-0.5)--(3.5,0.5);
    \fill[Red,fill opacity=0.6] (-3.5,-0.5)--(-0.5,2.5)--(3.5,2.5)--(0.5,-0.5);
    \fill[Red,fill opacity=0.6] (-3.5,-2.5)--(-0.5,0.5)--(3.5,0.5)--(0.5,-2.5);
    \draw[Blue](-0.5,2.5)--(-0.5,3.5)--(-3.5,0.5)--(-3.5,-3.5)--(-2.5,-2.5);
    \draw[Blue,dotted](-2.5,-2.5)--(-0.5,-0.5)--(-0.5,2.5);
    \draw[Blue](0.5,0.5)--(3.5,3.5)--(3.5,-0.5)--(0.5,-3.5)--cycle;
    \draw[Maroon,thick](2.5,2.5)--(-0.5,2.5)--(-3.5,-0.5)--(0.5,-0.5)--(3.5,2.5);
    \draw[Maroon,thick,dotted](3.5,2.5)--(2.5,2.5);
    \draw[Maroon,thick] (-1.5,-0.5)--(-3.5,-2.5)--(0.5,-2.5)--(3.5,0.5);
    \draw[Maroon,thick,dotted](3.5,0.5)--(-0.5,0.5)--(-1.5,-0.5);
    \draw[Blue,->-](-0.5,3.5)--(-3.5,0.5);

\end{scope}
}

\newcommand{\IsomorphismOne}{
    \begin{scope}
        \path[name path=Square] (-2,-2) rectangle (2,0);
        \path[name path= Ellipse] (0,-0.2) ellipse (1cm and 0.5cm);
        \path[name intersections={of=Ellipse and Square, by={A,B}}];
        \draw[Blue] (B)--(-2,0)--(-2,-2)--(2,-2)--(2,0)--(A);
        \fill[Red,fill opacity=0.6] (A)to[in=-90,out=-90,looseness=1.5] (B)--(A);
        \begin{scope}
            \clip (0,-0.2)ellipse (1cm and 0.5cm);
            \draw[Maroon,thick](A)to[in=-90,out=-90,looseness=1.5] (B)--(A);
        \end{scope}
        \begin{scope}[even odd rule]
            \clip (-2,-2) rectangle (2,0)
            (0,-0.2) ellipse (1cm and 0.5cm);
            \draw[Maroon,thick,dotted](A)to[in=-90,out=-90,looseness=1.5](B);
        \end{scope}
        \begin{scope}[even odd rule]
            \clip (-2,-2) rectangle (2,0);
            \draw[Blue] (0,-0.2) ellipse (1cm and 0.5cm);
            \clip (-2,-2) rectangle (2,0)
            (A)to[in=-90,out=-90,looseness=1.5] (B)--(A);
            \fill[Cyan,fill opacity=0.2] (-2,-2) rectangle (2,0);
            \clip (-2,-2) rectangle (2,0)
            (0,-0.2) ellipse (1cm and 0.5cm);
            \draw[Blue,fill=Cyan,fill opacity=0.2] (-1,-0.2) arc (-180:0:1cm);     
        \end{scope}
        \draw[Maroon,thick,-<-](A)--(B);
        \draw[Blue,-<-] (B)--(-2,0);
        \draw[Blue,-<-](2,0)--(A);
        \draw[Blue,-<-](2,-2)--(-2,-2);
        \draw[Blue,-<-](-0.01,-0.7)--(0.01,-0.7);
    \end{scope}
}

\newcommand{\IsomorphismOneSym}{
    \begin{scope}
        \path[name path=Square] (-2,2) rectangle (2,0);
        \path[name path= Ellipse] (0,0.2) ellipse (1cm and 0.5cm);
        \path[name intersections={of=Ellipse and Square, by={A,B}}];
        \draw[Blue] (B)--(-2,0)--(-2,2)--(2,2)--(2,0)--(A);
        \fill[Red,fill opacity=0.6] (A)to[in=90,out=90,looseness=1.5] (B)--(A);
        \begin{scope}
            \clip (0,0.2)ellipse (1cm and 0.5cm);
            \draw[Maroon,thick](A)to[in=90,out=90,looseness=1.5] (B)--(A);
        \end{scope}
        \begin{scope}[even odd rule]
            \clip (-2,2) rectangle (2,0)
            (0,0.2) ellipse (1cm and 0.5cm);
            \draw[Maroon,thick,dotted](A)to[in=90,out=90,looseness=1.5](B);
        \end{scope}
        \begin{scope}[even odd rule]
            \clip (-2,2) rectangle (2,0);
            \draw[Blue] (0,0.2) ellipse (1cm and 0.5cm);
            \clip (-2,2) rectangle (2,0)
            (A)to[in=90,out=90,looseness=1.5] (B)--(A);
            \fill[Cyan,fill opacity=0.2] (-2,2) rectangle (2,0);
            \clip (-2,2) rectangle (2,0)
            (0,0.2) ellipse (1cm and 0.5cm);
            \draw[Blue,fill=Cyan,fill opacity=0.2] (1,0.2) arc (0:180:1cm);           
        \end{scope}
        \draw[Maroon,thick,-<-](B)--(A);
        \draw[Blue,-<-] (A)--(-2,0);
        \draw[Blue,-<-](2,0)--(B);
        \draw[Blue,-<-](2,2)--(-2,2);
        \draw[Blue,-<-](-0.01,0.7)--(0.01,0.7);
    \end{scope}
}

\newcommand{\IsomorphismTwo}{
    \begin{scope}
        \draw[Blue,fill=Cyan,fill opacity=0.2] (-3.5,-0.5)--(-0.5,2.5)--(-0.5,-0.5)--(-3.5,-3.5)--cycle;
        \draw[Blue,fill=Cyan,fill opacity=0.2] (0.5,-0.5)--(3.5,2.5)--(3.5,-0.5)--(0.5,-3.5)--cycle;
        \draw[Maroon,thick](-2.5,0.5)--(1.5,0.5);
        \draw[Maroon,thick,->-](-1.5,1.5)--(2.5,1.5);
        \draw[Maroon,thick](-2,0)--(1,0);
        \draw[Maroon,thick,dotted](1,0)--(2,0);
        \draw[Maroon,thick](-2.5,0.5)to[out=-90,in=-180](-2,0);
        \draw[Maroon,thick](1.5,0.5)to[out=-90,in=-180](2,0);
        \draw[Maroon,thick](2,0)to[out=0,in=-45](2.5,1.5);
        \begin{scope}
            \clip (-2,0.5) rectangle (-1,0);
            \draw[Maroon,thick,dotted] (-2,0)to[out=0,in=-45](-1.5,1.5);
        \end{scope}
        \begin{scope}
            \clip (-2.5,0.5) rectangle (2.5,1.5);
            \draw[Maroon,thick] (-2,0)to[out=0,in=-45](-1.5,1.5);
        \end{scope}
        \fill[Red,fill opacity=0.6]
        (-2.5,0.5)to[out=-90,in=-180](-2,0)--(2,0)to[in=-90,out=-180](1.5,0.5)--cycle;
        \fill[Red,fill opacity=0.6] (2,0)to[out=0,in=-45](2.5,1.5)--(-1.5,1.5) to[in=0,out=-45] (-2,0)--cycle;
        
    \end{scope}
}

\newcommand{\IsomorphismTwoSym}{
    \begin{scope}
        \draw[Blue,fill=Cyan,fill opacity=0.2] (-3.5,-0.5)--(-0.5,2.5)--(-0.5,-0.5)--(-3.5,-3.5)--cycle;
        \draw[Blue,fill=Cyan,fill opacity=0.2] (0.5,-0.5)--(3.5,2.5)--(3.5,-0.5)--(0.5,-3.5)--cycle;
        \draw[Maroon,thick,-<-](-2.5,-2.5)--(0.5,-2.5);
        \draw[Maroon,thick,dotted](0.5,-2.5)--(1.5,-2.5);
        \draw[Maroon,thick,dotted](-1.5,-1.5)--(2.5,-1.5);
        \draw[Maroon,thick](-2,-1)--(0.5,-1);
        \draw[Maroon,thick,dotted](0.5,-1)--(2,-1);

        \draw[Maroon,thick](1.5,-2.5)to[out=135,in=-180](2,-1);
        \draw[Maroon,thick](2.5,-1.5)to[out=90,in=0](2,-1);

        \draw[Maroon,thick](-2.5,-2.5)to[out=135,in=-180](-2,-1);
        \draw[Maroon,thick,dotted](-1.5,-1.5)to[out=90,in=0](-2,-1);

        \fill[Red,fill opacity=0.6] (1.5,-2.5)to[out=135,in=-180](2,-1)--(-2,-1)to[in=135,out=-180](-2.5,-2.5)--cycle;

        \fill[Red,fill opacity=0.6] (2.5,-1.5)to[out=90,in=0](2,-1)--(-2,-1)to[out=0,in=90](-1.5,-1.5)--cycle;

    \end{scope}
}

\newcommand{\IsomorphismThree}{
    \begin{scope}
        \draw[Maroon,thick] (-2.5,0.5)--(-3.5,-0.5)--(-3.5,-3.5)--(-0.5,-0.5)--(-0.5,2.5)--(-1.5,1.5);
        \draw[Blue] (1.5,0.5)--(0.5,-0.5)--(0.5,-3.5)--(3.5,-0.5)--(3.5,2.5)--(2.5,1.5);

        \draw[Blue](-2.5,0.5)--(1.5,0.5);
        \draw[Blue,-<-](-1.5,1.5)--(2.5,1.5);
        \draw[Blue](-2,0)--(1,0);
        \draw[Blue,dotted](1,0)--(2,0);

        \draw[Maroon,thick](-2.5,0.5)to[out=-90,in=-180](-2,0);
        \draw[Maroon,thick](1.5,0.5)to[out=-90,in=-180](2,0);
        \draw[Maroon,thick](2,0)to[out=0,in=-45](2.5,1.5);

        \draw[Maroon,thick](-2.5,0.5)--(-1.5,1.5);
        \draw[Maroon,thick](1.5,0.5)--(2.5,1.5);

        \begin{scope}
            \clip (-2,0.5) rectangle (-1,0);
            \draw[Maroon,thick,dotted] (-2,0)to[out=0,in=-45](-1.5,1.5);
        \end{scope}

        \begin{scope}
            \clip (-2.5,0.5) rectangle (2.5,1.5);
            \draw[Maroon,thick] (-2,0)to[out=0,in=-45](-1.5,1.5);
        \end{scope}

        \fill[Cyan,fill opacity=0.2]
        (-2.5,0.5)to[out=-90,in=-180](-2,0)--(2,0)to[in=-90,out=-180](1.5,0.5)--cycle;

        \fill[Cyan,fill opacity=0.2]
        (2,0)to[out=0,in=-45](2.5,1.5)--(-1.5,1.5) to[in=0,out=-45] (-2,0)--cycle;

        \fill[Red,fill opacity=0.6]
            (1.5,0.5)to[out=-90,in=-180](2,0)to[out=0,in=-45](2.5,1.5)--cycle;
        \fill[Cyan,fill opacity=0.2]
            (-2.5,0.5)to[out=-90,in=-180](-2,0)to[out=0,in=-45] (-1.5,1.5)--cycle;

        \begin{scope}[even odd rule]
            \clip (0.5,-0.5)--(3.5,2.5)--(3.5,-0.5)--(0.5,-3.5)--cycle
            (1.5,0.5)to[out=-90,in=-180](2,0)to[out=0,in=-45](2.5,1.5)--cycle;
            \fill[Cyan,fill opacity=0.2]  (0.5,-0.5)--(3.5,2.5)--(3.5,-0.5)--(0.5,-3.5)--cycle;
        \end{scope}
        \begin{scope}[even odd rule]
            \clip  (-3.5,-0.5)--(-0.5,2.5)--(-0.5,-0.5)--(-3.5,-3.5)--cycle
            (-2.5,0.5)to[out=-90,in=-180](-2,0)to[out=0,in=-45] (-1.5,1.5)--cycle;
            \fill[Red,fill opacity=0.6] (-3.5,-0.5)--(-0.5,2.5)--(-0.5,-0.5)--(-3.5,-3.5)--cycle;
        \end{scope}
    \end{scope}
}

\newcommand{\IsomorphismThreeSym}{
    \begin{scope}
        \draw[Maroon,thick] (-1.5,-1.5)--(-0.5,-0.5)--(-0.5,2.5)--(-3.5,-0.5)--(-3.5,-3.5)--(-2.5,-2.5);
        \draw[Blue] (2.5,-1.5)--(3.5,-0.5)--(3.5,2.5)--(0.5,-0.5)--(0.5,-3.5)--(1.5,-2.5);
        \draw[Maroon,thick](1.5,-2.5)--(2.5,-1.5);
        \draw[Blue,dotted](-2.5,-2.5)--(-1.5,-1.5);
        \draw[Blue,dotted](-1.5,-1.5)--(2.5,-1.5);
        \draw[Blue,dotted](0.5,-2.5)--(1.5,-2.5);
        \draw[Blue,->-](-2.5,-2.5)--(0.5,-2.5);
        \draw[Blue](-2,-1)--(0.5,-1);
        \draw[Blue,dotted](0.5,-1)--(2,-1);
        \draw[Maroon,thick](1.5,-2.5)to[out=135,in=-180](2,-1);
        \draw[Maroon,thick](2.5,-1.5)to[out=90,in=0](2,-1);

        \draw[Blue](-2.5,-2.5)to[out=135,in=-180](-2,-1);
        \draw[Blue,dotted](-1.5,-1.5)to[out=90,in=0](-2,-1);

        \fill[Cyan,fill opacity=0.2]
        (2.5,-1.5)to[out=90,in=0](2,-1)--(-2,-1)to[out=0,in=90](-1.5,-1.5)--cycle;

        \fill[Cyan,fill opacity=0.2]
        (1.5,-2.5)to[out=135,in=-180](2,-1)--(-2,-1)to[in=135,out=-180](-2.5,-2.5)--cycle;

        \fill[Red,fill opacity=0.6]
            (2.5,-1.5)to[out=90,in=0](2,-1)to[in=135,out=-180](1.5,-2.5)--cycle;
        \fill[Cyan,fill opacity=0.2]
            (-2.5,-2.5)to[out=135,in=-180](-2,-1)to[out=0,in=90](-1.5,-1.5)--cycle;

        \begin{scope}[even odd rule]
            \clip (0.5,-0.5)--(3.5,2.5)--(3.5,-0.5)--(0.5,-3.5)--cycle
            (2.5,-1.5)to[out=90,in=0](2,-1)to[in=135,out=-180](1.5,-2.5)--cycle;
            \fill[Cyan,fill opacity=0.2]  (0.5,-0.5)--(3.5,2.5)--(3.5,-0.5)--(0.5,-3.5)--cycle;
        \end{scope}

        \begin{scope}[even odd rule]
            \clip  (-3.5,-0.5)--(-0.5,2.5)--(-0.5,-0.5)--(-3.5,-3.5)--cycle
            (-2.5,-2.5)to[out=135,in=-180](-2,-1)to[out=0,in=90](-1.5,-1.5)--cycle;
            \fill[Red,fill opacity=0.6] (-3.5,-0.5)--(-0.5,2.5)--(-0.5,-0.5)--(-3.5,-3.5)--cycle;
        \end{scope}

    \end{scope}
}
\newcommand{\BandOne}{
    \begin{scope}
        
        \path[name path=Id] (-1.5,0.5)--(1.5,3.5)--(1.5,-0.5)--(-1.5,-3.5)--cycle;
        \path[name path=ER] (1.5,1.5) ellipse (0.5cm and 1cm);
        \path[name path=ERD] (1.25,1.75) ellipse (0.5cm and 1cm);
        \path[name path=EL] (-1.25,-1.5) ellipse (0.5cm and 1cm);
        \path[name intersections={of=Id and EL, by={A,B}}];
        \path[name intersections={of=Id and ERD, by={C,D}}];
        \draw[draw=Blue,fill=Cyan,fill opacity=0.2] (A)--(-1.5,0.5)--(1.5,3.5)--(C);
        \draw[draw=Blue,fill=Cyan,fill opacity=0.2] (B)--(-1.5,-3.5)--(1.5,-0.5)--(D);
        \path[name path=ST] (A)--(C);
        \path[name path=SB] (B)--(D);
        \path[name intersections={of=ST and EL, by=E}];
        \path[name intersections={of=SB and EL, by={FU,F}}];
        \path[name path=B] (-0.9,-2.5)--(-1.25,-1.5);
        \path[name path=T] (1.75,0.5)--(1.25,2.5);
        \path[name intersections={of=T and ER, by={J,G}}];
        \path[name intersections={of=B and EL, by=H}];
        \fill[Red, fill opacity=0.6] (A)--(C)--(D)--(B)--cycle;
        \fill[Cyan,fill opacity=0.2] (1.5,1.5) ellipse (0.5cm and 1cm);
        \begin{scope}[even odd rule]
            \clip (-1.6,0.6)--(2.5,4.5)--(2.5,-0.5)--(-1.6,-3.6)
            (1.5,1.5) ellipse (0.5cm and 1cm);
            \clip (-1.6,0.6)--(2.5,4.5)--(2.5,-0.5)--(-1.6,-3.6)
            (-1.25,-1.5) ellipse (0.5cm and 1cm);
            \draw[draw=Blue,fill=Cyan,fill opacity=0.2]  (A)--(J)--(G)--(H);
        \end{scope}
        \begin{scope}
            \clip (-1.5,0.5)--(1.5,3.5)--(1.5,-0.5)--(-1.5,-3.5)--cycle;
            \draw[Blue] (-1.25,-1.5) ellipse (0.5cm and 1cm);
        \end{scope}
        \begin{scope}[even odd rule]
            \clip (-1.6,0.6)--(2.5,4.5)--(2.5,-0.5)--(-1.6,-3.6)
            (A)--(J)--(G)--(H);
            \draw[draw=Blue] (1.5,1.5) ellipse (0.5cm and 1cm);
        \end{scope}
        \draw[Maroon,thick](A)--(B);
        \draw[Maroon,thick](B)--(F);
        \draw[Maroon,thick,dotted](C)--(D);
        \draw[Maroon,thick,dotted,->-](B)--(D);
        \draw[Blue,dotted](G)to[out=-180,in=-90](D);

    \end{scope} 
    }
\newcommand{\BandOneCups}{
    \begin{scope}
        
        \path[name path=Id] (-1.5,0.5)--(1.5,3.5)--(1.5,-0.5)--(-1.5,-3.5)--cycle;
        \path[name path=ER] (1.5,1.5) ellipse (0.5cm and 1cm);
        \path[name path=ERD] (1.25,1.75) ellipse (0.5cm and 1cm);
        \path[name path=EL] (-1.25,-1.5) ellipse (0.5cm and 1cm);
        \path[name intersections={of=Id and EL, by={A,B}}];
        \path[name intersections={of=Id and ERD, by={C,D}}];
        \draw[draw=Blue] (A)--(-1.5,0.5)--(1.5,3.5)--(C);
        \draw[draw=Blue] (B)--(-1.5,-3.5)--(1.5,-0.5)--(D);
        \fill[Cyan, fill opacity=0.2] (-1.5,0.5)--(1.5,3.5)--(1.5,-0.5)--(-1.5,-3.5)--cycle;
        \path[name path=ST] (A)--(C);
        \path[name path=SB] (B)--(D);
        \path[name intersections={of=ST and EL, by=E}];
        \path[name intersections={of=SB and EL, by={FU,F}}];
        \path[name path=B] (-0.9,-2.5)--(-1.25,-1.5);
        \path[name path=T] (1.75,0.5)--(1.25,2.5);
        \path[name intersections={of=T and ER, by={J,G}}];
        \path[name intersections={of=B and EL, by=H}];
        \begin{scope}[even odd rule]
            \clip (-1.6,0.6)--(2.5,4.5)--(2.5,-0.5)--(-1.6,-3.6)
            (1.5,1.5) ellipse (0.5cm and 1cm);
            \clip (-1.6,0.6)--(2.5,4.5)--(2.5,-0.5)--(-1.6,-3.6)
            (-1.25,-1.5) ellipse (0.5cm and 1cm);
        \end{scope}
        \begin{scope}
            \clip (-1.5,0.5)--(1.5,3.5)--(1.5,-0.5)--(-1.5,-3.5)--cycle;
            \draw[Blue] (-1.25,-1.5) ellipse (0.5cm and 1cm);
        \end{scope}

        \draw[Maroon,thick,fill=Red,fill opacity=0.6,-<-](A)to [in=45,out=0](B);
        \draw[Maroon,thick](A)--(B);

        \draw[Maroon,thick,dotted,,fill=Red,fill opacity=0.6](C) to [in=-135,out=-135](D)-- cycle;
        \draw[Blue,fill=Cyan,fill opacity=0.2] (J)to[out=-135,in=-135, looseness=1.7](G);
        \begin{scope}[even odd rule]
            \clip (-1.6,0.6)--(2.5,4.5)--(2.5,-0.5)--(-1.6,-3.6)
            (A)--(J)--(G)--(H);
            \draw[draw=Blue,fill=Cyan,fill opacity=0.2] (1.5,1.5) ellipse (0.5cm and 1cm);
        \end{scope}
        \draw[Blue,dotted](G)to[out=-180,in=-90](D);
        \begin{scope}[even odd rule]
            \clip (-1.5,0.5)--(1.5,3.5)--(1.5,-0.5)--(-1.5,-3.5)--cycle
            (-1.25,-1.5) ellipse (0.5cm and 1cm);
            \draw[Blue,fill=Cyan,fill opacity=0.2] (H)to[out=45,in=45, looseness=2](A);            
        \end{scope}

    \end{scope} 
    }

    \newcommand{\BandTwo}{
    \begin{scope}
        
        \path[name path=Id] (-1.5,0.5)--(1.5,3.5)--(1.5,-0.5)--(-1.5,-3.5)--cycle;
        \path[name path=ER] (1.5,1.5) ellipse (0.5cm and 1cm);
        \path[name path=ERD] (1.25,1.75) ellipse (0.5cm and 1cm);
        \path[name path=EL] (-1.25,-1.5) ellipse (0.5cm and 1cm);
        \path[name intersections={of=Id and EL, by={A,B}}];
        \path[name intersections={of=Id and ERD, by={C,D}}];
        \draw[draw=Blue,fill=Cyan,fill opacity=0.2] (A)--(-1.5,0.5)--(1.5,3.5)--(C);
        \draw[draw=Blue,fill=Cyan,fill opacity=0.2] (B)--(-1.5,-3.5)--(1.5,-0.5)--(D);
        \path[name path=ST] (A)--(C);
        \path[name path=SB] (B)--(D);
        \path[name intersections={of=ST and EL, by=E}];
        \path[name intersections={of=SB and EL, by={FU,F}}];
        \path[name path=B] (-0.9,-2.5)--(-1.25,-1.5);
        \path[name path=T] (1.75,0.5)--(1.25,2.5);
        \path[name intersections={of=T and ER, by={J,G}}];
        \path[name intersections={of=B and EL, by=H}];
        \fill[Red, fill opacity=0.6] (A)--(C)--(D)--(B)--cycle;
        \fill[Cyan,fill opacity=0.2] (1.5,1.5) ellipse (0.5cm and 1cm);
        \begin{scope}[even odd rule]
            \clip (-1.6,0.6)--(2.5,4.5)--(2.5,-0.5)--(-1.6,-3.6)
            (1.5,1.5) ellipse (0.5cm and 1cm);
            \clip (-1.6,0.6)--(2.5,4.5)--(2.5,-0.5)--(-1.6,-3.6)
            (-1.25,-1.5) ellipse (0.5cm and 1cm);
            \draw[draw=Blue,fill=Cyan,fill opacity=0.2]  (A)--(J)--(G)--(H);
        \end{scope}
        \begin{scope}
            \clip (-1.5,0.5)--(1.5,3.5)--(1.5,-0.5)--(-1.5,-3.5)--cycle;
            \draw[Blue] (-1.25,-1.5) ellipse (0.5cm and 1cm);
        \end{scope}
        \begin{scope}[even odd rule]
            \clip (-1.6,0.6)--(2.5,4.5)--(2.5,-0.5)--(-1.6,-3.6)
            (A)--(J)--(G)--(H);
            \draw[draw=Blue] (1.5,1.5) ellipse (0.5cm and 1cm);
        \end{scope}
        \draw[Maroon,thick](A)--(B);
        \draw[Maroon,thick](B)--(F);
        \draw[Maroon,thick,dotted](C)--(D);
        \draw[Maroon,thick,dotted,-<-](B)--(D);
        \draw[Blue,dotted](G)to[out=-180,in=-90](D);

    \end{scope} 
    }
\newcommand{\BandTwoCups}{
    \begin{scope}
        
        \path[name path=Id] (-1.5,0.5)--(1.5,3.5)--(1.5,-0.5)--(-1.5,-3.5)--cycle;
        \path[name path=ER] (1.5,1.5) ellipse (0.5cm and 1cm);
        \path[name path=ERD] (1.25,1.75) ellipse (0.5cm and 1cm);
        \path[name path=EL] (-1.25,-1.5) ellipse (0.5cm and 1cm);
        \path[name intersections={of=Id and EL, by={A,B}}];
        \path[name intersections={of=Id and ERD, by={C,D}}];
        \draw[draw=Blue] (A)--(-1.5,0.5)--(1.5,3.5)--(C);
        \draw[draw=Blue] (B)--(-1.5,-3.5)--(1.5,-0.5)--(D);
        \fill[Cyan, fill opacity=0.2] (-1.5,0.5)--(1.5,3.5)--(1.5,-0.5)--(-1.5,-3.5)--cycle;
        \path[name path=ST] (A)--(C);
        \path[name path=SB] (B)--(D);
        \path[name intersections={of=ST and EL, by=E}];
        \path[name intersections={of=SB and EL, by={FU,F}}];
        \path[name path=B] (-0.9,-2.5)--(-1.25,-1.5);
        \path[name path=T] (1.75,0.5)--(1.25,2.5);
        \path[name intersections={of=T and ER, by={J,G}}];
        \path[name intersections={of=B and EL, by=H}];
        \begin{scope}[even odd rule]
            \clip (-1.6,0.6)--(2.5,4.5)--(2.5,-0.5)--(-1.6,-3.6)
            (1.5,1.5) ellipse (0.5cm and 1cm);
            \clip (-1.6,0.6)--(2.5,4.5)--(2.5,-0.5)--(-1.6,-3.6)
            (-1.25,-1.5) ellipse (0.5cm and 1cm);
        \end{scope}
        \begin{scope}
            \clip (-1.5,0.5)--(1.5,3.5)--(1.5,-0.5)--(-1.5,-3.5)--cycle;
            \draw[Blue] (-1.25,-1.5) ellipse (0.5cm and 1cm);
        \end{scope}

        \draw[Maroon,thick,fill=Red,fill opacity=0.6,->-](A)to [in=45,out=0](B);
        \draw[Maroon,thick](A)--(B);

        \draw[Maroon,thick,dotted,,fill=Red,fill opacity=0.6](C) to [in=-135,out=-135](D)-- cycle;
        \draw[Blue,fill=Cyan,fill opacity=0.2] (J)to[out=-135,in=-135, looseness=1.7](G);
        \begin{scope}[even odd rule]
            \clip (-1.6,0.6)--(2.5,4.5)--(2.5,-0.5)--(-1.6,-3.6)
            (A)--(J)--(G)--(H);
            \draw[draw=Blue,fill=Cyan,fill opacity=0.2] (1.5,1.5) ellipse (0.5cm and 1cm);
        \end{scope}
        \draw[Blue,dotted](G)to[out=-180,in=-90](D);
        \begin{scope}[even odd rule]
            \clip (-1.5,0.5)--(1.5,3.5)--(1.5,-0.5)--(-1.5,-3.5)--cycle
            (-1.25,-1.5) ellipse (0.5cm and 1cm);
            \draw[Blue,fill=Cyan,fill opacity=0.2] (H)to[out=45,in=45, looseness=2](A);            
        \end{scope}

    \end{scope} 
    }
\newcommand{\SqOne}{
        \begin{scope}
            \draw[draw=gray,dotted] (0,0) circle (2cm);
            \clip (0,0) circle (2cm);
            \draw[draw=Blue,->-] (-1,1)--(1,1);
            \draw[draw=Blue,-<-] (-1,-1)--(1,-1);
            \draw[draw=Maroon,thick,->>-] (-1,-1)--(-1,1);
            \draw[draw=Maroon,thick,-<<-] (1,-1)--(1,1);
            \draw[draw=Blue,postaction={decorate}, decoration={markings, mark=at position 0.3 with {\arrow{<}}}] (1,1)-- (2,2);
            \draw[draw=Blue,postaction={decorate}, decoration={markings, mark=at position 0.3 with {\arrow{>}}}] (-1,1)-- (-2,2) ;
            \draw[draw=Blue,postaction={decorate}, decoration={markings, mark=at position 0.3 with {\arrow{<}}}] (-1,-1)-- (-2,-2);
            \draw[draw=Blue,postaction={decorate}, decoration={markings, mark=at position 0.3 with {\arrow{>}}}] (1,-1)-- (2,-2);
        \end{scope}
}
\newcommand{\SqTwo}{
        \begin{scope}
            \draw[draw=gray,dotted] (0,0) circle (2cm);
            \clip (0,0) circle (2cm);
            \draw[draw=Blue,-<-] (-1,1)--(1,1);
            \draw[draw=Blue,->-] (-1,-1)--(1,-1);
            \draw[draw=Blue,->-] (-1,-1)--(-1,1);
            \draw[draw=Maroon,thick,->>-] (1,-1)--(1,1);
            \draw[draw=Blue,postaction={decorate}, decoration={markings, mark=at position 0.3 with {\arrow{>}}}] (1,1)-- (2,2);
            \draw[draw=Maroon,thick,postaction={decorate}, decoration={markings, mark=at position 0.3 with {\arrow{>>}}}] (-1,1)-- (-2,2) ;
            \draw[draw=Maroon,thick,postaction={decorate}, decoration={markings, mark=at position 0.3 with {\arrow{<<}}}] (-1,-1)-- (-2,-2);
            \draw[draw=Blue,postaction={decorate}, decoration={markings, mark=at position 0.3 with {\arrow{<}}}] (1,-1)-- (2,-2);
        \end{scope}
}
\newcommand{\GBigonRight}{
    \
        \begin{scope}
            \draw[draw=gray,dotted] (0,0) circle (2cm);
            \clip (0,0) circle (2cm);
            \draw[Blue,-<-] (0,-1) arc(-90:90:1cm);
            \draw[Blue,->-] (0,1)--(0,2);
            \draw[Blue,-<-] (0,-1)--(0,-2);
            \draw[Maroon,thick,->>-] (0,-1)--(0,1);
        \end{scope}
}

\newcommand{\GBigonLeft}{
    \
        \begin{scope}
            \draw[draw=gray,dotted] (0,0) circle (2cm);
            \clip (0,0) circle (2cm);
            \draw[Blue,->-] (0,1) arc(90:270:1cm);
            \draw[Blue,->-] (0,1)--(0,2);
            \draw[Blue,-<-] (0,-1)--(0,-2);
            \draw[Maroon,thick,->>-] (0,-1)--(0,1);
        \end{scope}
}

\newcommand{\GBigonCenter}{
    \
        \begin{scope}
            \draw[draw=gray,dotted] (0,0) circle (2cm);
            \clip (0,0) circle (2cm);
            \draw[Blue,-<-] (0,1) arc(90:270:1cm);
            \draw[Blue,->-] (0,-1) arc(-90:90:1cm);
            \draw[Maroon,thick,->>-] (0,1)--(0,2);
            \draw[Maroon,thick,-<<-] (0,-1)--(0,-2);

        \end{scope}
}

\newcommand{\BGId}{\begin{scope}
    \draw[draw=gray,dotted] (0,0) circle (2cm);
    \clip (0,0) circle (2cm);
    \draw[Blue,-<-] (0,2) -- (0,-2);
\end{scope}}
\newcommand{\RGId}{\begin{scope}
    \draw[draw=gray,dotted] (0,0) circle (2cm);
    \clip (0,0) circle (2cm);
    \draw[Maroon,thick,-<<-] (0,2) -- (0,-2);
\end{scope}}

\newcommand{\SqTwoId}{

    \begin{scope}
        \draw[draw=gray,dotted] (0,0) circle (2cm);
        \clip (0,0) circle (2cm);
        \draw[draw=Blue,->-,rounded corners] (2,-2)--(1,-1)--(1,1)--(2,2);
        \draw[draw=Maroon,thick,->>-, rounded corners] (-2,-2)--(-1,-1)--(-1,1)--(-2,2);
    \end{scope}
}

\newcommand{\SqOneId}{

    \begin{scope}
        \draw[draw=gray,dotted] (0,0) circle (2cm);
        \clip (0,0) circle (2cm);
        \draw[draw=Blue,->-,rounded corners] (2,2)--(1,1)--(-1,1)--(-2,2);
        \draw[draw=Blue,-<-,rounded corners] (2,-2)--(1,-1)--(-1,-1)--(-2,-2);
    \end{scope}
}

\newcommand{\Arrow}{
    \draw[Blue,->](0,-1.5)--(0,1.5);
}

\newcommand{\Circle}[1]{
    \ifthenelse{\equal{#1}{+}}{\draw[Blue,->-] (0,0) circle (0.75cm);}{}
    \ifthenelse{\equal{#1}{-}}{\draw[Blue,-<-] (0,0) circle (0.75cm);}{}
}

\newcommand{\PRight}{
    \begin{scope}[even odd rule]
        \clip (-0.1,1.6) rectangle (0.1,-1.6) (0.5,0) circle (0.75cm);
        \draw[Blue,->] (0,-1.5)--(0,1.5);
    \end{scope}
    \begin{scope}
        \clip (0,1) rectangle (1.5,-1);
        \draw[Blue] (0.5,0) circle (0.75cm);
    \end{scope}
    \draw[Blue,->-] (0.749cm,0.75cm)--(0.751cm,0.75cm);
    \begin{scope}
        \clip (0.5,0) circle (0.75cm);
        \draw[Maroon,thick,
        postaction={decorate}, decoration={markings, mark=at position 0.5 with {\arrow[scale=0.7]{>>}}}](0,-0.75)--(0,0.75);
    \end{scope}
}

\newcommand{\PLeft}{
    \begin{scope}[even odd rule]
        \clip (0.1,1.6) rectangle (-0.1,-1.6) (-0.5,0) circle (0.75cm);
        \draw[Blue,->] (0,-1.5)--(0,1.5);
    \end{scope}
    \begin{scope}
        \clip (0,1) rectangle (-1.5,-1);
        \draw[Blue] (-0.5,0) circle (0.75cm);
    \end{scope}
    \draw[Blue,-<-] (-0.751em,0.75cm)--(-0.749em,0.75cm);
    \begin{scope}
        \clip (-0.5,0) circle (0.75cm);
        \draw[Maroon,thick,
        postaction={decorate}, decoration={markings, mark=at position 0.5 with {\arrow[scale=0.7]{>>}}}](0,-0.75)--(0,0.75);
    \end{scope}
}

\newcommand{\HT}{

\draw[Blue] (0.5,-1.5)-- (0.5,-1.25) arc(0:180:0.5cm)--(-0.5,-1.5);
\draw[Blue,<->] (-0.5,1.5)--(-0.5,0.25) arc(-180:0:0.5cm)--(0.5,1.5);
\draw[Maroon,thick,
postaction={decorate}, decoration={markings, mark=at position 0.7 with {\arrow[scale=0.5]{>>}}}]
(0,-0.75)--(0,-0.25);
}

\newcommand{\HM}{
\draw[Blue ,<->] (0.5,1.5)-- (0.5,0.75) arc(0:-180:0.5cm)--(-0.5,1.5);
\draw[Blue] (-0.5,-1.5)--(-0.5,-0.75) arc(180:0:0.5cm)--(0.5,-1.5);
\draw[Maroon,thick,
postaction={decorate}, decoration={markings, mark=at position 0.7 with {\arrow[scale=0.5]{<<}}}]
(0,0.25)--(0,-0.25);
}

\newcommand{\HB}{

\draw[Blue ,<->] (0.5,1.5)-- (0.5,1.25) arc(0:-180:0.5cm)--(-0.5,1.5);
\draw[Blue] (-0.5,-1.5)--(-0.5,-0.25) arc(180:0:0.5cm)--(0.5,-1.5);
\draw[Maroon,thick,
postaction={decorate}, decoration={markings, mark=at position 0.7 with {\arrow[scale=0.5]{<<}}}]
(0,0.75)--(0,0.25);
}

\newcommand{\HTR}{

\draw[Blue,->] (0.5,-1.5)-- (0.5,-1.25) arc(0:180:0.5cm)--(-0.5,-1.5);
\draw[Blue,->] (-0.5,1.5)--(-0.5,0.25) arc(-180:0:0.5cm)--(0.5,1.5);
\draw[Maroon,thick] (0.5,0.5)--(-0.5,0.5);

}

\newcommand{\HBR}{

\draw[Blue ,<-] (0.5,1.5)-- (0.5,1.25) arc(0:-180:0.5cm)--(-0.5,1.5);
\draw[Blue,<-] (-0.5,-1.5)--(-0.5,-0.25) arc(180:0:0.5cm)--(0.5,-1.5);
\draw[Maroon,thick] (0.5,-0.5)--(-0.5,-0.5);

}

\newcommand{\HOR}{

\draw[Blue] (0.5,-1.5)-- (0.5,-1.25) arc(0:180:0.5cm)--(-0.5,-1.5);
\draw[Blue ,<->] (0.5,1.5)-- (0.5,1.25) arc(0:-180:0.5cm)--(-0.5,1.5);
\draw[Maroon,thick,
postaction={decorate}, decoration={markings, mark=at position 0.7 with {\arrow[scale=0.5]{>>}}}]
(0,-0.75)--(0,-0.5);
\draw[Maroon,thick,
postaction={decorate}, decoration={markings, mark=at position 0.7 with {\arrow[scale=0.5]{<<}}}]
(0,0.75)--(0,0.5);
\draw[Blue] (0,0) circle (0.5cm);
\draw[Blue,->](-0.5,-0.01)--(-0.5,0.01);
\draw[Blue,->](0.5,-0.01)--(0.5,0.01);
}

\newcommand{\HO}{

\draw[Blue,->] (0.5,-1.5)-- (0.5,-1.25) arc(0:180:0.5cm)--(-0.5,-1.5);
\draw[Blue,<-] (0.5,1.5)-- (0.5,1.25) arc(0:-180:0.5cm)--(-0.5,1.5);
\draw[Blue,-<-] (0,0) circle (0.5cm);
}

\newcommand{\SnakeRight}{
    
\draw[Blue,xshift=0.5cm] (0.5,-1.5)-- (0.5,-1.25) arc(0:180:0.5cm)--(-0.5,-1.5);
\draw[Blue,xshift=-0.5cm,<->] (0.5,1.5)-- (0.5,1.25) arc(0:-180:0.5cm)--(-0.5,1.5);
\draw[Blue,->] (-1,-1.5)--(-1,0) arc (180:0:0.5cm and 0.4cm)-- (0,0) arc(-180:0:0.5cm and 0.4cm)--(1,1.5);
\draw[Maroon,thick,
postaction={decorate}, decoration={markings, mark=at position 0.7 with {\arrow[scale=0.5]{>>}}}]
(-0.5,0.4)--(-0.5,0.75);
\draw[Maroon,thick,
postaction={decorate}, decoration={markings, mark=at position 0.7 with {\arrow[scale=0.5]{<<}}}]
(0.5,-0.4)--(0.5,-0.75);
}

\newcommand{\SnakeLeft}{
    \begin{scope}[xscale=-1]
        \SnakeRight
    \end{scope}
}

\newcommand{\PopPimpleRight}{
    \draw[Blue] (0.5,-1.5)-- (0.5,-1.25) arc(0:180:0.5cm)--(-0.5,-1.5);
    \draw[Blue ,<->] (0.5,1.5)-- (0.5,1.25) arc(0:-180:0.5cm)--(-0.5,1.5);
    \draw[Blue] (-1,-1.5)--(-1,-0.5) to [in=135,out=90] (-135:0.5cm) arc(-135:135:0.5cm) -- (135:0.5cm) to [in=-90,out=-135] (-1,0.5)--(-1,1.5);
    \draw[Maroon,thick,
    postaction={decorate}, decoration={markings, mark=at position 0.7 with {\arrow[scale=0.4]{>>}}}]
    (0,-0.75)--(0,-0.5);
    \draw[Maroon,thick,
    postaction={decorate}, decoration={markings, mark=at position 0.7 with {\arrow[scale=0.4]{<<}}}]
    (0,0.75)--(0,0.5);
    \draw[Maroon,thick,
    postaction={decorate}, decoration={markings, mark=at position 0.7 with {\arrow[scale=0.4]{>>}}}]
    (-0.75,-0.3)--(-0.75,0.3);
}

\newcommand{\PopPimpleLeft}{
    \begin{scope}[xscale=-1]
        \PopPimpleRight
    \end{scope}
}

\newcommand{\Sq}{
            \begin{scope}[rotate=90,]
            \clip (0,0) circle (2cm);
            \begin{scope}[scale=0.75]

            \draw[draw=Blue,->-] (-1,1)--(1,1);
            \draw[draw=Blue,-<-] (-1,-1)--(1,-1);
            \draw[draw=Maroon,thick,postaction={decorate}, decoration={markings, mark=at position 0.5 with {\arrow[scale=0.7]{>>}}}] (-1,-1)--(-1,1);
            \draw[draw=Maroon,thick,postaction={decorate}, decoration={markings, mark=at position 0.5 with {\arrow[scale=0.7]{<<}}}] (1,-1)--(1,1);
            \draw[draw=Blue,postaction={decorate}, decoration={markings, mark=at position 0.3 with {\arrow{<}}}] (1,1)-- (2,2);
            \draw[draw=Blue,postaction={decorate}, decoration={markings, mark=at position 0.3 with {\arrow{>}}}] (-1,1)-- (-2,2) ;
            \draw[draw=Blue,postaction={decorate}, decoration={markings, mark=at position 0.3 with {\arrow{<}}}] (-1,-1)-- (-2,-2);
            \draw[draw=Blue,postaction={decorate}, decoration={markings, mark=at position 0.3 with {\arrow{>}}}] (1,-1)-- (2,-2);
                    
            \end{scope}
        \end{scope}

}

\newcommand{\Trefoil}{
    \draw[shorten <= 0.1cm,shorten >= 0.1cm,->-] (-90:0.5cm)--(30:0.5cm)to[out=60,in=0](90:1cm)to[out=180,in=120](150:0.5cm);
    \draw[shorten <= 0.1cm,shorten >= 0.1cm] (30:0.5cm)--(150:0.5cm)to[out=180,in=120](-150:1cm)to[out=-60,in=-120](-90:0.5cm);
    \draw[shorten <= 0.1cm,shorten >= 0.1cm] (150:0.5cm)--(-90:0.5cm)to[out=-60,in=-120](-30:1cm)to[out=60,in=0](30:0.5cm);

}

\newcommand{\Oriented}{

\coordinate (x) at ($(30:0.5cm)+(90:0.4)$);
\coordinate (y) at ($(30:0.5cm)+(-30:0.4)$);
\coordinate (t) at ($(150:0.5)+(90:0.4)$);
\coordinate (w) at (90:1);
\coordinate (X) at ([rotate=120] x);
\coordinate (Y) at ([rotate=120] y);
\coordinate (T) at ([rotate=120] t);
\coordinate (W) at ([rotate=120] w);
\coordinate (XP) at ([rotate=120] X);
\coordinate (YP) at ([rotate=120] Y);
\coordinate (TP) at ([rotate=120] T);
\coordinate (WP) at ([rotate=120] W);
\draw [Blue,->-](y)to[in=-90,out=150](x)to[in=0,out=90](w)to[in=90,out=180](t);
\draw [Blue](Y)to[in=30,out=270](X)to[in=120,out=210](W)to[in=210,out=300](T);
\draw [Blue](YP)to[in=150,out=30](XP)to[in=240,out=-30](WP)to[in=-30,out=60](TP);
\draw[Blue,rounded corners=0.3cm,->-] (30:0.5)--(150:0.5)--(270:0.5)--cycle;
}

\newcommand{\ThreeCircles}{

\coordinate (x) at ($(30:0.5cm)+(90:0.4)$);
\coordinate (z) at ($(150:0.5)+(90:0.4)$);
\coordinate (y) at (90:1);
\path[name path=SOne] (30:0.5)--(150:0.5);
\path[name path=STwo] (0,0)--(0,1);
\path[name intersections={of=SOne and STwo, by={w}}];

\coordinate (X) at ([rotate=120] x);
\coordinate (Z) at ([rotate=120] z);
\coordinate (Y) at ([rotate=120] y);
\coordinate (W) at ([rotate=120] w);
\coordinate (XP) at ([rotate=120] X);
\coordinate (ZP) at ([rotate=120] Z);
\coordinate (YP) at ([rotate=120] Y);
\coordinate (WP) at ([rotate=120] W);

\draw[Blue] (x)to[in=0,out=90,looseness=0.9](y)to[in=90,out=180,looseness=0.9](z)to[in=180,out=-90,looseness=1.3](w)to[in=-90,out=0,looseness=1.3](x);
\draw[Blue] (X)to[in=120,out=210,looseness=0.9](Y)to[in=210,out=-60,looseness=0.9](Z)to[in=-60,out=30,looseness=1.3](W)to[in=30,out=120,looseness=1.3](X);
\draw[Blue] (XP)to[in=240,out=330,looseness=0.9](YP)to[in=330,out=60,looseness=0.9](ZP)to[in=60,out=150,looseness=1.3](WP)to[in=150,out=240,looseness=1.3](XP);

\draw[Blue,->-](0.01,1)--(-0.01,1);
\draw[Blue,->-,rotate=120](0.01,1)--(-0.01,1);
\draw[Blue,->-,rotate=240](0.01,1)--(-0.01,1);
\draw[Blue,->-]($(w)+(0.01,0)$)--($(w)+(-0.01,0)$);
\draw[Blue,-<](Z)to[in=-60,out=30,looseness=1.3](W);
\draw[Blue,-<](ZP)to[in=60,out=150,looseness=1.3](WP);

\begin{scope}[even odd rule]
    \clip (0,0) circle (1cm)  (x)to[in=0,out=90,looseness=0.9](y)to[in=90,out=180,looseness=0.9](z)to[in=180,out=-90,looseness=1.3](w)to[in=-90,out=0,looseness=1.3](x);
    \clip (0,0) circle (1cm) (X)to[in=120,out=210,looseness=0.9](Y)to[in=210,out=-60,looseness=0.9](Z)to[in=-60,out=30,looseness=1.3](W)to[in=30,out=120,looseness=1.3](X);
    \clip (0,0) circle (1cm) (XP)to[in=240,out=330,looseness=0.9](YP)to[in=330,out=60,looseness=0.9](ZP)to[in=60,out=150,looseness=1.3](WP)to[in=150,out=240,looseness=1.3](XP);
    \draw[Maroon,thick,looseness=1.2,
    postaction={decorate,
      decoration={markings,
        mark=at position 0.01 with {\arrow[scale=0.5]{>>}},
        mark=at position 0.35 with {\arrow[scale=0.5]{>>}},
        mark=at position 0.7 with {\arrow[scale=0.5]{>>}}
      }
    }
    ] (30:0.5)to[in=60,out=120](150:0.5)to[in=180,out=-120](270:0.5)to[in=-60,out=0](30:0.5);
\end{scope}

}

\newcommand{\OneCircle}{
\coordinate (x) at ($(30:0.5cm)+(90:0.4)$);
\coordinate (y) at ($(30:0.5cm)+(-30:0.4)$);
\coordinate (t) at ($(150:0.5)+(90:0.4)$);
\coordinate (w) at (90:1);
\coordinate (X) at ([rotate=120] x);
\coordinate (Y) at ([rotate=120] y);
\coordinate (T) at ([rotate=120] t);
\coordinate (W) at ([rotate=120] w);
\coordinate (XP) at ([rotate=120] X);
\coordinate (YP) at ([rotate=120] Y);
\coordinate (TP) at ([rotate=120] T);
\coordinate (WP) at ([rotate=120] W);

\coordinate (xs) at ($(30:0.5cm)+(90:0.4)$);
\coordinate (zs) at ($(150:0.5)+(90:0.4)$);
\coordinate (ys) at (90:1);
\path[name path=SOne] (30:0.5)--(150:0.5);
\path[name path=STwo] (0,0)--(0,1);
\path[name intersections={of=SOne and STwo, by={ws}}];

\coordinate (Xs) at ([rotate=120] xs);
\coordinate (Zs) at ([rotate=120] zs);
\coordinate (Ys) at ([rotate=120] ys);
\coordinate (Ws) at ([rotate=120] ws);
\coordinate (XPs) at ([rotate=120] Xs);
\coordinate (ZPs) at ([rotate=120] Zs);
\coordinate (YPs) at ([rotate=120] Ys);
\coordinate (WPs) at ([rotate=120] Ws);

\draw[Blue,->-](y)to[in=-90,out=150](x)to[in=0,out=90](w)to[in=90,out=180](t);
\draw[Blue](Y)to[in=30,out=270](X)to[in=120,out=210](W);
\draw[Blue] (Ys)to[in=210,out=-60,looseness=0.9](Zs)to[in=-60,out=30,looseness=1.3](Ws);
\draw[Blue] (WP)to[in=-30,out=60](TP);
\draw[Blue] (XPs)to[in=240,out=330,looseness=0.9](YPs);
\draw[Blue](WPs)to[in=150,out=240,looseness=1.3](XP);
\draw[Blue,looseness=1.6,->-](WPs)to[out=60,in=0](ws)to[out=180,in=120](Ws);
\begin{scope}[even odd rule]
    \clip (0,-0.5) circle (0.5cm)  (Xs)to[in=120,out=210,looseness=0.9](Ys)to[in=210,out=-60,looseness=0.9](Zs)to[in=-60,out=30,looseness=1.3](Ws)to[in=30,out=120,looseness=1.3](Xs);
    \clip (0,-0.5) circle (0.5cm) (XPs)to[in=240,out=330,looseness=0.9](YPs)to[in=330,out=60,looseness=0.9](ZPs)to[in=60,out=150,looseness=1.3](WPs)to[in=150,out=240,looseness=1.3](XPs);
    \draw[Maroon,thick,looseness=1.2,
    postaction={decorate,
      decoration={markings,
        mark=at position 0.01 with {\arrow[scale=0.5]{>>}},
        mark=at position 0.35 with {\arrow[scale=0.5]{>>}},
        mark=at position 0.7 with {\arrow[scale=0.5]{>>}}
      }
    }
    ] (30:0.5)to[in=60,out=120](150:0.5)to[in=180,out=-120](270:0.5)to[in=-60,out=0](30:0.5);
\end{scope}

}

\newcommand{\TwoCircles}{
\coordinate (x) at ($(30:0.5cm)+(90:0.4)$);
\coordinate (y) at ($(30:0.5cm)+(-30:0.4)$);
\coordinate (t) at ($(150:0.5)+(90:0.4)$);
\coordinate (w) at (90:1);
\coordinate (X) at ([rotate=120] x);
\coordinate (Y) at ([rotate=120] y);
\coordinate (T) at ([rotate=120] t);
\coordinate (W) at ([rotate=120] w);
\coordinate (XP) at ([rotate=120] X);
\coordinate (YP) at ([rotate=120] Y);
\coordinate (TP) at ([rotate=120] T);
\coordinate (WP) at ([rotate=120] W);

\coordinate (xs) at ($(30:0.5cm)+(90:0.4)$);
\coordinate (zs) at ($(150:0.5)+(90:0.4)$);
\coordinate (ys) at (90:1);
\path[name path=SOne] (30:0.5)--(150:0.5);
\path[name path=STwo] (0,0)--(0,1);
\path[name intersections={of=SOne and STwo, by={ws}}];

\coordinate (Xs) at ([rotate=120] xs);
\coordinate (Zs) at ([rotate=120] zs);
\coordinate (Ys) at ([rotate=120] ys);
\coordinate (Ws) at ([rotate=120] ws);
\coordinate (XPs) at ([rotate=120] Xs);
\coordinate (ZPs) at ([rotate=120] Zs);
\coordinate (YPs) at ([rotate=120] Ys);
\coordinate (WPs) at ([rotate=120] Ws);

\draw[Blue] (X)to[in=120,out=210](W)to[in=210,out=300](T);
\draw[Blue,->-] (YP)to[in=150,out=30](XP)to[in=240,out=-30](WP)to[in=-30,out=60](TP);
\draw[Blue,->-] (xs)to[in=0,out=90,looseness=0.9](ys)to[in=90,out=180,looseness=0.9](zs)to[in=180,out=-90,looseness=1.3](ws)to[in=-90,out=0,looseness=1.3](xs);
\draw[Blue] (ZPs)to[in=60,out=150,looseness=1.3](WPs);
\draw[Blue] (Ws)to[in=30,out=120,looseness=1.3](Xs);
\draw[Blue,-<-] (WPs) to[in=-60,out=-120](Ws);
\draw[Blue,->-]($(ws)+(0.01,0)$)--($(ws)+(-0.01,0)$);

\begin{scope}[even odd rule]
    
    \clip (-1,0) rectangle (1,1)  (xs)to[in=0,out=90,looseness=0.9](ys)to[in=90,out=180,looseness=0.9](zs)to[in=180,out=-90,looseness=1.3](ws)to[in=-90,out=0,looseness=1.3](xs);
    \clip (-1,0) rectangle (1,1) (Xs)to[in=120,out=210,looseness=0.9](Ys)to[in=210,out=-60,looseness=0.9](Zs)to[in=-60,out=30,looseness=1.3](Ws)to[in=30,out=120,looseness=1.3](Xs);
    \clip (-1,0) rectangle (1,1) (XPs)to[in=240,out=330,looseness=0.9](YPs)to[in=330,out=60,looseness=0.9](ZPs)to[in=60,out=150,looseness=1.3](WPs)to[in=150,out=240,looseness=1.3](XPs);
    \draw[Maroon,thick,looseness=1.2,
    postaction={decorate,
      decoration={markings,
        mark=at position 0.01 with {\arrow[scale=0.5]{>>}},
        mark=at position 0.35 with {\arrow[scale=0.5]{>>}},
        mark=at position 0.7 with {\arrow[scale=0.5]{>>}}
      }
    }
    ] (30:0.5)to[in=60,out=120](150:0.5)to[in=180,out=-120](270:0.5)to[in=-60,out=0](30:0.5);
\end{scope}

}

\newcommand{\Positive}{
\begin{scope}
    \draw[->, shorten <= 0.5cm, shorten >= 0.5cm] (0,0)to[in=180,out=10](2,1)-- node[at start, above]{\small s(c)=0}(3,1) ;
    \draw[->, shorten <= 0.5cm, shorten >= 0.5cm] (0,0)to[in=180,out=-10](2,-1)-- node[at start, above]{\small s(c)=1}(3,-1);
    \draw[->] (-60:0.4)--(120:0.4);
    \fill[white] (0,0) circle (0.1cm);
    \draw[->] (-120:0.4)--(60:0.4);
    \draw[->,xshift=3cm,yshift=1cm,Blue,bend right] (-120:0.4)to(120:0.4);
    \draw[->,xshift=3cm,yshift=1cm,Blue, bend left] (-60:0.4)to(60:0.4);
    \begin{scope}[xshift=3cm, yshift=-1cm]
        \draw[Blue,->] (0,0.2) to[in=-120,out=0] (60:0.4);
        \draw[Blue,->] (0,0.2) to[in=-60,out=180] (120:0.4);
        \draw[Maroon,thick,decoration={ markings, mark=at position 0.7 with {\arrow[scale=0.5]{>>}}},postaction={decorate}] (0,-0.2)--(0,0.2);
        \draw[Blue] (-60:0.4) to[in=0, out=120] (0,-0.2);
        \draw[Blue] (-120:0.4) to[in=180, out=60] (0,-0.2);
    \end{scope}

\end{scope}

}

\newcommand{\Negative}{
\begin{scope}
    \draw[->, shorten <= 0.5cm, shorten >= 0.5cm] (0,0)to[in=180,out=10](2,1)-- node[at start, above]{\small s(c)=0}(3,1) ;
    \draw[->, shorten <= 0.5cm, shorten >= 0.5cm] (0,0)to[in=180,out=-10](2,-1)-- node[at start, above]{\small s(c)=-1}(3,-1);
    \draw[->] (-120:0.4)--(60:0.4);
    \fill[white] (0,0) circle (0.1cm);
    \draw[->] (-60:0.4)--(120:0.4);
    \draw[->,xshift=3cm,yshift=1cm,Blue,bend right] (-120:0.4)to(120:0.4);
    \draw[->,xshift=3cm,yshift=1cm,Blue, bend left] (-60:0.4)to(60:0.4);
    \begin{scope}[xshift=3cm, yshift=-1cm]
        \draw[Blue,->] (0,0.2) to[in=-120,out=0] (60:0.4);
        \draw[Blue,->] (0,0.2) to[in=-60,out=180] (120:0.4);
        \draw[Maroon,thick,decoration={ markings, mark=at position 0.7 with {\arrow[scale=0.5]{>>}}},postaction={decorate}] (0,-0.2)--(0,0.2);
        \draw[Blue] (-60:0.4) to[in=0, out=120] (0,-0.2);
        \draw[Blue] (-120:0.4) to[in=180, out=60] (0,-0.2);
    \end{scope}

\end{scope}

}

\newcommand{\CommOne}{

\draw[Blue] (0,0) ellipse (1cm and 0.5cm);
\draw[Blue] (3,0.5) arc (90:270:1cm and 0.5cm);
\draw[Blue ] (3,-0.5)--(3.5,-0.5);

\draw[Blue] (-3,-0.5) arc (-90:90:1cm and 0.5cm);
\draw[Blue ] (-3,0.5)--(-3.5,0.5);

\draw[Blue] (3,0.5)--(3,-0.5);

\draw[Blue](3.5,-0.5)--(3.5,-3.5)--(-3,-3.5)--(-3,-0.5);

\draw[Blue, dotted] (3,-0.5)--(3,-3)--(-3,-3);
\draw[Blue](-3,-3)--(-3.5,-3)--(-3.5,0.5);

\draw[Maroon,thick,fill=Red,fill opacity=0.6] (-1,0)--(-2,0) arc (-180:0:0.5cm);

\draw[Maroon,thick,fill=Red,fill opacity=0.6](2,-2)-- (2,0)--(1,0)--(1,-2) arc (-180:0:0.5cm);

\fill[Cyan,fill opacity=0.2](-3,-0.5) arc (-90:0:1cm and 0.5cm)--(-2,0) arc (-180:0:0.5cm)--(-1,0) arc (-180:0:1cm and 0.5cm)--(1,-2) arc (-180:0:0.5cm)--(2,0)arc(-180:-90:1cm and 0.5cm)--(3.5,-0.5)--(3.5,-3.5)--(-3,-3.5)--cycle;

\fill[Cyan,fill opacity=0.2](-3.5,0.5)--(-3,0.5) arc (90:0:1cm and 0.5cm)--(-2,0) arc (-180:0:0.5cm)--(-1,0) arc (180:0:1cm and 0.5cm)--(1,-2) arc (-180:0:0.5cm)--(2,0)arc(180:90:1cm and 0.5cm)--(3,-3)--(-3.5,-3)--cycle;

}

\newcommand{\CommTwo}{

\draw[Blue] (0,0) ellipse (1cm and 0.5cm);
\draw[Blue] (3,0.5) arc (90:270:1cm and 0.5cm);
\draw[Blue ] (3,-0.5)--(3.5,-0.5);

\draw[Blue] (-3,-0.5) arc (-90:90:1cm and 0.5cm);
\draw[Blue ] (-3,0.5)--(-3.5,0.5);

\draw[Blue] (3,0.5)--(3,-0.5);

\draw[Blue](3.5,-0.5)--(3.5,-3.5)--(-3,-3.5)--(-3,-0.5);

\draw[Blue, dotted] (3,-0.5)--(3,-3)--(-3,-3);
\draw[Blue](-3,-3)--(-3.5,-3)--(-3.5,0.5);

\draw[Maroon,thick,fill=Red,fill opacity=0.6] (2,0)--(1,0) arc (-180:0:0.5cm);

\draw[Maroon,thick,fill=Red,fill opacity=0.6](-1,-2)--(-1,0) --(-2,0)--(-2,-2) arc (-180:0:0.5cm);

\fill[Cyan, fill opacity=0.2](-3,-0.5) arc (-90:0:1cm and 0.5cm)--(-2,-2) arc (-180:0:0.5cm)--(-1,0)  arc (-180:0:1cm and 0.5cm) arc (-180:0:0.5cm)arc(-180:-90:1cm and 0.5cm)--(3.5,-0.5)--(3.5,-3.5)--(-3,-3.5)--cycle;
\fill[Cyan, fill opacity=0.2](-3.5,0.5)--(-3,0.5) arc (90:0:1cm and 0.5cm)--(-2,0)--(-2,-2) arc (-180:0:0.5cm)--(-1,0) arc (180:0:1cm and 0.5cm)--(1,0) arc (-180:0:0.5cm)--(2,0)arc(180:90:1cm and 0.5cm)--(3,-3)--(-3.5,-3)--cycle;

}

\newcommand{\CommThree}{

\draw[Blue](0,-3.5)-- (-3,-3.5)--(-3,-0.5) arc(-90:90:1cm and 0.5cm)--(-3.5,0.5)--(-3.5,-2.5)--(-3,-2.5);
\draw[Blue](-3,0.5)--(-3.5,0.5);
\draw[Blue](2,-3) arc(-180:-90:1cm and 0.5cm)--(3.5,-3.5)--(3.5,-0.5)--(0,-0.5) arc(-90:-270:1cm and 0.5cm)--(3,0.5)--(3,-0.5);
\draw[Blue] (0,-3.5) arc (-90:0:1cm and 0.5cm);
\draw[Blue,dotted] (-3,-2.5)--(0,-2.5)arc(90:0:1cm and 0.5cm);
\draw[Blue,dotted] (3,-0.5)--(3,-2.5) arc(90:180:1cm and 0.5cm);

\draw[Maroon,thick,fill=Red,fill opacity=0.6] (2,-3) arc(0:180:0.5cm)--cycle;
\draw[Maroon,thick,fill=Red,fill opacity=0.6] (-2,0) arc(-180:0: 0.5cm)--cycle;

\fill[Cyan, fill opacity=0.2] (-3,-0.5)arc(-90:0:1cm and 0.5cm)--(-2,0)arc(-180:0: 0.5cm)--(-1,0)arc(-180:-90:1cm and 0.5cm)--(3.5,-0.5)--(3.5,-3.5)--(3,-3.5) arc(-90:-180:1cm and 0.5cm)--(2,-3) arc(0:180:0.5cm)--(1,-3) arc(0:-90:1cm and 0.5cm)--(-3,-3.5)--cycle;
\fill[Cyan, fill opacity=0.2] (-3.5,0.5)--(-3,0.5)arc(90:0:1cm and 0.5cm)--(-2,0)arc(-180:0: 0.5cm)--(-1,0)arc(180:90:1cm and 0.5cm)--(3,0.5)--(3,-2.5)arc(90:180:1cm and 0.5cm)--(2,-3)arc(0:180:0.5cm) --(1,-3) arc(0:90:1cm and 0.5cm)--(-3.5,-2.5)--cycle;

}

\newcommand{\CommFour}{

\draw[Blue](0,-3.5)-- (-3,-3.5)--(-3,-0.5) arc(-90:90:1cm and 0.5cm)--(-3.5,0.5)--(-3.5,-2.5)--(-3,-2.5);
\draw[Blue](-3,0.5)--(-3.5,0.5);
\draw[Blue](2,-3) arc(-180:-90:1cm and 0.5cm)--(3.5,-3.5)--(3.5,-0.5)--(0,-0.5) arc(-90:-270:1cm and 0.5cm)--(3,0.5)--(3,-0.5);
\draw[Blue] (0,-3.5) arc (-90:0:1cm and 0.5cm);
\draw[Blue,dotted] (-3,-2.5)--(0,-2.5)arc(90:0:1cm and 0.5cm);
\draw[Blue,dotted] (3,-0.5)--(3,-2.5) arc(90:180:1cm and 0.5cm);

\draw[Maroon,thick,fill=Red,fill opacity=0.6] (2,-3)--(2,-1.5) arc(0:180:0.5cm)--(1,-3)--cycle;
\draw[Maroon,thick,fill=Red,fill opacity=0.6] (-2,0)--(-2,-1.5) arc(-180:0: 0.5cm)--(-1,0)--cycle;

\fill[Cyan, fill opacity=0.2] (-3,-0.5)arc(-90:0:1cm and 0.5cm)--(-2,0)--(-2,-1.5)arc(-180:0: 0.5cm)--(-1,0)arc(-180:-90:1cm and 0.5cm)--(3.5,-0.5)--(3.5,-3.5)--(3,-3.5) arc(-90:-180:1cm and 0.5cm)--(2,-3)--(2,-1.5) arc(0:180:0.5cm)--(1,-3) arc(0:-90:1cm and 0.5cm)--(-3,-3.5)--cycle;
\fill[Cyan, fill opacity=0.2] (-3.5,0.5)--(-3,0.5)arc(90:0:1cm and 0.5cm)--(-2,0)--(-2,-1.5)arc(-180:0: 0.5cm)--(-1,0)arc(180:90:1cm and 0.5cm)--(3,0.5)--(3,-2.5)arc(90:180:1cm and 0.5cm)--(2,-3)--(2,-1.5)arc(0:180:0.5cm) --(1,-3) arc(0:90:1cm and 0.5cm)--(-3.5,-2.5)--cycle;

}

\newcommand{\CommFive}{

\draw[Blue](0,-3.5)-- (-3,-3.5)--(-3,-0.5) arc(-90:90:1cm and 0.5cm)--(-3.5,0.5)--(-3.5,-2.5)--(-3,-2.5);
\draw[Blue](-3,0.5)--(-3.5,0.5);
\draw[Blue](2,-3) arc(-180:-90:1cm and 0.5cm)--(3.5,-3.5)--(3.5,-0.5)--(0,-0.5) arc(-90:-270:1cm and 0.5cm)--(3,0.5)--(3,-0.5);
\draw[Blue] (0,-3.5) arc (-90:0:1cm and 0.5cm);
\draw[Blue,dotted] (-3,-2.5)--(0,-2.5)arc(90:0:1cm and 0.5cm);
\draw[Blue,dotted] (3,-0.5)--(3,-2.5) arc(90:180:1cm and 0.5cm);

\draw[Maroon,thick,fill=Red,fill opacity=0.6] (2,-3) arc(0:180:0.5cm)--cycle;
\draw[Blue] (-2,0) arc(-180:0: 0.5cm);

\fill[Cyan, fill opacity=0.2] (-3,-0.5)arc(-90:0:1cm and 0.5cm)--(-2,0)arc(-180:0: 0.5cm)--(-1,0)arc(-180:-90:1cm and 0.5cm)--(3.5,-0.5)--(3.5,-3.5)--(3,-3.5) arc(-90:-180:1cm and 0.5cm)--(2,-3) arc(0:180:0.5cm)--(1,-3) arc(0:-90:1cm and 0.5cm)--(-3,-3.5)--cycle;
\fill[Cyan, fill opacity=0.2] (-3.5,0.5)--(-3,0.5)arc(90:0:1cm and 0.5cm)--(-2,0)arc(-180:0: 0.5cm)--(-1,0)arc(180:90:1cm and 0.5cm)--(3,0.5)--(3,-2.5)arc(90:180:1cm and 0.5cm)--(2,-3)arc(0:180:0.5cm) --(1,-3) arc(0:90:1cm and 0.5cm)--(-3.5,-2.5)--cycle;

}

\newcommand{\CommSix}{

\draw[Blue](0,-3.5)-- (-3,-3.5)--(-3,-0.5) arc(-90:90:1cm and 0.5cm)--(-3.5,0.5)--(-3.5,-2.5)--(-3,-2.5);
\draw[Blue](-3,0.5)--(-3.5,0.5);
\draw[Blue](2,-3) arc(-180:-90:1cm and 0.5cm)--(3.5,-3.5)--(3.5,-0.5)--(0,-0.5) arc(-90:-270:1cm and 0.5cm)--(3,0.5)--(3,-0.5);
\draw[Blue] (0,-3.5) arc (-90:0:1cm and 0.5cm);
\draw[Blue,dotted] (-3,-2.5)--(0,-2.5)arc(90:0:1cm and 0.5cm);
\draw[Blue,dotted] (3,-0.5)--(3,-2.5) arc(90:180:1cm and 0.5cm);

\draw[Maroon,thick,fill=Red,fill opacity=0.6] (2,-3)--(2,-1.5) arc(0:180:0.5cm)--(1,-3)--cycle;
\draw[Blue] (-2,0)--(-2,-1.5) arc(-180:0: 0.5cm)--(-1,0);

\fill[Cyan, fill opacity=0.2] (-3,-0.5)arc(-90:0:1cm and 0.5cm)--(-2,0)--(-2,-1.5)arc(-180:0: 0.5cm)--(-1,0)arc(-180:-90:1cm and 0.5cm)--(3.5,-0.5)--(3.5,-3.5)--(3,-3.5) arc(-90:-180:1cm and 0.5cm)--(2,-3)--(2,-1.5) arc(0:180:0.5cm)--(1,-3) arc(0:-90:1cm and 0.5cm)--(-3,-3.5)--cycle;
\fill[Cyan, fill opacity=0.2] (-3.5,0.5)--(-3,0.5)arc(90:0:1cm and 0.5cm)--(-2,0)--(-2,-1.5)arc(-180:0: 0.5cm)--(-1,0)arc(180:90:1cm and 0.5cm)--(3,0.5)--(3,-2.5)arc(90:180:1cm and 0.5cm)--(2,-3)--(2,-1.5)arc(0:180:0.5cm) --(1,-3) arc(0:90:1cm and 0.5cm)--(-3.5,-2.5)--cycle;

}

\newcommand{\ROneRight}[1]{
\path[name path=One](0,-1.5)--(0,-1)to[out=90,in=225]($(1,0)+(135:0.75)$)arc(135:-135:0.75cm);
\path[name path=Two]($(1,0)+(-135:0.75)$)to[in=-90,out=135](0,1.5);
\path[name intersections={of=One and Two, by=c}];
\ifthenelse{\equal{#1}{-}}{
    \draw(0,-1.5)--(0,-1)to[out=90,in=225]($(1,0)+(135:0.75)$)arc(135:-135:0.75cm);
    \fill[white] (c) circle (7pt);
    \draw[->]($(1,0)+(-135:0.75)$)to[in=-90,out=135](0,1.5);
}{
    \draw[->]($(1,0)+(-135:0.75)$)to[in=-90,out=135](0,1.5);
    \fill[white] (c) circle (7pt);
    \draw(0,-1.5)--(0,-1)to[out=90,in=225]($(1,0)+(135:0.75)$)arc(135:-135:0.75cm);
}

}

\newcommand{\ROneLeft}[1]{

    \begin{scope}[xscale=-1]
        \ifthenelse{\equal{#1}{+}}{\ROneRight{-}}{\ROneRight{+}}
    \end{scope}
}

\newcommand{\RTwo}[2]{
    \path[name path=l] (-0.5,-1.5)to[in=-90,out=45](0.5,0)to[in=-45, out=90](-0.5,1.5);
    \path[name path=r] (0.5,-1.5)to[in=-90,out=135](-0.5,0)to[in=-135,out=90](0.5,1.5);
    \path[name intersections ={of=l and r,by={cOne,cTwo}}];

    \ifthenelse{\equal{#1}{+}}{
        \ifthenelse{\equal{#2}{r}}{
            
        \draw[->] (-0.5,-1.5)to[in=-90,out=45](0.5,0)to[in=-45, out=90](-0.5,1.5);
        \fill[white] (cOne) circle (7pt) (cTwo) circle (7pt);
        \draw[->] (0.5,-1.5)to[in=-90,out=135](-0.5,0)to[in=-135,out=90](0.5,1.5);

        }{

        \draw[->] (0.5,-1.5)to[in=-90,out=135](-0.5,0)to[in=-135,out=90](0.5,1.5);
        \fill[white] (cOne) circle (7pt) (cTwo) circle (7pt);
        \draw[->] (-0.5,-1.5)to[in=-90,out=45](0.5,0)to[in=-45, out=90](-0.5,1.5);

        }

    }{
        \ifthenelse{\equal{#2}{r}}{
        \draw[<-] (-0.5,-1.5)to[in=-90,out=45](0.5,0)to[in=-45, out=90](-0.5,1.5);
        \fill[white] (cOne) circle (7pt) (cTwo) circle (7pt);
        \draw[->] (0.5,-1.5)to[in=-90,out=135](-0.5,0)to[in=-135,out=90](0.5,1.5);
   
        }{

        \draw[->] (0.5,-1.5)to[in=-90,out=135](-0.5,0)to[in=-135,out=90](0.5,1.5);
        \fill[white] (cOne) circle (7pt) (cTwo) circle (7pt);
        \draw[<-] (-0.5,-1.5)to[in=-90,out=45](0.5,0)to[in=-45, out=90](-0.5,1.5);

        }
    }
}

\newcommand{\RTwoPlusf}{
    \draw[Blue](1.5,0) arc(-180:-90:1cm and 0.5cm)--(3.5,-0.5)--(3.5,3.5);
    \draw[Blue,->-](-3,3.5)--(3.5,3.5);
    \draw[Blue](-3,3.5)--(-3,-0.5)--(-2.5,-0.5)arc(-90:0:1cm and 0.5cm);
    \draw[Blue,->-](-1,0)arc(-180:0:1cm and 0.5cm);
    \draw[Blue,->-](-3.5,4)--(3,4);
    \draw[Blue](3,4)--(3,3.5);
    \draw[Blue](-3.5,4)--(-3.5,0.5)--(-3,0.5);
    \draw[Blue,dotted](-1,0)arc(180:0:1cm and 0.5cm);
    \draw[Blue,dotted](-3,0.5)--(-2.5,0.5)arc(90:0:1cm and 0.5cm);
    \draw[Blue,dotted](1.5,0) arc(180:90:1cm and 0.5cm);
    \draw[Blue,dotted](2.5,0.5)--(3,0.5)--(3,3.5);
    \draw[Maroon,thick,fill=Red,fill opacity=0.6] (1,0)arc(0:180:1cm)--(-1.5,0)arc(180:0:1.5cm)--cycle;;
}

\newcommand{\RTwoPlusg}{
    \draw[Blue](1.5,0) arc(-180:-90:1cm and 0.5cm)--(3.5,-0.5)--(3.5,-4);
    \draw[Blue,->-](-3,-4)--(3.5,-4);
    \draw[Blue](-3,-4)--(-3,-0.5)--(-2.5,-0.5)arc(-90:0:1cm and 0.5cm);
    \draw[Blue,->-](-1,0)arc(-180:0:1cm and 0.5cm);
    \draw[Blue,dotted](-3.5,-3.5)--(3,-3.5);
    \draw[Blue](-3.5,-3.5)--(-3,-3.5);
    \draw[Blue](2.5,0.5)--(3,0.5)--(3,-0.5);
    \draw[Blue,dotted](3,-0.5)--(3,-3.5);
    \draw[Blue](-3.5,-3.5)--(-3.5,0.5)--(-3,0.5);
    \draw[Blue](-1,0)arc(180:0:1cm and 0.5cm);
    \draw[Blue,](-3,0.5)--(-2.5,0.5)arc(90:0:1cm and 0.5cm);
    \draw[Blue,](1.5,0) arc(180:90:1cm and 0.5cm);
    \draw[Maroon,thick,fill=Red,fill opacity=0.6] (1,0)arc(0:-180:1cm)--(-1.5,0)arc(-180:0:1.5cm)--cycle;;
}

\newcommand{\PinchU}{
    \draw[Blue,-<-](3,-0.5)--(3,0.5)--(0,0.5);
    \draw[Blue](0,0.5)arc(90:270:1cm and 0.5cm);
    \draw[Blue,->-](0,-0.5)--(3.5,-0.5);
    \draw[Blue](3.5,-0.5)--(3.5,-4);
    \draw[Blue,-<-](3.5,-4)--(-3,-4);
    \draw[Blue](-3,-4)--(-3,-0.5)--(-2.5,-0.5)arc(-90:90:1cm and 0.5cm)--(-3.5,0.5)--(-3.5,-3.5)--(-3,-3.5);
    \draw[Blue,dotted](-3,-3.5)--(3,-3.5)--(3,-0.5);
    \draw[Maroon,thick,fill=Red,fill opacity=0.6](-1,0)--(-1.5,0)--(-1.5,-0.5)arc(-180:0:0.25cm)--cycle;

}

\newcommand{\PinchD}{
    \draw[Blue,dotted]--(3,0.5)--(2.5,0.5);
    \draw[Blue,dotted](2.5,0.5)arc(90:180:1cm and 0.5cm);
    \draw[Blue](1.5,0)arc(180:270:1cm and 0.5cm);
    \draw[Blue](3.5,-0.5)--(3.5,3.5);
    \draw[Blue,-<-](3.5,3.5)--(-3,3.5);
    \draw[Blue](-3,3.5)--(-3,-0.5)--(0,-0.5);
    \draw[Blue](0,-0.5)arc(-90:0:1cm and 0.5cm);
    \draw[Blue,->-](2.5,-0.5)--(3.5,-0.5);
    \draw[Blue,dotted](3,0.5)--(3,3.5);
    \draw[Blue,->-](-3.5,4)--(3,4);
    \draw[Blue](3,4)--(3,3.5);
    \draw[Blue](-3,0.5)--(-3.5,0.5)--(-3.5,4);
    \draw[Blue,dotted](1,0)arc(0:90:1cm and 0.5cm)--(-3,0.5);
    \draw[Maroon,thick,fill=Red,fill opacity=0.6](1,0)--(1.5,0)--(1.5,0.5)arc(0:180:0.25cm)--cycle;

}

\newcommand{\CandyCane}{
    \draw[Blue,dotted](3,0.5)--(2.5,0.5);
    \draw[Blue,dotted](2.5,0.5)arc(90:180:1cm and 0.5cm);
    \draw[Blue](1.5,0)arc(180:270:1cm and 0.5cm);
    \draw[Blue](3.5,-0.5)--(3.5,3.5);
    \draw[Blue,-<-](3.5,3.5)--(-3,3.5);
    \draw[Blue](-3,3.5)--(-3,-0.5)--(0,-0.5);
    \draw[Blue](0,-0.5)arc(-90:0:1cm and 0.5cm);
    \draw[Blue,->-](2.5,-0.5)--(3.5,-0.5);
    \draw[Blue,dotted](3,0.5)--(3,3.5);
    \draw[Blue,->-](-3.5,4)--(3,4);
    \draw[Blue](3,4)--(3,3.5);
    \draw[Blue](-3,0.5)--(-3.5,0.5)--(-3.5,4);
    \draw[Blue,dotted](1,0)arc(0:90:1cm and 0.5cm)--(-3,0.5);
    \draw[Maroon,thick,fill=Red,fill opacity=0.6](1.5,0)--(1.5,1.5)arc(0:180:1.5cm)--(-1.5,1.5)arc(-180:0:0.25cm)--(-1,1.5)arc(180:0:1cm)--(1,0)--cycle;

}

\newcommand{\CandyCaneTwo}{
    \draw[Blue,-<-](3,-0.5)--(3,0.5)--(0,0.5);
    \draw[Blue](0,0.5)arc(90:270:1cm and 0.5cm);
    \draw[Blue,->-](0,-0.5)--(3.5,-0.5);
    \draw[Blue](3.5,-0.5)--(3.5,-4);
    \draw[Blue,-<-](3.5,-4)--(-3,-4);
    \draw[Blue](-3,-4)--(-3,-0.5)--(-2.5,-0.5)arc(-90:90:1cm and 0.5cm)--(-3.5,0.5)--(-3.5,-3.5)--(-3,-3.5);
    \draw[Blue,dotted](-3,-3.5)--(3,-3.5)--(3,-0.5);
    \draw[Maroon,thick,fill=Red,fill opacity=0.6](-1,0)--(-1.5,0)--(-1.5,-1.5)arc(-180:0:1.5cm)--(1.5,-1.5)arc(0:180:0.25cm)--(1,-1.5)arc(0:-180:1cm)--cycle;
    
}

\newcommand{\RTwoPlusfg}{
\begin{scope}[yshift=-2cm]
    \clip (-3.6,2)rectangle (3.6,-0.6);
    \RTwoPlusf
\end{scope}
\begin{scope}[yshift=2cm]
    \clip (-3.6,-2)rectangle (3.6,0.6);
    \RTwoPlusg
\end{scope}

}

\newcommand{\RTwoPlusI}{
    \draw[Maroon,thick,fill=Red,fill opacity=0.6](-1.5,0)--(-1.5,-4)--(-1,-4)--(-1,0)--cycle;
    \draw[Maroon,thick,fill=Red,fill opacity=0.6](1.5,0)--(1.5,-4)--(1,-4)--(1,0)--cycle;

    \draw[Blue](1.5,-4) arc(-180:-90:1cm and 0.5cm)--(3.5,-4.5);
    \draw[Blue](-3,-4.5)--(-2.5,-4.5)arc(-90:0:1cm and 0.5cm);
    \draw[Blue,->-](-1,-4)arc(-180:0:1cm and 0.5cm);
    \draw[Blue](-3.5,-3.5)--(-3,-3.5);
    \draw[Blue,dotted](-1,-4)arc(180:0:1cm and 0.5cm);
    \draw[Blue,dotted](-3,-3.5)--(-2.5,-3.5)arc(90:0:1cm and 0.5cm);
    \draw[Blue,dotted](1.5,-4) arc(180:90:1cm and 0.5cm);
    \draw[Blue,dotted](2.5,-3.5)--(3,-3.5);

    \draw[Blue](1.5,0) arc(-180:-90:1cm and 0.5cm)--(3.5,-0.5);
    \draw[Blue](-3,-0.5)--(-2.5,-0.5)arc(-90:0:1cm and 0.5cm);
    \draw[Blue,->-](-1,0)arc(-180:0:1cm and 0.5cm);
    \draw[Blue](2.5,0.5)--(3,0.5)--(3,-0.5);
    \draw[Blue](-3.5,0.5)--(-3,0.5);
    \draw[Blue](-1,0)arc(180:0:1cm and 0.5cm);
    \draw[Blue,](-3,0.5)--(-2.5,0.5)arc(90:0:1cm and 0.5cm);
    \draw[Blue,](1.5,0) arc(180:90:1cm and 0.5cm);

    \draw[Blue](-3,-4.5)--(-3,-0.5);
    \draw[Blue](-3.5,-3.5)--(-3.5,0.5);
    \draw[Blue](3.5,-4.5)--(3.5,-0.5);
    \draw[Blue,dotted](3,-0.5)--(3,-3.5);

}

\newcommand{\RTwoPlusCups}{

    \draw[Blue](1.5,-4) arc(-180:-90:1cm and 0.5cm)--(3.5,-4.5);
    \draw[Blue](-3,-4.5)--(-2.5,-4.5)arc(-90:0:1cm and 0.5cm);
    \draw[Blue,->-](-1,-4)arc(-180:0:1cm and 0.5cm);
    \draw[Blue](-3.5,-3.5)--(-3,-3.5);
    \draw[Blue,dotted](-1,-4)arc(180:0:1cm and 0.5cm);
    \draw[Blue,dotted](-3,-3.5)--(-2.5,-3.5)arc(90:0:1cm and 0.5cm);
    \draw[Blue,dotted](1.5,-4) arc(180:90:1cm and 0.5cm);
    \draw[Blue,dotted](2.5,-3.5)--(3,-3.5);

    \draw[Blue](1.5,0) arc(-180:-90:1cm and 0.5cm)--(3.5,-0.5);
    \draw[Blue](-3,-0.5)--(-2.5,-0.5)arc(-90:0:1cm and 0.5cm);
    \draw[Blue,->-](-1,0)arc(-180:0:1cm and 0.5cm);
    \draw[Blue](2.5,0.5)--(3,0.5)--(3,-0.5);
    \draw[Blue](-3.5,0.5)--(-3,0.5);
    \draw[Blue](-1,0)arc(180:0:1cm and 0.5cm);
    \draw[Blue,](-3,0.5)--(-2.5,0.5)arc(90:0:1cm and 0.5cm);
    \draw[Blue,](1.5,0) arc(180:90:1cm and 0.5cm);

    \draw[Blue](-3,-4.5)--(-3,-0.5);
    \draw[Blue](-3.5,-3.5)--(-3.5,0.5);
    \draw[Blue](3.5,-4.5)--(3.5,-0.5);
    \draw[Blue,dotted](3,-0.5)--(3,-3.5);

    \draw[Maroon,thick,fill=Red,fill opacity=0.6](-1,-4)arc(180:0:1cm)--(1.5,-4)--(1.5,0)--(1,0)arc(0:-180:1cm)--(-1.5,0)--(-1.5,-4)--cycle;

}

\newcommand{\RTwoPlushOne}{

    \draw[Blue,dotted](0,-3.5)--(3,-3.5);
    \draw[Blue,dotted](0,-3.5)arc(90:180:1cm and 0.5cm);
    \draw[Blue](-1,-4)arc(180:270:1cm and 0.5cm);
    \draw[Blue](-3,-4.5)--(-2.5,-4.5);
    \draw[Blue](-2.5,-4.5)arc(-90:0:1cm and 0.5cm);
    \draw[Blue,->-](0,-4.5)--(3.5,-4.5);
    \draw[Blue](-3,-3.5)--(-3.5,-3.5);
    \draw[Blue,dotted](-1.5,-4)arc(0:90:1cm and 0.5cm)--(-3,-3.5);

    \draw[Blue](1.5,0) arc(-180:-90:1cm and 0.5cm)--(3.5,-0.5);
    \draw[Blue](-3,-0.5)--(-2.5,-0.5)arc(-90:0:1cm and 0.5cm);
    \draw[Blue,->-](-1,0)arc(-180:0:1cm and 0.5cm);
    \draw[Blue](2.5,0.5)--(3,0.5)--(3,-0.5);
    \draw[Blue](-3.5,0.5)--(-3,0.5);
    \draw[Blue](-1,0)arc(180:0:1cm and 0.5cm);
    \draw[Blue,](-3,0.5)--(-2.5,0.5)arc(90:0:1cm and 0.5cm);
    \draw[Blue,](1.5,0) arc(180:90:1cm and 0.5cm);

    \draw[Blue](-3,-4.5)--(-3,-0.5);
    \draw[Blue](-3.5,-3.5)--(-3.5,0.5);
    \draw[Blue](3.5,-4.5)--(3.5,-0.5);
    \draw[Blue,dotted](3,-0.5)--(3,-3.5);

    \draw[Maroon,thick,fill=Red,fill opacity=0.6] (-1,0)arc(-180:0:1cm)--(1.5,0)to[in=90,out=-90](-1,-4)--(-1.5,-4)--(-1.5,0)--cycle;

}

\newcommand{\RTwoPlushTwo}{

    \draw[Blue](1.5,-4) arc(-180:-90:1cm and 0.5cm)--(3.5,-4.5);
    \draw[Blue](-3,-4.5)--(-2.5,-4.5)arc(-90:0:1cm and 0.5cm);
    \draw[Blue,->-](-1,-4)arc(-180:0:1cm and 0.5cm);
    \draw[Blue](-3.5,-3.5)--(-3,-3.5);
    \draw[Blue,dotted](-1,-4)arc(180:0:1cm and 0.5cm);
    \draw[Blue,dotted](-3,-3.5)--(-2.5,-3.5)arc(90:0:1cm and 0.5cm);
    \draw[Blue,dotted](1.5,-4) arc(180:90:1cm and 0.5cm);
    \draw[Blue,dotted](2.5,-3.5)--(3,-3.5);

    \draw[Blue](3,-0.5)--(3,0.5)--(2.5,0.5);
    \draw[Blue](2.5,0.5)arc(90:270:1cm and 0.5cm);
    \draw[Blue](2.5,-0.5)--(3.5,-0.5);
    \draw[Blue](0,-0.5)arc(-90:90:1cm and 0.5cm)--(-3.5,0.5);
    \draw[Blue,->-](-3,-0.5)--(0,-0.5);

    \draw[Blue](-3,-4.5)--(-3,-0.5);
    \draw[Blue](-3.5,-3.5)--(-3.5,0.5);
    \draw[Blue](3.5,-4.5)--(3.5,-0.5);
    \draw[Blue,dotted](3,-0.5)--(3,-3.5);

    \draw[Maroon,thick,fill=Red,fill opacity=0.6] (1,0)--(1.5,0)--(1.5,-4)--(1,-4)arc(0:180:1cm)--(-1.5,-4)to[in=-90,out=90]cycle;
}

\newcommand{\RTwoPlusdFour}{
    \draw[Blue](1.5,0) arc(-180:-90:1cm and 0.5cm)--(3.5,-0.5);
    \draw[Blue](-3,-0.5)--(-2.5,-0.5)arc(-90:0:1cm and 0.5cm);
    \draw[Blue,->-](-1,0)arc(-180:0:1cm and 0.5cm);
    \draw[Blue](2.5,0.5)--(3,0.5)--(3,-0.5);
    \draw[Blue](-3.5,0.5)--(-3,0.5);
    \draw[Blue](-1,0)arc(180:0:1cm and 0.5cm);
    \draw[Blue,](-3,0.5)--(-2.5,0.5)arc(90:0:1cm and 0.5cm);
    \draw[Blue,](1.5,0) arc(180:90:1cm and 0.5cm);
    \draw[Blue,dotted]--(3,-3.5)--(2.5,-3.5);
    \draw[Blue,dotted](2.5,-3.5)arc(90:180:1cm and 0.5cm);
    \draw[Blue](1.5,-4)arc(180:270:1cm and 0.5cm);
    \draw[Blue](0,-4.5)arc(-90:0:1cm and 0.5cm);
    \draw[Blue,->-](2.5,-4.5)--(3.5,-4.5);
    \draw[Blue](-3,-3.5)--(-3.5,-3.5);
    \draw[Blue](-3,-4.5)--(0,-4.5);
    \draw[Blue,dotted](1,-4)arc(0:90:1cm and 0.5cm)--(-3,-3.5);
    \draw[Blue](-3,-4.5)--(-3,-0.5);
    \draw[Blue](-3.5,-3.5)--(-3.5,0.5);
    \draw[Blue](3.5,-4.5)--(3.5,-0.5);
    \draw[Blue,dotted](3,-0.5)--(3,-3.5);
    \draw[Maroon,thick,fill=Red,fill opacity=0.6](1,0)--(1.5,0)--(1.5,-4)--(1,-4)--cycle;
    \draw[Maroon,thick,fill=Red,fill opacity=0.6](-1.5,0) --(-1.5,-0.5)arc(-180:0:0.25cm)--(-1,0)--cycle;

}

\newcommand{\RTwoPlusdTwo}{
    \draw[Blue](1.5,-4) arc(-180:-90:1cm and 0.5cm)--(3.5,-4.5);
    \draw[Blue](-3,-4.5)--(-2.5,-4.5)arc(-90:0:1cm and 0.5cm);
    \draw[Blue,->-](-1,-4)arc(-180:0:1cm and 0.5cm);
    \draw[Blue](-3.5,-3.5)--(-3,-3.5);
    \draw[Blue,dotted](-1,-4)arc(180:0:1cm and 0.5cm);
    \draw[Blue,dotted](-3,-3.5)--(-2.5,-3.5)arc(90:0:1cm and 0.5cm);
    \draw[Blue,dotted](1.5,-4) arc(180:90:1cm and 0.5cm);
    \draw[Blue,dotted](2.5,-3.5)--(3,-3.5);

    \draw[Blue,-<-](3,-0.5)--(3,0.5)--(0,0.5);
    \draw[Blue](0,0.5)arc(90:270:1cm and 0.5cm);
    \draw[Blue,->-](0,-0.5)--(3.5,-0.5);
    \draw[Blue](-3,-0.5)--(-2.5,-0.5)arc(-90:90:1cm and 0.5cm)--(-3.5,0.5);

    \draw[Blue](-3,-4.5)--(-3,-0.5);
    \draw[Blue](-3.5,-3.5)--(-3.5,0.5);
    \draw[Blue](3.5,-4.5)--(3.5,-0.5);
    \draw[Blue,dotted](3,-0.5)--(3,-3.5);

    \draw[Maroon,thick,fill=Red,fill opacity=0.6](1.5,-4)--(1.5,-3.5)arc(0:180:0.25cm)--(1,-4)--cycle;
    \draw[Maroon,thick,fill=Red,fill opacity=0.6](-1.5,0)--(-1.5,-4)--(-1,-4)--(-1,0)--cycle;

}

\newcommand{\RTwoPlushd}{

    \draw[Blue](1.5,-4) arc(-180:-90:1cm and 0.5cm)--(3.5,-4.5);
    \draw[Blue](-3,-4.5)--(-2.5,-4.5)arc(-90:0:1cm and 0.5cm);
    \draw[Blue,->-](-1,-4)arc(-180:0:1cm and 0.5cm);
    \draw[Blue](-3.5,-3.5)--(-3,-3.5);
    \draw[Blue,dotted](-1,-4)arc(180:0:1cm and 0.5cm);
    \draw[Blue,dotted](-3,-3.5)--(-2.5,-3.5)arc(90:0:1cm and 0.5cm);
    \draw[Blue,dotted](1.5,-4) arc(180:90:1cm and 0.5cm);
    \draw[Blue,dotted](2.5,-3.5)--(3,-3.5);

    \draw[Blue](1.5,0) arc(-180:-90:1cm and 0.5cm)--(3.5,-0.5);
    \draw[Blue](-3,-0.5)--(-2.5,-0.5)arc(-90:0:1cm and 0.5cm);
    \draw[Blue,->-](-1,0)arc(-180:0:1cm and 0.5cm);
    \draw[Blue](2.5,0.5)--(3,0.5)--(3,-0.5);
    \draw[Blue](-3.5,0.5)--(-3,0.5);
    \draw[Blue](-1,0)arc(180:0:1cm and 0.5cm);
    \draw[Blue,](-3,0.5)--(-2.5,0.5)arc(90:0:1cm and 0.5cm);
    \draw[Blue,](1.5,0) arc(180:90:1cm and 0.5cm);

    \draw[Blue](-3,-4.5)--(-3,-0.5);
    \draw[Blue](-3.5,-3.5)--(-3.5,0.5);
    \draw[Blue](3.5,-4.5)--(3.5,-0.5);
    \draw[Blue,dotted](3,-0.5)--(3,-3.5);

    \draw[Maroon,thick,fill=Red,fill opacity=0.6] (-1.5,0)--(-1.5,-0.5)arc(-180:0:0.25cm)--(-1,0)--cycle;
    \draw[Maroon,thick,fill=Red,fill opacity=0.6] (1,0)--(1.5,0)--(1.5,-4)--(1,-4)arc(0:180:1cm)--(-1.5,-4)--(-1.5,-3.5)to[in=-90,out=90](1,-0.5)--cycle;

}
\newcommand{\RTwoPlusdh}{

    \draw[Blue](1.5,-4) arc(-180:-90:1cm and 0.5cm)--(3.5,-4.5);
    \draw[Blue](-3,-4.5)--(-2.5,-4.5)arc(-90:0:1cm and 0.5cm);
    \draw[Blue,->-](-1,-4)arc(-180:0:1cm and 0.5cm);
    \draw[Blue](-3.5,-3.5)--(-3,-3.5);
    \draw[Blue,dotted](-1,-4)arc(180:0:1cm and 0.5cm);
    \draw[Blue,dotted](-3,-3.5)--(-2.5,-3.5)arc(90:0:1cm and 0.5cm);
    \draw[Blue,dotted](1.5,-4) arc(180:90:1cm and 0.5cm);
    \draw[Blue,dotted](2.5,-3.5)--(3,-3.5);

    \draw[Blue](1.5,0) arc(-180:-90:1cm and 0.5cm)--(3.5,-0.5);
    \draw[Blue](-3,-0.5)--(-2.5,-0.5)arc(-90:0:1cm and 0.5cm);
    \draw[Blue,->-](-1,0)arc(-180:0:1cm and 0.5cm);
    \draw[Blue](2.5,0.5)--(3,0.5)--(3,-0.5);
    \draw[Blue](-3.5,0.5)--(-3,0.5);
    \draw[Blue](-1,0)arc(180:0:1cm and 0.5cm);
    \draw[Blue,](-3,0.5)--(-2.5,0.5)arc(90:0:1cm and 0.5cm);
    \draw[Blue,](1.5,0) arc(180:90:1cm and 0.5cm);

    \draw[Blue](-3,-4.5)--(-3,-0.5);
    \draw[Blue](-3.5,-3.5)--(-3.5,0.5);
    \draw[Blue](3.5,-4.5)--(3.5,-0.5);
    \draw[Blue,dotted](3,-0.5)--(3,-3.5);

    \draw[Maroon,thick,fill=Red,fill opacity=0.6] (1.5,-4)--(1.5,-3.5)arc(0:180:0.25cm)--(1,-4)--cycle;
    \draw[Maroon,thick,fill=Red,fill opacity=0.6] (-1,0)arc(-180:0:1cm)--(1.5,0)--(1.5,-0.5)to[in=90,out=-90](-1,-3.5)--(-1,-4)--(-1.5,-4)--(-1.5,0)--cycle;

}

\newcommand{\Anotherf}{

    \draw[Blue](1.5,-4) arc(-180:-90:1cm and 0.5cm)--(3.5,-4.5);
    \draw[Blue](-3,-4.5)--(-2.5,-4.5)arc(-90:0:1cm and 0.5cm);
    \draw[Blue,->-](-1,-4)arc(-180:0:1cm and 0.5cm);
    \draw[Blue](-3.5,-3.5)--(-3,-3.5);
    \draw[Blue,dotted](-1,-4)arc(180:0:1cm and 0.5cm);
    \draw[Blue,dotted](-3,-3.5)--(-2.5,-3.5)arc(90:0:1cm and 0.5cm);
    \draw[Blue,dotted](1.5,-4) arc(180:90:1cm and 0.5cm);
    \draw[Blue,dotted](2.5,-3.5)--(3,-3.5);

    \draw[Maroon,thick,fill=Red,fill opacity=0.6] (1,-1.5)to[in=90,out=90, looseness=1.5](1.5,-1.5)--(1.5,-4)--(1,-4)arc(0:180:1cm)--(-1.5,-4)to[in=-90,out=90]cycle;

    \draw[Blue,->-](-3,-0.5)--(3.5,-0.5);
    \draw[Blue,->-](-3.5,0)--(3,0);
    \draw[Blue](3,0)--(3,-0.5);

    \draw[Blue](-3,-4.5)--(-3,-0.5);
    \draw[Blue](-3.5,-3.5)--(-3.5,0);
    \draw[Blue](3.5,-4.5)--(3.5,-0.5);
    \draw[Blue,dotted](3,-0.5)--(3,-3.5);

}

\newcommand{\Anotherg}{

    \draw[Blue,-<-](3.5,-4)--(-3,-4);
    \draw[Blue](-3.5,-3.5)--(-3,-3.5);
    \draw[Blue,dotted](-3,-3.5)--(3,-3.5);

    \draw[Blue](1.5,0) arc(-180:-90:1cm and 0.5cm)--(3.5,-0.5);
    \draw[Blue](-3,-0.5)--(-2.5,-0.5)arc(-90:0:1cm and 0.5cm);
    \draw[Blue,->-](-1,0)arc(-180:0:1cm and 0.5cm);
    \draw[Blue](2.5,0.5)--(3,0.5)--(3,-0.5);
    \draw[Blue](-3.5,0.5)--(-3,0.5);
    \draw[Blue](-1,0)arc(180:0:1cm and 0.5cm);
    \draw[Blue,](-3,0.5)--(-2.5,0.5)arc(90:0:1cm and 0.5cm);
    \draw[Blue,](1.5,0) arc(180:90:1cm and 0.5cm);

    \draw[Blue](-3,-4)--(-3,-0.5);
    \draw[Blue](-3.5,-3.5)--(-3.5,0.5);
    \draw[Blue](3.5,-4)--(3.5,-0.5);
    \draw[Blue,dotted](3,-0.5)--(3,-3.5);

    \draw[Maroon,thick,fill=Red,fill opacity=0.6] (-1,0)arc(-180:0:1cm)--(1.5,0)to[in=90,out=-90](-1,-3)to[in=-90,out=-90, looseness=1.5](-1.5,-3)--(-1.5,0)--cycle;

}

\newcommand{\RTwoMinusf}{
    \draw[Blue,->-](1.5,0) arc(-180:-90:1cm and 0.5cm)--(3.5,-0.5);
    \draw[Blue](3.5,-0.5)--(3.5,3.5);
    \draw[Blue,->-](-3,3.5)--(3.5,3.5);
    \draw[Blue](-3,3.5)--(-3,-0.5);
    \draw[Blue,->-](-3,-0.5)--(-2.5,-0.5)arc(-90:0:1cm and 0.5cm);
    \draw[Blue,-<-](-1,0)arc(-180:0:1cm and 0.5cm);
    \draw[Blue,-<-](-3.5,4)--(3,4);
    \draw[Blue](3,4)--(3,3.5);
    \draw[Blue](-3.5,4)--(-3.5,0.5)--(-3,0.5);
    \draw[Blue,dotted](-1,0)arc(180:0:1cm and 0.5cm);
    \draw[Blue,dotted](-3,0.5)--(-2.5,0.5)arc(90:0:1cm and 0.5cm);
    \draw[Blue,dotted](1.5,0) arc(180:90:1cm and 0.5cm);
    \draw[Blue,dotted](2.5,0.5)--(3,0.5)--(3,3.5);
    \draw[Blue] (1,0) arc(0:180:1cm);
    \draw[Blue] (1.5,0) arc(0:180:1.5cm);
}

\newcommand{\RTwoMinusg}{
    \draw[Blue,->-](1.5,0) arc(-180:-90:1cm and 0.5cm)--(3.5,-0.5);
    \draw[Blue](3.5,-0.5)--(3.5,-4);
    \draw[Blue,->-](-3,-4)--(3.5,-4);
    \draw[Blue](-3,-4)--(-3,-0.5)--(-2.5,-0.5)arc(-90:0:1cm and 0.5cm);
    \draw[Blue,-<-](-1,0)arc(-180:0:1cm and 0.5cm);
    \draw[Blue,dotted,-<-](-3.5,-3.5)--(3,-3.5);
    \draw[Blue](-3.5,-3.5)--(-3,-3.5);
    \draw[Blue](2.5,0.5)--(3,0.5)--(3,-0.5);
    \draw[Blue,dotted](3,-0.5)--(3,-3.5);
    \draw[Blue](-3.5,-3.5)--(-3.5,0.5)--(-3,0.5);
    \draw[Blue](-1,0)arc(180:0:1cm and 0.5cm);
    \draw[Blue,-<-](-3,0.5)--(-2.5,0.5)arc(90:0:1cm and 0.5cm);
    \draw[Blue,](1.5,0) arc(180:90:1cm and 0.5cm);
    \draw[Blue] (1,0) arc(0:-180:1cm);
    \draw[Blue] (1.5,0) arc(0:-180:1.5cm);
}

\newcommand{\DiffOne}{
    \draw[Blue,->-](3,-0.5)--(3,0.5)--(0,0.5);
    \draw[Blue](0,0.5)arc(90:270:1cm and 0.5cm);
    \draw[Blue,-<-](0,-0.5)--(1,-0.5);
    \draw[Blue](1,-0.5)--(2,-0.5);
    \draw[Blue,->-](2,-0.5)--(3.5,-0.5);
    \draw[Blue](3.5,-0.5)--(3.5,-4.5);
    \draw[Blue,-<-](3.5,-4.5)--(2.5,-4.5);
    \draw[Blue,->-](2.5,-4.5)--(-1,-4.5);
    \draw[Blue,-<-](-1,-4.5)--(-3,-4.5);
    \draw[Blue](-3,-4.5)--(-3,-0.5);
    \draw[Blue,->-](-3,-0.5)--(-2.5,-0.5);
    \draw[Blue](-2.5,-0.5)arc(-90:90:1cm and 0.5cm)--(-3.5,0.5);
    \draw[blue](-3.5,0.5)--(-3.5,-3.5)--(-3,-3.5);
    \draw[Blue,dotted](-3,-3.5)--(3,-3.5)--(3,-0.5);
    \fill[Red,fill opacity=0.6] (1,0.5)--(2,-0.5)--(2,-4.5)--(1,-3.5)--cycle;
    \draw[Maroon,thick](1,-0.5)--(1,0.5)--(2,-0.5)--(2,-4.5);
    \draw[Maroon,thick,dotted](1,-0.5)--(1,-3.5)--(2,-4.5);
    \draw[Blue] (-1.5,0)--(-1.5,-1)arc(-180:0:0.25cm)--(-1,0);
    \fill[Red, fill opacity=0.6](-2,-3.5)--(-2,-2)to[out=90,in=135,looseness=1.3](-1.25,-1.25)to[in=90,out=-45,looseness=0.5](-1,-3)--(-1,-4.5)--cycle;
    \draw[Maroon,thick](-1.25,-1.25)to[in=90,out=-45,looseness=0.5](-1,-3)--(-1,-4.5);
    \draw[Maroon,thick,dotted](-1,-4.5)--(-2,-3.5)--(-2,-2)to[out=90,in=135,looseness=1.3](-1.25,-1.25);
    
}

\newcommand{\DiffThree}{
    \draw[Blue,dotted](3,0.5)--(2.5,0.5);
    \draw[Blue,dotted](2.5,0.5)arc(90:180:1cm and 0.5cm);
    \draw[Blue,->-](1.5,0)arc(180:270:1cm and 0.5cm);
    \draw[Blue](3.5,-0.5)--(3.5,3.5);
    \draw[Blue,](3.5,3.5)--(3,3.5);
    \draw[Blue,-<-](3,3.5)--(2,3.5);
    \draw[Blue,->-](2,3.5)--(-1.5,3.5);
    \draw[Blue,-<-](-1.5,3.5)--(-3,3.5);
    \draw[Blue](-3,3.5)--(-3,-0.5);
    \draw[Blue,->-](-3,-0.5)--(0,-0.5);
    \draw[Blue,-<-](0,-0.5)arc(-90:0:1cm and 0.5cm);
    \draw[Blue,](2.5,-0.5)--(3.5,-0.5);
    \draw[Blue,dotted](3,0.5)--(3,3.5);
    \draw[Blue,](-3.5,4.5)--(3,4.5);
    \draw[Blue](3,4.5)--(3,3.5);
    \draw[Blue](-3,0.5)--(-3.5,0.5)--(-3.5,4.5);
    \draw[Blue,dotted](1,0)arc(0:90:1cm and 0.5cm)--(-3,0.5);
    \fill[Red,fill opacity=0.6] (-1,-0.5)--(-2,0.5)--(-2,4.5)--(-1,3.5)--cycle;
    \draw[Maroon,thick,dotted](-1,-0.5)--(-1,-0.5)--(-2,0.5)--(-2,4.5);
    \draw[Maroon,thick](-1,-0.5)--(-1,3.5)--(-2,4.5);

    \draw[Blue] (1.5,0)--(1.5,1)arc(0:180:0.25cm)--(1,0);
    \fill[Red, fill opacity=0.6](2,3.5)--(2,2)to[out=-90,in=-45,looseness=1.3](1.25,1.25)to[in=-90,out=135,looseness=0.5](1,3)--(1,4.5)--cycle;
    \draw[Maroon,thick,dotted](1.25,1.25)to[in=-90,out=135,looseness=0.5](1,3)--(1,4.5);
    \draw[Maroon,thick](1,4.5)--(2,3.5)--(2,2)to[out=-90,in=-45,looseness=1.3](1.25,1.25);

}

\newcommand{\DiffFour}{

    \draw[Blue,dotted](3,0.5)--(2.5,0.5);
    \draw[Blue,dotted](2.5,0.5)arc(90:180:1cm and 0.5cm);
    \draw[Blue](1.5,0)arc(180:270:1cm and 0.5cm);
    \draw[Blue,->-](-3,-0.5)--(0,-0.5);
    \draw[Blue,-<-](0,-0.5)arc(-90:0:1cm and 0.5cm);
    \draw[Blue,->-](2.5,-0.5)--(3.5,-0.5);
    \draw[Blue,dotted](3,0.5)--(3,3.5);
    \draw[Blue](-3,0.5)--(-3.5,0.5);
    \draw[Blue,dotted](1,0)arc(0:90:1cm and 0.5cm)--(-3,0.5);

    \draw[Blue,->-](1.5,4) arc(-180:-90:1cm and 0.5cm)--(3.5,3.5);
    \draw[Blue,->-](-3,3.5)--(-2.5,3.5)arc(-90:0:1cm and 0.5cm);
    \draw[Blue,-<-](-1,4)arc(-180:0:1cm and 0.5cm);
    \draw[Blue](2.5,4.5)--(3,4.5)--(3,3.5);
    \draw[Blue](-3.5,4.5)--(-3,4.5);
    \draw[Blue,](-1,4)arc(180:0:1cm and 0.5cm);
    \draw[Blue,](-3,4.5)--(-2.5,4.5)arc(90:0:1cm and 0.5cm);
    \draw[Blue,](1.5,4) arc(180:90:1cm and 0.5cm);

    \draw[Blue] (-1.5,4)--(-1.5,3)arc(-180:0:0.25cm)--(-1,4);
    \fill[Red, fill opacity=0.6](-2,0.5)--(-2,2)to[out=90,in=135,looseness=1.3](-1.25,2.75)to[in=90,out=-45,looseness=0.5](-1,1)--(-1,-0.5)--cycle;
    \draw[Maroon,thick](-1.25,2.75)to[in=90,out=-45,looseness=0.5](-1,1)--(-1,-0.5);
    \draw[Maroon,thick,dotted](-1,-0.5)--(-2,0.5)--(-2,2)to[out=90,in=135,looseness=1.3](-1.25,2.75);

    \draw[Blue](-3,-0.5)--(-3,3.5);
    \draw[Blue](-3.5,0.5)--(-3.5,4.5);
    \draw[Blue](3.5,-0.5)--(3.5,3.5);
    \draw[Blue,dotted](3,3.5)--(3,-0.5);

    \draw[Blue](1,0)--(1,4);
    \draw[Blue](1.5,0)--(1.5,4);
    
}

\newcommand{\DiffTwo}{

    \draw[Blue,->-](1.5,-4) arc(-180:-90:1cm and 0.5cm)--(3.5,-4.5);
    \draw[Blue,->-](-3,-4.5)--(-2.5,-4.5)arc(-90:0:1cm and 0.5cm);
    \draw[Blue,-<-](-1,-4)arc(-180:0:1cm and 0.5cm);
    \draw[Blue](-3.5,-3.5)--(-3,-3.5);
    \draw[Blue,dotted](-1,-4)arc(180:0:1cm and 0.5cm);
    \draw[Blue,dotted](-3,-3.5)--(-2.5,-3.5)arc(90:0:1cm and 0.5cm);
    \draw[Blue,dotted](1.5,-4) arc(180:90:1cm and 0.5cm);
    \draw[Blue,dotted](2.5,-3.5)--(3,-3.5);

    \draw[Blue,->-](3,-0.5)--(3,0.5)--(0,0.5);
    \draw[Blue](0,0.5)arc(90:270:1cm and 0.5cm);
    \draw[Blue,-<-](0,-0.5)--(1,-0.5);
    \draw[Blue](1,-0.5)--(2,-0.5);
    \draw[Blue,->-](2,-0.5)--(3.5,-0.5);
    \draw[Blue](3.5,-0.5)--(3.5,-4.5);
    \draw[Blue](-3,-4.5)--(-3,-0.5);
    \draw[Blue,->-](-3,-0.5)--(-2.5,-0.5);
    \draw[Blue](-2.5,-0.5)arc(-90:90:1cm and 0.5cm)--(-3.5,0.5);
    \draw[blue](-3.5,0.5)--(-3.5,-3.5)--(-3,-3.5);

    \draw[Blue](-3,-4.5)--(-3,-0.5);
    \draw[Blue](-3.5,-3.5)--(-3.5,0.5);
    \draw[Blue](3.5,-4.5)--(3.5,-0.5);
    \draw[Blue,dotted](3,-0.5)--(3,-3.5);
    
    \begin{scope}[yshift=-4cm]
        
    \draw[Blue] (1.5,0)--(1.5,1)arc(0:180:0.25cm)--(1,0);
    \fill[Red, fill opacity=0.6](2,3.5)--(2,2)to[out=-90,in=-45,looseness=1.3](1.25,1.25)to[in=-90,out=135,looseness=0.5](1,3)--(1,4.5)--cycle;
    \draw[Maroon,thick,dotted](1.25,1.25)to[in=-90,out=135,looseness=0.5](1,3)--(1,4.5);
    \draw[Maroon,thick](1,4.5)--(2,3.5)--(2,2)to[out=-90,in=-45,looseness=1.3](1.25,1.25);
    \end{scope}

    \draw[Blue] (-1,-4)--(-1,0);
    \draw[Blue] (-1.5,-4)--(-1.5,0);
}

\newcommand{\aOne}{

    \draw[Blue,->-](-3,-4.5)--(3.5,-4.5);
    \draw[Blue,dotted,-<-](-3.5,-3.5)--(3,-3.5);
    \draw[Blue](-3.5,-3.5)--(-3,-3.5);

    \draw[Blue,->-](-3,-0.5)--(-1,-0.5);
    \draw[Blue,-<-](-1,-0.5)--(2,-0.5);
    \draw[Blue,->-](2,-0.5)--(3.5,-0.5);

    \draw[Blue,-<-](-3.5,0.5)--(-2,0.5);
    \draw[Blue,->-](-2,0.5)--(1,0.5);
    \draw[Blue,-<-](1,0.5)--(3,0.5);
    
    \draw[Blue](-3.5,0.5)--(-3.5,-3.5);
    \draw[Blue](-3,-0.5)--(-3,-4.5);
    \draw[Blue](3.5,-0.5)--(3.5,-4.5);
    \draw[Blue](3,0.5)--(3,-0.5);
    \draw[Blue,dotted](3,-0.5)--(3,-3.5);

    \begin{scope}
        \clip (0.5,-0.5) circle (1.5cm);
        \fill[Red,fill opacity=0.6](-1,-0.5)arc(-180:0:1.5cm)--(1,0.5)arc(0:-180:1.5cm)--cycle;
    \end{scope}
    \fill[Red,fill opacity=0.6](-1,-0.5)--(-2,0.5)arc(180:225:1.5cm)--($(-135:1.5)+(0.5,-0.5)$) arc(-135:-180:1.5cm)--cycle;
    \draw[Maroon,thick](1,0.5)--(2,-0.5)arc(0:-180:1.5cm)--(-2,0.5);
    \begin{scope}
        \clip (-3,-0.5)rectangle (3,0.5);
        \draw[Maroon,thick] (-2,0.5)arc(-180:0:1.5cm);
    \end{scope}
    \begin{scope}
        \clip (-3,-0.5) rectangle (3,-3);
        \clip (0.5,-0.5) circle (1.5cm);   
        \draw[Maroon,thick,dotted] (-0.5,0.5) circle (1.5cm);
    \end{scope}
    \draw[Maroon,thick,dotted] ($(-135:1.5)+(0.5,-0.5)$)--($(-135:1.5)+(-0.5,0.5)$);

}

\newcommand{\aTwo}{

    \draw[Blue,dotted,-<-](-3.5,-3.5)--(-1.5,-3.5);
    \draw[Blue,dotted,->-](-1.5,-3.5)--(1,-3.5);
    \draw[Blue,dotted,-<-](1,-3.5)--(3,-3.5);

    \draw[Blue](-3.5,-3.5)--(-3,-3.5);

    \draw[Blue,->-](-3,-4.5)--(-1,-4.5);
    \draw[Blue,-<-](-1,-4.5)--(2,-4.5);
    \draw[Blue,->-](2,-4.5)--(3.5,-4.5);

    \draw[Blue,](-3,-0.5)--(-1,-0.5);
    \draw[Blue,->-](-1,-0.5)--(2,-0.5);
    \draw[Blue,](2,-0.5)--(3.5,-0.5);

    \draw[Blue,](-3.5,0.5)--(-2,0.5);
    \draw[Blue,-<-](-2,0.5)--(1,0.5);
    \draw[Blue,](1,0.5)--(3,0.5);
    
    \draw[Blue](-3.5,0.5)--(-3.5,-3.5);
    \draw[Blue](-3,-0.5)--(-3,-4.5);
    \draw[Blue](3.5,-0.5)--(3.5,-4.5);
    \draw[Blue](3,0.5)--(3,-0.5);
    \draw[Blue,dotted](3,-0.5)--(3,-3.5);

    \begin{scope}
        \clip (0.5,-4.5) circle (1.5cm);
        \fill[Red,fill opacity=0.6] (1,-3.5)arc(0:90:1.5cm)--(2,-2)--(2,-4.5)--cycle;
    \end{scope}
    \fill[Red,fill opacity=0.6](-1,-4.5)--(-2,-3.5)arc(180:45:1.5cm)--($(45:1.5)+(0.5,-4.5)$) arc(45:180:1.5cm)--cycle;

    \draw[Maroon,thick](2,-4.5)arc(0:180:1.5cm);
    \begin{scope}
        \clip (0.5,-4.5) circle (1.5cm);
        \draw[Maroon,thick,dotted] (1,-3.5)arc(0:90:1.5cm)--(2,-2)--(2,-4.5)--cycle;
    \end{scope}
    \draw[Maroon,thick,dotted](-1,-4.5)--(-2,-3.5)arc(180:45:1.5cm)--($(45:1.5)+(0.5,-4.5)$);

}

\newcommand{\aTwoSmall}{

    \draw[Blue,dotted](-3.5,-3.5)--(-1.5,-3.5);
    \draw[Blue,dotted](-1.5,-3.5)--(1,-3.5);
    \draw[Blue,dotted](1,-3.5)--(3,-3.5);

    \draw[Blue](-3.5,-3.5)--(-3,-3.5);

    \draw[Blue](-3,-4.5)--(-1,-4.5);
    \draw[Blue](-1,-4.5)--(2,-4.5);
    \draw[Blue](2,-4.5)--(3.5,-4.5);

    \draw[Blue,](-3,-0.5)--(-1,-0.5);
    \draw[Blue,](-1,-0.5)--(2,-0.5);
    \draw[Blue,](2,-0.5)--(3.5,-0.5);

    \draw[Blue,](-3.5,0.5)--(-2,0.5);
    \draw[Blue,](-2,0.5)--(1,0.5);
    \draw[Blue,](1,0.5)--(3,0.5);
    
    \draw[Blue](-3.5,0.5)--(-3.5,-3.5);
    \draw[Blue](-3,-0.5)--(-3,-4.5);
    \draw[Blue](3.5,-0.5)--(3.5,-4.5);
    \draw[Blue](3,0.5)--(3,-0.5);
    \draw[Blue,dotted](3,-0.5)--(3,-3.5);

    \begin{scope}
        \clip (0.5,-4.5) circle (1.5cm);
        \fill[Red,fill opacity=0.6] (1,-3.5)arc(0:90:1.5cm)--(2,-2)--(2,-4.5)--cycle;
    \end{scope}
    \fill[Red,fill opacity=0.6](-1,-4.5)--(-2,-3.5)arc(180:45:1.5cm)--($(45:1.5)+(0.5,-4.5)$) arc(45:180:1.5cm)--cycle;

    \draw[Maroon,thick](2,-4.5)arc(0:180:1.5cm);
    \begin{scope}
        \clip (0.5,-4.5) circle (1.5cm);
        \draw[Maroon,thick,dotted] (1,-3.5)arc(0:90:1.5cm)--(2,-2)--(2,-4.5)--cycle;
    \end{scope}
    \draw[Maroon,thick,dotted](-1,-4.5)--(-2,-3.5)arc(180:45:1.5cm)--($(45:1.5)+(0.5,-4.5)$);

}

\newcommand{\RTwoMinushOne}{

    \draw[Blue,dotted](0,-3.5)--(3,-3.5);
    \draw[Blue,dotted](0,-3.5)arc(90:180:1cm and 0.5cm);
    \draw[Blue](-1,-4)arc(180:270:1cm and 0.5cm);
    \draw[Blue](-3,-4.5)--(-2.5,-4.5);
    \draw[Blue,->-](-2.5,-4.5)arc(-90:0:1cm and 0.5cm);
    \draw[Blue,-<-](0,-4.5)--(2,-4.5);
    \draw[Blue,->-](2,-4.5)--(3.5,-4.5);
    \draw[Blue](-3,-3.5)--(-3.5,-3.5);
    \draw[Blue,dotted](-1.5,-4)arc(0:90:1cm and 0.5cm)--(-3,-3.5);

    \draw[Blue,->-](1.5,0) arc(-180:-90:1cm and 0.5cm)--(3.5,-0.5);
    \draw[Blue,->-](-3,-0.5)--(-2.5,-0.5)arc(-90:0:1cm and 0.5cm);
    \draw[Blue,-<-](-1,0)arc(-180:0:1cm and 0.5cm);
    \draw[Blue](2.5,0.5)--(3,0.5)--(3,-0.5);
    \draw[Blue](-3.5,0.5)--(-3,0.5);
    \draw[Blue](-1,0)arc(180:0:1cm and 0.5cm);
    \draw[Blue,](-3,0.5)--(-2.5,0.5)arc(90:0:1cm and 0.5cm);
    \draw[Blue,](1.5,0) arc(180:90:1cm and 0.5cm);

    \draw[Blue](-3,-4.5)--(-3,-0.5);
    \draw[Blue](-3.5,-3.5)--(-3.5,0.5);
    \draw[Blue](3.5,-4.5)--(3.5,-0.5);
    \draw[Blue,dotted](3,-0.5)--(3,-3.5);

    \draw[Blue] (-1,0)arc(-180:0:1cm);
    \draw[Blue](1.5,0)to[in=90,out=-90](-1,-4);
    \draw[Blue](-1.5,-4)--(-1.5,0);

    \fill[Red, fill opacity=0.6](1,-3.5)--(1,-2)to[out=90,in=-150,looseness=1](1.25,-1)to[in=90,out=-30,looseness=0.9](2,-3)--(2,-4.5)--cycle;
    \draw[Maroon,thick](1.25,-1)to[in=90,out=-30,looseness=0.9](2,-3)--(2,-4.5);
    \draw[Maroon,thick,dotted](2,-4.5)--(1,-3.5)--(1,-2)to[out=90,in=-150,looseness=1](1.25,-1);

}

\newcommand{\RTwoMinushTwo}{

    \draw[Blue,->-](1.5,-4) arc(-180:-90:1cm and 0.5cm)--(3.5,-4.5);
    \draw[Blue,->-](-3,-4.5)--(-2.5,-4.5)arc(-90:0:1cm and 0.5cm);
    \draw[Blue,-<-](-1,-4)arc(-180:0:1cm and 0.5cm);
    \draw[Blue](-3.5,-3.5)--(-3,-3.5);
    \draw[Blue,dotted](-1,-4)arc(180:0:1cm and 0.5cm);
    \draw[Blue,dotted](-3,-3.5)--(-2.5,-3.5)arc(90:0:1cm and 0.5cm);
    \draw[Blue,dotted](1.5,-4) arc(180:90:1cm and 0.5cm);
    \draw[Blue,dotted](2.5,-3.5)--(3,-3.5);

    \draw[Blue](3,-0.5)--(3,0.5)--(2.5,0.5);
    \draw[Blue](2.5,0.5)arc(90:270:1cm and 0.5cm);
    \draw[Blue,->-](2.5,-0.5)--(3.5,-0.5);
    \draw[Blue,](0,-0.5)arc(-90:90:1cm and 0.5cm)--(-3.5,0.5);
    \draw[Blue,->-](-3,-0.5)--(-2,-0.5);
    \draw[Blue,](-2,-0.5)--(-1,-0.5);
    \draw[Blue,-<-](-1,-0.5)--(0,-0.5);

    \draw[Blue](-3,-4.5)--(-3,-0.5);
    \draw[Blue](-3.5,-3.5)--(-3.5,0.5);
    \draw[Blue](3.5,-4.5)--(3.5,-0.5);
    \draw[Blue,dotted](3,-0.5)--(3,-3.5);

    \draw[Blue] (-1.5,-4)to[in=-90,out=90](1,0);
    \draw[Blue] (1.5,0)--(1.5,-4);
    \draw[Blue] (1,-4)arc(0:180:1cm);

    \fill[Red, fill opacity=0.6](-1,-0.5)--(-1,-2)to[out=-90,in=30,looseness=1](-1.25,-3)to[in=-90,out=150,looseness=0.9](-2,-1)--(-2,0.5)--cycle;
    \draw[Maroon,thick,dotted](-1.25,-3)to[in=-90,out=150,looseness=0.9](-2,-1)--(-2,-0.5);
    \draw[Maroon,thick](-2,-0.5)--(-2,0.5)--(-1,-0.5)--(-1,-2)to[out=-90,in=30,looseness=1](-1.25,-3);
    
}

\newcommand{\RTwoMinusfg}{

    \begin{scope}[yshift=-4cm]
        
    \draw[Blue,->-](1.5,0) arc(-180:-90:1cm and 0.5cm)--(3.5,-0.5);
    \draw[Blue,->-](-3,-0.5)--(-2.5,-0.5)arc(-90:0:1cm and 0.5cm);
    \draw[Blue,-<-](-1,0)arc(-180:0:1cm and 0.5cm);
    \draw[Blue](-3.5,0.5)--(-3,0.5);
    \draw[Blue,dotted](-1,0)arc(180:0:1cm and 0.5cm);
    \draw[Blue,dotted](-3,0.5)--(-2.5,0.5)arc(90:0:1cm and 0.5cm);
    \draw[Blue,dotted](1.5,0) arc(180:90:1cm and 0.5cm);
    \draw[Blue,dotted](2.5,0.5)--(3,0.5)--(3,3.5);
    \draw[Blue] (1,0) arc(0:180:1cm);
    \draw[Blue] (1.5,0) arc(0:180:1.5cm);

    \end{scope}

    \draw[Blue,->-](1.5,0) arc(-180:-90:1cm and 0.5cm)--(3.5,-0.5);
    \draw[Blue](-3,-0.5)--(-2.5,-0.5)arc(-90:0:1cm and 0.5cm);
    \draw[Blue,-<-](-1,0)arc(-180:0:1cm and 0.5cm);
    \draw[Blue](2.5,0.5)--(3,0.5)--(3,-0.5);
    \draw[Blue](-3.5,0.5)--(-3,0.5);
    \draw[Blue](-1,0)arc(180:0:1cm and 0.5cm);
    \draw[Blue,-<-](-3,0.5)--(-2.5,0.5)arc(90:0:1cm and 0.5cm);
    \draw[Blue,](1.5,0) arc(180:90:1cm and 0.5cm);
    \draw[Blue] (1,0) arc(0:-180:1cm);
    \draw[Blue] (1.5,0) arc(0:-180:1.5cm);

    \draw[Blue](-3,-4.5)--(-3,-0.5);
    \draw[Blue](-3.5,-3.5)--(-3.5,0.5);
    \draw[Blue](3.5,-4.5)--(3.5,-0.5);
    \draw[Blue,dotted](3,-0.5)--(3,-3.5);

}

\newcommand{\RTwoMinusI}{

    \begin{scope}[yshift=-4cm]
        
    \draw[Blue,->-](1.5,0) arc(-180:-90:1cm and 0.5cm)--(3.5,-0.5);
    \draw[Blue,->-](-3,-0.5)--(-2.5,-0.5)arc(-90:0:1cm and 0.5cm);
    \draw[Blue,-<-](-1,0)arc(-180:0:1cm and 0.5cm);
    \draw[Blue](-3.5,0.5)--(-3,0.5);
    \draw[Blue,dotted](-1,0)arc(180:0:1cm and 0.5cm);
    \draw[Blue,dotted](-3,0.5)--(-2.5,0.5)arc(90:0:1cm and 0.5cm);
    \draw[Blue,dotted](1.5,0) arc(180:90:1cm and 0.5cm);
    \draw[Blue,dotted](2.5,0.5)--(3,0.5)--(3,3.5);

    \end{scope}

    \draw[Blue,->-](1.5,0) arc(-180:-90:1cm and 0.5cm)--(3.5,-0.5);
    \draw[Blue](-3,-0.5)--(-2.5,-0.5)arc(-90:0:1cm and 0.5cm);
    \draw[Blue,-<-](-1,0)arc(-180:0:1cm and 0.5cm);
    \draw[Blue](2.5,0.5)--(3,0.5)--(3,-0.5);
    \draw[Blue](-3.5,0.5)--(-3,0.5);
    \draw[Blue](-1,0)arc(180:0:1cm and 0.5cm);
    \draw[Blue,-<-](-3,0.5)--(-2.5,0.5)arc(90:0:1cm and 0.5cm);
    \draw[Blue,](1.5,0) arc(180:90:1cm and 0.5cm);

    \draw[Blue](-3,-4.5)--(-3,-0.5);
    \draw[Blue](-3.5,-3.5)--(-3.5,0.5);
    \draw[Blue](3.5,-4.5)--(3.5,-0.5);
    \draw[Blue,dotted](3,-0.5)--(3,-3.5);

    \draw[Blue] (-1.5,-4)--(-1.5,0);
    \draw[Blue] (1.5,-4)--(1.5,0);
    \draw[Blue] (-1,-4)--(-1,0);
    \draw[Blue] (1,-4)--(1,0);

}

\newcommand{\aTwog}{

    \draw[Blue,dotted,-<-](-3.5,-3.5)--(-1.5,-3.5);
    \draw[Blue,dotted,->-](-1.5,-3.5)--(1,-3.5);
    \draw[Blue,dotted,-<-](1,-3.5)--(3,-3.5);

    \draw[Blue](-3.5,-3.5)--(-3,-3.5);

    \draw[Blue,->-](-3,-4.5)--(-1,-4.5);
    \draw[Blue,-<-](-1,-4.5)--(2,-4.5);
    \draw[Blue,->-](2,-4.5)--(3.5,-4.5);

    \draw[Blue,->-](1.5,0) arc(-180:-90:1cm and 0.5cm)--(3.5,-0.5);
    \draw[Blue](3.5,-0.5)--(3.5,-4.5);
    \draw[Blue](-3,-4.5)--(-3,-0.5)--(-2.5,-0.5)arc(-90:0:1cm and 0.5cm);
    \draw[Blue,-<-](-1,0)arc(-180:0:1cm and 0.5cm);
    \draw[Blue](2.5,0.5)--(3,0.5)--(3,-0.5);
    \draw[Blue,dotted](3,-0.5)--(3,-3.5);
    \draw[Blue](-3.5,-3.5)--(-3.5,0.5)--(-3,0.5);
    \draw[Blue](-1,0)arc(180:0:1cm and 0.5cm);
    \draw[Blue,-<-](-3,0.5)--(-2.5,0.5)arc(90:0:1cm and 0.5cm);
    \draw[Blue,](1.5,0) arc(180:90:1cm and 0.5cm);
    \draw[Blue] (1,0) arc(0:-180:1cm);
    \draw[Blue] (1.5,0) arc(0:-180:1.5cm);

    \begin{scope}
        \clip (0.5,-4.5) circle (1.5cm);
        \fill[Red,fill opacity=0.6] (1,-3.5)arc(0:90:1.5cm)--(2,-2)--(2,-4.5)--cycle;
    \end{scope}
    \fill[Red,fill opacity=0.6](-1,-4.5)--(-2,-3.5)arc(180:45:1.5cm)--($(45:1.5)+(0.5,-4.5)$) arc(45:180:1.5cm)--cycle;

    \draw[Maroon,thick](2,-4.5)arc(0:180:1.5cm);
    \begin{scope}
        \clip (0.5,-4.5) circle (1.5cm);
        \draw[Maroon,thick,dotted] (1,-3.5)arc(0:90:1.5cm)--(2,-2)--(2,-4.5)--cycle;
    \end{scope}
    \draw[Maroon,thick,dotted](-1,-4.5)--(-2,-3.5)arc(180:45:1.5cm)--($(45:1.5)+(0.5,-4.5)$);

}

\newcommand{\faOne}{
    \begin{scope}[yshift=-4cm]
        
    \draw[Blue,->-](1.5,0) arc(-180:-90:1cm and 0.5cm)--(3.5,-0.5);
    \draw[Blue,->-](-3,-0.5)--(-2.5,-0.5)arc(-90:0:1cm and 0.5cm);
    \draw[Blue,-<-](-1,0)arc(-180:0:1cm and 0.5cm);
    \draw[Blue](-3.5,0.5)--(-3,0.5);
    \draw[Blue,dotted](-1,0)arc(180:0:1cm and 0.5cm);
    \draw[Blue,dotted](-3,0.5)--(-2.5,0.5)arc(90:0:1cm and 0.5cm);
    \draw[Blue,dotted](1.5,0) arc(180:90:1cm and 0.5cm);
    \draw[Blue,dotted](2.5,0.5)--(3,0.5)--(3,3.5);
    \draw[Blue] (1,0) arc(0:180:1cm);
    \draw[Blue] (1.5,0) arc(0:180:1.5cm);

    \end{scope}

    \draw[Blue,->-](-3,-0.5)--(-1,-0.5);
    \draw[Blue,-<-](-1,-0.5)--(2,-0.5);
    \draw[Blue,->-](2,-0.5)--(3.5,-0.5);

    \draw[Blue,-<-](-3.5,0.5)--(-2,0.5);
    \draw[Blue,->-](-2,0.5)--(1,0.5);
    \draw[Blue,-<-](1,0.5)--(3,0.5);
    
    \draw[Blue](-3.5,0.5)--(-3.5,-3.5);
    \draw[Blue](-3,-0.5)--(-3,-4.5);
    \draw[Blue](3.5,-0.5)--(3.5,-4.5);
    \draw[Blue](3,0.5)--(3,-0.5);
    \draw[Blue,dotted](3,-0.5)--(3,-3.5);

    \begin{scope}
        \clip (0.5,-0.5) circle (1.5cm);
        \fill[Red,fill opacity=0.6](-1,-0.5)arc(-180:0:1.5cm)--(1,0.5)arc(0:-180:1.5cm)--cycle;
    \end{scope}
    \fill[Red,fill opacity=0.6](-1,-0.5)--(-2,0.5)arc(180:225:1.5cm)--($(-135:1.5)+(0.5,-0.5)$) arc(-135:-180:1.5cm)--cycle;
    \draw[Maroon,thick](1,0.5)--(2,-0.5)arc(0:-180:1.5cm)--(-2,0.5);
    \begin{scope}
        \clip (-3,-0.5)rectangle (3,0.5);
        \draw[Maroon,thick] (-2,0.5)arc(-180:0:1.5cm);
    \end{scope}
    \begin{scope}
        \clip (-3,-0.5) rectangle (3,-3);
        \clip (0.5,-0.5) circle (1.5cm);   
        \draw[Maroon,thick,dotted] (-0.5,0.5) circle (1.5cm);
    \end{scope}
    \draw[Maroon,thick,dotted] ($(-135:1.5)+(0.5,-0.5)$)--($(-135:1.5)+(-0.5,0.5)$);

}

\newcommand{\dOnehOne}{

    \draw[Blue,dotted,-<-](-3.5,-3.5)--(-1.5,-3.5);
    \draw[Blue,dotted,->-](-1.5,-3.5)--(1,-3.5);
    \draw[Blue,dotted,-<-](1,-3.5)--(3,-3.5);

    \draw[Blue](-3.5,-3.5)--(-3,-3.5);

    \draw[Blue,->-](-3,-4.5)--(-1,-4.5);
    \draw[Blue,-<-](-1,-4.5)--(2,-4.5);
    \draw[Blue,->-](2,-4.5)--(3.5,-4.5);

    \draw[Blue,->-](1.5,0) arc(-180:-90:1cm and 0.5cm)--(3.5,-0.5);
    \draw[Blue](3.5,-0.5)--(3.5,-4.5);
    \draw[Blue](-3,-4.5)--(-3,-0.5)--(-2.5,-0.5)arc(-90:0:1cm and 0.5cm);
    \draw[Blue,-<-](-1,0)arc(-180:0:1cm and 0.5cm);
    \draw[Blue](2.5,0.5)--(3,0.5)--(3,-0.5);
    \draw[Blue,dotted](3,-0.5)--(3,-3.5);
    \draw[Blue](-3.5,-3.5)--(-3.5,0.5)--(-3,0.5);
    \draw[Blue](-1,0)arc(180:0:1cm and 0.5cm);
    \draw[Blue,-<-](-3,0.5)--(-2.5,0.5)arc(90:0:1cm and 0.5cm);
    \draw[Blue,](1.5,0) arc(180:90:1cm and 0.5cm);
    \draw[Blue] (1,0) arc(0:-180:1cm);
    \draw[Blue] (1.5,0) arc(0:-180:1.5cm);

    \fill[Red, fill opacity=0.6](1,-3.5)--(1,-2)to[out=90,in=-150,looseness=1](1.25,-0.9)to[in=90,out=-30,looseness=0.9](2,-3)--(2,-4.5)--cycle;
    \draw[Maroon,thick](1.25,-0.9)to[in=90,out=-30,looseness=0.9](2,-3)--(2,-4.5);
    \draw[Maroon,thick,dotted](2,-4.5)--(1,-3.5)--(1,-2)to[out=90,in=-150,looseness=1](1.25,-0.9);

    \begin{scope}[xshift=-2.5cm]
        \fill[Red, fill opacity=0.6](1,-3.5)--(1,-2)to[out=90,in=-150,looseness=1](1.25,-0.9)to[in=90,out=-30,looseness=0.9](2,-3)--(2,-4.5)--cycle;
        \draw[Maroon,thick](1.25,-0.9)to[in=90,out=-30,looseness=0.9](2,-3)--(2,-4.5);
        \draw[Maroon,thick,dotted](2,-4.5)--(1,-3.5)--(1,-2)to[out=90,in=-150,looseness=1](1.25,-0.9);
    \end{scope}

}

\newcommand{\hTwodThree}{

    \begin{scope}[yshift=-4cm]
        
    \draw[Blue,->-](1.5,0) arc(-180:-90:1cm and 0.5cm)--(3.5,-0.5);
    \draw[Blue,->-](-3,-0.5)--(-2.5,-0.5)arc(-90:0:1cm and 0.5cm);
    \draw[Blue,-<-](-1,0)arc(-180:0:1cm and 0.5cm);
    \draw[Blue](-3.5,0.5)--(-3,0.5);
    \draw[Blue,dotted](-1,0)arc(180:0:1cm and 0.5cm);
    \draw[Blue,dotted](-3,0.5)--(-2.5,0.5)arc(90:0:1cm and 0.5cm);
    \draw[Blue,dotted](1.5,0) arc(180:90:1cm and 0.5cm);
    \draw[Blue,dotted](2.5,0.5)--(3,0.5)--(3,3.5);
    \draw[Blue] (1,0) arc(0:180:1cm);
    \draw[Blue] (1.5,0) arc(0:180:1.5cm);

    \end{scope}

    \draw[Blue,->-](-3,-0.5)--(-1,-0.5);
    \draw[Blue,-<-](-1,-0.5)--(2,-0.5);
    \draw[Blue,->-](2,-0.5)--(3.5,-0.5);

    \draw[Blue,-<-](-3.5,0.5)--(-2,0.5);
    \draw[Blue,->-](-2,0.5)--(1,0.5);
    \draw[Blue,-<-](1,0.5)--(3,0.5);
    
    \draw[Blue](-3.5,0.5)--(-3.5,-3.5);
    \draw[Blue](-3,-0.5)--(-3,-4.5);
    \draw[Blue](3.5,-0.5)--(3.5,-4.5);
    \draw[Blue](3,0.5)--(3,-0.5);
    \draw[Blue,dotted](3,-0.5)--(3,-3.5);

    \fill[Red, fill opacity=0.6](-1,-0.5)--(-1,-2)to[out=-90,in=30,looseness=1](-1.25,-3.1)to[in=-90,out=150,looseness=0.9](-2,-1)--(-2,0.5)--cycle;
    \draw[Maroon,thick,dotted](-1.25,-3.1)to[in=-90,out=150,looseness=0.9](-2,-1)--(-2,-0.5);
    \draw[Maroon,thick](-2,-0.5)--(-2,0.5)--(-1,-0.5)--(-1,-2)to[out=-90,in=30,looseness=1](-1.25,-3.1);

    \begin{scope}[xshift=2.5cm]
            \fill[Red, fill opacity=0.6](-1,-0.5)--(-1,-2)to[out=-90,in=30,looseness=1](-1.25,-3.1)to[in=-90,out=150,looseness=0.9](-2,-1)--(-2,0.5)--cycle;
            \draw[Maroon,thick,dotted](-1.25,-3.1)to[in=-90,out=150,looseness=0.9](-2,-1)--(-2,-0.5);
            \draw[Maroon,thick](-2,-0.5)--(-2,0.5)--(-1,-0.5)--(-1,-2)to[out=-90,in=30,looseness=1](-1.25,-3.1);
    \end{scope}

}

\newcommand{\faOneBubble}{

    \begin{scope}[yshift=-4cm]
        
    \draw[Blue,->-](1.5,0) arc(-180:-90:1cm and 0.5cm)--(3.5,-0.5);
    \draw[Blue,->-](-3,-0.5)--(-2.5,-0.5)arc(-90:0:1cm and 0.5cm);
    \draw[Blue,-<-](-1,0)arc(-180:0:1cm and 0.5cm);
    \draw[Blue](-3.5,0.5)--(-3,0.5);
    \draw[Blue,dotted](-1,0)arc(180:0:1cm and 0.5cm);
    \draw[Blue,dotted](-3,0.5)--(-2.5,0.5)arc(90:0:1cm and 0.5cm);
    \draw[Blue,dotted](1.5,0) arc(180:90:1cm and 0.5cm);
    \draw[Blue,dotted](2.5,0.5)--(3,0.5)--(3,3.5);
    \draw[Blue] (1,0) arc(0:180:1cm);
    \draw[Blue] (1.5,0) arc(0:180:1.5cm);

    \end{scope}

    \draw[Blue,->-](-3,-0.5)--(-1,-0.5);
    \draw[Blue,-<-](-1,-0.5)--(2,-0.5);
    \draw[Blue,->-](2,-0.5)--(3.5,-0.5);

    \draw[Blue,-<-](-3.5,0.5)--(-2,0.5);
    \draw[Blue,->-](-2,0.5)--(1,0.5);
    \draw[Blue,-<-](1,0.5)--(3,0.5);
    
    \draw[Blue](-3.5,0.5)--(-3.5,-3.5);
    \draw[Blue](-3,-0.5)--(-3,-4.5);
    \draw[Blue](3.5,-0.5)--(3.5,-4.5);
    \draw[Blue](3,0.5)--(3,-0.5);
    \draw[Blue,dotted](3,-0.5)--(3,-3.5);

    \draw[Maroon,thick, fill=Red, fill opacity=0.6] (-2,0.5)to[in=-120,out=-30,looseness=0.5](1,0.5)--(2,-0.5)to[in=-30,out=-120,looseness=0.5](-1,-0.5)--cycle;
    
    \draw[Maroon,thick, ,dotted] ($(30:1.5cm)+(0,-4)$)to[out=60,in=0](0,-1.5)to[out=180,in=120]($(150:1.5cm)+(0,-4)$);
    \draw[Maroon,thick, ] ($(150:1.5cm)+(0,-4)$)to[out=80,in=180](0,-2)to[out=0,in=100]($(30:1.5cm)+(0,-4)$);
    \fill[Red, fill opacity=0.6] ($(30:1.5cm)+(0,-4)$)to[out=60,in=0](0,-1.5)to[out=180,in=120]($(150:1.5cm)+(0,-4)$)to[out=80,in=180](0,-2)to[out=0,in=100]cycle;

}

\newcommand{\RTwoMinusCups}{

    \begin{scope}[yshift=-4cm]
        
    \draw[Blue,->-](1.5,0) arc(-180:-90:1cm and 0.5cm)--(3.5,-0.5);
    \draw[Blue,->-](-3,-0.5)--(-2.5,-0.5)arc(-90:0:1cm and 0.5cm);
    \draw[Blue,-<-](-1,0)arc(-180:0:1cm and 0.5cm);
    \draw[Blue](-3.5,0.5)--(-3,0.5);
    \draw[Blue,dotted](-1,0)arc(180:0:1cm and 0.5cm);
    \draw[Blue,dotted](-3,0.5)--(-2.5,0.5)arc(90:0:1cm and 0.5cm);
    \draw[Blue,dotted](1.5,0) arc(180:90:1cm and 0.5cm);
    \draw[Blue,dotted](2.5,0.5)--(3,0.5)--(3,3.5);
    \draw[Blue] (1,0) arc(0:180:1cm);

    \end{scope}

    \draw[Blue,->-](1.5,0) arc(-180:-90:1cm and 0.5cm)--(3.5,-0.5);
    \draw[Blue](-3,-0.5)--(-2.5,-0.5)arc(-90:0:1cm and 0.5cm);
    \draw[Blue,-<-](-1,0)arc(-180:0:1cm and 0.5cm);
    \draw[Blue](2.5,0.5)--(3,0.5)--(3,-0.5);
    \draw[Blue](-3.5,0.5)--(-3,0.5);
    \draw[Blue](-1,0)arc(180:0:1cm and 0.5cm);
    \draw[Blue,-<-](-3,0.5)--(-2.5,0.5)arc(90:0:1cm and 0.5cm);
    \draw[Blue,](1.5,0) arc(180:90:1cm and 0.5cm);
    \draw[Blue] (1,0) arc(0:-180:1cm);

    \draw[Blue](-3,-4.5)--(-3,-0.5);
    \draw[Blue](-3.5,-3.5)--(-3.5,0.5);
    \draw[Blue](3.5,-4.5)--(3.5,-0.5);
    \draw[Blue,dotted](3,-0.5)--(3,-3.5);

    \draw[Blue](1.5,-4)--(1.5,0);
    \draw[Blue](-1.5,-4)--(-1.5,0);

}

\newcommand{\RTwoMinushd}{
    \begin{scope}[yshift=-4cm]
        
    \draw[Blue,->-](1.5,0) arc(-180:-90:1cm and 0.5cm)--(3.5,-0.5);
    \draw[Blue,->-](-3,-0.5)--(-2.5,-0.5)arc(-90:0:1cm and 0.5cm);
    \draw[Blue,-<-](-1,0)arc(-180:0:1cm and 0.5cm);
    \draw[Blue](-3.5,0.5)--(-3,0.5);
    \draw[Blue,dotted](-1,0)arc(180:0:1cm and 0.5cm);
    \draw[Blue,dotted](-3,0.5)--(-2.5,0.5)arc(90:0:1cm and 0.5cm);
    \draw[Blue,dotted](1.5,0) arc(180:90:1cm and 0.5cm);
    \draw[Blue,dotted](2.5,0.5)--(3,0.5)--(3,3.5);

    \end{scope}

    \draw[Blue,->-](1.5,0) arc(-180:-90:1cm and 0.5cm)--(3.5,-0.5);
    \draw[Blue](-3,-0.5)--(-2.5,-0.5)arc(-90:0:1cm and 0.5cm);
    \draw[Blue,-<-](-1,0)arc(-180:0:1cm and 0.5cm);
    \draw[Blue](2.5,0.5)--(3,0.5)--(3,-0.5);
    \draw[Blue](-3.5,0.5)--(-3,0.5);
    \draw[Blue](-1,0)arc(180:0:1cm and 0.5cm);
    \draw[Blue,-<-](-3,0.5)--(-2.5,0.5)arc(90:0:1cm and 0.5cm);
    \draw[Blue,](1.5,0) arc(180:90:1cm and 0.5cm);

    \draw[Blue](-3,-4.5)--(-3,-0.5);
    \draw[Blue](-3.5,-3.5)--(-3.5,0.5);
    \draw[Blue](3.5,-4.5)--(3.5,-0.5);
    \draw[Blue,dotted](3,-0.5)--(3,-3.5);

    \draw[Blue] (-1.5,0)--(-1.5,-0.5)arc(-180:0:0.25cm)--(-1,0);
    \draw[Blue] (1.5,0)--(1.5,-4) (1,-4)arc(0:180:1cm) (-1.5,-4)--(-1.5,-3.5)to[in=-90,out=90](1,-0.5)--(1,0);
    \begin{scope}[xshift=-0.1cm]
        \path[name path=PathOne](-1.25,0)--(-1.25,-4);
        \path[name path=PathTwo] (-1.5,-4)--(-1.5,-3.5)to[in=-90,out=90](1,-0.5)--(1,0);
        \path[name intersections={of=PathOne and PathTwo, by={A}}];
        \fill[Red, fill opacity=0.6](A)to[in=0,out=0,looseness=0.5](-1.25,-0.75)to[in=180,out=180,looseness=0.5](A);
        \draw[Maroon,thick](A)to[in=0,out=0,looseness=0.5](-1.25,-0.75);
        \draw[Maroon,thick,dotted](-1.25,-0.75)to[in=180,out=180,looseness=0.5](A);
    \end{scope}

}
\newcommand{\RTwoMinusdh}{
    \begin{scope}[yshift=-4cm]
        
    \draw[Blue,->-](1.5,0) arc(-180:-90:1cm and 0.5cm)--(3.5,-0.5);
    \draw[Blue,->-](-3,-0.5)--(-2.5,-0.5)arc(-90:0:1cm and 0.5cm);
    \draw[Blue,-<-](-1,0)arc(-180:0:1cm and 0.5cm);
    \draw[Blue](-3.5,0.5)--(-3,0.5);
    \draw[Blue,dotted](-1,0)arc(180:0:1cm and 0.5cm);
    \draw[Blue,dotted](-3,0.5)--(-2.5,0.5)arc(90:0:1cm and 0.5cm);
    \draw[Blue,dotted](1.5,0) arc(180:90:1cm and 0.5cm);
    \draw[Blue,dotted](2.5,0.5)--(3,0.5)--(3,3.5);

    \end{scope}

    \draw[Blue,->-](1.5,0) arc(-180:-90:1cm and 0.5cm)--(3.5,-0.5);
    \draw[Blue](-3,-0.5)--(-2.5,-0.5)arc(-90:0:1cm and 0.5cm);
    \draw[Blue,-<-](-1,0)arc(-180:0:1cm and 0.5cm);
    \draw[Blue](2.5,0.5)--(3,0.5)--(3,-0.5);
    \draw[Blue](-3.5,0.5)--(-3,0.5);
    \draw[Blue](-1,0)arc(180:0:1cm and 0.5cm);
    \draw[Blue,-<-](-3,0.5)--(-2.5,0.5)arc(90:0:1cm and 0.5cm);
    \draw[Blue,](1.5,0) arc(180:90:1cm and 0.5cm);

    \draw[Blue](-3,-4.5)--(-3,-0.5);
    \draw[Blue](-3.5,-3.5)--(-3.5,0.5);
    \draw[Blue](3.5,-4.5)--(3.5,-0.5);
    \draw[Blue,dotted](3,-0.5)--(3,-3.5);

    \draw[Blue] (1.5,-4)--(1.5,-3.5)arc(0:180:0.25cm)--(1,-4);
    \draw[Blue] (-1,0)arc(-180:0:1cm) (1.5,0)--(1.5,-0.5)to[in=90,out=-90](-1,-3.5)--(-1,-4) (-1.5,-4)--(-1.5,0);
    \begin{scope}[xshift=0.1cm]
        \path[name path=PathOne](1.25,0)--(1.25,-4);
        \path[name path=PathTwo] (1.5,0)--(1.5,-0.5)to[in=90,out=-90](-1,-3.5)--(-1,-4) ;
        \path[name intersections={of=PathOne and PathTwo, by={A}}];
        \fill[Red, fill opacity=0.6](A)to[in=0,out=0,looseness=0.5](1.25,-3.25)to[in=180,out=180,looseness=0.5](A);
        \draw[Maroon,thick](A)to[in=0,out=0,looseness=0.5](1.25,-3.25);
        \draw[Maroon,thick,dotted](1.25,-3.25)to[in=180,out=180,looseness=0.5](A);
    \end{scope}

}

\newcommand{\RThreeLeft}[1]{
\begin{tikzpicture}[
    every braid/.style={->},
    braid/.cd,
    height=1cm,
    gap=.2,
]
    \pic[scale=#1] at (0,0) {braid={s_1 s_2 s_1}};
\end{tikzpicture}
}

\newcommand{\RThreeRight}[1]{
\begin{tikzpicture}[
    every braid/.style={->},
    braid/.cd,
    height=1cm,
    gap=.2,
]
    \pic[scale=#1] at (0,0) {braid={s_2 s_1 s_2}};
\end{tikzpicture}
}

\newcommand{\TopLayer}[1]{
\begin{tikzpicture}[
    every braid/.style={->},
    braid/.cd,
    height=1cm,
    gap=.2,
]
    \pic[scale=#1] at (0,0) {braid={s_1 s_2}};
\end{tikzpicture}
}

\newcommand{\fOne}{
    
    \draw[Blue](1.5,-4) arc(-180:-90:1cm and 0.5cm)--(3.5,-4.5);
    \draw[Blue](-3,-4.5)--(-2.5,-4.5)arc(-90:0:1cm and 0.5cm);
    \draw[Blue](-3.5,-3.5)--(-3,-3.5);
    \draw[Blue,dotted](-3,-3.5)--(-2.5,-3.5)arc(90:0:1cm and 0.5cm);
    \draw[Blue,dotted](1.5,-4) arc(180:90:1cm and 0.5cm);
    \draw[Blue,dotted](2.5,-3.5)--(3,-3.5);

    \draw[Blue,->-](1.5,0) arc(-180:-90:1cm and 0.5cm)--(3.5,-0.5);
    \draw[Blue](-3,-0.5)--(-2.5,-0.5)arc(-90:0:1cm and 0.5cm);
    \draw[Blue,](-1,0)arc(-180:0:1cm and 0.5cm);
    \draw[Blue](2.5,0.5)--(3,0.5)--(3,-0.5);
    \draw[Blue](-3.5,0.5)--(-3,0.5);
    \draw[Blue](-1,0)arc(180:0:1cm and 0.5cm);
    \draw[Blue,](-3,0.5)--(-2.5,0.5)arc(90:0:1cm and 0.5cm);
    \draw[Blue,](1.5,0) arc(180:90:1cm and 0.5cm);

    \draw[Blue](-3,-4.5)--(-3,-0.5);
    \draw[Blue](-3.5,-3.5)--(-3.5,0.5);
    \draw[Blue](3.5,-4.5)--(3.5,-0.5);
    \draw[Blue,dotted](3,-0.5)--(3,-3.5);

    \draw[Maroon,thick,fill=Red,fill opacity=0.6](-1,-4)--(1.5,-4)--(1.5,0)--(1,0)arc(0:-180:1cm)--(-1.5,0)--(-1.5,-4)--cycle;
    \fill[Cyan,fill opacity=0.2](-2.5,-1.5)rectangle (4,-6);

    \begin{scope}[yshift=0.2cm]
        \BDot
    \end{scope}

}

\newcommand{\fTwo}{
    
    \draw[Blue](1.5,-4) arc(-180:-90:1cm and 0.5cm)--(3.5,-4.5);
    \draw[Blue](-3,-4.5)--(-2.5,-4.5)arc(-90:0:1cm and 0.5cm);
    \draw[Blue](-3.5,-3.5)--(-3,-3.5);
    \draw[Blue,dotted](-3,-3.5)--(-2.5,-3.5)arc(90:0:1cm and 0.5cm);
    \draw[Blue,dotted](1.5,-4) arc(180:90:1cm and 0.5cm);
    \draw[Blue,dotted](2.5,-3.5)--(3,-3.5);

    \draw[Blue,->-](1.5,0) arc(-180:-90:1cm and 0.5cm)--(3.5,-0.5);
    \draw[Blue](-3,-0.5)--(-2.5,-0.5)arc(-90:0:1cm and 0.5cm);
    \draw[Blue,](-1,0)arc(-180:0:1cm and 0.5cm);
    \draw[Blue](2.5,0.5)--(3,0.5)--(3,-0.5);
    \draw[Blue](-3.5,0.5)--(-3,0.5);
    \draw[Blue](-1,0)arc(180:0:1cm and 0.5cm);
    \draw[Blue,](-3,0.5)--(-2.5,0.5)arc(90:0:1cm and 0.5cm);
    \draw[Blue,](1.5,0) arc(180:90:1cm and 0.5cm);

    \draw[Blue](-3,-4.5)--(-3,-0.5);
    \draw[Blue](-3.5,-3.5)--(-3.5,0.5);
    \draw[Blue](3.5,-4.5)--(3.5,-0.5);
    \draw[Blue,dotted](3,-0.5)--(3,-3.5);

    \draw[Maroon,thick,fill=Red,fill opacity=0.6](-1,-4)--(1.5,-4)--(1.5,0)--(1,0)arc(0:-180:1cm)--(-1.5,0)--(-1.5,-4)--cycle;
    \fill[Cyan,fill opacity=0.2](-2.5,-1.5)rectangle (4,-6);

}

\newcommand{\fOnePrime}{

    \draw[Blue](1.5,-4) arc(-180:-90:1cm and 0.5cm)--(3.5,-4.5);
    \draw[Blue](-3,-4.5)--(-2.5,-4.5)arc(-90:0:1cm and 0.5cm);
    \draw[Blue,](-1,-4)arc(-180:0:1cm and 0.5cm);
    \draw[Blue](-3.5,-3.5)--(-3,-3.5);
    \draw[Blue,dotted](-1,-4)arc(180:0:1cm and 0.5cm);
    \draw[Blue,dotted](-3,-3.5)--(-2.5,-3.5)arc(90:0:1cm and 0.5cm);
    \draw[Blue,dotted](1.5,-4) arc(180:90:1cm and 0.5cm);
    \draw[Blue,dotted](2.5,-3.5)--(3,-3.5);

    \draw[Blue,->-](1.5,0) arc(-180:-90:1cm and 0.5cm)--(3.5,-0.5);
    \draw[Blue](-3,-0.5)--(-2.5,-0.5)arc(-90:0:1cm and 0.5cm);
    \draw[Blue](2.5,0.5)--(3,0.5)--(3,-0.5);
    \draw[Blue](-3.5,0.5)--(-3,0.5);
    \draw[Blue,](-3,0.5)--(-2.5,0.5)arc(90:0:1cm and 0.5cm);
    \draw[Blue,](1.5,0) arc(180:90:1cm and 0.5cm);

    \draw[Blue](-3,-4.5)--(-3,-0.5);
    \draw[Blue](-3.5,-3.5)--(-3.5,0.5);
    \draw[Blue](3.5,-4.5)--(3.5,-0.5);
    \draw[Blue,dotted](3,-0.5)--(3,-3.5);

    \draw[Maroon,thick,fill=Red,fill opacity=0.6](-1,-4)arc(180:0:1cm)--(1.5,-4)--(1.5,0)--(-1.5,0)--(-1.5,-4)--cycle;
    \fill[Cyan,fill opacity=0.2](-2.5,-1.5)rectangle (4,-6);

}

\newcommand{\fTwoPrime}{

    \draw[Blue](1.5,-4) arc(-180:-90:1cm and 0.5cm)--(3.5,-4.5);
    \draw[Blue](-3,-4.5)--(-2.5,-4.5)arc(-90:0:1cm and 0.5cm);
    \draw[Blue,](-1,-4)arc(-180:0:1cm and 0.5cm);
    \draw[Blue](-3.5,-3.5)--(-3,-3.5);
    \draw[Blue,dotted](-1,-4)arc(180:0:1cm and 0.5cm);
    \draw[Blue,dotted](-3,-3.5)--(-2.5,-3.5)arc(90:0:1cm and 0.5cm);
    \draw[Blue,dotted](1.5,-4) arc(180:90:1cm and 0.5cm);
    \draw[Blue,dotted](2.5,-3.5)--(3,-3.5);

    \draw[Blue,->-](1.5,0) arc(-180:-90:1cm and 0.5cm)--(3.5,-0.5);
    \draw[Blue](-3,-0.5)--(-2.5,-0.5)arc(-90:0:1cm and 0.5cm);
    \draw[Blue](2.5,0.5)--(3,0.5)--(3,-0.5);
    \draw[Blue](-3.5,0.5)--(-3,0.5);
    \draw[Blue,](-3,0.5)--(-2.5,0.5)arc(90:0:1cm and 0.5cm);
    \draw[Blue,](1.5,0) arc(180:90:1cm and 0.5cm);

    \draw[Blue](-3,-4.5)--(-3,-0.5);
    \draw[Blue](-3.5,-3.5)--(-3.5,0.5);
    \draw[Blue](3.5,-4.5)--(3.5,-0.5);
    \draw[Blue,dotted](3,-0.5)--(3,-3.5);

    \draw[Maroon,thick,fill=Red,fill opacity=0.6](-1,-4)arc(180:0:1cm)--(1.5,-4)--(1.5,0)--(-1.5,0)--(-1.5,-4)--cycle;
    \fill[Cyan,fill opacity=0.2](-2.5,-1.5)rectangle (4,-6);

    \begin{scope}[yshift=-4.2cm]
        \BDot
    \end{scope}

}

\newcommand{\Eff}{

    \fill[Red,fill opacity=0.6](-1.5,-1)--(-1.5,-4)--(1.5,-4)--(1.5,-1)--cycle;

    \draw[Maroon,thick,fill=Red,fill opacity=0.6](0.25,-1)arc(-180:0:0.75cm)--cycle;
    \fill[Cyan,fill opacity=0.2](-2.5,-1)--(0.25,-1)arc(-180:0:0.75cm)--(4,-1)--(4,-5.5)--(-2.5,-5.5)--cycle;
    \draw[Blue](-0.75,0)--(0.25,-1);
    \draw[Blue](0.75,0)--(1.75,-1);
    \fill[Cyan,fill opacity=0.2](-0.75,0)arc(-180:0:0.75cm);

    \begin{scope}

        \clip (-4,1)--(4,1)--(4,-1)--(-2.5,-1)--(-2.5,-5.5)--(-4,-5.5)--cycle;

        \begin{scope}[yshift=-4cm]
            \draw[Blue,](1.5,0) arc(-180:-90:1cm and 0.5cm)--(3.5,-0.5);
            \draw[Blue,](-3,-0.5)--(-2.5,-0.5)arc(-90:0:1cm and 0.5cm);
            \draw[Blue](-3.5,0.5)--(-3,0.5);
            \draw[Blue,dotted](-3,0.5)--(-2.5,0.5)arc(90:0:1cm and 0.5cm);
            \draw[Blue,dotted](1.5,0) arc(180:90:1cm and 0.5cm);
            \draw[Blue,dotted](2.5,0.5)--(3,0.5)--(3,3.5);

        \end{scope}

        \draw[Blue,->-](1.5,0) arc(-180:-90:1cm and 0.5cm)--(3.5,-0.5);
        \draw[Blue](-3,-0.5)--(-2.5,-0.5)arc(-90:0:1cm and 0.5cm);
        \draw[Blue](2.5,0.5)--(3,0.5)--(3,-0.5);
        \draw[Blue](-3.5,0.5)--(-3,0.5);
        \draw[Blue,](-3,0.5)--(-2.5,0.5)arc(90:0:1cm and 0.5cm);
        \draw[Blue,](1.5,0) arc(180:90:1cm and 0.5cm);

        \draw[Blue](-3,-4.5)--(-3,-0.5);
        \draw[Blue](-3.5,-3.5)--(-3.5,0.5);
        \draw[Blue](3.5,-4.5)--(3.5,-0.5);
        \draw[Blue,dotted](3,-0.5)--(3,-3.5);

        \draw[Blue](-135:0.75)--($(-135:0.75)+(1,-1)$);
        \fill[Red,fill opacity=0.6] (-1.5,-1)--(-1.5,0)--(-0.75,0)arc(-180:-135:0.75)--($(-135:0.75)+(1,-1)$)--cycle;
        \fill[Red,fill opacity=0.6] (0.75,0)--(1.5,0)--(1.5,-0.75)--cycle;
        
    \end{scope}

    \begin{scope}

        \clip (-2.5,-1)rectangle (4,-5.5);

        \begin{scope}[yshift=-4cm]
            \draw[Blue,dotted](1.5,0) arc(-180:-90:1cm and 0.5cm)--(3.5,-0.5);
            \draw[Blue,dotted](-3,-0.5)--(-2.5,-0.5)arc(-90:0:1cm and 0.5cm);
            \draw[Blue,dotted](-3.5,0.5)--(-3,0.5);
            \draw[Blue,dotted](-3,0.5)--(-2.5,0.5)arc(90:0:1cm and 0.5cm);
            \draw[Blue,dotted](1.5,0) arc(180:90:1cm and 0.5cm);
            \draw[Blue,dotted](2.5,0.5)--(3,0.5)--(3,3.5);

        \end{scope}

        \draw[Blue,dotted](1.5,0) arc(-180:-90:1cm and 0.5cm)--(3.5,-0.5);
        \draw[Blue,dotted](-3,-0.5)--(-2.5,-0.5)arc(-90:0:1cm and 0.5cm);
        \draw[Blue,dotted](2.5,0.5)--(3,0.5)--(3,-0.5);
        \draw[Blue,dotted](-3.5,0.5)--(-3,0.5);
        \draw[Blue,dotted](-3,0.5)--(-2.5,0.5)arc(90:0:1cm and 0.5cm);
        \draw[Blue,dotted](1.5,0) arc(180:90:1cm and 0.5cm);

        \draw[Blue,dotted](-3,-4.5)--(-3,-0.5);
        \draw[Blue,dotted](-3.5,-3.5)--(-3.5,0.5);
        \draw[Blue,dotted](3.5,-4.5)--(3.5,-0.5);
        \draw[Blue,dotted](3,-0.5)--(3,-3.5);

        \draw[Maroon,thick,dotted](-0.75,0)arc(-180:0:0.75cm)--(1.5,0)--(1.5,-4)--(-1.5,-4)--(-1.5,0)--cycle;
        
        \draw[Blue,dotted](-135:0.75)--($(-135:0.75)+(1,-1)$);

    \end{scope}

    \draw[Blue] (-0.75,0)--(0.75,0);
    \draw[Maroon,thick](0.75,0)--(1.5,0)--(1.5,-0.75) (-1.5,-1)--(-1.5,0)--(-0.75,0)arc(180:225:0.75);
    \draw[Maroon,thick,dotted](1.5,-0.75)--(1.5,-1);
    \draw[Maroon,thick](0.75,0)arc(0:-90:0.75);
    \draw[Maroon,dotted,thick](0,-0.75)arc(-90:-135:0.75);

    \draw[Blue](-2.5,-1)--(0.25,-1) (1.75,-1)--(4,-1)--(4,-5.5)--(-2.5,-5.5)--(-2.5,-1);

    \fill[Cyan, fill opacity=0.2] (0.25,-1) arc(-180:-135:0.75)-- (-135:0.75) arc(-135:-180:0.75)-- cycle;
    \fill[Cyan, fill opacity=0.2] (0.75,0)arc(0:-90:0.75)--(0.25,-1)--(1.75,-1)--cycle;

}
\newcommand{\EffPrime}{
    \draw[Blue,dotted]($(45:0.75)+(0,-4)$)--($(45:0.75)+(1,-5.5)$);
    \draw[Blue,dotted](0.75,-4)--(1.75,-5.5);
    \draw[Blue,dotted](-0.75,-4)--(0.25,-5.5);

    \draw[Blue](0.25,-5.5)--(-2.5,-5.5)--(-2.5,-1)--(4,-1)--(4,-5.5)--(1.75,-5.5);

    \draw[Maroon,thick,fill=Red,fill opacity=0.6](1.75,-5.5)arc(0:180:0.75)--cycle;
    \fill[Cyan,fill opacity=0.2](-2.5,-1)--(4,-1)--(4,-5.5)--(1.75,-5.5)arc(0:180:0.75)--(-2.5,-5.5)--cycle;
    \fill[Red,fill opacity=0.6](0.75,-4)arc(0:180:0.75)--(-1.5,-4)--(-1.5,0)--(1.5,0)--(1.5,-4)--cycle;
    \draw[Maroon,thick,dotted](1.5,-1)--(1.5,-4)--(0.75,-4)arc(0:180:0.75)--(-1.5,-4)--(-1.5,-1);
    \draw[Maroon,thick](-1.5,-1)--(-1.5,0)--(1.5,0)--(1.5,-1);

    \begin{scope}

        \clip (-4,1)--(4,1)--(4,-1)--(-2.5,-1)--(-2.5,-5.5)--(-4,-5.5)--cycle;

        \begin{scope}[yshift=-4cm]
            \draw[Blue,](1.5,0) arc(-180:-90:1cm and 0.5cm)--(3.5,-0.5);
            \draw[Blue,](-3,-0.5)--(-2.5,-0.5)arc(-90:0:1cm and 0.5cm);
            \draw[Blue](-3.5,0.5)--(-3,0.5);
            \draw[Blue,dotted](-3,0.5)--(-2.5,0.5)arc(90:0:1cm and 0.5cm);
            \draw[Blue,dotted](1.5,0) arc(180:90:1cm and 0.5cm);
            \draw[Blue,dotted](2.5,0.5)--(3,0.5)--(3,3.5);

        \end{scope}

        \draw[Blue,->-](1.5,0) arc(-180:-90:1cm and 0.5cm)--(3.5,-0.5);
        \draw[Blue](-3,-0.5)--(-2.5,-0.5)arc(-90:0:1cm and 0.5cm);
        \draw[Blue](2.5,0.5)--(3,0.5)--(3,-0.5);
        \draw[Blue](-3.5,0.5)--(-3,0.5);
        \draw[Blue,](-3,0.5)--(-2.5,0.5)arc(90:0:1cm and 0.5cm);
        \draw[Blue,](1.5,0) arc(180:90:1cm and 0.5cm);

        \draw[Blue](-3,-4.5)--(-3,-0.5);
        \draw[Blue](-3.5,-3.5)--(-3.5,0.5);
        \draw[Blue](3.5,-4.5)--(3.5,-0.5);
        \draw[Blue,dotted](3,-0.5)--(3,-3.5);

    \end{scope}

    \begin{scope}

        \clip (-2.5,-1)rectangle (4,-5.5);

        \begin{scope}[yshift=-4cm]
            \draw[Blue,dotted](1.5,0) arc(-180:-90:1cm and 0.5cm)--(3.5,-0.5);
            \draw[Blue,dotted](-3,-0.5)--(-2.5,-0.5)arc(-90:0:1cm and 0.5cm);
            \draw[Blue,dotted](-3.5,0.5)--(-3,0.5);
            \draw[Blue,dotted](-3,0.5)--(-2.5,0.5)arc(90:0:1cm and 0.5cm);
            \draw[Blue,dotted](1.5,0) arc(180:90:1cm and 0.5cm);
            \draw[Blue,dotted](2.5,0.5)--(3,0.5)--(3,3.5);

        \end{scope}

        \draw[Blue,dotted](1.5,0) arc(-180:-90:1cm and 0.5cm)--(3.5,-0.5);
        \draw[Blue,dotted](-3,-0.5)--(-2.5,-0.5)arc(-90:0:1cm and 0.5cm);
        \draw[Blue,dotted](2.5,0.5)--(3,0.5)--(3,-0.5);
        \draw[Blue,dotted](-3.5,0.5)--(-3,0.5);
        \draw[Blue,dotted](-3,0.5)--(-2.5,0.5)arc(90:0:1cm and 0.5cm);
        \draw[Blue,dotted](1.5,0) arc(180:90:1cm and 0.5cm);

        \draw[Blue,dotted](-3,-4.5)--(-3,-0.5);
        \draw[Blue,dotted](-3.5,-3.5)--(-3.5,0.5);
        \draw[Blue,dotted](3.5,-4.5)--(3.5,-0.5);
        \draw[Blue,dotted](3,-0.5)--(3,-3.5);

    \end{scope}

}

\newcommand{\gOne}{
    
    \draw[Blue](1.5,-4) arc(-180:-90:1cm and 0.5cm)--(3.5,-4.5);
    \draw[Blue](-3,-4.5)--(-2.5,-4.5)arc(-90:0:1cm and 0.5cm);
    \draw[Blue](-3.5,-3.5)--(-3,-3.5);
    \draw[Blue,dotted](-3,-3.5)--(-2.5,-3.5)arc(90:0:1cm and 0.5cm);
    \draw[Blue,dotted](1.5,-4) arc(180:90:1cm and 0.5cm);
    \draw[Blue,dotted](2.5,-3.5)--(3,-3.5);

    \draw[Blue,->-](1.5,0) arc(-180:-90:1cm and 0.5cm)--(3.5,-0.5);
    \draw[Blue](-3,-0.5)--(-2.5,-0.5)arc(-90:0:1cm and 0.5cm);
    \draw[Blue,](-1,0)arc(-180:0:1cm and 0.5cm);
    \draw[Blue](2.5,0.5)--(3,0.5)--(3,-0.5);
    \draw[Blue](-3.5,0.5)--(-3,0.5);
    \draw[Blue](-1,0)arc(180:0:1cm and 0.5cm);
    \draw[Blue,](-3,0.5)--(-2.5,0.5)arc(90:0:1cm and 0.5cm);
    \draw[Blue,](1.5,0) arc(180:90:1cm and 0.5cm);

    \draw[Blue](-3,-4.5)--(-3,-0.5);
    \draw[Blue](-3.5,-3.5)--(-3.5,0.5);
    \draw[Blue](3.5,-4.5)--(3.5,-0.5);
    \draw[Blue,dotted](3,-0.5)--(3,-3.5);

    \draw[Maroon,thick,fill=Red,fill opacity=0.6](-1,-4)--(1.5,-4)--(1.5,0)--(1,0)arc(0:-180:1cm)--(-1.5,0)--(-1.5,-4)--cycle;
    \draw[Blue](-3.5,-3)--(-4.25,-3)--(-4.25,1.25) (2.5,1.25)--(2.5,0.5);
    \draw[Blue,->-](-4.25,1.25)--(2.5,1.25);
    \draw[Blue,dotted](2.5,0.5)--(2.5,-3)--(-3.5,-3);
    \fill[Cyan, fill opacity=0.2] (2.5,1.25)--(2.5,0.5)arc(90:180:1cm and 0.5cm)--(1,0)arc(0:180:1cm and 0.5cm)--(-1.5,0)arc(0:90:1cm and 0.5cm)--(-3.5,0.5)--(-3.5,-3)--(-4.25,-3)--(-4.25,1.25)--cycle;
    \begin{scope}[yshift=0.2cm]
        \BDot
    \end{scope}

}

\newcommand{\gTwo}{
    
    \draw[Blue](1.5,-4) arc(-180:-90:1cm and 0.5cm)--(3.5,-4.5);
    \draw[Blue](-3,-4.5)--(-2.5,-4.5)arc(-90:0:1cm and 0.5cm);
    \draw[Blue](-3.5,-3.5)--(-3,-3.5);
    \draw[Blue,dotted](-3,-3.5)--(-2.5,-3.5)arc(90:0:1cm and 0.5cm);
    \draw[Blue,dotted](1.5,-4) arc(180:90:1cm and 0.5cm);
    \draw[Blue,dotted](2.5,-3.5)--(3,-3.5);

    \draw[Blue,->-](1.5,0) arc(-180:-90:1cm and 0.5cm)--(3.5,-0.5);
    \draw[Blue](-3,-0.5)--(-2.5,-0.5)arc(-90:0:1cm and 0.5cm);
    \draw[Blue,](-1,0)arc(-180:0:1cm and 0.5cm);
    \draw[Blue](2.5,0.5)--(3,0.5)--(3,-0.5);
    \draw[Blue](-3.5,0.5)--(-3,0.5);
    \draw[Blue](-1,0)arc(180:0:1cm and 0.5cm);
    \draw[Blue,](-3,0.5)--(-2.5,0.5)arc(90:0:1cm and 0.5cm);
    \draw[Blue,](1.5,0) arc(180:90:1cm and 0.5cm);

    \draw[Blue](-3,-4.5)--(-3,-0.5);
    \draw[Blue](-3.5,-3.5)--(-3.5,0.5);
    \draw[Blue](3.5,-4.5)--(3.5,-0.5);
    \draw[Blue,dotted](3,-0.5)--(3,-3.5);

    \draw[Maroon,thick,fill=Red,fill opacity=0.6](-1,-4)--(1.5,-4)--(1.5,0)--(1,0)arc(0:-180:1cm)--(-1.5,0)--(-1.5,-4)--cycle;
    \draw[Blue](-3.5,-3)--(-4.25,-3)--(-4.25,1.25) (2.5,1.25)--(2.5,0.5);
    \draw[Blue,->-](-4.25,1.25)--(2.5,1.25);
    \draw[Blue,dotted](2.5,0.5)--(2.5,-3)--(-3.5,-3);
    \fill[Cyan, fill opacity=0.2] (2.5,1.25)--(2.5,0.5)arc(90:180:1cm and 0.5cm)--(1,0)arc(0:180:1cm and 0.5cm)--(-1.5,0)arc(0:90:1cm and 0.5cm)--(-3.5,0.5)--(-3.5,-3)--(-4.25,-3)--(-4.25,1.25)--cycle;

}

\newcommand{\gOnePrime}{

    \draw[Blue](1.5,-4) arc(-180:-90:1cm and 0.5cm)--(3.5,-4.5);
    \draw[Blue](-3,-4.5)--(-2.5,-4.5)arc(-90:0:1cm and 0.5cm);
    \draw[Blue,](-1,-4)arc(-180:0:1cm and 0.5cm);
    \draw[Blue](-3.5,-3.5)--(-3,-3.5);
    \draw[Blue,dotted](-1,-4)arc(180:0:1cm and 0.5cm);
    \draw[Blue,dotted](-3,-3.5)--(-2.5,-3.5)arc(90:0:1cm and 0.5cm);
    \draw[Blue,dotted](1.5,-4) arc(180:90:1cm and 0.5cm);
    \draw[Blue,dotted](2.5,-3.5)--(3,-3.5);

    \draw[Blue,->-](1.5,0) arc(-180:-90:1cm and 0.5cm)--(3.5,-0.5);
    \draw[Blue](-3,-0.5)--(-2.5,-0.5)arc(-90:0:1cm and 0.5cm);
    \draw[Blue](2.5,0.5)--(3,0.5)--(3,-0.5);
    \draw[Blue](-3.5,0.5)--(-3,0.5);
    \draw[Blue,](-3,0.5)--(-2.5,0.5)arc(90:0:1cm and 0.5cm);
    \draw[Blue,](1.5,0) arc(180:90:1cm and 0.5cm);

    \draw[Blue](-3,-4.5)--(-3,-0.5);
    \draw[Blue](-3.5,-3.5)--(-3.5,0.5);
    \draw[Blue](3.5,-4.5)--(3.5,-0.5);
    \draw[Blue,dotted](3,-0.5)--(3,-3.5);

    \draw[Maroon,thick,fill=Red,fill opacity=0.6](-1,-4)arc(180:0:1cm)--(1.5,-4)--(1.5,0)--(-1.5,0)--(-1.5,-4)--cycle;

    \draw[Blue](-3.5,-3)--(-4.25,-3)--(-4.25,1.25) (2.5,1.25)--(2.5,0.5);
    \draw[Blue,->-](-4.25,1.25)--(2.5,1.25);
    \draw[Blue,dotted](2.5,0.5)--(2.5,-3)--(-3.5,-3);
    \fill[Cyan, fill opacity=0.2] (2.5,1.25)--(2.5,0.5)arc(90:180:1cm and 0.5cm)--(1,0)--(-1.5,0)arc(0:90:1cm and 0.5cm)--(-3.5,0.5)--(-3.5,-3)--(-4.25,-3)--(-4.25,1.25)--cycle;

}

\newcommand{\gTwoPrime}{

    \draw[Blue](1.5,-4) arc(-180:-90:1cm and 0.5cm)--(3.5,-4.5);
    \draw[Blue](-3,-4.5)--(-2.5,-4.5)arc(-90:0:1cm and 0.5cm);
    \draw[Blue,](-1,-4)arc(-180:0:1cm and 0.5cm);
    \draw[Blue](-3.5,-3.5)--(-3,-3.5);
    \draw[Blue,dotted](-1,-4)arc(180:0:1cm and 0.5cm);
    \draw[Blue,dotted](-3,-3.5)--(-2.5,-3.5)arc(90:0:1cm and 0.5cm);
    \draw[Blue,dotted](1.5,-4) arc(180:90:1cm and 0.5cm);
    \draw[Blue,dotted](2.5,-3.5)--(3,-3.5);

    \draw[Blue,->-](1.5,0) arc(-180:-90:1cm and 0.5cm)--(3.5,-0.5);
    \draw[Blue](-3,-0.5)--(-2.5,-0.5)arc(-90:0:1cm and 0.5cm);
    \draw[Blue](2.5,0.5)--(3,0.5)--(3,-0.5);
    \draw[Blue](-3.5,0.5)--(-3,0.5);
    \draw[Blue,](-3,0.5)--(-2.5,0.5)arc(90:0:1cm and 0.5cm);
    \draw[Blue,](1.5,0) arc(180:90:1cm and 0.5cm);

    \draw[Blue](-3,-4.5)--(-3,-0.5);
    \draw[Blue](-3.5,-3.5)--(-3.5,0.5);
    \draw[Blue](3.5,-4.5)--(3.5,-0.5);
    \draw[Blue,dotted](3,-0.5)--(3,-3.5);

    \draw[Maroon,thick,fill=Red,fill opacity=0.6](-1,-4)arc(180:0:1cm)--(1.5,-4)--(1.5,0)--(-1.5,0)--(-1.5,-4)--cycle;

    \draw[Blue](-3.5,-3)--(-4.25,-3)--(-4.25,1.25) (2.5,1.25)--(2.5,0.5);
    \draw[Blue,->-](-4.25,1.25)--(2.5,1.25);
    \draw[Blue,dotted](2.5,0.5)--(2.5,-3)--(-3.5,-3);
    \fill[Cyan, fill opacity=0.2] (2.5,1.25)--(2.5,0.5)arc(90:180:1cm and 0.5cm)--(1,0)--(-1.5,0)arc(0:90:1cm and 0.5cm)--(-3.5,0.5)--(-3.5,-3)--(-4.25,-3)--(-4.25,1.25)--cycle;

    \begin{scope}[yshift=-4.2cm]
        \BDot
    \end{scope}

}

\newcommand{\Gi}{
    \draw[Blue](1.5,-4) arc(-180:-90:1cm and 0.5cm)--(3.5,-4.5);
    \draw[Blue](-3,-4.5)--(-2.5,-4.5)arc(-90:0:1cm and 0.5cm);
    \draw[Blue](-3.5,-3.5)--(-3,-3.5);
    \draw[Blue,dotted](-3,-3.5)--(-2.5,-3.5)arc(90:0:1cm and 0.5cm);
    \draw[Blue,dotted](1.5,-4) arc(180:90:1cm and 0.5cm);
    \draw[Blue,dotted](2.5,-3.5)--(3,-3.5);

    \draw[Blue,->-](1.5,0) arc(-180:-90:1cm and 0.5cm)--(3.5,-0.5);
    \draw[Blue](-3,-0.5)--(-2.5,-0.5)arc(-90:0:1cm and 0.5cm);
    \draw[Blue](2.5,0.5)--(3,0.5)--(3,-0.5);
    \draw[Blue](-3.5,0.5)--(-3,0.5);
    \draw[Blue,](-3,0.5)--(-2.5,0.5)arc(90:0:1cm and 0.5cm);
    \draw[Blue,](1.5,0) arc(180:90:1cm and 0.5cm);

    \draw[Blue](-3,-4.5)--(-3,-0.5);
    \draw[Blue](-3.5,-3.5)--(-3.5,0.5);
    \draw[Blue](3.5,-4.5)--(3.5,-0.5);
    \draw[Blue,dotted](3,-0.5)--(3,-3.5);

    \draw[Maroon,thick,fill=Red,fill opacity=0.6](-1,-4)--(1.5,-4)--(1.5,0)--(0.75,0)arc(0:-180:0.75cm)--(-1.5,0)--(-1.5,-4)--cycle;
    \fill[Red,fill opacity=0.6](-0.5,1.25)arc(0:-180:0.75cm)--cycle;
    \draw[Blue](-3.5,-3)--(-4.25,-3)--(-4.25,1.25) (2.5,1.25)--(2.5,0.5);
    \draw[Blue,->](-4.25,1.25)--(2,1.25);
    \draw[Blue](2,1.25)--(2.5,1.25);
    \draw[Blue,dotted](2.5,0.5)--(2.5,-3)--(-3.5,-3);
    \draw[Blue](0.75,0)--(-0.5,1.25) (-0.75,0)--(-2,1.25);
    
    \begin{scope}
        \clip(-1,-4)--(1.5,-4)--(1.5,0)--(0.75,0)arc(0:-180:0.75cm)--(-1.5,0)--(-1.5,-4)--cycle;
        \draw[Blue,dotted](-135:0.75)--($(-135:0.75)+(-1.25,1.25)$);
    \end{scope}
    \begin{scope}
        \clip (2.5,1.25)--(2.5,0.5)arc(90:180:1cm and 0.5cm)--(1,0)--(-1.5,0)arc(0:90:1cm and 0.5cm)--(-3.5,0.5)--(-3.5,-3)--(-4.25,-3)--(-4.25,1.25)--cycle;
        \draw[Blue](-135:0.75)--($(-135:0.75)+(-1.25,1.25)$);
    \end{scope}
    \draw[Blue](-0.75,0)--(0.75,0);
    \draw[Maroon,dotted,thick](-0.5,1.25)arc(0:-135:0.75cm);
    \draw[Maroon,thick](-2,1.25)arc(-180:-135:0.75cm);

}

\newcommand{\GiPrime}{

    \draw[Blue](1.5,-4) arc(-180:-90:1cm and 0.5cm)--(3.5,-4.5);
    \draw[Blue](-3,-4.5)--(-2.5,-4.5)arc(-90:0:1cm and 0.5cm);
    \draw[Blue](-3.5,-3.5)--(-3,-3.5);
    \draw[Blue,dotted](-3,-3.5)--(-2.5,-3.5)arc(90:0:1cm and 0.5cm);
    \draw[Blue,dotted](1.5,-4) arc(180:90:1cm and 0.5cm);
    \draw[Blue,dotted](2.5,-3.5)--(3,-3.5);

    \draw[Blue](-0.75,-4)--(0.75,-4);
    \draw[Blue,->-](1.5,0) arc(-180:-90:1cm and 0.5cm)--(3.5,-0.5);
    \draw[Blue](-3,-0.5)--(-2.5,-0.5)arc(-90:0:1cm and 0.5cm);
    \draw[Blue](2.5,0.5)--(3,0.5)--(3,-0.5);
    \draw[Blue](-3.5,0.5)--(-3,0.5);
    \draw[Blue,](-3,0.5)--(-2.5,0.5)arc(90:0:1cm and 0.5cm);
    \draw[Blue,](1.5,0) arc(180:90:1cm and 0.5cm);

    \draw[Blue](-3,-4.5)--(-3,-0.5);
    \draw[Blue](-3.5,-3.5)--(-3.5,0.5);
    \draw[Blue](3.5,-4.5)--(3.5,-0.5);
    \draw[Blue,dotted](3,-0.5)--(3,-3.5);

    \draw[Maroon,thick,fill=Red,fill opacity=0.6](-0.75,-4)arc(180:0:0.75cm)--(1.5,-4)--(1.5,0)--(-1.5,0)--(-1.5,-4)--cycle;

    \draw[Blue](-3.5,-3)--(-4.25,-3)--(-4.25,1.25) (2.5,1.25)--(2.5,0.5);
    \draw[Blue,->-](-4.25,1.25)--(2.5,1.25);
    \draw[Blue,dotted](2.5,0.5)--(2.5,-3)--(-3.5,-3);

    \draw[Maroon,thick,dotted,fill=Red,fill opacity=0.6](-0.25,-3)arc(0:180:0.75cm)--cycle;
    
    \draw[Blue,dotted](0.75,-4)--(-0.25,-3) (-0.75,-4)--(-1.75,-3);
    
    \begin{scope}
        \draw[Blue,dotted]($(45:0.75)+(0,-4)$)--($(45:0.75)+(-1,-3)$);
    \end{scope}
    \draw[Blue](0.75,-4)--(0,-3.25);

}

\newcommand{\MapOne}{

    \draw[Blue] (-3.5,0.5)--(-2,0.5)arc(90:-90: 0.5cm and 0.25cm)--(-3,0)
    (3,0)--(2,0)arc(90:270:0.5cm and 0.25cm)--(3.5,-0.5)
    (2.5,0.5)--(0,0.5)arc(90:270:0.5cm and 0.25cm)--(0,0)arc(90:-90: 0.5cm and 0.25cm)--(-2.5,-0.5);

    \draw[Blue] (-0.75,-4.25)arc(-180:0:0.75cm and 0.25cm)
    (1.5,-4.25)arc(-180:-90:0.5cm and 0.25cm)--(3.5,-4.5)
    (-1.5,-4.25)arc(0:-90:0.5cm and 0.25cm)--(-2.5,-4.5);
    \draw[Blue,dotted](-0.75,-4.25)arc(180:0:0.75cm and 0.25cm)
    (1.5,-4.25)arc(180:90:0.5cm and 0.25cm)--(3,-4)
    (-1.5,-4.25)arc(0:90:0.5cm and 0.25cm)--(-2.5,-4)
    (2.5,-3.5)--(-3,-3.5);
    \draw[Blue] (-2.5,-4)--(-3,-4)
    (-3,-3.5)--(-3.5,-3.5);

    \draw[Blue](3.5,-0.5)--(3.5,-4.5) (3,0)--(3,-0.5) (2.5,0.5)--(2.5,0)
    (-3.5,0.5)--(-3.5,-3.5) (-3,0)--(-3,-4) (-2.5,-0.5)--(-2.5,-4.5);
    \draw[Blue,dotted](3,-0.5)--(3,-4) (2.5,0)--(2.5,-3.5);

    \draw[Maroon,thick,fill=Red,fill opacity=0.6] (-1.5,0.25)arc(-180:0:0.5cm)--cycle;
    \draw[Maroon,thick,fill=Red, fill opacity=0.6](0.5,-0.25)--(1.5,-0.25)--(1.5,-4.25)--(0.75,-4.25)arc(0:180:0.75cm)--(-1.5,-4.25)to[out=90,in=-90]cycle;

}

\newcommand{\MapTwo}{
 
    \draw[Blue](-3.5,0.5)--(0,0.5)arc(90:-90:0.5cm and 0.25cm)--(0,0)arc(90:270:0.5cm and 0.25cm)--(3.5,-0.5)
    (2.5,0.5)--(2,0.5)arc(90:270:0.5cm and 0.25cm)--(3,0)
    (-3,0)--(-2,0)arc(90:-90:0.5cm and 0.25cm)--(-2.5,-0.5);

    \draw[Blue] (-0.75,-4.25)arc(-180:0:0.75cm and 0.25cm)
    (1.5,-4.25)arc(-180:-90:0.5cm and 0.25cm)--(3.5,-4.5)
    (-1.5,-4.25)arc(0:-90:0.5cm and 0.25cm)--(-2.5,-4.5);
    \draw[Blue,dotted](-0.75,-4.25)arc(180:0:0.75cm and 0.25cm)
    (1.5,-4.25)arc(180:90:0.5cm and 0.25cm)--(3,-4)
    (-1.5,-4.25)arc(0:90:0.5cm and 0.25cm)--(-2.5,-4)
    (2.5,-3.5)--(-3,-3.5);
    \draw[Blue] (-2.5,-4)--(-3,-4)
    (-3,-3.5)--(-3.5,-3.5);

    \draw[Blue](3.5,-0.5)--(3.5,-4.5) (3,0)--(3,-0.5) (2.5,0.5)--(2.5,0)
    (-3.5,0.5)--(-3.5,-3.5) (-3,0)--(-3,-4) (-2.5,-0.5)--(-2.5,-4.5);
    \draw[Blue,dotted](3,-0.5)--(3,-4) (2.5,0)--(2.5,-3.5);

    \draw[Maroon,thick,fill=Red,fill opacity=0.6] (1.5,0.25)arc(0:-180:0.5cm)--cycle;
    \draw[Maroon,thick,fill=Red,fill opacity=0.6] (-1.5,-0.25)--(-1.5,-4.25)--(-0.75,-4.25)arc(180:0:0.75cm)--(1.5,-4.25)to[out=90,in=-90](-0.5,-0.25)--cycle;

}

\newcommand{\MapThree}{

    \draw[Maroon,thick, fill=Red, fill opacity=0.6](-1.5,-4.25)arc(180:0:0.5cm)--cycle;
    \draw[Maroon,thick, fill=Red, fill opacity=0.6](-1.5,-0.5)--(-1.5,0.25)--(-0.75,0.25)arc(-180:0:0.75cm)--(1.5,0.25)--(1.5,-0.5);
    \draw[Maroon,thick,dotted,fill=Red,fill opacity=0.6](-1.5,-0.5)to[out=-90,in=90](0.5,-3.75)--(1.5,-3.75)--(1.5,-0.5);
 
    \draw[Blue] (-3.5,0.5)--(-2,0.5)arc(90:-90: 0.5cm and 0.25cm)--(-3,0)
    (2.5,0.5)--(2,0.5)arc(90:270:0.5cm and 0.25cm)--(3,0)
    (-2.5,-0.5)--(3.5,-0.5)
    (0,0.25)ellipse(0.75cm and 0.25cm);

    \draw[Blue](-2.5,-4.5)--(-2,-4.5)arc(-90:0:0.5cm and 0.25cm)
    (3.5,-4.5)--(0,-4.5)arc(-90:-180:0.5cm and 0.25cm);
    \draw[Blue,dotted] (-3,-4)--(-2,-4)arc(90:0:0.5cm and 0.25cm)
    (3,-4)--(2,-4)arc(-90:-270:0.5cm and 0.25cm)--(2.5,-3.5)
    (-3,-3.5)--(0,-3.5)arc(90:-90:0.5cm and 0.25cm)--(0,-4)arc(90:180:0.5cm and 0.25cm);
    \draw[Blue] (-2.5,-4)--(-3,-4)
    (-3,-3.5)--(-3.5,-3.5);

    \draw[Blue](3.5,-0.5)--(3.5,-4.5) (3,0)--(3,-0.5) (2.5,0.5)--(2.5,0)
    (-3.5,0.5)--(-3.5,-3.5) (-3,0)--(-3,-4) (-2.5,-0.5)--(-2.5,-4.5);
    \draw[Blue,dotted](3,-0.5)--(3,-4) (2.5,0)--(2.5,-3.5);

}

\newcommand{\MapFour}{

    \draw[Maroon,thick, fill=Red, fill opacity=0.6](1.5,-4.25)arc(0:180:1.5cm)--(-0.75,-4.25)arc(180:0:0.75cm)--cycle;
    \fill[Red,fill opacity=0.6](1.5,0.25)arc(0:-180:1.5cm)--(-0.75,0.25)arc(-180:0:0.75cm)--cycle;
    \begin{scope}
        \clip (-2,0.5)rectangle (2,-0.5);
        \draw[Maroon,thick](1.5,0.25)arc(0:-180:1.5cm)--(-0.75,0.25)arc(-180:0:0.75cm)--cycle;
    \end{scope}
    \begin{scope}
        \clip (-2,-2)rectangle (2,-0.5);
        \draw[Maroon,thick,dotted](1.5,0.25)arc(0:-180:1.5cm)--(-0.75,0.25)arc(-180:0:0.75cm)--cycle;
    \end{scope}
    \draw[Blue] (-3.5,0.5)--(-2,0.5)arc(90:-90: 0.5cm and 0.25cm)--(-3,0)
    (2.5,0.5)--(2,0.5)arc(90:270:0.5cm and 0.25cm)--(3,0)
    (-2.5,-0.5)--(3.5,-0.5)
    (0,0.25)ellipse(0.75cm and 0.25cm);

    \draw[Blue] (-0.75,-4.25)arc(-180:0:0.75cm and 0.25cm)
    (1.5,-4.25)arc(-180:-90:0.5cm and 0.25cm)--(3.5,-4.5)
    (-1.5,-4.25)arc(0:-90:0.5cm and 0.25cm)--(-2.5,-4.5);
    \draw[Blue,dotted](-0.75,-4.25)arc(180:0:0.75cm and 0.25cm)
    (1.5,-4.25)arc(180:90:0.5cm and 0.25cm)--(3,-4)
    (-1.5,-4.25)arc(0:90:0.5cm and 0.25cm)--(-2.5,-4)
    (2.5,-3.5)--(-3,-3.5);
    \draw[Blue] (-2.5,-4)--(-3,-4)
    (-3,-3.5)--(-3.5,-3.5);

    \draw[Blue](3.5,-0.5)--(3.5,-4.5) (3,0)--(3,-0.5) (2.5,0.5)--(2.5,0)
    (-3.5,0.5)--(-3.5,-3.5) (-3,0)--(-3,-4) (-2.5,-0.5)--(-2.5,-4.5);
    \draw[Blue,dotted](3,-0.5)--(3,-4) (2.5,0)--(2.5,-3.5);

}

\newcommand{\MapFive}{

    \draw[Maroon,thick,fill=Red,fill opacity=0.6] (1.5,-0.25)arc(0:-180:0.5cm)--cycle;
    \draw[Maroon,thick,fill=Red,fill opacity=0.6,dotted] (-1.5,-0.5)--(-1.5,-3.75)--(-0.75,-3.75)arc(180:0:0.75cm)--(1.5,-3.75)to[out=90,in=-90](-0.5,-0.5);
    \draw[Maroon,thick,fill=Red,fill opacity=0.6] (-1.5,-0.5)--(-1.5,0.25)--(-0.5,0.25)--(-0.5,-0.5);

    \draw[Blue] (-3.5,0.5)--(-2,0.5)arc(90:-90: 0.5cm and 0.25cm)--(-3,0)
    (3,0)--(2,0)arc(90:270:0.5cm and 0.25cm)--(3.5,-0.5)
    (2.5,0.5)--(0,0.5)arc(90:270:0.5cm and 0.25cm)--(0,0)arc(90:-90: 0.5cm and 0.25cm)--(-2.5,-0.5);

    \draw[Blue](-2.5,-4.5)--(3.5,-4.5)
    (-3,-4)--(-2.5,-4)
    (-3.5,-3.5)--(-3,-3.5);
    \draw[Blue,dotted](-2.5,-4)--(-2,-4)arc(-90:90:0.5cm and 0.25cm)--(-3,-3.5)
    (0,-3.75)ellipse(0.75cm and 0.25cm)
    (3,-4)--(2,-4)arc(-90:-270:0.5cm and 0.25cm)--(2.5,-3.5);

    \draw[Blue](3.5,-0.5)--(3.5,-4.5) (3,0)--(3,-0.5) (2.5,0.5)--(2.5,0)
    (-3.5,0.5)--(-3.5,-3.5) (-3,0)--(-3,-4) (-2.5,-0.5)--(-2.5,-4.5);
    \draw[Blue,dotted](3,-0.5)--(3,-4) (2.5,0)--(2.5,-3.5);

}
\newcommand{\MapSix}{

    \draw[Maroon,thick,fill=Red,fill opacity=0.6] (-1.5,-0.25)arc(-180:0:0.5cm)--cycle;
    \draw[Maroon,thick,fill=Red,fill opacity=0.6,dotted] (0.5,-0.5)to[in=90,out=-90](-1.5,-3.75)--(-0.75,-3.75)arc(180:0:0.75cm)--(1.5,-3.75)--(1.5,-0.5);
    \draw[Maroon,thick,fill=Red,fill opacity=0.6] (1.5,-0.5)--(1.5,0.25)--(0.5,0.25)--(0.5,-0.5);

    \draw[Blue](-3.5,0.5)--(0,0.5)arc(90:-90:0.5cm and 0.25cm)--(0,0)arc(90:270:0.5cm and 0.25cm)--(3.5,-0.5)
    (2.5,0.5)--(2,0.5)arc(90:270:0.5cm and 0.25cm)--(3,0)
    (-3,0)--(-2,0)arc(90:-90:0.5cm and 0.25cm)--(-2.5,-0.5);

    \draw[Blue](-2.5,-4.5)--(3.5,-4.5)
    (-3,-4)--(-2.5,-4)
    (-3.5,-3.5)--(-3,-3.5);
    \draw[Blue,dotted](-2.5,-4)--(-2,-4)arc(-90:90:0.5cm and 0.25cm)--(-3,-3.5)
    (0,-3.75)ellipse(0.75cm and 0.25cm)
    (3,-4)--(2,-4)arc(-90:-270:0.5cm and 0.25cm)--(2.5,-3.5);

    \draw[Blue](3.5,-0.5)--(3.5,-4.5) (3,0)--(3,-0.5) (2.5,0.5)--(2.5,0)
    (-3.5,0.5)--(-3.5,-3.5) (-3,0)--(-3,-4) (-2.5,-0.5)--(-2.5,-4.5);
    \draw[Blue,dotted](3,-0.5)--(3,-4) (2.5,0)--(2.5,-3.5);

}
\newcommand{\MapSeven}{

    \draw[Maroon,thick, fill=Red, fill opacity=0.6,yshift=0.5cm,dotted](1.5,-4.25)arc(0:180:1.5cm)--(-0.75,-4.25)arc(180:0:0.75cm)--cycle;
    \draw[Maroon,thick, fill=Red, fill opacity=0.6,yshift=-0.5cm](1.5,0.25)arc(0:-180:1.5cm)--(-0.75,0.25)arc(-180:0:0.75cm)--cycle;
    \draw[Blue] (-3,0)--(-2,0)arc(90:-90: 0.5cm and 0.25cm)--(-2.5,-0.5)
    (3,0)--(2,0)arc(90:270:0.5cm and 0.25cm)--(3.5,-0.5)
    (0,-0.25)ellipse(0.75cm and 0.25cm)
    (-3.5,0.5)--(2.5,0.5);

    \draw[Blue](-2.5,-4.5)--(3.5,-4.5)
    (-3,-4)--(-2.5,-4)
    (-3.5,-3.5)--(-3,-3.5);
    \draw[Blue,dotted](-2.5,-4)--(-2,-4)arc(-90:90:0.5cm and 0.25cm)--(-3,-3.5)
    (0,-3.75)ellipse(0.75cm and 0.25cm)
    (3,-4)--(2,-4)arc(-90:-270:0.5cm and 0.25cm)--(2.5,-3.5);

    \draw[Blue](3.5,-0.5)--(3.5,-4.5) (3,0)--(3,-0.5) (2.5,0.5)--(2.5,0)
    (-3.5,0.5)--(-3.5,-3.5) (-3,0)--(-3,-4) (-2.5,-0.5)--(-2.5,-4.5);
    \draw[Blue,dotted](3,-0.5)--(3,-4) (2.5,0)--(2.5,-3.5);

}
\newcommand{\MapEight}{

    \draw[Maroon,thick,dotted,fill=Red, fill opacity=0.6](0.5,-3.75)arc(180:0:0.5cm)--cycle;
    \draw[Maroon,thick,fill=Red, fill opacity=0.6](-1.5,-0.25)--(-1.5,-4.25)--(-0.5,-4.24)to[in=-90,out=90](1.5,-0.25)--(0.75,-0.25)arc(0:-180:0.75cm)--cycle;

    \draw[Blue] (-3,0)--(-2,0)arc(90:-90: 0.5cm and 0.25cm)--(-2.5,-0.5)
    (3,0)--(2,0)arc(90:270:0.5cm and 0.25cm)--(3.5,-0.5)
    (0,-0.25)ellipse(0.75cm and 0.25cm)
    (-3.5,0.5)--(2.5,0.5);

    \draw[Blue](-2.5,-4.5)--(-2,-4.5)arc(-90:0:0.5cm and 0.25cm)
    (3.5,-4.5)--(0,-4.5)arc(-90:-180:0.5cm and 0.25cm);
    \draw[Blue,dotted] (-3,-4)--(-2,-4)arc(90:0:0.5cm and 0.25cm)
    (3,-4)--(2,-4)arc(-90:-270:0.5cm and 0.25cm)--(2.5,-3.5)
    (-3,-3.5)--(0,-3.5)arc(90:-90:0.5cm and 0.25cm)--(0,-4)arc(90:180:0.5cm and 0.25cm);
    \draw[Blue] (-2.5,-4)--(-3,-4)
    (-3,-3.5)--(-3.5,-3.5);

    \draw[Blue](3.5,-0.5)--(3.5,-4.5) (3,0)--(3,-0.5) (2.5,0.5)--(2.5,0)
    (-3.5,0.5)--(-3.5,-3.5) (-3,0)--(-3,-4) (-2.5,-0.5)--(-2.5,-4.5);
    \draw[Blue,dotted](3,-0.5)--(3,-4) (2.5,0)--(2.5,-3.5);

}
\newcommand{\MapNine}{

    \draw[Maroon,thick,fill=Red,fill opacity=0.6](0.25,-1)arc(-180:0:0.75cm)--cycle;
    \draw[Blue](-0.75,0)--(0.25,-1);
    \draw[Blue](0.75,0)--(1.75,-1);

    \begin{scope}

        \clip (-4,1)--(4,1)--(4,-1)--(-2.5,-1)--(-2.5,-5.5)--(-4,-5.5)--cycle;

        \begin{scope}[yshift=-4cm]
            \draw[Blue,](1.5,0) arc(-180:-90:1cm and 0.5cm)--(3.5,-0.5);
            \draw[Blue,](-3,-0.5)--(-2.5,-0.5)arc(-90:0:1cm and 0.5cm);
            \draw[Blue](-3.5,0.5)--(-3,0.5);
            \draw[Blue,dotted](-3,0.5)--(-2.5,0.5)arc(90:0:1cm and 0.5cm);
            \draw[Blue,dotted](1.5,0) arc(180:90:1cm and 0.5cm);
            \draw[Blue,dotted](2.5,0.5)--(3,0.5)--(3,3.5);

        \end{scope}

        \draw[Blue,->-](1.5,0) arc(-180:-90:1cm and 0.5cm)--(3.5,-0.5);
        \draw[Blue](-3,-0.5)--(-2.5,-0.5)arc(-90:0:1cm and 0.5cm);
        \draw[Blue](2.5,0.5)--(3,0.5)--(3,-0.5);
        \draw[Blue](-3.5,0.5)--(-3,0.5);
        \draw[Blue,](-3,0.5)--(-2.5,0.5)arc(90:0:1cm and 0.5cm);
        \draw[Blue,](1.5,0) arc(180:90:1cm and 0.5cm);

        \draw[Blue](-3,-4.5)--(-3,-0.5);
        \draw[Blue](-3.5,-3.5)--(-3.5,0.5);
        \draw[Blue](3.5,-4.5)--(3.5,-0.5);
        \draw[Blue,dotted](3,-0.5)--(3,-3.5);

        \draw[Blue](-135:0.75)--($(-135:0.75)+(1,-1)$);
        \fill[Red,fill opacity=0.6] (-1.5,-1)--(-1.5,0)--(-0.75,0)arc(-180:-135:0.75)--($(-135:0.75)+(1,-1)$)--cycle;
        \fill[Red,fill opacity=0.6] (0.75,0)--(1.5,0)--(1.5,-0.75)--cycle;
        
    \end{scope}

    \begin{scope}

        \clip (-2.5,-1)rectangle (4,-5.5);

        \begin{scope}[yshift=-4cm]
            \draw[Blue,dotted](1.5,0) arc(-180:-90:1cm and 0.5cm)--(3.5,-0.5);
            \draw[Blue,dotted](-3,-0.5)--(-2.5,-0.5)arc(-90:0:1cm and 0.5cm);
            \draw[Blue,dotted](-3.5,0.5)--(-3,0.5);
            \draw[Blue,dotted](-3,0.5)--(-2.5,0.5)arc(90:0:1cm and 0.5cm);
            \draw[Blue,dotted](1.5,0) arc(180:90:1cm and 0.5cm);
            \draw[Blue,dotted](2.5,0.5)--(3,0.5)--(3,3.5);

        \end{scope}

        \draw[Blue,dotted](1.5,0) arc(-180:-90:1cm and 0.5cm)--(3.5,-0.5);
        \draw[Blue,dotted](-3,-0.5)--(-2.5,-0.5)arc(-90:0:1cm and 0.5cm);
        \draw[Blue,dotted](2.5,0.5)--(3,0.5)--(3,-0.5);
        \draw[Blue,dotted](-3.5,0.5)--(-3,0.5);
        \draw[Blue,dotted](-3,0.5)--(-2.5,0.5)arc(90:0:1cm and 0.5cm);
        \draw[Blue,dotted](1.5,0) arc(180:90:1cm and 0.5cm);

        \draw[Blue,dotted](-3,-4.5)--(-3,-0.5);
        \draw[Blue,dotted](-3.5,-3.5)--(-3.5,0.5);
        \draw[Blue,dotted](3.5,-4.5)--(3.5,-0.5);
        \draw[Blue,dotted](3,-0.5)--(3,-3.5);

        \draw[Maroon,thick,dotted](-0.75,0)arc(-180:0:0.75cm)--(1.5,0)--(1.5,-4)--(1,-4)arc(0:180:1cm)--(-1,-4)--(-1.5,-4)--(-1.5,0)--cycle;
        
        \draw[Blue,dotted](-135:0.75)--($(-135:0.75)+(1,-1)$);

    \end{scope}

    \begin{scope}[even odd rule]
        \clip(-1.5,-1)rectangle(1.5,-4)(0.25,-1)arc(-180:0:0.75cm)--cycle;
        
        \fill[Red,fill opacity=0.6]($(-135:0.75)+(1,-1)$)--(-135:0.75)--(-1.5,-1)--(-1.5,-4)--(-1,-4)arc(180:0:1cm)--(1,-4)--(1.5,-4)--(1.5,-1);

    \end{scope}

    \draw[Blue,dotted](0,-4)ellipse(1cm and 0.5cm);
    \draw[Blue] (-0.75,0)--(0.75,0);
    \draw[Maroon,thick](0.75,0)--(1.5,0)--(1.5,-0.75) (-1.5,-1)--(-1.5,0)--(-0.75,0)arc(180:225:0.75);
    \draw[Maroon,thick,dotted](1.5,-0.75)--(1.5,-1);
    \draw[Maroon,thick](0.75,0)arc(0:-90:0.75);
    \draw[Maroon,dotted,thick](0,-0.75)arc(-90:-135:0.75);

    \draw[Blue](-2.5,-1)--(0.25,-1) (1.75,-1)--(4,-1)--(4,-5.5)--(-2.5,-5.5)--(-2.5,-1);

}
\newcommand{\MapTen}{

    \draw[Blue](1.5,-4) arc(-180:-90:1cm and 0.5cm)--(3.5,-4.5);
    \draw[Blue](-3,-4.5)--(-2.5,-4.5)arc(-90:0:1cm and 0.5cm);
    \draw[Blue](-3.5,-3.5)--(-3,-3.5);
    \draw[Blue,dotted](-3,-3.5)--(-2.5,-3.5)arc(90:0:1cm and 0.5cm);
    \draw[Blue,dotted](1.5,-4) arc(180:90:1cm and 0.5cm);
    \draw[Blue,dotted](2.5,-3.5)--(3,-3.5);

    \draw[Blue,->-](1.5,0) arc(-180:-90:1cm and 0.5cm)--(3.5,-0.5);
    \draw[Blue](-3,-0.5)--(-2.5,-0.5)arc(-90:0:1cm and 0.5cm);
    \draw[Blue](2.5,0.5)--(3,0.5)--(3,-0.5);
    \draw[Blue](-3.5,0.5)--(-3,0.5);
    \draw[Blue,](-3,0.5)--(-2.5,0.5)arc(90:0:1cm and 0.5cm);
    \draw[Blue,](1.5,0) arc(180:90:1cm and 0.5cm);

    \draw[Blue](-3,-4.5)--(-3,-0.5);
    \draw[Blue](-3.5,-3.5)--(-3.5,0.5);
    \draw[Blue](3.5,-4.5)--(3.5,-0.5);
    \draw[Blue,dotted](3,-0.5)--(3,-3.5);

    \draw[Maroon,thick,fill=Red,fill opacity=0.6](1.5,0)--(0.75,0)arc(0:-180:0.75cm)--(-1.5,0)--(-1.5,-4)--(-1,-4)arc(180:0:1cm)--(1,-4)--(1.5,-4)--cycle;

    \draw[Maroon,dotted,thick](-0.5,1.25)arc(0:-135:0.75cm);
    \draw[Maroon,thick](-2,1.25)arc(-180:-135:0.75cm);
    \draw[Blue,dotted](-1,-4)arc(180:0:1cm and 0.5cm);
    \draw[Blue](-1,-4)arc(-180:0:1cm and 0.5cm);
    
    \fill[Red,fill opacity=0.6](-0.5,1.25)arc(0:-180:0.75cm)--cycle;
    \draw[Blue](-3.5,-3)--(-4.25,-3)--(-4.25,1.25) (2.5,1.25)--(2.5,0.5);
    \draw[Blue,->](-4.25,1.25)--(2,1.25);
    \draw[Blue](2,1.25)--(2.5,1.25);
    \draw[Blue,dotted](2.5,0.5)--(2.5,-3)--(-3.5,-3);
    \draw[Blue](0.75,0)--(-0.5,1.25) (-0.75,0)--(-2,1.25);
    
    \begin{scope}
        \clip(-1,-4)--(1.5,-4)--(1.5,0)--(0.75,0)arc(0:-180:0.75cm)--(-1.5,0)--(-1.5,-4)--cycle;
        \draw[Blue,dotted](-135:0.75)--($(-135:0.75)+(-1.25,1.25)$);
    \end{scope}
    \begin{scope}
        \clip (2.5,1.25)--(2.5,0.5)arc(90:180:1cm and 0.5cm)--(1,0)--(-1.5,0)arc(0:90:1cm and 0.5cm)--(-3.5,0.5)--(-3.5,-3)--(-4.25,-3)--(-4.25,1.25)--cycle;
        \draw[Blue](-135:0.75)--($(-135:0.75)+(-1.25,1.25)$);
    \end{scope}
    \draw[Blue](-0.75,0)--(0.75,0);

}
\newcommand{\FrobOne}[2]{
\begin{scope}[scale=#2]
        
\draw[#1](-1,0)arc(-180:0:1cm and 0.5cm)--(1,0)arc(0:180:1cm);
\draw[dotted,#1](-1,0)arc(180:0:1cm and 0.5cm);
\end{scope}

}
\newcommand{\FrobEpsilon}[2]{
\begin{scope}[scale=#2]

\draw[#1](1,0)arc(0:-180:1cm);
\draw[#1](0,0)ellipse (1cm and 0.5cm);
\end{scope}
}
\newcommand{\Frobm}[2]{
\begin{scope}[scale=#2]

\draw[#1](1.5,0)ellipse(1cm and 0.5cm)
(-1.5,0)ellipse(1cm and 0.5cm)
(-1,-3.5)arc(-180:0:1cm and 0.5cm)
(-0.5,0)to[in=-90,out=-90,looseness=3](0.5,0)
(-1,-3.5)to[out=90,in=-90](-2.5,0)
(1,-3.5)to[out=90,in=-90](2.5,0)
;
\draw[dotted,#1](-1,-3.5)arc(180:0:1cm and 0.5cm);
\end{scope}
}
\newcommand{\FrobCom}[2]{
\begin{scope}[scale=#2]

\draw[#1](2.5,0)arc(0:-180:1cm and 0.5cm)
(-2.5,0)arc(-180:0:1cm and 0.5cm)
(0,3.5)ellipse(1cm and 0.5cm)
(-0.5,0)to[in=90,out=90,looseness=3](0.5,0)
(-1,3.5)to[out=-90,in=90](-2.5,0)
(1,3.5)to[out=-90,in=90](2.5,0)
;
\draw[dotted,#1](2.5,0)arc(0:180:1cm and 0.5cm)
(-2.5,0)arc(180:0:1cm and 0.5cm);
\end{scope}
}
\newcommand{\FrobArc}[2]{
\begin{scope}[scale=#2]

\draw[#1](2.5,0)arc(0:-180:1cm and 0.5cm)
(-2.5,0)arc(-180:0:1cm and 0.5cm)
(-2.5,0)to[in=90,out=90,looseness=2](2.5,0)
(-0.5,0)to[in=90,out=90,looseness=3](0.5,0)
;
\draw[dotted,#1](2.5,0)arc(0:180:1cm and 0.5cm)
(-2.5,0)arc(180:0:1cm and 0.5cm);
\end{scope}

}
\newcommand{\Smoothing}{
\begin{scope}
    \draw[ shorten <= 0.5cm, shorten >= 0.5cm] (0,0)to[in=180,out=10](2,1)-- node[at start, above]{\small s(c)=0}(3,1) ;
    \draw[ shorten <= 0.5cm, shorten >= 0.5cm] (0,0)to[in=180,out=-10](2,-1)-- node[at start, above]{\small s(c)=1}(3,-1);
    \draw[] (-60:0.4)--(120:0.4);
    \fill[white] (0,0) circle (0.1cm);
    \draw[] (-120:0.4)--(60:0.4);
    \draw[,xshift=3cm,yshift=1cm,bend right] (-120:0.4)to(120:0.4);
    \draw[,xshift=3cm,yshift=1cm, bend left] (-60:0.4)to(60:0.4);
    \begin{scope}[xshift=3cm, yshift=-1cm]
        \draw[] (0,0.2) to[in=-120,out=0] (60:0.4);
        \draw[] (0,0.2) to[in=-60,out=180] (120:0.4);
        \draw[] (-60:0.4) to[in=0, out=120] (0,-0.2);
        \draw[] (-120:0.4) to[in=180, out=60] (0,-0.2);
    \end{scope}

\end{scope}

}

\newcommand{\SmoothingN}{
\begin{scope}
    \draw[ shorten <= 0.5cm, shorten >= 0.5cm] (0,0)to[in=180,out=10](2,1)-- node[at start, above]{\small s(c)=1}(3,1) ;
    \draw[ shorten <= 0.5cm, shorten >= 0.5cm] (0,0)to[in=180,out=-10](2,-1)-- node[at start, above]{\small s(c)=0}(3,-1);
    \draw (-120:0.4)--(60:0.4);
    \fill[white] (0,0) circle (0.1cm);
    \draw (-60:0.4)--(120:0.4);
    \draw[xshift=3cm,yshift=1cm,bend right] (-120:0.4)to(120:0.4);
    \draw[xshift=3cm,yshift=1cm, bend left] (-60:0.4)to(60:0.4);
    \begin{scope}[xshift=3cm, yshift=-1cm]
        \draw (0,0.2) to[in=-120,out=0] (60:0.4);
        \draw (0,0.2) to[in=-60,out=180] (120:0.4);
        \draw (-60:0.4) to[in=0, out=120] (0,-0.2);
        \draw (-120:0.4) to[in=180, out=60] (0,-0.2);
    \end{scope}

\end{scope}

}

\newcommand{\MFromShading}{
\begin{scope}
    \clip (110:2)to[out=-70,in=0](180:2)arc(180:110:2cm);
    \clip (50:2) to[out=-130, in=-30] (150:2)arc(150:50:2cm);
    \fill[gray,fill opacity=0.5] (0,0) circle (2cm);
\end{scope}

\begin{scope}
    \clip (50:2) to[out=-130, in=-30] (150:2)arc(150:50:2cm);
    \clip (70:2)to[out=-110,in=120](-60:2)arc(-60:70:2cm);
    \fill[gray,fill opacity=0.5] (0,0) circle (2cm);

\end{scope}

\begin{scope}
    \clip (5:2) to[out=-175, in=60] (-120:2) arc(-120:5:2cm);
    \clip (70:2)to[out=-110,in=120](-60:2)arc(-60:70:2cm);
    \fill[gray,fill opacity=0.5] (0,0) circle (2cm);

\end{scope}

\begin{scope}[even odd rule]
    \clip (0,0) circle (2cm)  (50:2) to[out=-130, in=-30] (150:2)arc(150:50:2cm);
    \clip (0,0) circle (2cm) (5:2) to[out=-175, in=60] (-120:2) arc(-120:5:2cm);
    \clip (0,0) circle (2cm) (70:2)to[out=-110,in=120](-60:2)arc(-60:70:2cm);
    \clip (0,0) circle (2cm) (110:2)to[out=-70,in=0](180:2)arc(180:110:2cm);
    \fill[gray,fill opacity=0.5] (0,0) circle (2cm);

\end{scope}

\draw (0,0) circle (2cm);
\draw[Blue] (50:2) to[out=-130, in=-30] (150:2);
\draw[Blue] (5:2) to[out=-175, in=60] (-120:2);
\draw[Maroon,thick,postaction={decorate}, decoration={markings, mark=at position 0.3 with {\arrow{>>}}}]
     (110:2)to[out=-70,in=0](180:2);
\draw[Maroon,thick,postaction={decorate}, decoration={markings, mark=at position 0.5 with {\arrow{>>}}}]
     (70:2)to[out=-110,in=120](-60:2);
\draw[->](100:1.3)arc(-170:170:0.2cm);
\draw[->](30:1.3)arc(-170:170:0.2cm);
\draw[->](165:1.75)arc(-170:170:0.2cm);
\draw[->](-90:1.5)arc(-170:170:0.2cm);

}

\newcommand{\MFromShadingTwo}{

\draw (0,0) circle (2cm);
\draw[Blue,postaction={decorate}, decoration={markings, mark=at position 0.5 with {\arrow{<}}},postaction={decorate}, decoration={markings, mark=at position 0.2 with {\arrow{>}}},postaction={decorate}, decoration={markings, mark=at position 0.9 with {\arrow{>}}}] (50:2) to[out=-130, in=-30] (150:2);
\draw[Blue,postaction={decorate}, decoration={markings, mark=at position 0.7 with {\arrow{>}}},postaction={decorate}, decoration={markings, mark=at position 0.3 with {\arrow{<}}}] (5:2) to[out=-175, in=60] (-120:2);
\draw[gray,dashed] (110:2)to[out=-70,in=0](180:2);
\draw[gray,dashed] (70:2)to[out=-110,in=120](-60:2);

\begin{scope}
    \clip  (50:2) to[out=-130, in=-30] (150:2)arc(150:50:2cm);
    \draw[Maroon,thick,postaction={decorate}, decoration={markings, mark=at position 0.3 with {\arrow{>>}}}]
     (110:2)to[out=-70,in=0](180:2); 
\end{scope}

\begin{scope}[even odd rule]

    \clip (0,0) circle (2cm)  (50:2) to[out=-130, in=-30] (150:2)arc(150:50:2cm);
    \clip (0,0) circle (2cm) (5:2) to[out=-175, in=60] (-120:2) arc(-120:5:2cm);
    \draw[Maroon,thick,postaction={decorate}, decoration={markings, mark=at position 0.5 with {\arrow{>>}}}]
     (70:2)to[out=-110,in=120](-60:2);   
\end{scope}
}

\newcommand{\ShadingFromM}{
    
\draw (0,0) circle (2cm);
\draw[Blue,postaction={decorate}, decoration={markings, mark=at position 0.5 with {\arrow{<}}},postaction={decorate}, decoration={markings, mark=at position 0.2 with {\arrow{>}}},postaction={decorate}, decoration={markings, mark=at position 0.9 with {\arrow{>}}}] (50:2) to[out=-130, in=-30] (150:2);
\draw[Blue,postaction={decorate}, decoration={markings, mark=at position 0.7 with {\arrow{>}}},postaction={decorate}, decoration={markings, mark=at position 0.3 with {\arrow{<}}}] (5:2) to[out=-175, in=60] (-120:2);

\begin{scope}
    \clip  (50:2) to[out=-130, in=-30] (150:2)arc(150:50:2cm);
    \draw[Maroon,thick,postaction={decorate}, decoration={markings, mark=at position 0.3 with {\arrow{>>}}}]
     (110:2)to[out=-70,in=0](180:2); 
\end{scope}

\begin{scope}[even odd rule]

    \clip (0,0) circle (2cm)  (50:2) to[out=-130, in=-30] (150:2)arc(150:50:2cm);
    \clip (0,0) circle (2cm) (5:2) to[out=-175, in=60] (-120:2) arc(-120:5:2cm);
    \draw[Maroon,thick,postaction={decorate}, decoration={markings, mark=at position 0.5 with {\arrow{>>}}}]
     (70:2)to[out=-110,in=120](-60:2);   
\end{scope}
\fill (80:2)circle(2pt);
\fill (-70:2)circle(2pt);
\fill (-170:2)circle(2pt);
\node at (15:2.2) {\small w};
\node at (100:2.2) {\small w};
\node at (175:2.2) {\small w};
\node at (-100:2.2) {\small w};
\node at (70:2.2) {\small b};
\node at (135:2.2) {\small b};
\node at (-150:2.2) {\small b};
\node at (-30:2.2) {\small b};

}

\newcommand{\ShadingFromMTwo}{
    
\draw[
    postaction={decorate}, decoration={markings, mark=at position 0.1 with {\arrow{>}}},
    postaction={decorate}, decoration={markings, mark=at position 0.18 with {\arrow{<}}},
    postaction={decorate}, decoration={markings, mark=at position 0.27 with {\arrow{>}}},
    postaction={decorate}, decoration={markings, mark=at position 0.35 with {\arrow{<}}},
    postaction={decorate}, decoration={markings, mark=at position 0.48 with {\arrow{>}}},
    postaction={decorate}, decoration={markings, mark=at position 0.6 with {\arrow{<}}},
    postaction={decorate}, decoration={markings, mark=at position 0.75 with {\arrow{>}}},
    postaction={decorate}, decoration={markings, mark=at position 0.9 with {\arrow{<}}},
] (0,0) circle (2cm);
\draw[Blue,postaction={decorate}, decoration={markings, mark=at position 0.5 with {\arrow{<}}},postaction={decorate}, decoration={markings, mark=at position 0.2 with {\arrow{>}}},postaction={decorate}, decoration={markings, mark=at position 0.9 with {\arrow{>}}}] (50:2) to[out=-130, in=-30] (150:2);
\draw[Blue,postaction={decorate}, decoration={markings, mark=at position 0.7 with {\arrow{>}}},postaction={decorate}, decoration={markings, mark=at position 0.3 with {\arrow{<}}}] (5:2) to[out=-175, in=60] (-120:2);

\begin{scope}
    \clip  (50:2) to[out=-130, in=-30] (150:2)arc(150:50:2cm);
    \draw[Maroon,thick,postaction={decorate}, decoration={markings, mark=at position 0.3 with {\arrow{>>}}}]
     (110:2)to[out=-70,in=0](180:2); 
\end{scope}

\begin{scope}[even odd rule]

    \clip (0,0) circle (2cm)  (50:2) to[out=-130, in=-30] (150:2)arc(150:50:2cm);
    \clip (0,0) circle (2cm) (5:2) to[out=-175, in=60] (-120:2) arc(-120:5:2cm);
    \draw[Maroon,thick,postaction={decorate}, decoration={markings, mark=at position 0.5 with {\arrow{>>}}}]
     (70:2)to[out=-110,in=120](-60:2);   
\end{scope}
\fill (80:2)circle(2pt);
\fill (-70:2)circle(2pt);
\fill (-170:2)circle(2pt);

}

\newcommand{\PushAndShade}{
    
\draw (0,0) circle (2cm);
\draw[Blue,] (50:2) to[out=-130, in=-30] (150:2);
\draw[Blue,] (5:2) to[out=-175, in=60] (-120:2);
\begin{scope}
    \clip  (50:2) to[out=-130, in=-30] (150:2)arc(150:50:2cm);
    \draw[Maroon,thick,postaction={decorate}, decoration={markings, mark=at position 0.3 with {\arrow{>>}}}]
     (110:2)to[out=-70,in=0](180:2); 
\end{scope}

\begin{scope}[even odd rule]

    \clip (0,0) circle (2cm)  (50:2) to[out=-130, in=-30] (150:2)arc(150:50:2cm);
    \clip (0,0) circle (2cm) (5:2) to[out=-175, in=60] (-120:2) arc(-120:5:2cm);
    \draw[Maroon,thick,postaction={decorate}, decoration={markings, mark=at position 0.5 with {\arrow{>>}}}]
     (70:2)to[out=-110,in=120](-60:2);   
\end{scope}
\draw[Maroon,-<-](130:1.6)circle(0.2cm);

\path[name path=RL](110:2)to[out=-70,in=0](180:2);
\path[name path= BT](50:2) to[out=-130, in=-30] (150:2);
\path[name path= BB](5:2) to[out=-175, in=60] (-120:2);
\path[name path= RR](70:2)to[out=-110,in=120](-60:2);

\path[name intersections={of= RL and BT, by=A}];
\path[name intersections={of= RR and BT, by=B}];
\path[name intersections={of= RR and BB, by=C}];

\fill (80:2)circle(2pt);
\fill (-70:2)circle(2pt);
\fill (-170:2)circle(2pt);

\draw[Maroon,->-](A)to[out=-80,in=90](-0.1,0)to[out=-90,in=0](-145:1.6)to[out=180,in=-80](-170:2);
\draw[Maroon,->-] (80:2)to[out=-45,in=90](60:1.6)to[out=-90,in=70](B);
\draw[Maroon,->-](C)to[out=-80,in=180](-20:1.5)to[out=0,in=70](-70:2);

\begin{scope}[even odd rule]
    \clip (50:2) to[out=-130, in=-30] (150:2)arc(150:50:2cm) (130:1.6)circle(0.2cm);
    \clip (110:2)to[out=-70,in=0](180:2)arc(180:110:2cm);
    \fill[gray,fill opacity=0.5] (0,0) circle (2cm);
\end{scope}
\begin{scope}[even odd rule]
    \clip (50:2) to[out=-130, in=-30] (150:2)arc(150:50:2cm) 
    (80:2)to[out=-45,in=90](60:1.6)to[out=-90,in=70](B)--(-2,0)--(-2,2)--(0.5,2)--cycle;
    \fill[gray,fill opacity=0.5] (50:2) to[out=-130, in=-30] (150:2)arc(150:50:2cm) ;    
\end{scope}
\begin{scope}[even odd rule]
    \clip (5:2) to[out=-175, in=60] (-120:2) arc(-120:5:2cm) 
    (C)to[out=-80,in=180](-20:1.5)to[out=0,in=70](-70:2)--(0.5,-2)--(-2,-2)--(-2,0.5)--cycle;
    \fill[gray, fill opacity=0.5](5:2) to[out=-175, in=60] (-120:2) arc(-120:5:2cm) ;
\end{scope}

\begin{scope}[even odd rule]
    \clip (0,0) circle (2cm)  (50:2) to[out=-130, in=-30] (150:2)arc(150:50:2cm);
    \clip (0,0) circle (2cm) (5:2) to[out=-175, in=60] (-120:2) arc(-120:5:2cm);
    \clip (0,0) circle (2cm) (70:2)to[out=-110,in=120](-60:2)arc(-60:70:2cm);
    \clip (0,0) circle (2cm) (110:2)to[out=-70,in=0](180:2)arc(180:110:2cm);
    \clip (0,0) circle (2cm) (A)to[out=-80,in=90](-0.1,0)to[out=-90,in=0](-145:1.6)to[out=180,in=-80](-170:2)arc(-170:-270:2cm)--cycle;
    \fill[gray,fill opacity=0.5] (0,0) circle (2cm);
    
\end{scope}
}
\newcommand{\SaddleKinda}{
\draw[Blue,->-](2,2)ellipse(1cm and 0.5cm);
\path[name path= Top](2,2)ellipse(1cm and 0.5cm);
\path[name path=Rectangle] (-1.5,0.5)--(1.5,3.5)--(1.5,-0.5)--(-1.5,-3.5)--cycle;
\path[name intersections={of=Top and Rectangle, by={A,B}}];
\draw[Blue](1.5,3.5)--(A);
\draw[Blue,->-] (-1.5,0.5)--(1.5,3.5);
\draw[Blue](-1.5,-3.5)--(-1.5,0.5);
\draw[Blue](3,2)--(3,-2);
\draw[Blue,->-](-1.5,-3.5)to[out=45,in=180](2,-2.5)arc(-90:0:1cm and 0.5cm);
\draw[Blue,dotted](3,-2)arc(0:90:1cm and 0.5cm)to[out=180,in=-135](1.5,-0.5)--(A);
\draw[Blue,](1,2)--(1,1)to[out=-90,in=-30,looseness=1.2](0,0);
}

\newcommand{\MMElevenId}{
\draw[Blue](1,1)arc(180:0:1cm);
\path[name path= Top](2,1)ellipse(1cm and 1cm);
\path[name path=Rectangle] (-1.5,0.5)--(1.5,3.5)--(1.5,-0.5)--(-1.5,-3.5)--cycle;
\path[name intersections={of=Top and Rectangle, by={A,B}}];
\draw[Blue](1.5,3.5)--(A);
\draw[Blue,->-] (-1.5,0.5)--(1.5,3.5);
\draw[Blue](-1.5,-3.5)--(-1.5,0.5);
\draw[Blue](3,1)--(3,-2);
\draw[Blue,->-](-1.5,-3.5)to[out=45,in=180](2,-2.5)arc(-90:0:1cm and 0.5cm);
\draw[Blue,dotted](3,-2)arc(0:90:1cm and 0.5cm)to[out=180,in=-135](1.5,-0.5)--(A);
\draw[Blue,](1,1)to[out=-90,in=-30,looseness=1.2](0,0);
}

\newcommand{\BCapE}[2]{

\draw[draw=Blue]

  (1,0) arc (0:180:1cm)     
  -- (-1,0) arc  (-180:0 :1cm and 0.5cm)
  -- cycle ; 
\draw[draw=Blue, dotted] (1,0) arc (0:180:1cm and 0.5cm) ;
\ifthenelse{\equal{#1}{1}}{
    \fill[Blue] (0,0.7) circle (2pt);
    }{}
        
\ifthenelse{\equal{#1}{2}}{
    \fill[Blue] (0.1,0.7) circle (2pt);
    \fill[Blue] (-0.1,0.7) circle (2pt);
    }{}
\ifthenelse{\equal{#2}{+}}{
\draw[draw=Blue,-<-]  (-1,0) arc  (-180:0 :1cm and 0.5cm);

}{}
\ifthenelse{\equal{#2}{-}}{
\draw[draw=Blue,->-]  (-1,0) arc  (-180:0 :1cm and 0.5cm);

}{}
}

\newcommand{\BRCapE}[4]{

\draw[draw=Blue]

  (1,0.5) arc (0:180:1cm)     
  -- (-1,0) arc  (-180:0 :1cm and 0.5cm)
  -- cycle ; 
\draw[draw=Blue, dotted] (1,0) arc (0:180:1cm and 0.5cm) ;
\fill[Red, fill opacity= 0.60] (0,0.5) ellipse (1cm and 0.5cm);
\draw[draw=Maroon, thick,dotted] (1,0.5) arc (0:180:1cm and 0.5cm) ;
\draw[draw=Maroon, thick]  (-1,0.5) arc  (-180:0 :1cm and 0.5cm);

\ifthenelse{\equal{#1}{1}}{
    \fill[Blue] (0,1.2) circle (2pt);
    }{}
        
\ifthenelse{\equal{#1}{2}}{
    \fill[Blue] (0.1,1.2) circle (2pt);
    \fill[Blue] (-0.1,1.2) circle (2pt);
    }{}

\ifthenelse{\equal{#2}{+}}{
\draw[draw=Blue,-<-]  (-1,0) arc  (-180:0 :1cm and 0.5cm);
\draw[draw=Maroon,thick,->-]  (-1,0.5) arc  (-180:0 :1cm and 0.5cm);

}{}
\ifthenelse{\equal{#2}{-}}{
\draw[draw=Blue,->-]  (-1,0) arc  (-180:0 :1cm and 0.5cm);
\draw[draw=Maroon,thick,-<-]  (-1,0.5) arc  (-180:0 :1cm and 0.5cm);

}{}

\ifthenelse{\equal{#3}{1}}{
    \fill[Maroon] (0,0.5) circle (2pt);
    }{}

\ifthenelse{\equal{#4}{1}}{
    \fill[Blue] (0,-0.2) circle (2pt);
    }{}

}

\newcommand{\BCupE}[2]{

\draw[draw=Blue,]

  (1,0) arc (0:180:1cm and 0.5cm)     
  -- (-1,0) arc  (-180:0 :1cm)
  -- cycle ; 
\draw[draw=Blue] (-1,0) arc (-180:0:1cm and 0.5cm) ;
\ifthenelse{\equal{#1}{1}}{
    \fill[Blue] (0,-0.7) circle (2pt);
    }{}
        
\ifthenelse{\equal{#1}{2}}{
    \fill[Blue] (0.1,-0.7) circle (2pt);
    \fill[Blue] (-0.1,-0.7) circle (2pt);
    }{}
\ifthenelse{\equal{#2}{+}}{
\draw[draw=Blue,->-] (-1,0) arc (-180:0:1cm and 0.5cm) ;

}{}
\ifthenelse{\equal{#2}{-}}{
\draw[draw=Blue,-<-]  (1,0) arc  (0:180 :1cm and 0.5cm);

}{}
}

\newcommand{\BRCupE}[4]{

\draw[draw=Blue, ]

  (1,0) arc (0:180:1cm and 0.5cm)     
  -- (-1,-0.5) arc  (-180:0 :1cm)
  -- cycle ; 
\draw[draw=Blue] (-1,0) arc (-180:0:1cm and 0.5cm) ;

\fill[Red, fill opacity= 0.60] (0,-0.5) ellipse (1cm and 0.5cm);
\draw[draw=Maroon, thick,dotted] (1,-0.5) arc (0:180:1cm and 0.5cm) ;
\draw[draw=Maroon, thick]  (-1,-0.5) arc  (-180:0 :1cm and 0.5cm);

\ifthenelse{\equal{#1}{1}}{
    \fill[Blue] (0,-1.2) circle (2pt);
    }{}
        
\ifthenelse{\equal{#1}{2}}{
    \fill[Blue] (0.1,-1.2) circle (2pt);
    \fill[Blue] (-0.1,-1.2) circle (2pt);
    }{}

\ifthenelse{\equal{#2}{+}}{
\draw[draw=Blue,->-]  (1,0) arc  (0:180:1cm and 0.5cm);
\draw[draw=Blue,->-] (-1,0) arc (-180:0:1cm and 0.5cm) ;
\draw[draw=Maroon, thick,->-] (-1,-0.5) arc (-180:0:1cm and 0.5cm) ;

}{}
\ifthenelse{\equal{#2}{-}}{
\draw[draw=Blue,-<-]  (1,0) arc  (0:180 :1cm and 0.5cm);
\draw[draw=Blue,-<-] (-1,0) arc (-180:0:1cm and 0.5cm) ;
\draw[draw=Maroon, thick,-<-] (-1,-0.5) arc (-180:0:1cm and 0.5cm) ;
}{}

\ifthenelse{\equal{#3}{1}}{
    \fill[Maroon] (0,-0.6) circle (2pt);
    }{}

\ifthenelse{\equal{#4}{1}}{
    \fill[Blue] (0,0.1) circle (2pt);
    }{}

}

\newcommand{\BCylE}[3]{
    \begin{scope}
        \draw[draw=Blue,] (1,2) arc (0:180: 1 cm and 0.5 cm) -- (-1,-2) arc (-180:0: 1 cm and 0.5 cm) -- cycle;
        \draw[draw=Blue] (-1,2) arc (-180: 0: 1 cm and 0.5 cm);
        \draw[draw=Blue,dotted] (1,-2) arc (0:180: 1 cm and 0.5 cm);
        \ifthenelse{\equal{#1}{1}}{
        \fill[Blue] (0,0) circle (2pt);
        }{}
        \ifthenelse{\equal{#1}{2}}{
        \fill[Blue] (0,-0.5) circle (2pt);
        \fill[Blue] (0,0.5) circle (2pt);
        }{}    
        \ifthenelse{\equal{#2}{+}}{
        \draw[draw=Blue,->-] (-1,2)  arc (-180:0:1cm and 0.5 cm);
        }{
        \draw[draw=Blue,-<-] (-1,2)  arc (-180:0:1cm and 0.5 cm);
        }
        \ifthenelse{\equal{#3}{+}}{
        \draw[draw=Blue,->-] (-1,-2) arc (-180:0: 1 cm and 0.5 cm);    
        }{
        \draw[draw=Blue,-<-] (-1,-2) arc (-180:0: 1 cm and 0.5 cm);    
        }
    \end{scope}

}

\newcommand{\BIdE}[3]{
\begin{scope}
    \draw[draw=Blue,]  
    (-1.5,0.5)--(1.5,3.5)--(1.5,-0.5)--(-1.5,-3.5)--cycle;
    \ifthenelse{\equal{#1}{1}}{
    \fill[Blue] (0,0) circle (2pt);
    }{}
    \ifthenelse{\equal{#1}{2}}{
    \fill[Blue] (0,-0.5) circle (2pt);
    \fill[Blue] (0,0.5) circle (2pt);
    }{}   
    \ifthenelse{\equal{#2}{+}}{
    \draw[draw=Blue,->-] (-1.5,0.5)--(1.5,3.5);
    
    }{}
    \ifthenelse{\equal{#2}{-}}{
    \draw[draw=Blue,-<-] (-1.5,0.5)--(1.5,3.5);
   
    }{}
    \ifthenelse{\equal{#3}{+}}{
    \draw[draw=Blue,->-] (1.5,-0.5)--(-1.5,-3.5);
    
    }{}
    \ifthenelse{\equal{#3}{-}}{

    \draw[draw=Blue,-<-] (1.5,-0.5)--(-1.5,-3.5);
   
    }{}
\end{scope}

}

\newcommand{\DoubleBubbleB}[1]{
        \begin{scope}[scale=#1]
        \draw[Blue](-1,-3)--(-1.5,-3.5)--(-1.5,0.5)--(1.5,3.5)--(1.5,-0.5)--(1,-1);
        \draw[Maroon,thick,](-1,-3)--(-0.5,-2.5); 
        \draw[Maroon,thick,dotted](0.5,-1.5)--(1,-1);
        \begin{scope}[even odd rule]
            \clip (0,-2)rectangle(1,-1) (0,-1)arc(90:0:1cm);
        \draw[Maroon,thick](0.5,-1.5)--(1,-1);
        \end{scope}
        \draw[Blue](-0.5,-2.5)to[out=-10,in=-100](1,-2);
        \draw[Blue,dotted](1,-2)to[out=90,in=0](0.5,-1.5);
        \draw[Maroon,thick](-0.5,-2.5)to[out=100,in=180](0,-1);
        \draw[Maroon,thick,dotted](0,-1)to[out=0,in=90](0.5,-1.5);
        \draw[Blue](0,-1)arc(90:0:1cm);
        \draw[Maroon,thick](-1,-3)to[out=100,in=180] (0,0)to[out=0,in=90](1,-1);
        \draw[Blue,dotted](-0.5,-2.5)--(0.5,-1.5);
        \draw[Blue,dotted](0,0)arc(90:180:2cm);
        \draw[Blue,dotted](-1,-3)to[out=180,in=-90](-2,-2)to[out=45,in=180,looseness=0.7](1,-1);
        \begin{scope}
            \clip(-1.5,-3.5)rectangle (-2.5,0.5);
            \draw[Blue](-1,-3)to[out=180,in=-90](-2,-2);
            \draw[Blue](0,0)arc(90:180:2cm);
        \end{scope}
        \begin{scope}[even odd rule]
            \clip (-1.5,-3.5)rectangle(1.5,3.5) (0,-1)arc(90:0:1cm);
            \fill[Red,fill opacity=0.6] (-0.5,-2.5)to[out=100,in=180](0,-1)to[out=0,in=90](0.5,-1.5)--(1,-1)to[out=90,in=0](0,0)to[in=100,out=180](-1,-3)--cycle;
        \end{scope}
    \end{scope}
}

\newcommand{\DoubleBubbleT}[1]{
        \begin{scope}[scale=#1]
        \draw[Blue] (-1,1)--(-1.5,0.5)--(-1.5,-3.5)--(1.5,-0.5)--(1.5,3.5)--(1,3);
        \draw[Maroon,thick](-1,1)--(-0.5,1.5); 
        \draw[Maroon,thick](0.5,2.5)--(1,3);      
        \draw[Blue](-0.5,1.5)--(0.5,2.5);     
        \draw[Blue](-0.5,1.5)to[out=-10,in=-90](1,2)to[out=90,in=0,looseness=0.7](0.5,2.5);
        \draw[Blue](-1,1)to[out=180,in=-90](-2,2)to[out=90,in=180,looseness=0.7](1,3);
        \draw[Blue](-0.5,1.5)to[out=-90,in=180](0,1);
        \draw[Blue,dotted](0,1)to[out=0,in=90](0.5,2.5);
        \begin{scope}
            \clip (-0.5,1.5)to[out=-10,in=-90](1,2)to[out=90,in=0,looseness=0.7](0.5,2.5);
            \draw[Blue](0,1)to[out=0,in=90](0.5,2.5);        
        \end{scope}
        \draw[Blue](0,1)arc(-90:0:1cm);
        \begin{scope}[even odd rule]
            \clip (-1.5,-3.5)rectangle(1.5,3.5) (0.5,2.5)to[out=0,in=0,looseness=0.7](1,2)arc(0:-90:1cm);
           \draw[Maroon,thick](0,0)to[out=0,in=-90,looseness=0.7](1,3); 
           \fill[Red,fill opacity=0.6] (-0.5,1.5)to[out=-90,in=180](0,1)to[out=0,in=90](0.5,2.5)--(1,3)to[out=90,in=0](0,0)to[in=-90,out=180](-1,1)--cycle;      
        \end{scope}
        \draw[Maroon,thick,dotted](0,0)to[out=0,in=-90,looseness=0.7](1,3);
        \draw[Maroon,thick](0,0)to[in=-90,out=180](-1,1);
        \begin{scope}
            \clip (1.5,3.5)--(-2.5,3.5)--(-2.5,-3.5)--(-1.5,-3.5)--(-1.5,0.5)--(1.5,3.5);
            \draw[Blue](-2,2)arc(-180:-90:2cm);
        \end{scope}
        \begin{scope}
            \clip (1.5,3.5)--(-1.5,0.5)--(-1.5,-3.5)--(1.5,-0.5)--(1.5,3.5);
            \draw[Blue,dotted](-2,2)arc(-180:-90:2cm);
        \end{scope}
    \end{scope}
   
}

\newcommand{\BridgeBubbleOne}[1]{
    \begin{scope}[scale=#1]
        
    \draw[Blue] (-1.5,2.5)to[bend right](1.5,5.5);
    \draw[Blue] (-1.5,-3.5) to[bend right](0,-2);
    \draw[Blue] (1,-2)arc(-180:0:1cm and 0.5cm);
    \draw[Blue] (4,-2)arc(-180:0:1cm and 0.5cm);
    \draw[Blue] (1.5,5.5)to[out=-90,in=90](6,-2);
    \draw[Blue] (1,-2)arc(180:0:1cm);
    \draw[Blue] (0,-2)arc(180:0:2cm);
    \draw[Blue](-1.5,2.5)--(-1.5,-3.5);
    \draw[Blue,dotted] (1,-2)arc(180:0:1cm and 0.5cm);
    \draw[Blue,dotted] (4,-2)arc(180:0:1cm and 0.5cm);
    \node at (0,1) {\tiny a};
    \node at (2,-2) {\tiny a};
    \begin{scope}[xshift=8.5cm]
        \draw[Blue] (-1.5,2.5)to[bend left](1.5,5.5)--(1.5,-0.5)to[bend right](-1.5,-3.5)--cycle;
        \node at (0,0) {\tiny a};
    \end{scope}
    \end{scope}

}

\newcommand{\BridgeBubbleTwo}[1]{
    \begin{scope}[scale=#1]
    \draw[Blue,scale=-1] (1,-2)arc(-180:0:1cm and 0.5cm);
    \draw[Blue,scale=-1] (4,-2)arc(-180:0:1cm and 0.5cm);
    \draw[Blue,scale=-1] (1.5,5.5)to[out=-90,in=90](6,-2);
    \draw[Blue,scale=-1] (1,-2)arc(180:0:1cm);
    \draw[Blue,scale=-1] (0,-2)arc(180:0:2cm);
    \draw[Blue,scale=-1] (1,-2)arc(180:0:1cm and 0.5cm);
    \draw[Blue,scale=-1] (4,-2)arc(180:0:1cm and 0.5cm);
    \draw[Blue] (0,2)to[bend left](1.5,3.5);
    \draw[Blue](-1.5,-5.5)to[bend left](1.5,-3.5);
    \draw[Blue](0,2)to[in=90,out=-90](-1.5,-5.5);
    \draw[Blue](1.5,3.5)--(1.5,-3.5);
    \node at (0,-1) {\tiny a};
    \node at (-2,2) {\tiny a};
    \begin{scope}[xshift=-8.5cm,scale=-1]
        \draw[Blue] (-1.5,2.5)to[bend left](1.5,5.5)--(1.5,-0.5)to[bend right](-1.5,-3.5)--cycle;
        \node at (0,0) {\tiny a};
    \end{scope}
    \end{scope}
}

\newcommand{\MergeOne}[1]{
    \begin{scope}[scale=#1]
        \draw[Blue] (0,0.2) to[in=-120,out=0] (60:0.4);
        \draw[Blue] (0,0.2) to[in=-60,out=180] (120:0.4);
        \draw[dashed,Maroon](0,-0.4)--(0,0.4);
        \draw[Maroon,thick] (0,-0.2)--(0,0.2);
        \draw[Blue] (-60:0.4) to[in=0, out=120] (0,-0.2);
        \draw[Blue] (-120:0.4) to[in=180, out=60] (0,-0.2);
        \draw[Blue,dotted](60:0.4)to[in=90,out=30](0.6,0)to[in=-30,out=-90](-60:0.4);
        \draw[Blue,dotted,xscale=-1](60:0.4)to[in=90,out=30](0.6,0)to[in=-30,out=-90](-60:0.4);
    \end{scope}
}
\newcommand{\SplitOne}[1]{
    \begin{scope}[scale=#1]
        \draw[Blue,bend right] (-120:0.4)to(120:0.4);
        \draw[Blue, bend left] (-60:0.4)to(60:0.4);
        \draw[dashed,Maroon](0,-0.4)--(0,0.4);
        \draw[Blue,dotted](60:0.4)to[in=90,out=30](0.6,0)to[in=0,out=-90](-60:0.4);
        \draw[Blue,dotted,xscale=-1](60:0.4)to[in=90,out=30](0.6,0)to[in=0,out=-90](-60:0.4);
    \end{scope}
}

\newcommand{\MergeTwo}[1]{
    \begin{scope}[scale=#1]
        \draw[Blue] (0,0.2) to[in=-120,out=0] (60:0.4);
        \draw[Blue] (0,0.2) to[in=-60,out=180] (120:0.4);
        \draw[dashed,Maroon](0,-0.4)--(0,0.4);
        \draw[Maroon,thick] (0,-0.2)--(0,0.2);
        \draw[Blue] (-60:0.4) to[in=0, out=120] (0,-0.2);
        \draw[Blue] (-120:0.4) to[in=180, out=60] (0,-0.2);
        \draw[Blue,dotted](60:0.4)to[in=90,out=30](0.6,0)to[in=-30,out=-90](-60:0.4);
        \draw[Blue,dotted,looseness=1.5](120:0.4)to[in=90,out=100](1,0)to[in=-100,out=-90](-120:0.4);
    \end{scope}
}
\newcommand{\SplitTwo}[1]{
    \begin{scope}[scale=#1]
        \draw[Blue,bend right] (-120:0.4)to(120:0.4);
        \draw[Blue, bend left] (-60:0.4)to(60:0.4);
        \draw[dashed,Maroon](0,-0.4)--(0,0.4);
        \draw[Blue,dotted](60:0.4)to[in=90,out=30](0.6,0)to[in=-30,out=-90](-60:0.4);
        \draw[Blue,dotted,looseness=1.5](120:0.4)to[in=90,out=100](1,0)to[in=-100,out=-90](-120:0.4);
    \end{scope}
}

\newcommand{\SplitThree}[1]{
    \begin{scope}[scale=#1]
        \draw[Blue] (0,0.2) to[in=-120,out=0] (60:0.4);
        \draw[Blue] (0,0.2) to[in=-60,out=180] (120:0.4);
        \draw[dashed,Maroon](0,-0.4)--(0,0.4);
        \draw[Maroon,thick] (0,-0.2)--(0,0.2);
        \draw[Blue] (-60:0.4) to[in=0, out=120] (0,-0.2);
        \draw[Blue] (-120:0.4) to[in=180, out=60] (0,-0.2);
        \draw[Blue,dotted](60:0.4)to[in=0,out=60](0,0.6)to[in=120,out=180](120:0.4);
        \draw[Blue,dotted,yscale=-1](60:0.4)to[in=0,out=60](0,0.6)to[in=120,out=180](120:0.4);
    \end{scope}
}
\newcommand{\MergeThree}[1]{
    \begin{scope}[scale=#1]
        \draw[Blue,bend right] (-120:0.4)to(120:0.4);
        \draw[Blue, bend left] (-60:0.4)to(60:0.4);
        \draw[dashed,Maroon](0,-0.4)--(0,0.4);
        \draw[Blue,dotted](60:0.4)to[in=0,out=60](0,0.6)to[in=120,out=180](120:0.4);
        \draw[Blue,dotted,yscale=-1](60:0.4)to[in=0,out=60](0,0.6)to[in=120,out=180](120:0.4);
    \end{scope}
}

\newcommand{\SplitFour}[1]{
    \begin{scope}[scale=#1]
        \draw[Blue] (0,0.2) to[in=-120,out=0] (60:0.4);
        \draw[Blue] (0,0.2) to[in=-60,out=180] (120:0.4);
        \draw[dashed,Maroon](0,-0.4)--(0,0.4);
        \draw[Maroon,thick] (0,-0.2)--(0,0.2);
        \draw[Blue] (-60:0.4) to[in=0, out=120] (0,-0.2);
        \draw[Blue] (-120:0.4) to[in=180, out=60] (0,-0.2);
        \draw[Blue,dotted](60:0.4)to[in=0,out=60](0,0.6)to[in=120,out=180](120:0.4);
        \draw[Blue,dotted,looseness=1.5](-60:0.4)to[in=0,out=-60](0,1)to[in=-120,out=180](-120:0.4);
    \end{scope}
}
\newcommand{\MergeFour}[1]{
    \begin{scope}[scale=#1]
        \draw[Blue,bend right] (-120:0.4)to(120:0.4);
        \draw[Blue, bend left] (-60:0.4)to(60:0.4);
        \draw[dashed,Maroon](0,-0.4)--(0,0.4);
        \draw[Blue,dotted](60:0.4)to[in=0,out=60](0,0.6)to[in=120,out=180](120:0.4);
        \draw[Blue,dotted,looseness=1.5](-60:0.4)to[in=0,out=-60](0,1)to[in=-120,out=180](-120:0.4);
    \end{scope}
}

\newcommand{\hTh}[1]{
    \begin{scope}[scale=#1]
        
    \draw[Blue] (-3,-4.5)--(-3.5,-4.5)--(-3.5,0)--(-0.5,0)--(1,-1.5)--(2,-1.5)--(0.5,0)--(3.5,0)--(3.5,-1);
    \draw[Blue] (1,-1.5)arc(-180:0:0.5cm);
    \draw[Blue](0.25,-1.5)arc(0:90:1cm and 0.5cm)--(-3,-1)--(-3,-5.5)--(-2.5,-5.5);
    \draw[Blue] (4,-2)--(4,-1)--(3.75,-1)arc(90:270:1cm and 0.5cm)--(4.5,-2)--(4.5,-6.5)--(-2.5,-6.5)--(-2.5,-2)--(-0.75,-2)arc(-90:0:1cm and 0.5cm);
    \draw[Maroon,thick,fill=Red,fill opacity=0.6](2,-1.5)--(2.75,-1.5)arc(0:-180:1.25cm)--(1,-1.5)arc(-180:0:0.5cm);
    \draw[Maroon,thick](0.5,0)--(-0.5,0)arc(-180:-90:0.5cm);
    \draw[Maroon,thick,dotted](0,-0.5)arc(-90:0:0.5cm);
    \draw[Maroon,thick,dotted,fill=Red, fill opacity=0.6](-1.25,-5)--(-0.75,-5)arc(180:0:0.75cm)--(1.25,-5)arc(0:180:1.25cm);
    \draw[Blue,dotted]
    (-3,-4.5)--(-2.25,-4.5)arc(90:-90:1cm and 0.5cm)--(-2.5,-5.5)
    (3.5,-1)--(3.5,-4.5)--(2.25,-4.5)arc(90:270:1cm and 0.5cm)--(4,-5.5)--(4,-2)
    (-0.75,-5)arc(-180:180:0.75cm and 0.5cm)
    (-0.75,-5)arc(180:0:0.75cm)
    ;
    \draw[Blue](-0.35,-0.35)--(0.75,-1.5);
    \draw[Blue,dotted](0.75,-1.5)--(1.15,-1.85);
    \fill[Red,fill opacity=0.6](-0.5,0)arc(-180:0:0.5cm);
    \end{scope}
}

\newcommand{\hFo}[1]{
    \begin{scope}[scale=#1]
    \draw[Blue] (-3,-4.5)--(-3.5,-4.5)--(-3.5,0)--(-0.5,0)--(1,-1.5)--(2,-1.5)--(0.5,0)--(3.5,0)--(3.5,-1);
    \draw[Blue] (1,-1.5)arc(-180:0:0.5cm);
    \draw[Blue](0.25,-1.5)arc(0:90:1cm and 0.5cm)--(-3,-1)--(-3,-5.5)--(-2.5,-5.5);
    \draw[Blue] (4,-2)--(4,-1)--(3.75,-1)arc(90:270:1cm and 0.5cm)--(4.5,-2)--(4.5,-6.5) (-2.5,-6.5)--(-2.5,-2)--(-0.75,-2)arc(-90:0:1cm and 0.5cm);
    \draw[Maroon,thick](0.5,0)--(-0.5,0)arc(-180:-90:0.5cm);
    \draw[Maroon,thick,dotted](0,-0.5)arc(-90:0:0.5cm);
    \draw[Blue,dotted]
    (-3,-4.5)--(0.5,-4.5)arc(90:-90:1cm and 0.5cm)--(0.5,-5.5)arc(90:180:1cm and 0.5cm)
    (3.5,-1)--(3.5,-4.5)arc(90:270:1cm and 0.5cm)--(4,-5.5)--(4,-2)
    (-2.5,-5.5)arc(90:0:1cm and 0.5cm)
    ;
    \draw[Blue](-0.35,-0.35)--(0.75,-1.5);
    \draw[Blue,dotted](0.75,-1.5)--(1.15,-1.85);
    \fill[Red,fill opacity=0.6](-0.5,0)arc(-180:0:0.5cm);
    \draw[Blue](-1.5,-6)arc(0:-90:1cm and 0.5cm)
    (-0.5,-6)arc(-180:-90:1cm and 0.5cm)--(4.5,-6.5)
    ;
    \draw[Maroon,thick,dotted,fill=Red,fill opacity=0.6](1.5,-5)--(2.5,-5)arc(0:180:0.5cm);
    \draw[Maroon,thick,fill=Red,fill opacity=0.6](-1.5,-6)to[out=90,in=-90](0.25,-1.5)--(1,-1.5)arc(-180:0:0.5cm)--(2.75,-1.5)to[out=-90,in=90](-0.5,-6)--cycle;
    \end{scope}
    }

\newcommand{\hOn}[1]{
    \begin{scope}[scale=#1]
        
    \draw[Maroon,thick, fill=Red, fill opacity=0.6,yshift=-0.5cm](1.5,0.25)arc(0:-180:1.5cm)--(-0.75,0.25)arc(-180:0:0.75cm)--cycle;
    \draw[Maroon, thick, dotted, fill=Red, fill opacity=0.6](0.5,-3.75)arc(180:0:0.5cm)--cycle;
    \draw[Blue] (-3,0)--(-2,0)arc(90:-90: 0.5cm and 0.25cm)--(-2.5,-0.5)
    (3,0)--(2,0)arc(90:270:0.5cm and 0.25cm)--(3.5,-0.5)
    (0,-0.25)ellipse(0.75cm and 0.25cm)
    (-3.5,0.5)--(2.5,0.5);

    \draw[Blue](-2.5,-4.5)--(3.5,-4.5)
    (-3,-4)--(-2.5,-4)
    (-3.5,-3.5)--(-3,-3.5);
    \draw[Blue,dotted](-2.5,-4)--(0,-4)arc(-90:90:0.5cm and 0.25cm)--(-3,-3.5)
    (3,-4)--(2,-4)arc(-90:-270:0.5cm and 0.25cm)--(2.5,-3.5);

    \draw[Blue](3.5,-0.5)--(3.5,-4.5) (3,0)--(3,-0.5) (2.5,0.5)--(2.5,0)
    (-3.5,0.5)--(-3.5,-3.5) (-3,0)--(-3,-4) (-2.5,-0.5)--(-2.5,-4.5);
    \draw[Blue,dotted](3,-0.5)--(3,-4) (2.5,0)--(2.5,-3.5);

    \end{scope}
}

\newcommand{\hTw}[1]{
    \begin{scope}[scale=#1]
        
    \draw[Maroon,thick,fill=Red,fill opacity=0.6] (-1.5,-0.25)arc(-180:0:0.5cm)--cycle;
    \draw[Maroon,thick,fill=Red,fill opacity=0.6,dotted] (0.5,-0.5)--(0.5,-3.75)--(1.5,-3.75)--(1.5,-0.5);
    \draw[Maroon,thick,fill=Red,fill opacity=0.6] (1.5,-0.5)--(1.5,0.25)--(0.5,0.25)--(0.5,-0.5);

    \draw[Blue](-3.5,0.5)--(0,0.5)arc(90:-90:0.5cm and 0.25cm)--(0,0)arc(90:270:0.5cm and 0.25cm)--(3.5,-0.5)
    (2.5,0.5)--(2,0.5)arc(90:270:0.5cm and 0.25cm)--(3,0)
    (-3,0)--(-2,0)arc(90:-90:0.5cm and 0.25cm)--(-2.5,-0.5);

    \draw[Blue](-2.5,-4.5)--(3.5,-4.5)
    (-3,-4)--(-2.5,-4)
    (-3.5,-3.5)--(-3,-3.5);
    \draw[Blue,dotted](-2.5,-4)--(0,-4)arc(-90:90:0.5cm and 0.25cm)--(-3,-3.5)
    (3,-4)--(2,-4)arc(-90:-270:0.5cm and 0.25cm)--(2.5,-3.5);

    \draw[Blue](3.5,-0.5)--(3.5,-4.5) (3,0)--(3,-0.5) (2.5,0.5)--(2.5,0)
    (-3.5,0.5)--(-3.5,-3.5) (-3,0)--(-3,-4) (-2.5,-0.5)--(-2.5,-4.5);
    \draw[Blue,dotted](3,-0.5)--(3,-4) (2.5,0)--(2.5,-3.5);

    \end{scope}
}

\usepackage[a4paper,top=3cm, bottom=3cm, left=2.5cm, right=2.5cm]{geometry}

\usepackage[backend=biber,style=alphabetic]{biblatex}
\title{On the functoriality of Khovanov homology}
\date{}
\author{Dorra Hamza 
\\
Supervisor: Emmanuel Wagner}

\begin{document}
\maketitle
\begin{abstract}
        We provide an overview of Blanchet’s oriented model of Khovanov homology, which resolves the sign ambiguity in functoriality. We give a detailed and self-contained exposition of the construction. We establish strict functoriality for the oriented model and briefly outline its extension to cobordisms with singularities recently studied by Scott Carter, Benjamin Cooper, Mikhail Khovanov and Vyacheslav Krushkal in \cite{ImmersedCobordisms}.
\end{abstract}
\tableofcontents
\section{Introduction}
In \cite{Khovanov2000} Khovanov defined a homology theory categorifying the Jones polynomial. It was later proven \cite{Fun} that Khovanov homology of links is functorial up to sign ambiguity: cobordisms between links induce maps between their respective Khovanov complexes defined up to homotopy and overall sign. Since then there were several attempts to fix the sign ambiguity (\cite{Caprau2008}, \cite{FixSign}). One method, widely considered to be the most natural, was introduced by Blanchet \cite{BLANCHET2010}. He defined an oriented model of Khovanov homology replacing circles and surfaces in the original definition with \textit{trivalent graphs} and \textit{foams}. complete proof for $gl_n$ homology using similar tools was later given by Michael Ehrig, Daniel Tubbenhauer, Paul Wedrich in \cite{gln}. 
\par
Functoriality of Khovanov homology has important applications in low dimensional topology. Namely
it was used in the definition of Rasmussen’s s-invariant \cite{Rasmussen2010}, which gives a bound on the slice genus of knots and was used to give a combinatorial proof of the Milnor conjecture on the slice genus of torus knots. Functoriality has also been used in the development of the Skein–Lasagna invariant, which has significant applications in 4-dimensional topology (\cite{Lasagna2024}, \cite{Morrison2022}).\par
Recently, the construction of Khovanov homology of cobordisms was generalised to cobordisms with double point singularities \cite{ImmersedCobordisms}. The maps associated to such cobordisms are proven to be defined up to sign. \par 
However, there are not many accessible references that detail Blanchet’s construction and establish strict functoriality. The aim of this paper is to provide such a reference, accessible to readers from different backgrounds, including those who may not be familiar with the original sources. At the end of the paper, we also discuss how to extend the construction to cobordisms with singularities, and we also give a direct proof of strict functoriality in this setting as well. A different proof using functoriality of Khovanov homology in $\mathbb{S}^3$ is given in \cite{S3}.\par
The structure of the paper is as follows. In Section \ref{TheFoamCategory} we define trivalent graphs and foams, which form the foam category. In \cite{BLANCHET2010} Blanchet defines an integral evaluation of decorated closed foams, which he uses to construct a TQFT functor from the foam category to the category of graded modules. In Section \ref{EvaluationOfFoams} we present a more general formula of evaluation of foams into the ring of symmetric polynomials in two variables with integer coefficients. This is a specialization of a more general construction defined in \cite{Robert2020} for $sln-$foams with arbitrary $n$. We describe in Section \ref{universal} a way to define tautological functors from the foam category into the category of graded modules as in \cite{BLANCHET2010}. This construction turns out to be a categorification of the evaluation of closed graphs defined in Section \ref{EvaluationOfGraphs}. \par
We then proceed in Section \ref{TheDef} to define the Khovanov comlex of tangles using foams. We adapt a similar approach to the one used in \cite{BarNatan2005}, starting with a pictorial definition and working with formal complexes in the foam category, and then proving invariance up to homotopy with explicit maps. The following Section \ref{TrivalentTQFT} relates the evaluation of foams used in this paper to the TQFT functor defined by Blanchet. We then recover the Jones polynomial in Section \ref{Categorification} as the graded Euler characteristic of Khovanov homology. \par
Up until this point, there is no clear reason for the Khovanov complex defined by Blanchet to coincide with the original Khovanov complex. We present a proof that they are indeed isomorphic based on \cite{CupFoams} in Section \ref{sec:RecoveringKh}. \par
A final ingredient to the proof of functoriality would be Lee Khovanov homology, discussed in Section \ref{sec: Lee}. In this section we use the Karoubi envelope approach in \cite{BarNatan2006} to prove Lee degeneration in the foam setting. \par
Section \ref{Funct} is devoted for the proof of functoriality of Khovanov homology in the embedded case, starting with functoriality up to sign as in \cite{BarNatan2005}, and the fixing the sign using Lee/ Equivariant Khovanov homology. Finally, in Section \ref{Imm} we extend the construction to immersed cobordisms as in \cite{ImmersedCobordisms}, and we finish by proving strict functoriality in the category of immersed cobordisms. 
\section{The foam category}
\label{TheFoamCategory}
In the first two sections we want to define the evaluation formula of $sl_n-$Foams, introduced in \cite{Robert2020}, and specialized here to the case of $sl_2$. Let us first start by defining the foam category. 
\begin{deff}
    A closed trivalent graph is an oriented 3-valent graph embedded in $\mathbb{R}^2$ such that: 
    \begin{itemize}
        \item Edges are labeled with red or blue.
        \item At each 3-valent vertex we have two blue and one red edge. The orientation of these edges must satisfy the flow condition as in Figure \ref{fig:TrivalentVertices}.
        \item An order on two blue edges meeting at a trivalent vertex is fixed as shown in Figure \ref{fig:TrivalentVertices}.
    \end{itemize} 
    
    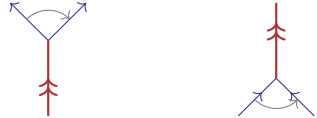
\begin{figure}[h]
    \centering
    \begin{tikzpicture}
        \begin{scope}
            \draw[draw=Maroon, thick,->>-] (0,-1)--(0,0);
            \draw[draw=Blue,->] (0,0)--(0.5,0.5);
            \draw[draw=Blue,->] (0,0)--(-0.5,0.5);   
            \draw[draw=gray,->] (0,0) ++(135:0.4) arc (135:45:0.4cm);

        \end{scope}

        \begin{scope}[xshift=3cm,yshift=-0.5cm]
            \draw[draw=Maroon, thick,-<<-] (0,1)--(0,0);
            \draw[draw=Blue,-<-] (0,0)--(0.5,-0.5);
            \draw[draw=Blue,-<-] (0,0)--(-0.5,-0.5);   
            \draw[draw=gray,->] (0,0) ++(-135:0.4) arc (-135:-45:0.4cm);
        \end{scope}

    \end{tikzpicture}
    \caption{Trivalent vertices.}
    \label{fig:TrivalentVertices}
    \end{figure}
    We also allow blue and red oriented circles. \\
    Let $B$ be a set of points in $\mathbb{S}^1 \subset \mathbb{R}^2$. A trivalent graph $\Gamma$ with boundary $B$ is an intersection $\Gamma'\pitchfork \mathbb{D}^2$ for a certain closed trivalent graph embedded in $\mathbb{R}^2$, such that $\Gamma' \cap \mathbb{S}^1 = B$. \\
    For simplicity we choose that all the arcs in a neighborhood of the boundary are blue.

\end{deff}

\begin{figure}[h]
    \centering
    \begin{tikzpicture}
        \begin{scope}[xshift=4cm,scale=0.6]
            \draw[draw=gray,dotted] (0,0) circle (2cm);
            \clip (0,0) circle (2cm);
            \draw[draw=Blue,->-] (-1,1)--(1,1);
            \draw[draw=Blue,-<-] (-1,-1)--(1,-1);
            \draw[draw=Maroon,thick,->>-] (-1,-1)--(-1,1);
            \draw[draw=Maroon,thick,-<<-] (1,-1)--(1,1);
            \draw[draw=Blue,postaction={decorate}, decoration={markings, mark=at position 0.3 with {\arrow{<}}}] (1,1)-- (2,2);
            \draw[draw=Blue,postaction={decorate}, decoration={markings, mark=at position 0.3 with {\arrow{>}}}] (-1,1)-- (-2,2) ;
            \draw[draw=Blue,postaction={decorate}, decoration={markings, mark=at position 0.3 with {\arrow{<}}}] (-1,-1)-- (-2,-2);
            \draw[draw=Blue,postaction={decorate}, decoration={markings, mark=at position 0.3 with {\arrow{>}}}] (1,-1)-- (2,-2);
        \end{scope}
        \begin{scope}[xshift=-4cm,scale=0.5]
            \draw[draw=Blue,postaction={decorate}, decoration={markings, mark=at position 0 with {\arrow{<}},mark=at position 0.5 with {\arrow{>}}}] (0,0) circle (2cm);
            \draw[draw=Blue,postaction={decorate}, decoration={markings, mark=at position 0 with {\arrow{>}},mark=at position 0.5 with {\arrow{<}}}] (0,0) circle (1cm);
            \draw[draw=Maroon,thick,->>-] (0,-2)--(0,-1);
            \draw[draw=Maroon,thick,->>-] (0,1)--(0,2);
        \end{scope}
        \begin{scope}[,scale=0.4]
             \begin{scope}
                \clip (0,3) rectangle (4.5,-3);
                \draw[draw=Blue,postaction={decorate}, decoration={markings, mark=at position 0.25 with {\arrow{<}}}] (1.5,0) circle (2cm);
            \end{scope}

            \begin{scope}
                \clip (0,3) rectangle (-4.5,-3);
                \draw[draw=Blue,postaction={decorate}, decoration={markings, mark=at position 0.25 with {\arrow{>}}}] (-1.5,0) circle (2cm);
            \end{scope}
            \begin{scope}
                \clip  (1.5,0) circle (2cm);
                \draw[draw=Maroon,thick,->>-] (0,-2)--(0,2);
            \end{scope}
        \end{scope}
       
    \end{tikzpicture}
    \caption{Some examples of trivalent graphs. The red edges are marked with two arrows so that they are distinguishable in grayscale.  }
\end{figure}
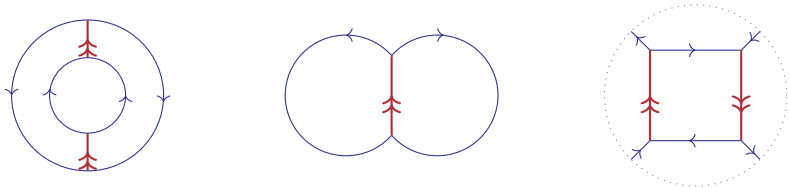

\begin{deff}
    A \textbf{foam} is a closed 2-complex whose regular faces (called \textbf{facets}) are labeled with blue or red, verifying:  
    \begin{itemize}
        \item The singular locus is called \textit{the binding}. A neighborhood of a connected component of the binding consists of two blue pages and one red page.
        \item The binding is oriented. 
        \item We choose the orientation of the red page  to respect that of the binding, and we take the opposite orientation on the two blue pages. 
        \item The blue pages are ordered by the right hand rule with respect to the orientation of the binding (Figure\ref{fig:order}).
        \item We may add decorations on the facets of the complex, points for the blue facets, and points and squares for the red facets.
    \end{itemize}
    \begin{figure}[h]
        \centering
        \begin{tikzpicture}
            \begin{scope}
                \draw[draw=Blue, fill = Cyan,fill opacity=0.2] (0,0) circle (1cm);
                \draw[draw=Maroon, thick, dotted] (1,0) arc (0:180:1cm and 0.5cm) ;
                \draw[draw=Maroon, thick,-<-](1,0) arc (0:-180:1cm and 0.5cm) ;
                \fill[Red, fill opacity= 0.60] (0,0) ellipse (1cm and 0.5cm);
                \fill[Blue] (0,0.9) circle (1pt);
                \fill[Maroon] (-2pt,-2pt) rectangle (2pt,2pt);
            \end{scope}
        \end{tikzpicture}
        \caption{Example of a (closed) foam. Red surfaces should show darker in grayscale.}
    \end{figure}
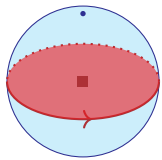
    Let $\Gamma_1$ and $\Gamma_2$ be two trivalent graphs with the same boundary $B$, potentially empty.
    A \textbf{cobordism} from $\Gamma_1$ to $\Gamma_2$, or a $\Gamma_1$-$foam$-$\Gamma_2$, is the intersection of a decorated foam $\Sigma$ embedded in $\mathbb{R}^3$ with $\mathbb{D}^2\times[0,1]$ such that:
    \begin{itemize}
        \item There exists an $\epsilon_0 > 0$ such that $\Sigma \cap \mathbb{D}^2\times[0,\epsilon_0]=(\Gamma_1)\times[0,\epsilon_0]$.
        \item There exists an $\epsilon_1 < 1$ such that $\Sigma \cap \mathbb{D}^2\times[\epsilon_1,1]=(-\Gamma_2)\times[\epsilon_1,1]$.
        \item $\Sigma \cap (\mathbb{S}^1\times[0,1])$ is the boundary $B$ times $[0,1]$
        \item The boundary of $\Sigma$ is $\Sigma\cap(\mathbb{S}^1\times [0,1])$
    \end{itemize}

    $(-\Gamma)$ is the graph $\Gamma$ with opposite orientation. \\
    We consider cobordisms up to ambient isotopy fixing the boundary. They will be drawn pointing down. \\
    We then have the \textbf{vertical composition} of cobordisms:\\ Let $\Gamma_0,\Gamma_1 $ and $\Gamma_2$ be trivalent graphs with the same boundary. Let $\Sigma_0$ be a cobordism from $\Gamma_0$ to $\Gamma_1$ and  $\Sigma_1$ a cobordism from $\Gamma_1$ to $\Gamma_2$ . The cobordism $\Sigma_1\circ \Sigma_0$ going from $\Gamma_0$ to $\Gamma_2$  is obtained by gluing the two along $\Gamma_1$ and rescaling to fit into $\mathbb{D}^2\times[0,1]$.\\
    A closed foam is regarded as a cobordism from the empty set to the empty set.
    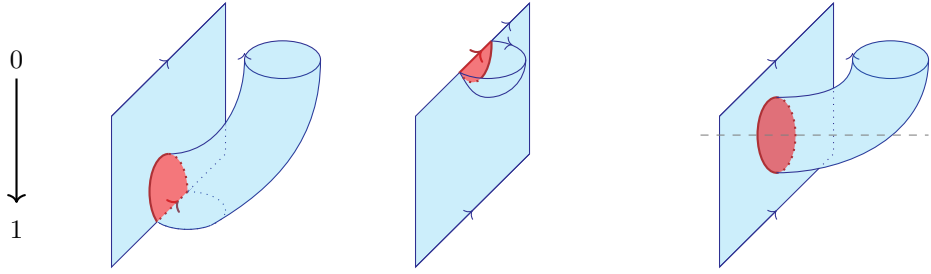
\begin{figure}[h]
        \centering
        \begin{tikzpicture}[scale=0.5]
            \node at (-4,2) {0};
            \node at (-4,-2.5) {1};
            \draw[thick,->](-4,1.5)--(-4,-1.8);
            \begin{scope}
                \dOne
            \end{scope}
            \begin{scope}[xshift=8cm]
                \dTwo
            \end{scope}
            \begin{scope}[xshift=16cm]
                \dThree
            \end{scope}
        \end{tikzpicture}

        \caption{Some cobordisms between trivalent graphs. The cobordism on the right is the result of the composition of the other two. }
    \end{figure}
    
\end{deff}

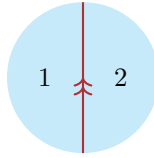
\begin{figure}[h]
    \centering
    \begin{tikzpicture}
        \begin{scope}
            \fill[Cyan!20] (0,0) circle (1cm);
            \draw[draw=Maroon,thick,->>-] (0,-1)--(0,1);
            \node at (0.5,0) {$2$};
            \node at (-0.5,0) {$1$};
        \end{scope}
    \end{tikzpicture}
    \caption{The order on the blue pages around the binding.}
    \label{fig:order}
\end{figure}

\begin{deff}
    The foam category $\mathcal{F}(B)$ has for objects the trivalent graphs with boundary $B$ (potentially empty), and morphisms $Hom_{\mathcal{F}}(\Gamma_0,\Gamma_1)$ the free $\mathbb{Z}[X_1,X_2]$-modules spanned by the cobordisms between $\Gamma_0$ and $\Gamma_1$. The identity morphism for a graph $\Gamma$ is the cobordism $\Gamma\times[0,1]$. The composition in the foam category is obtained by extending the vertical composition of cobordisms in the natural bilinear manner. 
\end{deff}
The foam category has a monoidal structure given by the disjoint union of trivalent graphs when $B$ is empty.\\
For a cobordism $\Sigma$ in the category $\mathcal{F}(B)$, let $\Sigma_b$ be the blue subsurface, p the number of points  and $s$ the number of squares in $\Sigma$. Let $P$ be a monomial in $\mathbb{Z}[X_1,X_2]$ of degree $d$.
Then the degree of $P\Sigma$ is $deg(P\Sigma)=\chi(\Sigma_b)-2(p+2s+d) - \frac{1}{2}|B|$. 
This makes $\mathcal{F}(B)$ into a graded category.

\begin{deff}
    A \textbf{graded category} is a pre-additive category $\mathcal{C}$ with the following two additional properties:
    \begin{itemize}
        \item For any two objects $O_{1,2}$ in $\mathcal{C}$, the morphisms $Mor(O_1,O_2)$ form a graded abelian group, the composition maps respect the gradings (ie, $deg(f\circ g) = deg(f) + deg(g)$ whenever this makes sense) and all identity maps are of degree 0.
        \item There is a $\mathbb{Z}$–action $(m,O) \rightarrow O\{m\}$, \textit{called grading shift by $m$}, on the objects of $\mathcal{C}$. As plain abelian groups, morphisms are unchanged by this action, $Mor(O_1\{m_1\},O_2\{m_2\}) = Mor(O_1,O_2)$, but gradings do change under the action; so if $f \in Mor(O_1,O_2)$ and $\deg(f) = d$, then as an element of $Mor(O_1\{m_1\},O_2\{m_2\})$ the degree of $f$ is $d + m_2 - m_1$.
    \end{itemize}
    A pre-additive category verifying only the first property can be seen as a graded category by adding artificial objects $O\{m\}$ for $m\in \mathbb{Z}$. 
\end{deff}

\section{Evaluation of foams}
\label{EvaluationOfFoams}
In this section we define an invariant of closed foams. We start with a few preliminary definitions.
\begin{deff} 
    A \textbf{coloring of a foam $\Sigma$ by $\{1,2\}$} assigns to each blue facet either $1$ or $2$, with the condition that both colors appear around every binding. Each red facet is assigned the full set $\{1,2\}$.\\
    Let $c$ be a coloring of a foam $\Sigma$. We denote by $\Sigma_1$ the subsurface of $\Sigma$ consisting of the facets of $\Sigma$ colored with $1$ or $\{1,2\}$, and by $\Sigma_b$ the subsurface of $\Sigma$ consisting of all the blue facets. The red subsurface is denoted $\Sigma_r$.\\
    In a neighborhood of a connected component of the binding we have one of two possible ways to color the two blue pages, either the order of the blue pages follows the order 1,2 or it doesn't.  Let $n_{1\uparrow 2 }$ be the number of the bindings where the coloring of the neighboring two blue pages agrees with their order (Figure \ref{fig:order}). \\ 
\end{deff}

\begin{claim}
    The subsurfaces $\Sigma_1$ and $\Sigma_b$ of a closed foam $\Sigma$ are closed.\par
\end{claim}
\begin{proof}
    Fix a coloring on a foam $\Sigma$. If $x$ is in the boundary of $\Sigma_1$ it has to belong to a component of the binding $\gamma$, but in a neighborhood of $\gamma$ we have exactly two pages containing the color 1, one blue and one red, giving a neighborhood of $x$ homeomorphic to $\mathbb{R}^2$. Therefore $\partial\Sigma_1$ is empty. The proof of the statement for $\Sigma_b$ works the same way.
\end{proof}

\begin{deff}
    Let F be a facet of the foam $\Sigma$. The evaluation $P_F$ of $F$ is the polynomial in $\mathbb{Z}[X_1,X_2]$ defined as follows:
    \begin{itemize}
        \item If $F$ is a blue facet colored by $c \in \{1,2\}$ and decorated with $p$ points $P_F(X_1,X_2)=X_c^p$.
        \item If $F$ is red and decorated by $p$ points and $s$ squares $P_F(X_1,X_2)= (X_1 + X_2)^p(X_1X_2)^s$.
    \end{itemize}
    Let $c$ be a coloring of the foam $\Sigma$, the evaluation $<\Sigma,c>$: $$<\Sigma,c>=(-1)^{\frac{\chi(\Sigma_1)}{2}}(-1)^{n_{1\uparrow 2 }}\frac{\prod_{F facet} P_F(X_c)}{(X_1-X_2)^{\frac{\chi(\Sigma_b)}{2}}}\in \mathbb{Q}(X_1,X_2)$$
    and the \textbf{evaluation of a closed foam $\Sigma$}, denoted $<\Sigma>$, is the sum of $<\Sigma,c>$ for all possible colorings of $\Sigma$:$$<\Sigma>=\sum_{c}<\Sigma,c>.$$
\end{deff}
\begin{example}
    A simple example would be computing the evaluation of a monochrome sphere.\\
    A red sphere $S_R$ has only one coloring $\{1,2\}$, for which $\Sigma_b$ is empty, $\Sigma_1$ is the sphere itself, and $n_{1\uparrow 2 }$ is zero. Therefore $<S_R>=<S_R,\{1,2\}>=(-1)^{\frac{2}{2}}\times (-1)^0 \frac{1}{(X_1-X_2)^0}=-1 $.\\
    A blue sphere $S_B$ has two colorings, $1$ and $2$. In both cases $\Sigma_b$ is the sphere and $n_{1\uparrow 2 }$ is zero. For the coloring $1$ we have $\Sigma_1=S_B$ and for the coloring $2$, $\Sigma_1$ is empty.\\
    If we have no decorations on the sphere this gives us $<S_B,1>=-<S_B,2>$ and $<S_B>=0$.\\
    Finally if we have a point on $S_B$ then:  $<S_B,1>=-\frac{X_1}{X_1-X_2}$, $<S_B,2>=\frac{X_2}{X_1-X_2}$ and \\$<S_B>=<S_B,1>+<S_B,2>=-1$.\\This works for any choice of orientation on the spheres.
\end{example}
All foams admit at least one coloring. In fact we can start at a blue facet $F$ and color it with 1. Naturally, for the coloring to be valid the color of the adjacent facets has to be $2$, then we proceed recursively. One would then expect that this fully determines the coloring of the connected component $C$ of $\Sigma_b$ that $F$ belongs to. We just have to check that two paths from $F$ to a facet of $F'$ of $C$ determine the same color on $F'$, a path being a sequence of components of the binding that connect $F$ to $F'$. This works because, each time we pass through a component of the binding, we switch the orientation. So the numbers of components of the binding in each path from $F$ to $F'$ has to be the same modulo $2$. Therefore the color of $F'$ is uniquely defined.\\
This process also shows that each coloring of a foam $\Sigma$ is determined by the choice of coloring of one facet from each component of $\Sigma_b$. If $\Sigma_b$ has $k$ connected components, then the number of colorings of $\Sigma$ is exactly $2^k$.

\begin{prop}\label{symmetricpolynomial}
    The evaluation of a foam $<\Sigma>$ is a symmetric polynomial with coefficients in $\mathbb{Z}$.
\end{prop}
\begin{proof}
    If we add $p$ points and $s$ squares on the red subsurface of $\Sigma$ we multiply $<\Sigma>$ by $(X_1+X_2)^p(X_1X_2)^s$. Therefore it suffices to show the result for foams without any decorations on the red subsurface.\\
    We fix a coloring $c$. Let $\Sigma_{b,1},\ldots,\Sigma_{b,k}$ be the connected components of $\Sigma_b$. For each $0\leq \ell \leq k$, we can define a new foam $\Sigma_{\ell}$ by removing all the components of $\Sigma_b$ other than $\Sigma_{b,\ell}$, and gluing disks on the components of the binding that don't touch $\Sigma_{b,\ell}$ (see Figure below). \par
\begin{figure}[h]
    \centering
    
    \begin{tikzpicture}[scale=0.6]
        \fill[Cyan,fill opacity=0.2] (1.5,0) circle (1cm);
        \fill[Cyan,fill opacity=0.2] (-1.5,0) circle (1cm);
        \fill[Red,fill opacity=0.6] (0,0) circle (1cm);
        \begin{scope}[even odd rule]
            \clip (-3,1.5) rectangle (3,-1.5) (0,0) circle (1cm);
            \draw[Blue](1.5,0) circle (1cm);
            \draw[Blue](-1.5,0) circle (1cm);
        \end{scope}
        \begin{scope}
            \clip (0,0) circle (1cm);
            \draw[Maroon, thick, dotted](1.5,0) circle (1cm);
            \draw[Maroon, thick, dotted](-1.5,0) circle (1cm);
        \end{scope}        
        \begin{scope}
            \clip (1.5,0) circle (1cm) (-1.5,0) circle (1cm);
            \draw[Maroon, thick, dotted](0,0) circle (1cm);
        \end{scope}
        \begin{scope}[even odd rule]
            \clip (-3,1.5) rectangle (3,-1.5) (1.5,0) circle (1cm);
            \clip (-3,1.5) rectangle (3,-1.5) (-1.5,0)circle (1cm);
            \draw[Maroon, thick](0,0) circle (1cm);
        \end{scope}
        \begin{scope}[xshift=3cm,yshift=-4cm]
            \fill[Cyan,fill opacity=0.2] (1.5,0) circle (1cm);
            \fill[Red,fill opacity=0.6] (0,0) circle (1cm);
            \begin{scope}[even odd rule]
                \clip (-3,1.5) rectangle (3,-1.5) (0,0) circle (1cm);
                \draw[Blue](1.5,0) circle (1cm);
            \end{scope}
            \begin{scope}
                \clip (0,0) circle (1cm);
                \draw[Maroon, thick, dotted](1.5,0) circle (1cm);
                \draw[Maroon, thick, dotted](-1.5,0) circle (1cm);
            \end{scope}        
            \begin{scope}
                \clip (1.5,0) circle (1cm) (-1.5,0) circle (1cm);
                \draw[Maroon, thick, dotted](0,0) circle (1cm);
            \end{scope}
            \begin{scope}[even odd rule]
                \clip (-3,1.5) rectangle (3,-1.5) (1.5,0) circle (1cm);
                \draw[Maroon, thick](0,0) circle (1cm);
            \end{scope}

        \end{scope}
        \begin{scope}[xshift=-3cm,yshift=-4cm]
            \fill[Cyan,fill opacity=0.2] (-1.5,0) circle (1cm);
            \fill[Red,fill opacity=0.6] (0,0) circle (1cm);
            \begin{scope}[even odd rule]
                \clip (-3,1.5) rectangle (3,-1.5) (0,0) circle (1cm);
                \draw[Blue](-1.5,0) circle (1cm);
            \end{scope}
            \begin{scope}
                \clip (0,0) circle (1cm);
                \draw[Maroon, thick, dotted](1.5,0) circle (1cm);
                \draw[Maroon, thick, dotted](-1.5,0) circle (1cm);
            \end{scope}        
            \begin{scope}
                \clip (1.5,0) circle (1cm) (-1.5,0) circle (1cm);
                \draw[Maroon, thick, dotted](0,0) circle (1cm);
            \end{scope}
            \begin{scope}[even odd rule]
                \clip (-3,1.5) rectangle (3,-1.5) (-1.5,0) circle (1cm);
                \draw[Maroon, thick](0,0) circle (1cm);
            \end{scope}
        \end{scope}
        \draw[->](2.5,-1.5)--(3.5,-2.5);
        \draw[->](-2.5,-1.5)--(-3.5,-2.5);
  
    \end{tikzpicture}
    
\end{figure}
    Let $P_{\ell}$ be the evaluation of $\Sigma_{\ell}$, i.e. the sum over all colorings of $\Sigma_{\ell}$ of the fractions:$$(-1)^{\frac{\chi(\Sigma_{1,{\ell}})}{2}}(-1)^{n_{1\uparrow 2,{\ell}}}\frac{\prod_{F \text{ facet of } \Sigma_{b,{\ell}}} P_F(X_c)}{(X_1-X_2)^{\frac{\chi(\Sigma_{b,{\ell}})}{2}}}$$
    
    Let $c_{0,\ell}, c_{1,\ell}$ be the two possible colorings of each component $\Sigma_{b,\ell}$. Then the colorings of $\Sigma$ are fully determined by a sequence in $\{0,1\}^k$: a choice of $c_0$ or $c_1$ for each component of $\Sigma_b$. For $c \in \{0,1\}^k $
    $$<\Sigma,c>=(-1)^{\frac{1}{2}(\chi(\Sigma_1)-\sum_{\ell}(\chi(\Sigma_{1,{\ell}}))}\prod_{1\leq \ell \leq k}(-1)^{\frac{\chi(\Sigma_{1,{\ell}})}{2}}(-1)^{n_{1\uparrow 2,{\ell}}}\frac{\prod_{F \text{ facet of } \Sigma_{b,{\ell}}} P_F(X_c)}{(X_1-X_2)^{\frac{\chi(\Sigma_{b,{\ell}})}{2}}} $$
    Using additivity of the Euler characteristic it can be shown that $\frac{1}{2}(\chi(\Sigma_1)-\sum_{\ell}(\chi(\Sigma_{1,{\ell}}))$ is independent of the choice of c and is equal to $r= \frac{1}2{}(k -1)\chi(\Sigma_{r})$ where (only here) $\Sigma_{r}$ is the red subsurface of $\Sigma$ obtained by gluing disks along the components of the binding. This proves that:
    $$<\Sigma>=(-1)^r\prod_{1\leq \ell \leq k}P_{\ell}.$$
    It suffices to show the result when $\Sigma_{b}$ is connected. [note that the surface $\Sigma_b$ does not depend on the coloring of $\Sigma$].\par
    Let $\Sigma$ be a foam such that $\Sigma_b$ is connected, and there are no decorations on the red subsurface. For a coloring $c$ define $s(\Sigma,c)=\frac{\chi(\Sigma_1)}{2}-n_{1\uparrow2}$. Since $\Sigma_b$ is connected we have exactly two colorings $c$ and $c'$. Let $s=s(\Sigma,c)$ and $s'=s(\Sigma,c')$. Let $p$ be the number of points figuring in $\Sigma_1$ for the coloring $c$, and $p'$ the number of points in $\Sigma_1$ for the coloring $c'$. Suppose $p'\leq p$.
    \begin{equation*}
    \begin{split}
        <\Sigma>  & =\frac{(-1)^sX_1^pX_2^{p'}+(-1)^{s'}X_2^pX_1^{p'}}{(X_1-X_2)^{\frac{\chi(\Sigma_b)}{2}}}\\
        &=(-1)^s(X_1X_2)^{p'}\frac{X_1^{(p-p')} + (-1)^{(s'-s)}X_2^{(p-p')}}{(X_1-X_2)^{\frac{\chi(\Sigma_b)}{2}}}
    \end{split}   
    \end{equation*}
    
    First let us show that $s'-s\equiv\frac{\chi(\Sigma_b)}{2}\mod2$.
    Note that $\Sigma_1'$ the subsurface containting the color 1 for the second coloring is exactly the subsurface containing the color 2 for the first coloring, then $n_{1\uparrow2}'$ is the number of components of the binding $b$ minus $n_{1\uparrow2}$. 
    \begin{equation*}
    \begin{split}
            s'-s &= \frac{\chi(\Sigma_1)}{2} - \frac{\chi(\Sigma_2)}{2} + 2 n_{1\uparrow2} - b\\
            &\equiv \frac{\chi(\Sigma_1)+\chi(\Sigma_2)}{2}-b \mod 2
            \\& \equiv \frac{\chi(\Sigma_b)}{2} - \chi(\Sigma_{red})-b  \mod 2
            \\& \equiv \frac{\chi(\Sigma_b)}{2} \mod 2
    \end{split}
    \end{equation*}
    the last equality holds because the Euler characteristic of the red subsurface plus the number of the bindings is the Euler charachteristic of the red subsurface "closed up", which is even.\\
    Finally let us look at the fraction $
        \frac{X_1^{p} + (-1)^{\frac{\chi(\Sigma_b)}{2}}X_2^{p}}{(X_1-X_2)^{\frac{\chi(\Sigma_b)}{2}}} $ for an integer $p$ and for each possible value of $\chi(\Sigma_b)$.
    \begin{itemize}
        \item $\frac{\chi(\Sigma_b)}{2}=1$: $\frac{X_1^{p} -X_2^{p}}{X_1-X_2}$ is a symmetric polynomial with coefficients in $\mathbb{Z}$.
        \item $\frac{\chi(\Sigma_b)}{2}=n<0$ and is odd: $(X_1-X_2)^n(X_1^p-X_2^p)$ is the product of two antisymmetric polynomials, thus a symmetric polynomial , with coefficients in $\mathbb{Z}$.
        \item $\frac{\chi(\Sigma_b)}{2}= 2n\leq0$ for a certain integer n: $(X_1-X_2)^{2n}(X_1^p+X_2^p)$ is a symmetric polynomial with coefficients in $\mathbb{Z}$. 
      
    \end{itemize}
This finishes the proof. 
\end{proof}

\begin{lemma}
    \label{localrelations}
    The evaluation of foams $<.>$ satisfies the following local relations ($\ell$)\par \vspace{1cm} \begin{itemize}
    \item Spheres:     \vspace{0.5cm} \\ 
    \inlinepic{\begin{scope}[scale=0.5]
        \BSphere{}
    \end{scope}}\quad = 0 \quad, \hspace{1cm}
    \inlinepic{\begin{scope}[scale=0.5]
        \BSphere{1}
    \end{scope}}\quad = -1 \quad, \hspace{1cm}    
      \inlinepic{\begin{scope}[scale=0.5]
        \RSphere
    \end{scope}}\quad = -1 
     \quad .
    \vspace{0.5cm}
    \item Bubble relations:
    \vspace{0.5cm}\\
    \inlinepic{\begin{scope}[scale=0.7]
        \Pimple{+}
    \end{scope}} \quad =  \quad
    \inlinepic{\begin{scope}[scale=0.9]
        \BDisk{+}
    \end{scope}}\quad , \hspace{1cm}
    \inlinepic{\begin{scope}[scale=0.7]
        \Pimple{-}
    \end{scope}} \quad =  \quad
    \inlinepic{\begin{scope}[scale=0.9]
        \BDisk{-}
    \end{scope}}\quad . \vspace{1cm} \\
    \inlinepic{\begin{scope}[scale=0.7]
        \Saturn{} 
    \end{scope}} \quad = \quad 0.
   \vspace{1cm}\\
   \inlinepic{\begin{scope}[scale=0.7]
       \Saturn{+} 
   \end{scope}} \quad = \quad - \quad 
   \inlinepic{\begin{scope}[scale=0.7]
       \Saturn{-} 
   \end{scope}} \quad = \quad
   \inlinepic{\begin{scope}[scale=0.9]
       \RDisk{+} \end{scope}}  .
    \vspace{0.5cm}
    \item Bigon relation:
    \vspace{0.5cm}\\
    \inlinepic{\begin{scope}[scale=0.5]
        \BigonId \end{scope}}
    \quad = \quad
    \inlinepic{\begin{scope}[scale=0.5] 
        \BigonCups{+} \end{scope}}
    \quad - \quad
    \inlinepic{\begin{scope}[scale=0.5]
        \BigonCups{-} \end{scope}}.
    \vspace{0.5cm}
    \item Neck cutting relations:
    \vspace{0.5cm}\\
    \inlinepic{\begin{scope}[scale=0.5]\BCyl{}{+}{}\end{scope}} \quad = \quad \inlinepic{    \begin{scope}[yshift=1cm,scale=0.5]
        \BCup{1}{+}
    \end{scope}
    \begin{scope}[yshift=-1cm,scale=0.5]
        \BRCap{0}{}{0}{0}
        
        \draw[draw=Blue,-<-]  (-1,0) arc  (-180:0 :1cm and 0.5cm);
        \draw[draw=Maroon,thick,-<-]  (-1,0.5) arc  (-180:0 :1cm and 0.5cm);
        
    \end{scope}} \quad - \quad \inlinepic{    \begin{scope}[yshift=1cm,scale=0.5]
        \BCup{0}{+}
    \end{scope}
    \begin{scope}[yshift=-1cm,scale=0.5]
        \BRCap{1}{}{0}{0}
        \draw[draw=Blue,-<-]  (-1,0) arc  (-180:0 :1cm and 0.5cm);
        \draw[draw=Maroon,thick,-<-]  (-1,0.5) arc  (-180:0 :1cm and 0.5cm);

    \end{scope}}\quad. \hspace{1.5cm}
    \inlinepic{\begin{scope}[scale=0.5]\RCyl{}{+}{}\end{scope}} \quad = \quad - \quad \inlinepic{    \begin{scope}[yshift=1cm,scale=0.5]
        \RCup{+}
    \end{scope}
    \begin{scope}[yshift=-1cm,scale=0.5]
        \RCap{+}
    \end{scope}}\quad. 
    \vspace{0.5cm}\\    
    \item Migration relation:
    \vspace{0.5cm}\\
    \inlinepic{\begin{scope}[scale=0.5]
        \Migration{R}
    \end{scope}} \quad = \quad 
    \inlinepic{\begin{scope}[scale=0.5]
        \Migration{B2}
    \end{scope}} \quad + \quad
    \inlinepic{\begin{scope}[scale=0.5]
        \Migration{B1}
    \end{scope}}.
    \vspace{0.5cm}\\
    \item Action of decorations:
    \vspace{0.5cm}\\
    \inlinepic{\begin{scope}[scale=0.3]
        \BId{2}{}{}
    \end{scope}} \quad = \quad $ (X_1+X_2)$\quad
    \inlinepic{\begin{scope}[scale=0.3]
        \BId{1}{}{}
    \end{scope}}\quad - \quad $ X_1X_2$\quad
    \inlinepic{\begin{scope}[scale=0.3]
        \BId{}{}{}
    \end{scope}}\quad.
    \vspace{1cm}\\
    \inlinepic{\begin{scope}[scale=0.3]
        \RId{1}{}{}{}
    \end{scope}} \quad = \quad $ (X_1+X_2)$\quad
    \inlinepic{\begin{scope}[scale=0.3]
        \RId{}{}{}{}
    \end{scope}}\quad  ,
    \hspace{1cm}
    \inlinepic{\begin{scope}[scale=0.3]
        \RId{}{}{}{1}
    \end{scope}} \quad = \quad $ X_1X_2$\quad
    \inlinepic{\begin{scope}[scale=0.3]
        \RId{}{}{}{}
    \end{scope}}\quad  .  \vspace{0.5cm}   \\ 
    \end{itemize}
\end{lemma}
\textbf{Note:} if the orientation is ignored that means that the relation holds for any choice of orientation. \par
\textbf{Fact:} the relations ($\ell$) above comes from the kernel of the bilinear map $<,>$.
\begin{proof} (Lemma \ref{localrelations})\par
    We already proved the sphere relations. We start with the first bubble relations. Let $\Sigma$ be the surface with the bubble 
    \inlinepic{
        \begin{scope}[scale=0.3]
        \Pimple{+} 
         \end{scope}}, 
    and $\Sigma '$ be the surface where the bubble is replaced with a disk. There is a one to one correspondance between colorings of $\Sigma$ and colorings of $\Sigma '$. We are going to compare $<\Sigma,c>$ and $<\Sigma',c'>$ for each class of colorings $c$. \\
    The first class is where the disk is colored by $1$: \\    By looking at the picture we see that:\\

    \vspace{0.5cm}
    \begin{center}
        
        \begin{tabular}{c c c}
            \inlinepic{\begin{scope}[scale=0.7]
            \Pimple{+}
            \node at (0,0.9) [ font=\small]  {2};
            \node at (-1.25,0) [ font=\small] {1};
        \end{scope}} 
        &
        \inlinepic{\begin{scope}[scale=0.9]
            \draw[->] (-1,0) -- (1,0);
        \end{scope}}
        &
        \inlinepic{\begin{scope}[scale=0.9]
            \BDisk{+}
            \node at (0,0) [ font=\small] {1};
        \end{scope}}
        \\[1cm]
        $ \Sigma_1$
        &
        \inlinepic{\begin{scope}[scale=0.9]
            \draw[->] (-1,0) -- (1,0);
        \end{scope}}
        &
        $\Sigma_1 '=\Sigma_1$
        \\[0.5cm]
        $\Sigma_b$
        &
        \inlinepic{\begin{scope}[scale=0.9]
            \draw[->] (-1,0) -- (1,0);
        \end{scope}}
        &
        $\Sigma_b '=\Sigma_b$
        \\[0.5cm]
        $ n _{1\uparrow 2}$
        &
        \inlinepic{\begin{scope}[scale=0.9]
            \draw[->] (-1,0) -- (1,0);
        \end{scope}}
        &
         $ n '_{1\uparrow 2}=  n _{1\uparrow 2}$
        \\[0.5cm]
        $\prod _{F \text{ facet of } \Sigma} P_F$
        &
        \inlinepic{\begin{scope}[scale=0.9]
            \draw[->] (-1,0) -- (1,0);
        \end{scope}}
        &
        $\prod _{F \text{ facet of } \Sigma '} P_{F'} = \prod _{ F \text{ facet of } \Sigma} P_F$
        \\
         
        \end{tabular}
    \end{center}
\vspace{0.5cm}
    Thus $<\Sigma,c> = <\Sigma',c>$.\vspace{0.5cm}\\
    For The second class of colorings $c'$, where the disk is marked with $2$:\\
    \begin{center}
        \begin{tabular}{c c c}
            \inlinepic{\begin{scope}[scale=0.7]
            \Pimple{+}
            \node at (0,0.9) [ font=\small]  {1};
            \node at (-1.25,0) [ font=\small] {2};
        \end{scope}} 
        &
        \inlinepic{\begin{scope}[scale=0.9]
            \draw[->] (-1,0) -- (1,0);
        \end{scope}}
        &
        \inlinepic{\begin{scope}[scale=0.9]
            \BDisk{+}
            \node at (0,0) [ font=\small] {2};
        \end{scope}}
        \\[1cm]
        $ \Sigma_1$
        &
        \inlinepic{\begin{scope}[scale=0.9]
            \draw[->] (-1,0) -- (1,0);
        \end{scope}}
        &
        $\Sigma_1 '=\Sigma_1 \sqcup$ \inlinepic{\begin{scope}[scale=0.2]
            \BSphere{}
        \end{scope}} 
        \\[0.5cm]
        & & $\Rightarrow \chi(\Sigma_1 ')=\chi(\Sigma_1) + 2$
        \\[0.5cm]
        $\Sigma_b$
        &
        \inlinepic{\begin{scope}[scale=0.9]
            \draw[->] (-1,0) -- (1,0);
        \end{scope}}
        &
        $\Sigma_b '=\Sigma_b$
        \\[0.5cm]
        $ n _{1\uparrow 2}$
        &
        \inlinepic{\begin{scope}[scale=0.9]
            \draw[->] (-1,0) -- (1,0);
        \end{scope}}
        &
         $ n '_{1\uparrow 2}=  n _{1\uparrow 2}+1$
        \\[0.5cm]
        $\prod _{F \text{ facet of } \Sigma} P_F$
        &
        \inlinepic{\begin{scope}[scale=0.9]
            \draw[->] (-1,0) -- (1,0);
        \end{scope}}
        &
        $\prod _{F \text{ facet of } \Sigma '} P_{F'} = \prod _{ F \text{ facet of } \Sigma} P_F$
        \\[0.5cm]
         
        \end{tabular}
    \end{center}
\par

    Therefore 
    \begin{equation*}
    \begin{split}
    <\Sigma ',c'>&= \frac{(-1)^{\frac{\chi(\Sigma_1 ')}{2}} (-1)^{n_{1\uparrow 2}'}\prod _{F \text{ facet of } \Sigma '} P_{F'}}{(X_1-X_2)^{\chi(\Sigma_b')}}\\&= \frac{(-1)^{\frac{\chi(\Sigma_1)+2}{2}} (-1)^{n_{1\uparrow 2}+1}\prod _{F \text{ facet of } \Sigma} P_{F}}{(X_1-X_2)^{\chi(\Sigma_b)}}\\&= (-1)^2 <\Sigma ,c'>\\&=<\Sigma ,c'>\\
    \end{split}
    \end{equation*}
    Summing over all colorings of $\Sigma$ and $\Sigma'$ we get $<\Sigma > = <\Sigma '>$.\par
    When we change the orientation of the surfaces, the only thing that changes is the relation between $n_{1\uparrow 2}$ and $n' _{1\uparrow 2}$ for each class of colorings. The result of this is an overall multiplication by $-1$. We get the first two bubble relations. The proof of the other bubble relations are similar.\par

    Moving on to the Bigon relation.\vspace{0.5cm}\\
    Let $\Sigma=$ \inlinepic{\begin{scope}[scale=0.3]
        \BigonId
    \end{scope}}, \quad $\Sigma_+=$ \inlinepic{\begin{scope}[scale=0.3]
        \BigonCups{+}
    \end{scope}} \quad and \quad $\Sigma_-=$ \inlinepic{\begin{scope}[scale=0.3]
        \BigonCups{-} \end{scope}}.\vspace{0.5cm}\\

    The possible colorings on $\Sigma$ are divided into two classes, those of  $\Sigma_+$ and $\Sigma_-$ into four classes. There is a correspondance between them:
    \begin{itemize}
        \item \textbf{$c_1$:} \quad \inlinepic{\begin{scope}[scale=0.3]
        \BigonId \node at (0,0) [ font=\tiny] {1}; \node at (0,2) [ font=\tiny] {2};\end{scope}}
        \quad, \hspace{1cm}
        \inlinepic{\begin{scope}[scale=0.3] 
            \BigonCups{+} \node at (0.2,1.25) [ font=\tiny] {1}; \node at(0.2,2.25) [ font=\tiny] {2};\node at (0.2,-2.25) [font=\tiny] {1}; \node at (0.2,-1.25) [font=\tiny] {2};\end{scope}}
        \quad , \hspace{1cm}
        \inlinepic{\begin{scope}[scale=0.3] 
            \BigonCups{-} \node at (0.2,1.25) [ font=\tiny] {1}; \node at (0.2,2.25) [ font=\tiny] {2} ;\node at (0.2,-2.25) [font=\tiny] {1}; \node at (0.2,-1.25) [font=\tiny] {2};\end{scope}}
        \quad .\vspace{0.5cm}\\
        \item \textbf{$c_2$:} \quad \inlinepic{\begin{scope}[scale=0.3]
        \BigonId \node at (0,0) [ font=\tiny] {2}; \node at (0,2) [ font=\tiny] {1};\end{scope}}
        \quad, \hspace{1cm}
        \inlinepic{\begin{scope}[scale=0.3] 
            \BigonCups{+} \node at (0.2,1.25) [ font=\tiny] {2}; \node at(0.2,2.25) [ font=\tiny] {1};\node at (0.2,-2.25) [font=\tiny] {2}; \node at (0.2,-1.25) [font=\tiny] {1};\end{scope}}
        \quad , \hspace{1cm}
        \inlinepic{\begin{scope}[scale=0.3] 
            \BigonCups{-} \node at (0.2,1.25) [ font=\tiny] {2}; \node at (0.2,2.25) [ font=\tiny] {1} ;\node at (0.2,-2.25) [font=\tiny] {2}; \node at (0.2,-1.25) [font=\tiny] {1};\end{scope}}  \quad . \vspace{0.5cm}\\
        \item \textbf{$c_3$:}\quad 
        \inlinepic{\begin{scope}[scale=0.3] 
            \BigonCups{+} \node at (0.2,1.25) [ font=\tiny] {1}; \node at(0.2,2.25) [ font=\tiny] {2};\node at (0.2,-2.25) [font=\tiny] {2}; \node at (0.2,-1.25) [font=\tiny] {1};\end{scope}}
        \quad , \hspace{1cm}
        \inlinepic{\begin{scope}[scale=0.3] 
            \BigonCups{-} \node at (0.2,1.25) [ font=\tiny] {1}; \node at (0.2,2.25) [ font=\tiny] {2} ;\node at (0.2,-2.25) [font=\tiny] {2}; \node at (0.2,-1.25) [font=\tiny] {1};\end{scope}}  \quad . \vspace{0.5cm}\\
        \item \textbf{$c_4$:}\quad 
        \inlinepic{\begin{scope}[scale=0.3] 
            \BigonCups{+} \node at (0.2,1.25) [ font=\tiny] {2}; \node at(0.2,2.25) [ font=\tiny] {1};\node at (0.2,-2.25) [font=\tiny] {1}; \node at (0.2,-1.25) [font=\tiny] {2};\end{scope}}
        \quad , \hspace{1cm}
        \inlinepic{\begin{scope}[scale=0.3] 
            \BigonCups{-} \node at (0.2,1.25) [ font=\tiny] {2}; \node at (0.2,2.25) [ font=\tiny] {1} ;\node at (0.2,-2.25) [font=\tiny] {1}; \node at (0.2,-1.25) [font=\tiny] {2};\end{scope}}  \quad . \vspace{0.5cm}\\
    \end{itemize}
    We are going to proceed as before, starting with $c_1$:
    \begin{center}
        \begin{tabular}{c c c}
        \inlinepic{\begin{scope}[scale=0.3]
        \BigonId \node at (0,0) [ font=\tiny] {1}; \node at (0,2) [ font=\tiny] {2};\end{scope}} \hspace{0.5cm}
        &
        \inlinepic{\begin{scope}[scale=0.3] 
            \BigonCups{+} \node at (0.2,1.25) [ font=\tiny] {1}; \node at(0.2,2.25) [ font=\tiny] {2};\node at (0.2,-2.25) [font=\tiny] {1}; \node at (0.2,-1.25) [font=\tiny] {2};\end{scope}}\hspace{0.5cm}
        &
        \inlinepic{\begin{scope}[scale=0.3] 
            \BigonCups{-} \node at (0.2,1.25) [ font=\tiny] {1}; \node at (0.2,2.25) [ font=\tiny] {2} ;\node at (0.2,-2.25) [font=\tiny] {1}; \node at (0.2,-1.25) [font=\tiny] {2};\end{scope}}\hspace{0.5cm}
        \\[1cm]
        $\Sigma_1$
        & 
        $\Sigma_1$
        & 
        $\Sigma_1$
        \\[0.5cm]
        $\chi(\Sigma_b)$
        &
        $\chi(\Sigma_b)+2$
        &
        $\chi(\Sigma_b)+2$
        \\[0.5cm]
        
        $n_{1\uparrow2}$
        &
        $n_{1\uparrow2}$
        &
        $n_{1\uparrow2}$
        \\[0.5cm]
        $P$
        &
        $P\times X_1$
        &
        $P\times X_2$
        \end{tabular}
    \end{center}
    \vspace{1cm}
    
    The relation between the blue subsurfaces is the attachment  of a handle, giving the relation between the Euler characteristics.\\
    $<\Sigma_+,c_1> - <\Sigma_-,c_1 >= \frac{X_1-X_2}{X_1-X_2}<\Sigma,c_1>$, which is what we want.\par
    For $c_2$ $\Sigma_1$ is always the same. The blue subsurfaces are related by the attachment of a handle again. As in before $P_+=X_2 P$ and $P_-=X_1 P$, but $n_{1\uparrow2}$ changes depending on the number of components of the binding.
    \begin{center}
        \inlinepic{\begin{scope}[scale=0.5]
        \BigonId \node at (0,0) [ font=\tiny] {2}; \node at (0,2) [ font=\tiny] {1};
             \node at (1.2,0) [font=\small,text=Yellow] {$\gamma_1$};
             \node at (-1.2,0) [ font=\small,text=Yellow] {$\gamma_2$};
        \end{scope}}
        \quad, \hspace{1cm}
        \inlinepic{\begin{scope}[scale=0.5] 
            \BigonCups{+} \node at (0.2,1.25) [ font=\tiny] {2}; \node at(0.2,2.25) [ font=\tiny] {1};\node at (0.2,-2.25) [font=\tiny] {2}; \node at (0.2,-1.25) [font=\tiny] {1};
            \node at (0,0.8) [ font=\small,text=Yellow] {$\gamma_3$};
            \node at (0,-0.8) [font=\small,text=Yellow] {$\gamma_4$};
        \end{scope}}
        \quad , \hspace{1cm}
        \inlinepic{\begin{scope}[scale=0.5] 
            \BigonCups{-} \node at (0.2,1.25) [ font=\tiny] {2}; \node at (0.2,2.25) [ font=\tiny] {1} ;\node at (0.2,-2.25) [font=\tiny] {2}; \node at (0.2,-1.25) [font=\tiny] {1};        
        \end{scope}} 
        \quad . \vspace{0.5cm}\\
    \end{center}
    If the arcs $\gamma_1$ and $\gamma_2$ are in the same component, then $\gamma_3$ and $\gamma_4$ are in different components, $n_{1\uparrow2}$ increases by 1 in $\Sigma_+$ and $\Sigma_-$. \\
    If $\gamma_1$ and $\gamma_2$ are in different components, then $\gamma_3$ and $\gamma_4$ are in the same component, $n_{1\uparrow2}$ decreases by 1 in $\Sigma_+$ and $\Sigma_-$.\\
    \begin{center}
        \begin{tikzpicture}[scale=0.7]
            \begin{scope}[xshift=-8cm]
                \draw (-0.5,-1) arc (-180:0:0.5cm)--(0.5,1) arc(0:180:0.5cm)--cycle;
                \node at (0.6,0) [font=\small,text=Red] {$\gamma_1$};
                \node at (-0.6,0) [font=\small,text=Red] {$\gamma_2$};
            \end{scope}
            \draw[<->] (-7,0)--(-5,0);
            \begin{scope}[xshift=-4cm]
                \draw (0,1) circle (0.5cm);
                \draw (0,-1) circle (0.5cm);
                \node at (0,0.4) [font=\small,text=Red] {$\gamma_3$};
                \node at (0,-0.4) [font=\small,text=Red] {$\gamma_4$};                
            \end{scope}
            \begin{scope}[xshift=2cm]
                \draw (-0.5,1) arc (0:180:0.5cm)--(-1.5,-1.5) arc(-180:-90:1cm)--(0.5,-2.5) arc (-90:0:1cm)--(1.5,1) arc (0:180:0.5cm)--(0.5,-1) arc(0:-180:0.5cm)--cycle;
                \node at (0.6,0) [font=\small,text=Red] {$\gamma_1$};
                \node at (-0.6,0) [font=\small,text=Red] {$\gamma_2$};
            \end{scope}
            \draw[<->] (6,0)--(4,0);
            \begin{scope}[xshift=8cm]
                \draw (-0.5,1) arc (0:180:0.5cm)--(-1.5,-1.5) arc(-180:-90:1cm)--(0.5,-2.5) arc (-90:0:1cm)--(1.5,1) arc (0:180:0.5cm)--(0.5,1) arc(0:-180:0.5cm)--cycle;
                \draw (0,-1) circle (0.5cm);
                \node at (0,0.4) [font=\small,text=Red] {$\gamma_3$};
                \node at (0,-0.4) [font=\small,text=Red] {$\gamma_4$};
            \end{scope}
        \end{tikzpicture}\\

    \end{center}

        Either way, $n_{1\uparrow2}$ changes by 1, therefore we get an overall multiplication by -1.
        \begin{equation*}
            \begin{split}
                <\Sigma_+,c_2> - <\Sigma_-,c_2> &= \frac{(-1)X_2}{X_1-X_2}<\Sigma,c_2> - \frac{(-1)X_1}{X_1-X_2}<\Sigma,c_2> \\
                 &=<\Sigma,c_2>
            \end{split}
        \end{equation*}
        Finally, we prove that $<\Sigma_+,c_3>=<\Sigma_-,c_3>$ and $<\Sigma_+,c_4>=<\Sigma_-,c_4>$.\\
        For $c_3$ and $c_4$, it is clear that $\Sigma_{+,1}=\Sigma_{-,1}$, $\Sigma_{+,b}=\Sigma_{-,b}$, and $n^+_{1\uparrow2}=n^-_{1\uparrow2}$. The only thing that could change is the  polynomial decoration. The points in both $\Sigma_+$ and $\Sigma_-$ are placed on subsurfaces with the same label ($1$ for $c_4$ and $1$ for $c_3$). Therefore, we have $<\Sigma_+,c_3>=<\Sigma_-,c_3>$ and $<\Sigma_+,c_4>=<\Sigma_-,c_4>$, like we wanted. Summing over all colorings we have $<\Sigma>=<\Sigma_+>-<\Sigma_->$.\par
        The proof of the neck cutting relations are similar. The migration relation and the action of the decorations are left to the reader.

\end{proof}

More relations follow, which will be useful in following sections.
\begin{lemma} \label{localrelationsagain}The following relations hold:\vspace{1cm}\\
    \begin{minipage}[t]{0.48\textwidth}
        \begin{itemize}
            \item Tube relations: \vspace{0.5cm}\\
            \inlinepic{\begin{scope}[scale=0.3]
                \TubeTwo
            \end{scope}} \quad = \quad - \quad
            \inlinepic{\begin{scope}[scale=0.3]
            \begin{scope}[xshift=-2cm]
                    \BId{}{+}{}
            \end{scope}
            \begin{scope}[xshift=2cm]
                \BId{}{-}{}
            \end{scope}
            \end{scope}}  \vspace{0.5cm}\\
            \inlinepic{\begin{scope}[scale=0.3]
                \TubeThree
            \end{scope}} \quad = \quad 
            \inlinepic{\begin{scope}[scale=0.3]
            \begin{scope}[xshift=-2cm]
                    \RId{}{-}{}{}
            \end{scope}
            \begin{scope}[xshift=2cm]
                \BId{}{-}{}
            \end{scope}
            \end{scope}}   \vspace{0.5cm}\\
            \inlinepic{\begin{scope}[scale=0.3]
                \TubeFour
            \end{scope}} \quad = \quad - 
            \inlinepic{\begin{scope}[scale=0.3]
            \begin{scope}[xshift=-2cm]
                    \RId{}{+}{}{}
            \end{scope}
            \begin{scope}[xshift=2cm]
                \BId{}{+}{}
            \end{scope}
            \end{scope}}  \vspace{1cm}\\
            \inlinepic{\begin{scope}[scale=0.5]
                \BCyl{}{+}{+}
                \draw[Maroon,thick,-<-] (1,0) arc (0:-180:1cm and 0.5cm);
                \draw[Maroon,dotted] (1,0) arc (0:180:1cm and 0.5cm);
                \fill[Red, fill opacity=0.6] (0,0) ellipse (1cm and 0.5cm);
            \end{scope}} \quad = \quad
            \inlinepic{\begin{scope}[scale=0.5]
                \begin{scope}[yshift=2cm]
                \BCup{}{+} 
                \end{scope}
                \begin{scope}[yshift=-2cm]
                \BCap{1}{-} 
                \end{scope}
            \end{scope}} \quad - \quad
            \inlinepic{\begin{scope}[scale=0.5]
                \begin{scope}[yshift=2cm]
                \BCup{1}{+} 
                \end{scope}
                \begin{scope}[yshift=-2cm]
                \BCap{}{-} 
                \end{scope}\end{scope}}
                \vspace{0.5cm}\\
                \item $(I2)$ relation:
                \vspace{0.5cm}\\
                \inlinepic{\begin{scope}[scale=0.3] \ITwo \end{scope}} =\quad-\:  \inlinepic{\begin{scope}[scale=0.3] \IdTwo\end{scope}}  
                \vspace{0.5cm}\\
        \end{itemize}
    \end{minipage}\hfill
    \begin{minipage}[t]{0.48\textwidth}
        \begin{itemize}
            \item Band relations: \vspace{0.5cm}\\
            \inlinepic{\begin{scope}[scale=0.4]
                \BandOne 
            \end{scope}} \quad = \quad \inlinepic{\begin{scope}[scale=0.4]
                \BandOneCups
            \end{scope}}  \vspace{0.5cm}\\
            \inlinepic{\begin{scope}[scale=0.4]
                \BandTwo 
            \end{scope}} \quad = \quad $-$ \quad \inlinepic{\begin{scope}[scale=0.4]
                \BandTwoCups
            \end{scope}}
            \vspace{0.5cm}\\
            \item $(I1)$ relation:
            \vspace{0.5cm}\\
            \inlinepic{\begin{scope}[scale=0.3] \IOne 
                \draw[Maroon,thick,->-](-0.5,3.5)--(-3.5,0.5);
            \end{scope}} =\quad \inlinepic{\begin{scope}[scale=0.3] \IdOne  \draw[Maroon,thick,->-](-0.5,3.5)--(-3.5,0.5);\end{scope}}                         
            \vspace{0.5cm}\\
            \inlinepic{\begin{scope}[scale=0.3] \IOne     \draw[Maroon,thick,-<-](-0.5,3.5)--(-3.5,0.5);
             \end{scope}} =\quad $-$ \inlinepic{\begin{scope}[scale=0.3] \IdOne \draw[Maroon,thick,-<-](-0.5,3.5)--(-3.5,0.5);\end{scope}}                         
        \end{itemize}
    \end{minipage}
\end{lemma}
\begin{proof}
    The tube relations clearly follow from the neck cutting relations along with the bubble relations. The remaining relations can be proven as in the previous lemma. 
\end{proof}
\section{The universal construction}\label{universal}

\begin{deff}
    Given an object $O \in \mathcal{F}(B)$, we can define a functor $\mathcal{G_{O}}$
    \begin{equation*}
    \begin{split}
        \mathcal{G_O:}& \mathcal{F}(B) \rightarrow  \mathbb{Z}[X_1,X_2]Mod\\& \Gamma \longmapsto Hom_{\mathcal{F}(B)}(O,\Gamma)
    \end{split}
\end{equation*}
The morphisms are sent to the map defined by the composition.
\end{deff}
In the case $O=\emptyset$ we can further quotient $Hom_{\mathcal{F}}(\emptyset,\Gamma)$ for each $\Gamma$ by the kernel of the bilinear form $b$:
\begin{equation*}
\begin{split}
        b:& Hom(\emptyset,\Gamma)\times Hom(\Gamma,\emptyset)\rightarrow \mathbb{Z}[X_1,X_2]\\&
        \Sigma_1,\Sigma_2 \longmapsto <\Sigma_2\circ\Sigma_1>
\end{split}
\end{equation*}
to obtain a functor
\begin{equation*}
    \begin{split}
        \mathcal{G:}& \mathcal{F} \rightarrow  \mathbb{Z}[X_1,X_2]Mod\\& \Gamma \longmapsto Hom_{\mathcal{F}}(\emptyset,\Gamma)_{/ker(b)}
    \end{split}
\end{equation*}

\begin{deff}
    The relations ($\ell$) from the Lemma \ref{localrelations} are in the kernel of the bilinear form $b$. Therefore the functor $\mathcal{G}$ factor through a category $\mathcal{F}_{\ell}$, that has the same objects as $\mathcal{F}$, and for morphism the quotient of the morphisms of $\mathcal{F}$ by ($\ell$). We keep the same notation for the new functor:
    \begin{equation*}
    \begin{split}
        \mathcal{G:}& \mathcal{F_{/\ell}} \rightarrow  \mathbb{Z}[X_1,X_2]Mod\\& \Gamma \longmapsto Hom_{\mathcal{F}_{/\ell}}(\emptyset,\Gamma)
    \end{split}
\end{equation*}
\end{deff}

We similarly define the category $\mathcal{F}(B)_{/\ell}$ by taking the quotient of $\mathcal{F}(B)$ by the local relations ($\ell$) on the morphisms, always keeping the boundary fixed.\\
The relations ($\ell$) respect the grading. The categories $\mathcal{F}(B)_{/\ell}$ remains graded. \par

\begin{example} Computing the image by $\mathcal{G}$  of some graphs: 
Let $\Gamma$ be closed graph. Let $S(\Gamma)=\mathcal{G}(\Gamma)$.
\begin{itemize}
    \item $\Gamma=\emptyset$: 
    Let $\Sigma \in S(\emptyset)$. Let $\Sigma'$ be a closed foams. Then $<\Sigma\circ \Sigma'>=<\Sigma><\Sigma'>=<\Sigma><\emptyset\circ\Sigma'>$. This is true for any such $\Sigma'$. Therefore, for all  $\Sigma \in S(\emptyset)$, we have $\Sigma=<\Sigma>\emptyset$ in $S(\emptyset)$, and $S(\emptyset) \simeq \mathbb{Z}[X_1,X_2]$ is generated by the empty foam. 
    \item $\Gamma$ is a red circle, then $S(\textcolor{Maroon}{\bigcirc})$ is generated by the red disk. This follows from the neck cutting relation applied to a neighborhood of the circle.
    \item $\Gamma$ is a blue circle: consider the maps $f=\begin{pmatrix}
        f_1\\f_2
    \end{pmatrix}$ , and $g=\begin{pmatrix}
        g_1 & g_2
    \end{pmatrix}$: \vspace{0.5cm} \\
    \begin{center}
    \begin{tikzcd}[row sep= small, column sep=large]
    &     
    S(\emptyset)\text{\{-1\}}
    \arrow[dr]
    \arrow[dr, phantom,dash, "{g_1= \inlinepic{\begin{scope}[scale=0.25]
        \BRCap{}{-}{}{} \end{scope} }}" above right] 
    &
    \\
    S(\inlinepic{\draw[Blue,->-](0,0) circle (0.2cm);}) 
    \arrow[ur] 
    \arrow[ur, phantom,dash, "{f_1=\inlinepic{\begin{scope}[scale=0.25]
        \BCup{1}{+}\end{scope} }}" above left] 
    \arrow[dr]
    \arrow[dr, phantom,dash, "{f_2=\inlinepic{\begin{scope}[scale=0.25]
        \BCup{}{+}\end{scope} }}" below left] 
    &&  
    S(\inlinepic{\draw[Blue,->-](0,0) circle (0.2cm);}) 
    \\
    &    
    S(\emptyset)\text{\{1\}}
    \arrow[ur]
    \arrow[ur, phantom,dash, "{g_2= -\inlinepic{\begin{scope}[scale=0.25]
        \BRCap{1}{-}{}{} \end{scope} }}" below right] 
    &
    \end{tikzcd}
    \end{center}\vspace{0.5cm}
    Then: $g\circ f=$
    \inlinepic{   
    \begin{scope}[yshift=0.5cm,scale=0.25]
        \BCup{1}{+}
    \end{scope}
    \begin{scope}[yshift=-0.5cm,scale=0.25]
        \BRCap{0}{-}{0}{0}
    \end{scope}} \: - \: \inlinepic{    \begin{scope}[yshift=0.5cm,scale=0.25]
        \BCup{0}{+}
    \end{scope}
    \begin{scope}[yshift=-0.5cm,scale=0.25]
        \BRCap{1}{-}{0}{0}
    \end{scope}}\: is the identity on S(\inlinepic{\draw[Blue,->-](0,0) circle (0.2cm);}) by the neck cutting relation, and \vspace{1cm}\\
    $f\circ g= 
    \begin{pmatrix}
        &\inlinepic{\begin{scope}[scale=0.25]\BSphere{}\Med \begin{scope}[yshift=-0.7cm]  \BDot \end{scope}  \end{scope}}
        &
        \inlinepic{\begin{scope}[scale=0.25]\BSphere{}\Med \end{scope}}&
        \\[0.5cm]
        &\inlinepic{\begin{scope}[scale=0.25]\BSphere{1}\Med \begin{scope}[yshift=-0.7cm] \BDot \end{scope}  \end{scope}}
        &
        \inlinepic{\begin{scope}[scale=0.25]\BSphere{1}\Med \end{scope}}&
    \end{pmatrix}
    = \begin{pmatrix}
        1 & 0\\
        0 & 1
    \end{pmatrix}\vspace{1cm}
    $.
    \\ Therefore $S(\inlinepic{\draw[Blue,->-](0,0) circle (0.2cm);}) \simeq S(\emptyset)\text{\{-1\}} \oplus S(\emptyset)\text{\{1\}} \simeq  \mathbb{Z}[X_1,X_2]\text{\{-1\}}\oplus \mathbb{Z}[X_1,X_2]\text{\{1\}} $.
\end{itemize}
    
\end{example}

\section{Evaluation of graphs}
\label{EvaluationOfGraphs}\label{EvaluationOfGraphs}

\begin{deff}
    Let $\Gamma \in \mathcal{F}$ be a closed trivalent graph. The blue edges in $\Gamma$ form a set of disjoint loops $\Gamma_1$. The evaluation of $\Gamma$ is defined by the formula    $V(\Gamma)= (q+q^{-1})^{n}$, where $n$ is the number of components of $\Gamma_1$.
\end{deff}
\begin{thh}\label{gradedDim}
    The graded dimension of $S(\Gamma)$ is equal to its evaluation $V(\Gamma)$.
\end{thh}
To prove this we will need some intermediate results:
\begin{lemma}\label{graph theory}
    For all $\Gamma$ be a closed trivalent graph that contains at least one red edge, then there has to be at least one sub-graph that is either a bigon or a square (Figure \ref{fig:Bigons} and \ref{fig:Squares}).
    \begin{figure}[h]
        \centering
        \inlinepic{\begin{scope}[scale=0.6]
            \begin{scope}[xshift=-6cm]
                \GBigonLeft
            \end{scope}
            \GBigonRight
            \begin{scope}[xshift=6cm]
                \GBigonCenter
            \end{scope}
        \end{scope}}
        \caption{Bigons}
        \label{fig:Bigons}
    \end{figure}

    \begin{figure}[h]
        \centering
        \inlinepic{\begin{scope}[scale=0.6]
            \begin{scope}[xshift=-3cm]
                \SqOne
            \end{scope}     
            \begin{scope}[xshift=3cm]
                \SqTwo
            \end{scope}
        \end{scope}}
        \caption{Squares}
    \end{figure}
    \label{fig:Squares}
\end{lemma}

\begin{proof}
    Suppose $\Gamma$ is connected. Otherwise we do the same for a connected component of $\Gamma$.\\
    Let $v$, $e$ and $f$ be respectively the number of vertices, edges and faces. We know that $\chi(\Gamma)=v-e+f=2$.\\
    For each face $F$ assign $\chi_F$ defined by
    $\chi_F:=\left\{\begin{array}{l}\frac{1}{3}  \text {  the number of vertices of } F \\ +\frac{1}{3} \text {  the number of vertices of } F \text { for which all adjacent} \\ \text{edges belong to } F \\ +\frac{1}{3} \text { the number of vertices of } F \text { that belong only to } F \\ -\frac{1}{2}  \text { the number of edges of } F \\ -\frac{1}{2} \text { the number of edges of } F \text { that belong only to } F \\ +1.\end{array}\right.$
    Then:
    \begin{enumerate}
        \item $\sum_{F \text { face of } \Gamma} \chi_F=2 $.
        \item If a face $F$ is bounded, and $v_F$ is the number of vertices of F, then $\chi_F=\frac{1}{3}v_F-\frac{1}{2}v_F+1=1-\frac{1}{6}v_F$.
        \item  If the unbounded face $F'$ does not contain an edge that belongs to it only, then $\chi_{F'}\leq 1$   
    \end{enumerate}
 
    Suppose that $v_F \geq 6 $ for all bounded faces of $\Gamma$, then the sum $\sum_{F \text { bounded face of } \Gamma} \chi_F$ would be at most 1. Therefore there has to be a face of $\Gamma$ with at most 5 vertices. Checking all configurations of the faces possible up to 5 vertices that respect the flow condition, we obtain exactly the ones in the statement of the lemma.    
\end{proof}

\begin{lemma}  \label{TheLemma}
    \begin{enumerate}
        \item $S(\inlinepic{\begin{scope}[scale=0.4] \GBigonCenter \end{scope}})\:  
        \simeq  \: S(\inlinepic{\begin{scope}[scale=0.4] \RGId \end{scope}}) \text{\{-1\}}\: \oplus 
        \: S(\inlinepic{\begin{scope}[scale=0.4] \RGId \end{scope}})\text{\{1\}}
        $.       
        \vspace{0.5cm}
        \item $S(\inlinepic{\begin{scope}[scale=0.4] \GBigonLeft \end{scope}})\: 
        \simeq \: S(\inlinepic{\begin{scope}[scale=0.4] \GBigonRight \end{scope}}) \: 
        \simeq \: S(\inlinepic{\begin{scope}[scale=0.4] \BGId        
        \end{scope}}). 
        $\vspace{0.5cm}
        \item $
        S(
            \inlinepic{\begin{scope}[scale=0.4] \SqOne \end{scope}}
        )
        \: \simeq \: 
        S(
            \inlinepic{\begin{scope}[scale=0.4] \SqOneId \end{scope}}
        )
        $. \vspace{0.5cm}
        \item $
        S(
            \inlinepic{\begin{scope}[scale=0.4] \SqTwo \end{scope}}
        )
        \: \simeq \: 
        S(
            \inlinepic{\begin{scope}[scale=0.4] \SqTwoId \end{scope}}
        )
        $.

    \end{enumerate}
\end{lemma}
\begin{proof}(Lemma \ref{TheLemma})
    \par $1.$ Let $f=\begin{pmatrix}
        f_1\\f_2
    \end{pmatrix}$ , and $g=\begin{pmatrix}
        g_1 & g_2
    \end{pmatrix}$ be the maps: \vspace{0.5cm} \\
    \begin{center}
    \begin{tikzcd}[row sep= small, column sep=large]
    &     
    S(\inlinepic{\begin{scope}[scale=0.25] \RGId \end{scope}})\text{\{-1\}}
    \arrow[dr]
    \arrow[dr, phantom,dash, "{g_1=  - \inlinepic{\begin{scope}[scale=0.25]\draw[Maroon,thick](-2,0)--(2,0);\clip (-2.5,0)rectangle (2.5,-2.5); \BigonCups{-}\end{scope}}
    }" above right] 
    &
    \\
    S(\inlinepic{\begin{scope}[scale=0.4] \GBigonCenter \end{scope}}) 
    \arrow[ur] 
    \arrow[ur, phantom,dash, "{f_1=\inlinepic{\begin{scope}[scale=0.25]\draw[Maroon,thick](-2,0)--(2,0);\clip (-2.5,0)rectangle (2.5,2.5); \BigonCups{-}\end{scope}}
    }" above left] 
    \arrow[dr]
    \arrow[dr, phantom,dash, "{f_2=\inlinepic{\begin{scope}[scale=0.25]\draw[Maroon,thick](-2,0)--(2,0);\clip (-2.5,0)rectangle (2.5,2.5); \BigonCups{+}\end{scope}}
    }" below left] 
    &&  
    S(\inlinepic{\begin{scope}[scale=0.4] \GBigonCenter \end{scope}})
    \\
    &    
    S(\inlinepic{\begin{scope}[scale=0.25] \RGId \end{scope}})\text{\{1\}}
    \arrow[ur]
    \arrow[ur, phantom,dash, "{g_2=     \inlinepic{\begin{scope}[scale=0.25]\draw[Maroon,thick](-2,0)--(2,0);\clip (-2.5,0)rectangle (2.5,-2.5); \BigonCups{+}\end{scope}}
    }" below right] 
    &
    \end{tikzcd}
    \end{center}
    Then $g\circ f=Id$ by the Bigons relation and $f\circ g=Id$ by the bubble relations.\vspace{0.5cm}\par
    For the other points we state the maps. Checking that they are isomorphisms follows directly from the relations in Lemma \ref{localrelations} and \ref{localrelationsagain}.\vspace{0.5cm}
    \\$2.$ For the identity on the right:\vspace{0.5cm}\\
    $f = $\: \inlinepic{\begin{scope}[scale=0.4]\IsomorphismOne \end{scope}}  , \: $g=-$ \: \inlinepic{\begin{scope}[scale=0.4]\IsomorphismOneSym \end{scope}} \quad  (Band and bubble relations).\vspace{0.5cm}\\The isomorphisms are symmetrical for the first graph, with a sign change.
    \vspace{0.5cm}\\ 
    $3.$
    $f = $\: \inlinepic{\begin{scope}[scale=0.25]\IsomorphismTwo \end{scope}}  , \: $g=-$ \: \inlinepic{\begin{scope}[scale=0.25]\IsomorphismTwoSym \end{scope}} \quad ($(I1)$ and tube relations). 
    \vspace{0.5cm}\\
    $4.$
    $f=$ \: \inlinepic{\begin{scope}[scale=0.25]\IsomorphismThree \end{scope}}  , \: $g=$ \: \inlinepic{\begin{scope}[scale=0.25]\IsomorphismThreeSym \end{scope}} \quad ($(I2)$ and tube relations).

\end{proof}
\begin{proof} (Theorem \ref{gradedDim})\par
Let $r$ be the number of red edges in the graph $\Gamma$. We know that \\$S(\inlinepic{\draw[Blue,->-](0,0) circle (0.2cm);}) \simeq \mathbb{Z}[X_1,X_2]\text{\{-1\}}\oplus \mathbb{Z}[X_1,X_2]\text{\{1\}}$. Therefore $deg(S(\inlinepic{\draw[Blue,->-](0,0) circle (0.2cm);}))=(q+q^{-1})$. Suppose then that $r=0$ and $n$ is the number of blue circles of $\Gamma$, it follows that $deg(S(\Gamma))=(q+q^{-1})^n$.\\
We proceed by induction, let $\Gamma$ be a graph for which $r$ is positive. Then there is a face $F$ of $\Gamma$ that is either a bigon or a square. Let $n$ be the number of blue circles in $\Gamma$:
\begin{itemize}
    \item If  $F=$\:\inlinepic{\begin{scope}[scale=0.15]       \begin{scope}
            \draw[draw=gray,dotted] (0,0) circle (2cm);
            \clip (0,0) circle (2cm);
            \draw[Blue,] (0,1) arc(90:270:1cm);
            \draw[Blue,] (0,-1) arc(-90:90:1cm);
            \draw[Maroon,thick,] (0,1)--(0,2);
            \draw[Maroon,thick,] (0,-1)--(0,-2);

        \end{scope}\end{scope}},
    then we know from the previous lemma that $S(\Gamma)\simeq S(\Gamma') \text{\{-1\}} \oplus S(\Gamma')\text{\{1\}}$ where $\Gamma'$ is the graph obtained from $\Gamma$ by replacing $F$ with a simple red line, and\\ $deg(S(\Gamma))=(q+q^{-1})deg(S(\Gamma'))$. The number of circles in $\Gamma'$ is $n-1$. If the number of red edges in $\Gamma'$ is stricly lower than that in $\Gamma$ then by induction $deg(S(\Gamma))=(q+q^{-1})(q+q^{-1})^{n-1}=(q+q^{-1})^n$. This is the case unless the component of $F$ is of this form \inlinepic{\draw[Blue] (0,0)circle (0.15cm); \draw[Maroon,thick](0,0.15)to[in=90,out=90](-0.4,0)to[in=-90,out=-90](0,-0.15);}, and is replaced by a red circle in $\Gamma'$, but  $S(\inlinepic{\draw[Maroon,thick] (0,0) circle (0.2cm);})\simeq S(\emptyset)$. We can suppose that $\Gamma'$ does have fewer red edges than $\Gamma$.
    \item For the other options of faces $F$, a different bigon or a square, it is clear that the isomorphisms from Lemma \ref{TheLemma} send $\Gamma$ to a graph $\Gamma'$ with strictly fewer red edges and the same number of blue circles $n$. By induction $deg(S(\Gamma))=(q+q^{-1})^n$. This finishes the proof of the theorem.

\end{itemize}
\end{proof}

\section{The Khovanov complex using foams}\label{TheDef}
\subsection{Formal chain complexes in the category of foams}
Similarly to \cite{BarNatan2005}, we start with some definitions to lay the groundwork for the rest of the section. 
\begin{deff}
    Let $R$ be a ring. 
    A category $\mathcal{C}$ is said to be $R$-linear if for any two objects $O,O' $ in $ \mathcal{C}$, $Hom_{\mathcal{C}}(O,O')$ is an $R$-module, and the composition of morphisms is bilinear. 
\end{deff}
For example the categories $\mathcal{F}(B)$ and $\mathcal{F}(B)_{/\ell}$ are $\mathbb{Z}[X_1,X_2]$-linear. 
\begin{deff}
    Given a $R$-linear category $\mathcal{C}$, $Mat(\mathcal{C})$ is the $R$-linear category defined as follows:
    \begin{itemize}
        \item The objects of $\operatorname{Mat}(\mathcal{C})$ are formal direct sums (possibly empty) $\oplus_{i=1}^n \mathcal{O}_i$ of objects $\mathcal{O}_i$ of $\mathcal{C}$.
        \item If $\mathcal{O}=\oplus_{i=1}^m \mathcal{O}_i$ and $\mathcal{O}^{\prime}=\oplus_{j=1}^n \mathcal{O}_j^{\prime}$, then a morphism $F: \mathcal{O}^{\prime} \rightarrow \mathcal{O}$ in $\operatorname{Mat}(\mathcal{C})$ will be an $m \times n$ matrix $F=\left(F_{i j}\right)$ of morphisms $F_{i j}: \mathcal{O}_j^{\prime} \rightarrow \mathcal{O}_i$ in $\mathcal{C}$.
        \item Linear combinations of morphisms in $\operatorname{Mat}(\mathcal{C})$ are the linear combinations of the corresponding matrices.
        \item Compositions of morphisms in $\operatorname{Mat}(\mathcal{C})$ are defined by a rule modeled on matrix multiplication, but with compositions in $\mathcal{C}$ replacing the multiplication.
$$
\left(\left(F_{i j}\right) \circ\left(G_{j k}\right)\right)_{i k}:=\sum_j F_{i j} \circ G_{j k} .
$$
    \end{itemize}

\end{deff}
Objects in $Mat(\mathcal{C})$ can be represented by columns of objects of $\mathcal{C}$, and morphisms by arrows between them. The empty column $()$ is the zero object (Figure \ref{fig:Mat}).\\

\begin{figure}
\centering
\begin{tikzpicture}
\draw[->,shorten <=10pt,shorten >=10pt] (0,1)--(4,1.5) node[midway,above] {$F_{1,1}$};
\draw[->,shorten <=10pt,shorten >=10pt] (0,1)--(4,-1.5) node[pos=0.8,above] {$F_{3,1}$};
\draw[->,shorten <=10pt,shorten >=10pt] (0,1)--(4,0) node[pos=0.8,above] {$F_{2,1}$};
\draw[->,shorten <=10pt,shorten >=10pt,dashed] (0,-1)--(4,1.5) node[pos=0.7,above] {$F_{1,2}$};
\draw[->,shorten <=10pt,shorten >=10pt,dashed] (0,-1)--(4,-1.5) node[near end,below] {$F_{3,2}$};
\draw[->,shorten <=10pt,shorten >=10pt,dashed] (0,-1)--(4,0) node[midway,below] {$F_{2,2}$};

\node at (0,1) {$\mathcal{O}_{1}$};
\node at (0,-1) {$\mathcal{O}_{2}$};
\node at (4,1.5) {$\mathcal{O}'_{1}$};
\node at (4,0) {$\mathcal{O}'_{2}$};
\node at (4,-1.5) {$\mathcal{O}'_{3}$};

\end{tikzpicture}

\caption{A representation of a morphism in $Mat(\mathcal{C})$. }
\label{fig:Mat}
\end{figure}
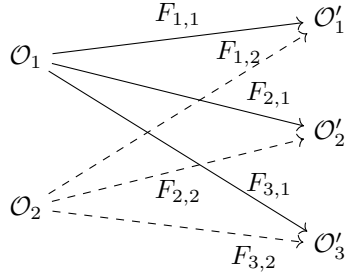
A morphism in $Mat(\mathcal{C})$ is homogeneous of degree $d$ if all its components are of degree $d$.

\begin{deff}
    \textbf{Formal chain complexes}\\Let $R$ be a ring and $\mathcal{D}$ be an $R$-linear category. $Kom(\mathcal{D})$ is the category of formal chain complexes in $\mathcal{D}$ defined by:
    \begin{itemize}
        \item Objects $(O,d) \in Kom(\mathcal{D})$ are chains of morphisms $$\ldots\xrightarrow{d_{i-2}}O_{i-1}\xrightarrow{d_{i-1}} O_i\xrightarrow{d_{i}}O_{i+1}\xrightarrow{d_{i+1}}\ldots$$ such that the composition $d_{i+1}\circ d_{i}=0$ for all $i$. 
        \item Morphisms $f:(O,d)\rightarrow(O',d')$ are collections of morphisms $f_i:O_i\rightarrow O'_i$ in $\mathcal{D}$ such that for all $i$, $f_{i+1}\circ d_{i}= d'_{i}\circ f_{i}$.
        \item The composition of two morphisms $f$ and $g$ is given by $(f\circ g)_i=f_i \circ g_i$.
    \end{itemize}
    
\end{deff}
\begin{deff}
    \textbf{Homotopy in formal complexes}\\
    Let $R$ be a ring. Let $\mathcal{D}$ be an $R$-linear category. We say that two morphisms $F, G:\left(O_i\right) \rightarrow\left(O_i'\right)$ in $Kom(\mathcal{D})$ are homotopic (and we write
$F \sim G$ ) if there exists  morphisms $h_i: O_i \rightarrow O_{i-1}$ so that $F_i-G_i=h_{i+1} d_i+d_{i-1} h_i$ for all $i$.

\end{deff}
Just like for chain complexes of abelian groups, homotopy is an equivalence relation and is invariant under  composition on the left and on the right. Therefore for a category $\mathcal{D}$ we can quotient the morphisms of $Kom(\mathcal{D})$ by homotopy equivalences to obtain the category $Kom_{/h}(\mathcal{D})$.\par
For readability we write $Kob(B)$ for $Kom(Mat(\mathcal{F_{/\ell}}(B)))$ and $Kob_{/h}(B)$ for $Kom_{/h}(Mat(\mathcal{F_{/\ell}}(B)))$. When B is empty we just write $Kob$ and $Kob_{/h}$.

\subsection{The definition of the Khovanov complex: A picture}\label{DefPic}

Let $D$ be a diagram of a link $L$ with n crossings. 
Let $n_-$ be the number of negative crossings of D, and $n_+$ the number of positive crossings.
Choose an order on the crossings. \\
We describe the construction of a "cube of resolution" of $D$. It is a cube with vertices trivalent graphs corresponding to different ways of resolving the crossings of the diagram, and with edges signed cobordisms between them. It can be seen as a way to modify the classic cube of resolution (see Section \ref{sec:RecoveringKh}) in order to keep track of orientation. This is the key idea for Blanchet's oriented model of Khovanov homology \cite{BLANCHET2010}.\\

An example of a cube of resolution of the trefoil \inlinepic{\begin{scope}[scale=0.3]\Trefoil\end{scope}} is given in Figure \ref{fig:Picture}.
\par
\textbf{The vertices}\\
A state $s$ associates to each positive (resp. negative) crossing a number $s(c)\in \{0,1\}$ (resp. $s(c) \in \{0,-1\}$). For a state s we construct a trivalent graph $D_s$, obtained from $D$ b replacing each crossing $c$ by: 
\begin{center}
    
\begin{tabular}{c @{\hspace{2cm}} c}

if $c$ is a positive crossing &
if $c$ is a negative crossing \vspace{0.5cm}\\ 
\inlinepic{\Positive}&
\inlinepic{\Negative}

\end{tabular}

\end{center}

The height of $s$ is the sum $|s|=\sum_cs(c)$. The graphs $D_s$ are the vertices of the cube, and the are grouped by height. \par
\textbf{The edges}\\
Let $s$ and $s'$ be two states respectively at height $i$ and $i+1$, such that they differ by only one crossing $c$ in position $j$. Then $d_s^j$ is the $D_s$-$Foam$-$D_{s'}$ defined by: 
\begin{itemize}
    \item outside of a neighborhood of  $c$  $d_s^j$ is the identity cobordism.
    \item in the neighborhood of $c$ we have, depending on the sign of the crossing one of the two cobordisms:
    \begin{figure}[h]
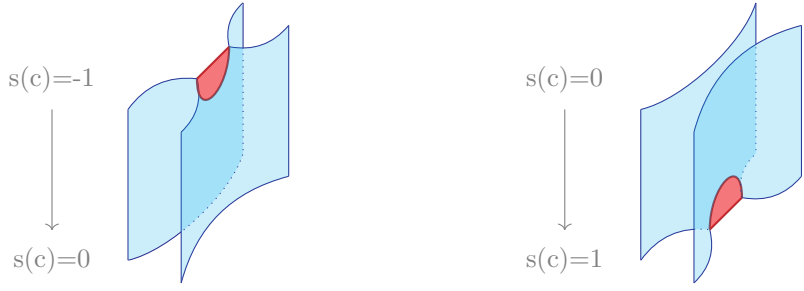

        \centering
        \inlinepic{\SaddleUp}\hspace{3cm}\inlinepic{\SaddleDown}
        \caption{The differentials.}
        \label{fig:Saddles}
    \end{figure}
\end{itemize}
Let $o_s$ be the set of integers $j$ such that either the $ j^{th}$ crossing is positive and $s(j)=0$, or it's negative and $s(j)=-1$. The edges of the cube correspond to the maps $\{d_s^j\}_{s,j\in o_s}$.
\par
\textbf{The signs}\\
We work in the category $Kob$.\\
Let $\llbracket D \rrbracket _i=\bigoplus_{|s|=i}D_s $, and $d_i=\sum_{\substack{{|s|=i}\\j\in o_s}} \epsilon_s^jd_s^j$, where $\epsilon_s^j$ is the sign assigned to the corresponding edge.\\ We want the bracket   $\llbracket D \rrbracket$ defined above to be a chain complex: 
$$\llbracket D \rrbracket_{-n_{-}}\rightarrow \llbracket D \rrbracket_{-n_{-}+1}\rightarrow \ldots \rightarrow\llbracket D \rrbracket_{n_{+}}$$
First notice that when we forget the signs the faces of the cube commute. \vspace{0.5cm}\\
\begin{center}
    
 \inlinepic{\begin{scope}[scale=0.3]\CommOne\end{scope}}\quad =\quad \inlinepic{\begin{scope}[scale=0.3]\CommTwo\end{scope}}
\vspace{0.5cm}\\
and
\vspace{0.5cm}\\
\inlinepic{\begin{scope}[scale=0.3]\CommThree\end{scope}}\quad=\quad \inlinepic{\begin{scope}[scale=0.3]\CommFour\end{scope}} 

\end{center}

It is enough to choose the signs $\epsilon_s^j$ well so that they anticommute instead. then $d\circ d=0$ and $ \llbracket D \rrbracket $ is a chain complex. .\par
Let $s$ be a state and $n_s=\sum_c |s(c)|$,. Let $\Delta_s$ be the abelian group spanned by crossings $c$ such that $|s(c)|=1$.  \\
Let $j$ be such a crossing. If $j$ is positive let $\epsilon_s^j$ be the map $.\wedge j: \bigwedge^{n_s}\Delta_s \rightarrow \bigwedge^{n_s+1}\Delta_{s'} $ . If $c$ is negative let let $\epsilon_s^j$ be the map $.\lrcorner j: \bigwedge^{n_s}\Delta_s \rightarrow \bigwedge^{n_s-1}\Delta_{s'}$ where $\lrcorner$ is the antisymmetrization of the exterior product. These maps can be interpreted as signs. \\
Let $j_1$ and $j_2$ be distinct elements of $o_s$. The exterior product and antisymmetrization are anticommutative. So if $w$ is an element of $\bigwedge^{n_s}\Delta_s$ then $w\wedge j_1 \wedge j_2=w \wedge j_2 \wedge j_1$ and $w\lrcorner j_1 \lrcorner j_2=w \lrcorner j_2 \lrcorner j_1$. By definition we also have $w\wedge j_1 \lrcorner j_2=w \lrcorner j_2 \wedge j_1$.  By adding the signs the faces do anticommute. \\
Therefore the cube of resolution of a diagram defines a formal chain complex  $\llbracket D \rrbracket \in Kob $.\\
Finally, \textbf{the Khovanov complex of a link diagram $D$}, $ CKh(D)$ is constructed from the bracket $\llbracket D \rrbracket $ by shifting the degrees \\ $ CKh^i(D)=\llbracket D \rrbracket_i\{i+n_+-n_-\}$, that is $deg( CKh^i(D))=q^{i+n_+-n_-}deg(\llbracket D \rrbracket_i)$.\par

\begin{example}: \textbf{The cube of resolution of a Trefoil}\\
    Consider the following diagram $D$ of a trefoil, with ordered crossings:\begin{center}
    \inlinepic{\Trefoil 
    \node[below] at (-90:0.5cm) {\tiny 1};
    \node[above right] at (30:0.5cm){\tiny 2};
    \node[above left] at (150:0.5cm){\tiny 3};
    }
    \end{center}
All the crossings are positive. We draw the vertices of the cube of resolution, label them with $s$ and indicate the signs. The cobordisms are harder to draw. A change in the state of a crossing $s(c)$ always corresponds to one of the cobordisms in Figure \ref{fig:Saddles}, everywhere else we have the identity cobordism. Refer to section \ref{sec:Invariance} for examples of smaller chain complexes, where differentials are explicitely drawn.\\
The picture we get is then seen as a formal chain complex of objects in the category $Mat(\mathcal{F})$, as in figure \ref{fig:Mat}.

\begin{figure}[h]

\[
\begin{tikzcd}[column sep =2cm,row sep=0.5cm]
	& {\inlinepic{\begin{scope}[scale=0.8]\OneCircle\end{scope}\node[darkgray] at (0,1.2) {\small 100};}}
    & {\inlinepic{\begin{scope}[scale=0.8,rotate=240]\TwoCircles\end{scope}\node[darkgray] at (0,1.2) {\small 110};}} \\
	{\inlinepic{\begin{scope}[scale=0.8]\Oriented\end{scope}\node[darkgray] at (0,1.2) {\small 000};}} 
    & {\quad\inlinepic{\begin{scope}[scale=0.8,rotate=120]\OneCircle\end{scope}\node[darkgray] at (0,1.2) {\small 010};}} 
    & {\inlinepic{\begin{scope}[scale=0.8,rotate=120]\TwoCircles\end{scope}\node[darkgray] at (0,1.2) {\small 101};}} 
    & {\inlinepic{\begin{scope}[scale=0.8]\ThreeCircles\end{scope}\node[darkgray] at (0,1.2) {\small 111};}} \\
	& {\quad \inlinepic{\begin{scope}[scale=0.8,rotate=240]\OneCircle\end{scope}\node[darkgray] at (0,1.2) {\small 001};}} 
    & {\inlinepic{\begin{scope}[scale=0.8]\TwoCircles\end{scope}\node[darkgray] at (0,1.2) {\small 011};}}\\
    \llbracket D \rrbracket_0 \arrow[r]&\llbracket D \rrbracket_1 \arrow[r]&\llbracket D \rrbracket_2\arrow[r]&\llbracket D \rrbracket_3\\
	\arrow[from=1-2, to=1-3, "+",near end]
	\arrow[from=1-2, to=2-3,"+",near end]
    \arrow[from=3-2, to=2-3,"-",near end]
	\arrow[from=1-3, to=2-4,"+",near end]
	\arrow[from=2-1, to=1-2,"+",near end]
	\arrow[from=2-1, to=2-2,"+",near end]
	\arrow[from=2-1, to=3-2,"+",near end]
	\arrow[from=2-2, to=1-3,"-", near end,crossing over]
	\arrow[from=2-2, to=3-3,"+",near end,crossing over]
	\arrow[from=2-3, to=2-4,"-",near end]
	\arrow[from=3-2, to=3-3,"-",near end]
	\arrow[from=3-3, to=2-4,"+",near end]
\end{tikzcd}
\]
\caption{The cube of resolution of a trefoil.}
\label{fig:Picture}
\end{figure}
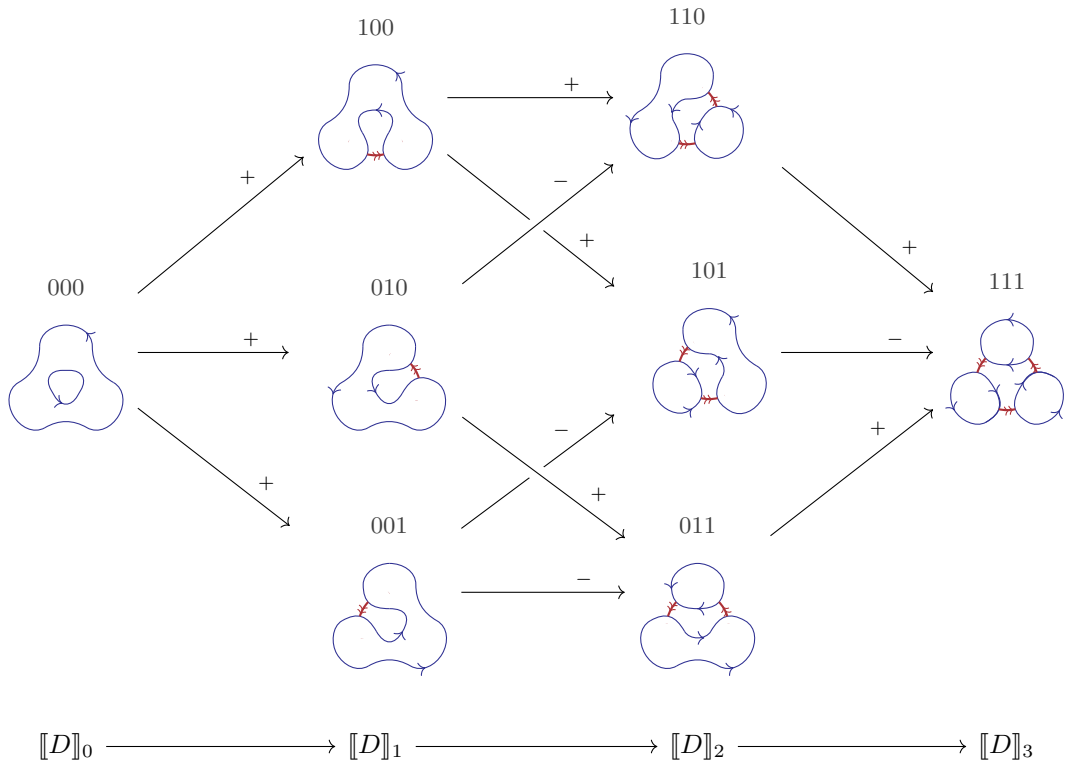

\end{example}

\subsection{Invariance} \label{sec:Invariance}

Let $D$ be a planar diagram of a tangle with boundary $B$. Then we can define $CKh(D)\in Kob(B)$ in the same way we did for links. 
\begin{thh}\label{invariance}
Let $D$ be a diagram of a tangle $T$. Then:
    \begin{enumerate}

    \item  All differentials in $ CKh(D)$ are of degree $0$.
    \item The complex $ CKh(T)$ is an invariant of the tangle $T$ up to degree $0$ homotopy equivalences. That is, if $D_1$ and $D_2$ are two diagrams which differ by some Reidemeister moves and planar isotopy, then there is a homotopy equivalence $F:  CKh\left(D_1\right) \rightarrow CKh\left(D_2\right)$ with $deg(F)=0$. 
    \end{enumerate}
\end{thh}
\begin{proof}
    The cobordisms $d_s^j$ defined above are of degree $-1$, so the differentials $d_i$ in the complex $\llbracket D \rrbracket$ are all of degree $-1$. 
    But after the shift of degree  $deg( CKh^i(D))=q^{-i+n_--n_+}deg(\llbracket D \rrbracket_i)$, the degrees of the differentials $d_i$ in $CKh$ are shifted up by the difference $(i+1+n_--n_+)-(i+n_--n_+)=1$. That is $deg(d)=0$. \par 
    Now we define the homotopy equivalence maps for each Reidemeister move. 
    \begin{itemize}
    \item         $R1+$: The chain complex for $R1+$ is:\vspace{0.5cm}\\
        $\llbracket 
        \inlinepic{\begin{scope}[scale=0.4]
        \ROneRight{+}
        \end{scope}}\rrbracket =$
        \begin{tikzcd} [column sep=1cm]
            {0} \arrow[r] & \inlinepic{\begin{scope}[scale=0.4]\draw (0,-1.7)--(2.25,-1.7);\Arrow \begin{scope}[xshift=1.5cm] \Circle{-}\end{scope}\end{scope}} 
            \arrow[r,"d"] & \inlinepic{\begin{scope}[scale=0.4]\PRight\end{scope}} \arrow[r] & 0
        \end{tikzcd}.\vspace{0.5cm}\\ The graph at height 0 is underlined. We want to find chain maps $f$ and $g$ \vspace{0.5cm}\\
        \begin{tikzcd}[column sep=1cm]
            {0} \arrow[r] & 
            {\inlinepic{\begin{scope}[scale=0.4];\Arrow \begin{scope}[xshift=1.5cm] \Circle{-}\end{scope}\end{scope}}} 
            \arrow[r,"d"] \arrow[d,"f",swap, shift right = 3] & 
            \inlinepic{\begin{scope}[scale=0.4]\PRight\end{scope}} \arrow[r] \arrow[l,"h", shift left=3]& 0            \\
            0 \arrow[r]& \inlinepic{\begin{scope}[scale=0.3]  \Arrow  \end{scope}}
            \arrow[u, "g",swap]\arrow[r] & 0
        \end{tikzcd}\vspace{0.5cm}\\ and a homotopy map $h$. We have: \vspace{0.5cm} \\
        $d=\inlinepic{\begin{scope}[scale=0.3]\dOne\end{scope}}$. \quad Let $h=$ \: 
        $\inlinepic{
        \begin{scope}[scale=0.3]
            \dTwo
            \begin{scope}[xshift=3cm,yshift=-2cm]
                \BCap{}{+}
            \end{scope}
        \end{scope}}$ . \vspace{0.5cm}\\            
        Let $f=$ \: $\inlinepic{
        \begin{scope}[scale=0.3]
            \BId{}{+}{-}
            \begin{scope}[xshift=3cm,yshift=2cm]
                \BCup{}{-}
            \end{scope}
        \end{scope}

        }   $ \quad and \quad  $g=$ \:  $-$ \:
        $\inlinepic{
        \begin{scope}[scale=0.3]
            \BId{1}{+}{-}
            \begin{scope}[xshift=3cm,yshift=-2cm]
                \BRCap{}{+}{}{}
            \end{scope}
        \end{scope}}$ \:  $+$ \:
        $\inlinepic{
        \begin{scope}[scale=0.3]
            \BId{}{+}{-}
            \begin{scope}[xshift=3cm,yshift=-2cm]
                \BRCap{1}{+}{}{}
            \end{scope}
        \end{scope}}$ 
        \vspace{0.5cm}\\Then: 
            \begin{enumerate}
                \item \textbf{$d\circ g =0$:}
            \begin{longtable}{r l p{5cm}}
            $d\circ g=$ & $-$ \inlinepic{\begin{scope}[scale=0.3] 
            \ROnePlusCommutativity 
            \begin{scope}[yshift=1cm]
                \BDot
            \end{scope}
            \end{scope}}
            \quad $+$ \inlinepic{\begin{scope}[scale=0.3] 
            \ROnePlusCommutativity 
            \begin{scope}[xshift=3cm,yshift=1.8cm]
                \BDot
            \end{scope}
            \end{scope}}     \\=
            & $-$ \quad \inlinepic{\begin{scope}[scale=0.3] \dFour 
            \begin{scope}[yshift=1cm]
                \BDot
            \end{scope}\end{scope}} \quad $-$  
            \inlinepic{\begin{scope}[scale=0.3] 
            \ROnePlusCommutativity 
            \begin{scope}[xshift=3cm]
                \BDot
            \end{scope}
            \end{scope}} \quad $+$  
            \inlinepic{\begin{scope}[scale=0.3] 
            \ROnePlusCommutativity 
            \begin{scope}[xshift=3cm,yshift=1cm]
                \RDot
            \end{scope}
            \end{scope}} &\textcolor{darkgray}{Bubble relation on the first term and migration on the second term}\\
            =& $-$ \quad \inlinepic{\begin{scope}[scale=0.3] \dFour 
            \begin{scope}[yshift=1cm]
                \BDot
            \end{scope}\end{scope}} \quad $-$ \quad \inlinepic{\begin{scope}[scale=0.3] \dFour 
            \begin{scope}[xshift=0.5cm,yshift=-2cm]
                \BDot
            \end{scope}\end{scope}}  \quad $+$ \quad \inlinepic{\begin{scope}[scale=0.3] \dFour 
            \begin{scope}[yshift=-1.5cm]
                \RDot
            \end{scope}\end{scope}} \vspace{0.5cm} & 
            \textcolor{darkgray}{Placing the red dot on a different component doesn't change the morphism, then we use bubble relations again}\\=
            & $0$ & \textcolor{darkgray}{This is the migration relation again.}

            \end{longtable}
            \item \textbf{$f\circ g =Id$:}  \\$f\circ g$ \: $=$ \: $-$ \:
            \inlinepic{
                \begin{scope}[scale=0.3]
                    \BId{1}{+}{-}
                    \begin{scope}[yshift=0.5cm,xshift=3cm]
                        \BSphere{}  
                        \Med
                    \end{scope}
                \end{scope}
            } \: $+$ \:
            \inlinepic{
                \begin{scope}[scale=0.3]
                    \BId{}{+}{-}
                    \begin{scope}[yshift=0.5cm,xshift=3cm]
                        \BSphere{1}  
                        \Med
                    \end{scope}
                \end{scope}
            }\vspace{0.5cm}\\
            but a sphere without points is just 0, and \quad 
            \inlinepic{\begin{scope}[scale=0.3]\BSphere{1}\Med \draw[Maroon,thick,->-](-0.01,-0.5)--(0.01,-0.5);\end{scope}}\: =\: - 
            \inlinepic{\begin{scope}[scale=0.3]\BSphere{1}\end{scope}} \: $=1$. \\ So that $f\circ g=$ \inlinepic{
                \begin{scope}[scale=0.3]
                    \BId{}{+}{-}
                \end{scope} 
            } \: = $Id$.
            \item \textbf{$g \circ f -Id = d \circ h$:}\\

            \begin{longtable}{r l }
            $g\circ f -Id=$ & -\quad
            \inlinepic{
                \begin{scope}[scale=0.3]
                    \BId{1}{+}{-}
                    \begin{scope}[xshift=3cm,yshift=2cm]
                        \BCup{}{-}
                    \end{scope}
                    \begin{scope}[xshift=3cm,yshift=-2cm]
                        \BRCap{}{+}{}{}
                    \end{scope}
                \end{scope}
            }   + \quad
            \inlinepic{
                \begin{scope}[scale=0.3]
                    \BId{}{+}{-}
                    \begin{scope}[xshift=3cm,yshift=2cm]
                        \BCup{}{-}
                    \end{scope}
                    \begin{scope}[xshift=3cm,yshift=-2cm]
                        \BRCap{1}{+}{}{}
                    \end{scope}
                \end{scope}        
            }    -            \quad
            \inlinepic{
                \begin{scope}[scale=0.3]
                    \BId{}{+}{-}
                    \begin{scope}[xshift=3cm]
                        \BCyl{}{+}{-}
                    \end{scope}
                \end{scope}
            }
              \vspace{0.5cm}\\=& - \quad
            \inlinepic{
                \begin{scope}[scale=0.3]
                    \BId{1}{+}{-}
                    \begin{scope}[xshift=3cm,yshift=2cm]
                        \BCup{}{-}
                    \end{scope}
                    \begin{scope}[xshift=3cm,yshift=-2cm]
                        \BRCap{}{+}{}{}
                    \end{scope}
                \end{scope}
            }   + \quad
            \inlinepic{
                \begin{scope}[scale=0.3]
                    \BId{}{+}{-}
                    \begin{scope}[xshift=3cm,yshift=2cm]
                        \BCup{}{-}
                    \end{scope}
                    \begin{scope}[xshift=3cm,yshift=-2cm]
                        \BRCap{1}{+}{}{}
                    \end{scope}
                \end{scope}        
            }    +            \quad
            \inlinepic{
                \begin{scope}[scale=0.3]
                    \BId{}{+}{-}
                    \begin{scope}[xshift=3cm,yshift=2cm]
                        \BCup{1}{-}
                    \end{scope}
                    \begin{scope}[xshift=3cm,yshift=-2cm]
                        \BRCap{}{+}{}{}
                    \end{scope}
                \end{scope}
            }   - \quad
            \inlinepic{
                \begin{scope}[scale=0.3]
                    \BId{}{+}{-}
                    \begin{scope}[xshift=3cm,yshift=2cm]
                        \BCup{}{-}
                    \end{scope}
                    \begin{scope}[xshift=3cm,yshift=-2cm]
                        \BRCap{1}{+}{}{}
                    \end{scope}
                \end{scope}        
            }   
            
            \vspace{0.5cm}\\
            =& - \quad \inlinepic{
                \begin{scope}[scale=0.3]
                    \BId{1}{+}{-}
                    \begin{scope}[xshift=3cm,yshift=2cm]
                        \BCup{}{-}
                    \end{scope}
                    \begin{scope}[xshift=3cm,yshift=-2cm]
                        \BCap{}{+}
                    \end{scope}
                \end{scope}
            }   + \quad
            \inlinepic{
                \begin{scope}[scale=0.3]
                    \BId{}{+}{-}
                    \begin{scope}[xshift=3cm,yshift=2cm]
                        \BCup{1}{-}
                    \end{scope}
                    \begin{scope}[xshift=3cm,yshift=-2cm]
                        \BCap{}{+}
                    \end{scope}
                \end{scope}        
            }   \vspace{0.5cm}\\ =& \inlinepic{
                \begin{scope}[scale=0.3]
                    \dThree
                    \begin{scope}[xshift=3cm,yshift=-2cm]
                        \BCap{}{+}
                    \end{scope}     
                    \node at (4.5,1) {$d$};
                    \node at (4.5,-1) {$h$};    

                \end{scope}

            }
            \end{longtable}
            First, we used the neck cutting relation, then the second and last term cancel out. Applying the bubble relations to the rest we get the third line, and the last line follows from a neck cutting relation again. 
            \item Finally \textbf{$-Id=h\circ d$,} this is just the tube relation: \vspace{0.5cm}\\
            $h\circ d=$ \inlinepic{
            \begin{scope}[scale=0.3]
                \clip (-2,0) rectangle (2,4);
                \dTwo
            \end{scope}
            \begin{scope}[scale=0.3]
                \clip (-2,-4) rectangle (2,0);
                \dFour
            \end{scope}

            }   $=-Id$

            \end{enumerate}
    
    \item $R1-$: \vspace{0.5cm}\\
    $\llbracket 
    \inlinepic{\begin{scope}[scale=0.4]
        \ROneRight{-}
    \end{scope}}\rrbracket =$
    \begin{tikzcd} [column sep=1cm]
        {0} \arrow[r]& \inlinepic{\begin{scope}[scale=0.4]\PRight\end{scope}} \arrow[r,"d"] & \inlinepic{\begin{scope}[scale=0.4]\draw (0,-1.7)--(2.25,-1.7);\Arrow \begin{scope}[xshift=1.5cm] \Circle{-}\end{scope}\end{scope}} 
        \arrow[r]  & 0
    \end{tikzcd}.\vspace{0.5cm}\\ We define the maps:\vspace{0.5cm}\\
        \begin{tikzcd}[column sep=1cm]
            {0} \arrow[r] & 
            \inlinepic{\begin{scope}[scale=0.4]\PRight\end{scope}}
            \arrow[r,"d"]  & {\inlinepic{\begin{scope}[scale=0.4];\Arrow \begin{scope}[xshift=1.5cm] \Circle{-}\end{scope}\end{scope}}}
             \arrow[r] \arrow[l,"h", shift left=3]\arrow[d,"f",swap, shift right = 3]& 0            \\&
            0 \arrow[r]& \inlinepic{\begin{scope}[scale=0.3]  \Arrow  \end{scope}}
            \arrow[u, "g",swap]\arrow[r] & 0
        \end{tikzcd}\vspace{0.5cm}\\
        $f=$  \:
        \inlinepic{
        \begin{scope}[scale=0.3]
            \BId{1}{+}{-}
            \begin{scope}[xshift=3cm,yshift=2cm]
                \BRCup{}{-}{}{}
            \end{scope}
        \end{scope}} \:  $-$ \:
        \inlinepic{
        \begin{scope}[scale=0.3]
            \BId{}{+}{-}
            \begin{scope}[xshift=3cm,yshift=2cm]
                \BRCup{1}{-}{}{}
            \end{scope}\end{scope}}\quad , \quad $g=$ \: \inlinepic{
        \begin{scope}[scale=0.3]
            \BId{}{+}{-}
            \begin{scope}[xshift=3cm,yshift=-2cm]
                \BCap{}{+}
            \end{scope}
        \end{scope}}  \quad and \quad
        $h  =$ \:-\: \inlinepic{
            \begin{scope}[scale=0.3]
                \dFour
                \begin{scope}[xshift=3cm, yshift=2cm]
                    \BCup{}{-}
                \end{scope}
            \end{scope}
        }. \vspace{0.5cm}\\
        We have $d=$ \inlinepic{\begin{scope}[scale=0.3]\dFive\end{scope}}.Then:\vspace{0.5cm}\\
            \begin{enumerate}
                \item \textbf{$f\circ d =0$:}
            \begin{longtable}{r l p{5cm}}
            $f\circ d=$ & $-$ \inlinepic{\begin{scope}[scale=0.3] 
            \dSix 
            \begin{scope}[yshift=-0.5cm]
                \BDot
            \end{scope}
            \end{scope}}
            \quad $+$ \inlinepic{\begin{scope}[scale=0.3] 
            \dSix 
            \begin{scope}[xshift=2.8cm,yshift=-2.8cm]
                \BDot
            \end{scope}
            \end{scope}}     \\
            =& $-$ \quad \inlinepic{\begin{scope}[scale=0.3] \dTwo 
            \begin{scope}[yshift=0cm]
                \BDot
            \end{scope}\end{scope}} \quad $-$ \quad \inlinepic{\begin{scope}[scale=0.3] \dTwo
            \begin{scope}[xshift=0.5cm,yshift=1.3cm]
                \BDot
            \end{scope}\end{scope}}  \quad $+$ \quad \inlinepic{\begin{scope}[scale=0.3] \dTwo 
            \begin{scope}[yshift=1.8cm]
                \RDot
            \end{scope}\end{scope}} \vspace{0.5cm} & 
            \textcolor{darkgray}{(Migration then bubble relations)}\\=
            & $0$.&\textcolor{darkgray}{(Migration)}\\

            \end{longtable}
            \item \textbf{$f\circ g =Id$:}  \\$f\circ g$ \: $=$ \:  
            \inlinepic{
                \begin{scope}[scale=0.3]
                    \BId{1}{+}{-}
                    \begin{scope}[yshift=0.5cm,xshift=3cm]
                        \BSphere{}  
                        \Med
                    \end{scope}
                \end{scope}
            } \: $-$ \:
            \inlinepic{
                \begin{scope}[scale=0.3]
                    \BId{}{+}{-}
                    \begin{scope}[yshift=0.5cm,xshift=3cm]
                        \BSphere{1}  
                        \Med
                    \end{scope}
                \end{scope}
            }\vspace{0.5cm}\\
            but \quad 
            \inlinepic{\begin{scope}[scale=0.3]\BSphere{1}\Med \draw[Maroon,thick,-<-](-0.01,-0.5)--(0.01,-0.5);\end{scope}}\: =\:  
            \inlinepic{\begin{scope}[scale=0.3]\BSphere{1}\end{scope}} \: $=-1$ \\ and $f\circ g= -( -$ \inlinepic{
                \begin{scope}[scale=0.3]
                    \BId{}{+}{-}
                \end{scope} 
            }\:) \: = $Id$.
            \item \textbf{$g \circ f -Id = d \circ h$:}\\

            \begin{longtable}{r l }
            $g\circ f -Id=$ & 
            \inlinepic{
                \begin{scope}[scale=0.3]
                    \BId{1}{+}{-}
                    \begin{scope}[xshift=3cm,yshift=2cm]
                        \BRCup{}{-}{}{}
                    \end{scope}
                    \begin{scope}[xshift=3cm,yshift=-2cm]
                        \BCap{}{+}
                    \end{scope}
                \end{scope}
            }   - \quad
            \inlinepic{
                \begin{scope}[scale=0.3]
                    \BId{}{+}{-}
                    \begin{scope}[xshift=3cm,yshift=2cm]
                        \BRCup{1}{-}{}{}
                    \end{scope}
                    \begin{scope}[xshift=3cm,yshift=-2cm]
                        \BCap{}{+}
                    \end{scope}
                \end{scope}        
            }    -            \quad
            \inlinepic{
                \begin{scope}[scale=0.3]
                    \BId{}{+}{-}
                    \begin{scope}[xshift=3cm]
                        \BCyl{}{+}{-}
                    \end{scope}
                \end{scope}
            }
              \vspace{0.5cm}\\=& 
            \inlinepic{
                \begin{scope}[scale=0.3]
                    \BId{1}{+}{-}
                    \begin{scope}[xshift=3cm,yshift=2cm]
                        \BRCup{}{-}{}{}
                    \end{scope}
                    \begin{scope}[xshift=3cm,yshift=-2cm]
                        \BCap{}{+}
                    \end{scope}
                \end{scope}
            }   - \quad
            \inlinepic{
                \begin{scope}[scale=0.3]
                    \BId{}{+}{-}
                    \begin{scope}[xshift=3cm,yshift=2cm]
                        \BRCup{1}{-}{}{}
                    \end{scope}
                    \begin{scope}[xshift=3cm,yshift=-2cm]
                        \BCap{}{+}
                    \end{scope}
                \end{scope}        
            }    +            \quad
            \inlinepic{
                \begin{scope}[scale=0.3]
                    \BId{}{+}{-}
                    \begin{scope}[xshift=3cm,yshift=2cm]
                        \BRCup{1}{-}{}{}
                    \end{scope}
                    \begin{scope}[xshift=3cm,yshift=-2cm]
                        \BCap{}{+}
                    \end{scope}
                \end{scope}
            }   - \quad
            \inlinepic{
                \begin{scope}[scale=0.3]
                    \BId{}{+}{-}
                    \begin{scope}[xshift=3cm,yshift=2cm]
                        \BRCup{}{-}{}{}
                    \end{scope}
                    \begin{scope}[xshift=3cm,yshift=-2cm]
                        \BCap{1}{+}
                    \end{scope}
                \end{scope}        
            }   
            
            \vspace{0.5cm}\\
            =& 
            \inlinepic{
                \begin{scope}[scale=0.3]
                    \BId{}{+}{-}
                    \begin{scope}[xshift=3cm,yshift=2cm]
                        \BCup{}{-}
                    \end{scope}
                    \begin{scope}[xshift=3cm,yshift=-2cm]
                        \BCap{1}{+}
                    \end{scope}
                \end{scope}        
            } - \quad \inlinepic{
                \begin{scope}[scale=0.3]
                    \BId{1}{+}{-}
                    \begin{scope}[xshift=3cm,yshift=2cm]
                        \BCup{}{-}
                    \end{scope}
                    \begin{scope}[xshift=3cm,yshift=-2cm]
                        \BCap{}{+}
                    \end{scope}
                \end{scope}
            }     \vspace{0.5cm}\\ =&-\: \inlinepic{
                \begin{scope}[scale=0.3]
                    \dSeven
                    \begin{scope}[xshift=3cm,yshift=2cm]
                        \BCup{}{+}
                    \end{scope}     
                    \node at (4.5,1) {$d$};
                    \node at (4.5,-1) {$h$};    

                \end{scope}

            }
            \end{longtable}
            \item 
            $h\circ d=$ \inlinepic{
            \begin{scope}[scale=0.3]
                \clip (-2,0) rectangle (2,4);
                \dTwo
            \end{scope}
            \begin{scope}[scale=0.3]
                \clip (-2,-4) rectangle (2,0);
                \dFour
            \end{scope}

            }   $=-Id$

            \end{enumerate}
    From now on we only fill in the blue surfaces if it makes reading the cobordisms easier.
    \item $R2+$: The chain complex is:\vspace{0.5cm}\\
    $\llbracket$ \inlinepic{\begin{scope}[scale=0.4]\RTwo{+}{r}
    \node at (0.7,1.2) {\tiny 1};
    \node at (0.7,-1.2) {\tiny 2};
    \end{scope}} $\rrbracket=$
    \begin{tikzcd}[ampersand replacement=\&, column sep=large]
    0 \arrow[r]\& \inlinepic{
        \begin{scope}[scale=0.4]
            \HT
        \end{scope}
    }
    \arrow[r, "{\begin{pmatrix} d_1 \\ d_2 \end{pmatrix}}"]
    \& \inlinepic{
    \begin{scope}[scale=0.4]
        \Arrow
        \begin{scope}[xshift=1cm]
            \Arrow
        \end{scope}
    \end{scope}
    }
    \oplus 
    \inlinepic{
        \begin{scope}[scale=0.4]
            \HOR
        \end{scope}
    }
    \arrow[r, "{\begin{pmatrix} d_3 & -d_4 \end{pmatrix}}"]
    \& 
    \inlinepic{
        \begin{scope}[scale=0.4]
            \HB
        \end{scope}
    }\arrow[r]
    \& 0 
    \end{tikzcd}\vspace{0.5cm}
    \\ where:\vspace{0.5cm}\\
    $d_1$=\inlinepic{\begin{scope}[scale=0.3]\PinchU\end{scope}}\quad, \quad
    $d_2$=\inlinepic{\begin{scope}[scale=0.3]\RTwoPlusdTwo\end{scope}}\quad , \quad
    $d_3$=\inlinepic{\begin{scope}[scale=0.3]\PinchD\end{scope}} \quad and \quad
    $d_4$=\inlinepic{\begin{scope}[scale=0.3]\RTwoPlusdFour\end{scope}}
    \vspace{0.5cm}\\  
    Let: \vspace{0.5cm}\\
    $f$=\inlinepic{\begin{scope}[scale=0.3]\RTwoPlusf\end{scope}}\quad , \quad
    $g=-$ \:\inlinepic{\begin{scope}[scale=0.3]\RTwoPlusg\end{scope}}\quad , \quad
    $h_1$=\inlinepic{\begin{scope}[scale=0.3]\RTwoPlushOne\end{scope}} \quad and \quad
    $h_2 = -$\:\inlinepic{\begin{scope}[scale=0.3]\RTwoPlushTwo\end{scope}}\vspace{0.5cm}\\
    Then we have a homotopy equivalence:
    \vspace{0.5cm}\\ 
    \begin{tikzcd}[ampersand replacement=\&, column sep=large, row sep=2cm] \llbracket \inlinepic{
    \begin{scope}[scale=0.4]
        \draw[->](0,-1.5)--(0,1.5);
        \begin{scope}[xshift=1cm]
            \draw[->](0,-1.5)--(0,1.5);
        \end{scope}
    \end{scope}
    } \rrbracket = \arrow[d , "F"] \&\& 0 \arrow[r] \&  \inlinepic{
    \begin{scope}[scale=0.4]
        \Arrow
        \begin{scope}[xshift=1cm]
            \Arrow
        \end{scope}
    \end{scope}
    } \arrow[d,equal, shift right=0.5cm] \arrow[d,"f",shift left=0.5cm]
         \arrow[r] \& 0 \& \\ \llbracket \inlinepic{\begin{scope}[scale=0.4]\RTwo{+}{r}
    \node at (0.7,1.2) {\tiny 2};
    \node at (0.7,-1.2) {\tiny 1};
    \end{scope}} \rrbracket = 
    \arrow[u,"G",shift left=0.3cm]
    \&
    0 \arrow[r]\& \inlinepic{
        \begin{scope}[scale=0.4]
            \HT
        \end{scope}
    }
    \arrow[r, "{\begin{pmatrix} d_1 \\ d_2 \end{pmatrix}}"]
    \& \inlinepic{
    \begin{scope}[scale=0.4]
        \Arrow
        \begin{scope}[xshift=1cm]
            \Arrow
        \end{scope}
    \end{scope}
    }
    \oplus 
    \inlinepic{
        \begin{scope}[scale=0.4]
            \HOR
        \end{scope}
    }
    \arrow[r, "{\begin{pmatrix} d_3 & -d_4 \end{pmatrix}}"]
     \arrow[u,"g",shift right= 0.3cm]
    \& 
    \inlinepic{
        \begin{scope}[scale=0.4]
            \HB
        \end{scope}
    }\arrow[r]
    \& 0 
    \arrow[from =2-4, to= 2-3, shift left=0.3cm, "{\begin{pmatrix} 0 & h_1 \end{pmatrix}}"]
    \arrow[from =2-5, to= 2-4, shift left=0.3cm, "{\begin{pmatrix} 0 \\ h_2 \end{pmatrix}}"]
    \end{tikzcd}\vspace{0.5cm}\\ To check this:
    \begin{enumerate}
        \item $F$ is a chain map: $d_3-d_4\circ f=0$. \vspace{0.5cm}\\ $d_4\circ f=$ 
        \inlinepic{\begin{scope}[scale=0.3]\CandyCane\end{scope}} = \inlinepic{\begin{scope}[scale=0.3]\PinchD\end{scope}} $=d_3$.\vspace{0.5cm}
        \item $G$ is a chain map: $d_1+g\circ d_2=0$\vspace{0.5cm}\\ $g\circ d_2=$ -\inlinepic{\begin{scope}[scale=0.3]\CandyCaneTwo\end{scope}} = - \inlinepic{\begin{scope}[scale=0.3]\PinchU\end{scope}} $=-d_1$.
        \item $G\circ F=Id$ because $g\circ f=0$. This follows from the bubble relation \inlinepic{\begin{scope}[scale=0.3]
        \Saturn{}
        \end{scope}
        }=0. 
        \item $F\circ G - I= h\circ d +d \circ h$. \vspace{0.5cm}\\
        $F\circ G=\begin{pmatrix}
            Id & f \\ g & f\circ g
        \end{pmatrix}$. \vspace{0.5cm}\\We have that $d_1\circ h_1$\: = \:-\:\inlinepic{\begin{scope}[scale=0.3]\Anotherg  
        \end{scope}
        }\:=\:$g$ \hspace{1cm} $h_2\circ d_3=$ \inlinepic{\begin{scope}[scale=0.3]\Anotherf  
        \end{scope}
        }\:=\:$f$. \vspace{0.5cm}\\
        Therefore $F\circ G - I=\begin{pmatrix}
            0 &  h_2 \circ d_3\\ d_1\circ h_1 & f\circ g -Id
        \end{pmatrix}$ and we just need that $f\circ g -Id=-h_2\circ d_4 + d_2 \circ h_1$. \vspace{0.5cm}\\
        $f \circ g- Id=$ \: - \: \inlinepic{
        \begin{scope}[scale=0.3]
            \RTwoPlusfg
        \end{scope}
        }\quad - \: \inlinepic{
        \begin{scope}[scale=0.3]
            \RTwoPlusI
        \end{scope}
        } \vspace{0.5cm}\\We can apply the bigon relation on both terms: here \: \inlinepic{
        \begin{scope}[scale=0.3]
            \RTwoPlusfg
            \draw[dashed] (0,0) circle (0.9cm);
        \end{scope}
        } \quad and here \: \inlinepic{
        \begin{scope}[scale=0.3]
            \RTwoPlusI
            \draw[dashed] (0,-2) circle (1.3cm);
        \end{scope}
        } \vspace{0.5cm}\\
        and obtain \vspace{0.5cm}\\ $f \circ g- Id=$ \: - \: \inlinepic{
        \begin{scope}[scale=0.3]
            \RTwoPlusCups
            \begin{scope}[xshift=-3cm,yshift=0]
                \BDot
            \end{scope}
        \end{scope}
        } \: + \: \inlinepic{
        \begin{scope}[scale=0.3]
            \RTwoPlusCups
            \begin{scope}[xshift=3.2cm,yshift=-1cm]
                \BDot
            \end{scope}
        \end{scope}
        } \: - \: \inlinepic{
        \begin{scope}[scale=0.3]
            \RTwoPlusCups
            \begin{scope}[xshift=0.1cm,yshift=-3.3cm]
                \BDot
            \end{scope}
        \end{scope}
        } \: + \: \inlinepic{
        \begin{scope}[scale=0.3]
            \RTwoPlusCups
            \begin{scope}[xshift=0.1cm,yshift=-0.8cm]
                \BDot
            \end{scope}
        \end{scope}
        }\vspace{0.5cm}\\
        By the bigon relation again we can connect the blue dots differently:\vspace{0.5cm}\\
         \inlinepic{
        \begin{scope}[scale=0.3]
            \RTwoPlusCups
            \begin{scope}[xshift=3.2cm,yshift=-1cm]
                \BDot
            \end{scope}
        \end{scope}
        } \: - \: \inlinepic{
        \begin{scope}[scale=0.3]
            \RTwoPlusCups
            \begin{scope}[xshift=0.1cm,yshift=-3.3cm]
                \BDot
            \end{scope}
        \end{scope}
        } \: = \: 
        \inlinepic{
            \begin{scope}[scale=0.3]
                \RTwoPlusdh
            \end{scope}
        } \: = $ d_2 \circ h_1$
        \vspace{0.5cm}\\  and \quad \inlinepic{
        \begin{scope}[scale=0.3]
            \RTwoPlusCups
            \begin{scope}[xshift=0.1cm,yshift=-0.8cm]
                \BDot
            \end{scope}
        \end{scope}} \: - \: \inlinepic{
        \begin{scope}[scale=0.3]
            \RTwoPlusCups
            \begin{scope}[xshift=-3cm,yshift=0]
                \BDot
            \end{scope}
        \end{scope}
        } = \:\inlinepic{
            \begin{scope}[scale=0.3]
                \RTwoPlushd
            \end{scope}
        }\: = $- h_2\circ d_4$ \:
        \vspace{0.5cm}\\
            Therefore $f\circ g -Id=-h_2\circ d_4 + d_2 \circ h_1$ like we wanted.\\Finally $-d_4\circ h_2=-Id$ and $h_1\circ d_1=-Id$ by isotopy. 
    \end{enumerate}
In the same way there is a homotopy equivalence between the complexes 
\inlinepic{    \begin{scope}[scale=0.25]
        \draw[->](0,-1.5)--(0,1.5);
        \begin{scope}[xshift=1cm]
            \draw[->](0,-1.5)--(0,1.5);
        \end{scope}
    \end{scope}} 
    and
$\: \llbracket \inlinepic{\begin{scope}[scale=0.25]\RTwo{+}{l}\end{scope}} \: \rrbracket$ given by. 
    \vspace{0.5cm}\\ 
    \begin{tikzcd}[ampersand replacement=\&, column sep=large, row sep=2cm] \llbracket \inlinepic{
    \begin{scope}[scale=0.4]
        \draw[->](0,-1.5)--(0,1.5);
        \begin{scope}[xshift=1cm]
            \draw[->](0,-1.5)--(0,1.5);
        \end{scope}
    \end{scope}
    } \rrbracket = \arrow[d , "F'"] \&\& 0 \arrow[r] \&  \inlinepic{
    \begin{scope}[scale=0.4]
        \Arrow
        \begin{scope}[xshift=1cm]
            \Arrow
        \end{scope}
    \end{scope}
    } \arrow[d,equal, shift right=0.5cm] \arrow[d,"f",shift left=0.5cm]
         \arrow[r] \& 0 \& \\ \llbracket \inlinepic{\begin{scope}[scale=0.4]\RTwo{+}{l}
    \node at (0.7,1.2) {\tiny 2};
    \node at (0.7,-1.2) {\tiny 1};
    \end{scope}} \rrbracket = 
    \arrow[u,"G'",shift left=0.3cm]
    \&
    0 \arrow[r]\& \inlinepic{
        \begin{scope}[scale=0.4]
            \HB
        \end{scope}
    }
    \arrow[r, "{\begin{pmatrix} d_1 \\ -d_2 \end{pmatrix}}"]
    \& \inlinepic{
    \begin{scope}[scale=0.4]
        \Arrow
        \begin{scope}[xshift=1cm]
            \Arrow
        \end{scope}
    \end{scope}
    }
    \oplus 
    \inlinepic{
        \begin{scope}[scale=0.4]
            \HOR
        \end{scope}
    }
    \arrow[r, "{\begin{pmatrix} d_3 & d_4 \end{pmatrix}}"]
     \arrow[u,"g",shift right= 0.3cm]
    \& 
    \inlinepic{
        \begin{scope}[scale=0.4]
            \HT
        \end{scope}
    }\arrow[r]
    \& 0 
    \arrow[from =2-4, to= 2-3, shift left=0.3cm, "{\begin{pmatrix} 0 & h_1 \end{pmatrix}}"]
    \arrow[from =2-5, to= 2-4, shift left=0.3cm, "{\begin{pmatrix} 0 \\ h_2 \end{pmatrix}}"]
    \end{tikzcd}\vspace{0.5cm}
    \item $R2-$:Let $F$ and $G$ be the maps:
\vspace{0.5cm}\\
%
%
%
%
\begin{tikzcd}[ampersand replacement=\&, column sep=large, row sep=2cm] \llbracket \inlinepic{
    \begin{scope}[scale=0.4]
        \draw[<-](0,-1.5)--(0,1.5);
        \begin{scope}[xshift=1cm]
            \draw[->](0,-1.5)--(0,1.5);
        \end{scope}
    \end{scope}
    } \rrbracket = \arrow[d , "F"] \&\& 0 \arrow[r] \&  \inlinepic{
    \begin{scope}[scale=0.4]
        \draw[Blue,->](0,1.5)--(0,-1.5);
        \begin{scope}[xshift=1cm]
            \Arrow
        \end{scope}
    \end{scope}
    } 
    \arrow[r] \& 0 \& \\ \llbracket 
    \inlinepic{
    \begin{scope}[scale=0.4]\RTwo{-}{r}
    \node at (0.7,1.2) {\tiny 1};
    \node at (0.7,-1.2) {\tiny 2};
    \end{scope}} 
    \rrbracket = 
    \arrow[u,"G",shift left=0.3cm]
    \&
    0 \arrow[r]\& \inlinepic{
        \begin{scope}[scale=0.4]
            \HTR
        \end{scope}
    }
    \arrow[r, "{\begin{pmatrix} d_1 \\ d_2 \end{pmatrix}}"]
    \& \inlinepic{
    \begin{scope}[scale=0.4]
        \Sq
    \end{scope}
    }
    \oplus 
    \inlinepic{
        \begin{scope}[scale=0.4]
            \HO
        \end{scope}
    }
    \arrow[r, "{\begin{pmatrix} -d_3 & d_4 \end{pmatrix}}"]
    \& 
    \inlinepic{
        \begin{scope}[scale=0.4]
            \HBR
        \end{scope}
    }\arrow[r]
    \& 0 
    \arrow[from =2-4, to= 2-3, shift left=0.3cm, "{\begin{pmatrix} 0 & h_1 \end{pmatrix}}"]
    \arrow[from =2-5, to= 2-4, shift left=0.3cm, "{\begin{pmatrix} 0 \\ h_2 \end{pmatrix}}"]
    \arrow[from =1-4, to= 2-4, out=-45, in=90, "f"
    ,
    start anchor={[xshift=0.2cm]}
    ,
    end anchor={[xshift=1.2cm]}
    ]    
    \arrow[from =2-4, to= 1-4, in=-45, out=90, "g",
    start anchor={[xshift=1cm]}
    ,
    end anchor={[xshift=0cm]}
    ]
    \arrow[from =2-4, to= 1-4, in=-135, out=90, "a_1"
    ,
    start anchor={[xshift=-1.2cm]}
    ,
    end anchor={[xshift=-0.2cm]}
    ]    
    \arrow[from =1-4, to= 2-4, out=-135, in=90, "a_2"{right},
    start anchor={[xshift=0cm]}
    ,
    end anchor={[xshift=-1cm]}
    ]
\end{tikzcd}\vspace{0.5cm}\\ 
Where:
\vspace{0.5cm}\\ 
$f=$ \inlinepic{
    \begin{scope}[scale=0.3]
        \RTwoMinusf
    \end{scope}
}\quad
$g=$ \inlinepic{
    \begin{scope}[scale=0.3]
        \RTwoMinusg
    \end{scope}
}\quad
$a_1=-$ \inlinepic{
    \begin{scope}[scale=0.3]
        \aOne
    \end{scope}
}\quad
$a_2=$ \inlinepic{
    \begin{scope}[scale=0.3]
        \aTwo
    \end{scope}
}\quad
\vspace{0.5cm}\\
Let 
$h_1=-$ \inlinepic{
    \begin{scope}[scale=0.3]
        \RTwoMinushOne
    \end{scope}
}\quad
$h_2= -$ \inlinepic{
    \begin{scope}[scale=0.3]
        \RTwoMinushTwo
    \end{scope}
}\quad
\vspace{0.5cm}\\
The differentials are: 
\vspace{0.5cm}\\
$d_1=$\inlinepic{
    \begin{scope}[scale=0.3]
        \DiffOne
    \end{scope}
}\quad
$d_2=$\inlinepic{
    \begin{scope}[scale=0.3]
        \DiffTwo
    \end{scope}
}\quad
$d_3=$\inlinepic{
    \begin{scope}[scale=0.3]
        \DiffThree
    \end{scope}
}\quad
$d_4=$\inlinepic{
    \begin{scope}[scale=0.3]
        \DiffFour
    \end{scope}
}\quad
Then: 
\begin{enumerate}
    \item $F$ and $G$ are chain maps: $-d_3\circ a_2 + d_4 \circ f=0$ and $g\circ d_2 + a_1 \circ d_1=0$ (By Isotopy).
    \item $G\circ F=Id$: in fact $g\circ  f=0$ because \inlinepic{\begin{scope}[scale=0.2]
    \draw[draw=Blue ] (0,0) circle (1cm);
    \draw[draw=Blue, dotted] (1,0) arc (0:180:1cm and 0.5cm) ;
    \draw [draw=Blue](1,0) arc (0:-180:1cm and 0.5cm) ;
    \end{scope}}$=0$ and $a_1\circ a_2=Id$ from the tube relations.
    \item $F\circ G -Id=d\circ h +h\circ d$:\vspace{0.5cm}\\
    $F\circ G = \begin{pmatrix}
        a_2\circ a_1 &a_2 \circ g\\ f\circ a_1 & f \circ g
    \end{pmatrix}$\vspace{0.5cm}\\
    We have $a_2\circ a_1=Id$ from the $(I2)$ relation in Lemma \ref{localrelationsagain}.\vspace{0.5cm}\\
    On top of that we want to prove that$f\circ a_1= -h_2 \circ d_3$, $a_2\circ g=d_1\circ h_1$
    and $f\circ g-Id=d_2 \circ h_1 + h_2\circ d_4 $.\vspace{0.5cm}\\
    $f\circ a_1=$\quad \inlinepic{\begin{scope}[scale=0.3]\faOne\end{scope}}\quad and \quad 
    $h_2 \circ d_3=-$ \quad \inlinepic{\begin{scope}[scale=0.3]\hTwodThree\end{scope}}\vspace{0.5cm}\\
    The $(I2)$ relation gives that $h_2 \circ d_3=$ \quad \inlinepic{\begin{scope}[scale=0.3]\faOneBubble\end{scope}}. Then using the bubble relations $h_2 \circ d_3=-$\quad \inlinepic{\begin{scope}[scale=0.3]\faOne\end{scope}} \quad $=f\circ a_1$.\vspace{0.5cm}\\ The same works for the $a_2\circ g=$\quad \inlinepic{\begin{scope}[scale=0.3]\aTwog\end{scope}} 
    \quad $=$ \quad
    \inlinepic{
        \begin{scope}[scale=0.3]
            \dOnehOne
        \end{scope}
    }    
    \quad $=$ \quad
    $d_1\circ h_1$.
    \vspace{0.5cm} \\
    Now:\vspace{0.5cm}\\
    
    \begin{longtable}[l]{r l  }
        $h_2\circ d_4$ 
        &
         $=$ \quad$-$ \quad 
        \inlinepic{\begin{scope}[scale=0.3] \RTwoMinushd\end{scope}}
        \vspace{0.5cm}
        
        \\
        &
        $=$ \quad
        \inlinepic{\begin{scope}[scale=0.3] \RTwoMinusCups
        \begin{scope}[yshift=-2cm,xshift=-2.5cm]
            \BDot
        \end{scope}
        \end{scope}} \quad $-$ \quad
        \inlinepic{\begin{scope}[scale=0.3] \RTwoMinusCups
        \begin{scope}[yshift=-0.8cm,xshift=0.2cm]
            \BDot
        \end{scope}        
        \end{scope}} 
        \vspace{0.5cm}
        
        \\
        $d_2\circ h_1$
        &
        $=$ \quad$-$
        \quad
        \inlinepic{\begin{scope}[scale=0.3] \RTwoMinusdh\end{scope}}
        \vspace{0.5cm}
        \\
        &
        $=$ 
        \quad
        \inlinepic{\begin{scope}[scale=0.3] \RTwoMinusCups
        \begin{scope}[yshift=-2cm,xshift=2.5cm]
            \BDot
        \end{scope}
        \end{scope}} \quad $-$ \quad
        \inlinepic{\begin{scope}[scale=0.3] \RTwoMinusCups
        \begin{scope}[yshift=-3.2cm,xshift=0cm]
            \BDot
        \end{scope}        
        \end{scope}} 
        \vspace{0.5cm}
        \\
    \end{longtable}
    Let us take a look back at the neck cutting relation and try to get rid of the red disks in it
    \vspace{0.5cm}\\ 
    \begin{longtable}[l]{r l}
    \inlinepic{\begin{scope}[scale=0.4]\BCylE{}{+}{+}\end{scope}} &
    $=$
    \quad \inlinepic{    \begin{scope}[yshift=0.8cm,scale=0.4]
        \BCupE{1}{+}
    \end{scope}
    \begin{scope}[yshift=-0.8cm,scale=0.4]
        \BRCapE{0}{-}{0}{0}
        
    \end{scope}} \quad - \quad \inlinepic{    \begin{scope}[yshift=0.8cm,scale=0.4]
        \BCupE{0}{+}
    \end{scope}
    \begin{scope}[yshift=-0.8cm,scale=0.4]
        \BRCapE{1}{-}{0}{0}

    \end{scope}}
    \vspace{0.5cm}
    \\ &
    $=$ \quad
    $-$
    \inlinepic{    \begin{scope}[yshift=0.8cm,scale=0.4]
        \BCupE{1}{+}
    \end{scope}
    \begin{scope}[yshift=-0.8cm,scale=0.4]
        \BCapE{0}{-}
        
    \end{scope}} \quad 
    $+$
    \quad \inlinepic{    \begin{scope}[yshift=0.8cm,scale=0.4]
        \BCupE{0}{+}
    \end{scope}
    \begin{scope}[yshift=-0.8cm,scale=0.4]
        \BRCapE{}{-}{0}{1}
    \end{scope}
    }\quad
    $+$
    \quad \inlinepic{    \begin{scope}[yshift=0.8cm,scale=0.4]
        \BCupE{0}{+}
    \end{scope}
    \begin{scope}[yshift=-0.8cm,scale=0.4]
        \BRCapE{}{-}{1}{}    

    \end{scope}}\quad
    \vspace{0.5cm}
    \\ & $=$
    \quad
    $-$
    \inlinepic{    \begin{scope}[yshift=0.8cm,scale=0.4]
        \BCupE{1}{+}
    \end{scope}
    \begin{scope}[yshift=-0.8cm,scale=0.4]
        \BCapE{0}{-}
        
    \end{scope}} \quad 
    $-$
    \quad \inlinepic{    \begin{scope}[yshift=0.8cm,scale=0.4]
        \BCupE{0}{+}
    \end{scope}
    \begin{scope}[yshift=-0.8cm,scale=0.4]
        \BCapE{1}{-}
    \end{scope}}\quad
    $+$ \quad $(X_1+X_2)$
    \quad \inlinepic{    \begin{scope}[yshift=0.8cm,scale=0.4]
        \BCupE{0}{+}
    \end{scope}
    \begin{scope}[yshift=-0.8cm,scale=0.4]
        \BCapE{}{-}  \vspace{0.5cm} 

    \end{scope}}
    
    \end{longtable}

    Then:
    \begin{longtable}[l]{r l}
        $f\circ g$ 
        &
        $=$ \quad$-$\quad
        \inlinepic{
            \begin{scope}[scale=0.3]
                \RTwoMinusfg
            \end{scope}
        }
        \vspace{0.5cm}\\
        &
        $=$\quad
        \inlinepic{\begin{scope}[scale=0.3] \RTwoMinusCups
        \begin{scope}[yshift=-2cm,xshift=-2.5cm]
            \BDot
        \end{scope}
        \end{scope}} \quad
        $+$\quad
        \inlinepic{\begin{scope}[scale=0.3] \RTwoMinusCups
        \begin{scope}[yshift=-2cm,xshift=2.5cm]
            \BDot
        \end{scope}
        \end{scope}} \quad 
        $-$\quad$(X_1+X_2)$
        \quad
        \inlinepic{\begin{scope}[scale=0.3] \RTwoMinusCups
        \end{scope}} \quad 
        \vspace{0.5cm}\\
        $Id$
        &
        $=$\quad
        \inlinepic{
            \begin{scope}[scale=0.3]
                \RTwoMinusI
            \end{scope}
        }
        \vspace{0.5cm}\\
        &
        $=$\quad 
        
        \inlinepic{\begin{scope}[scale=0.3] \RTwoMinusCups
        \begin{scope}[yshift=-3.2cm,xshift=0cm]
            \BDot
        \end{scope}        
        \end{scope}} 
        \quad $+$ \quad 
        \inlinepic{\begin{scope}[scale=0.3] \RTwoMinusCups
        \begin{scope}[yshift=-0.8cm,xshift=0.2cm]
            \BDot
        \end{scope}        
        \end{scope}}\quad $-$\quad$(X_1+X_2)$
        \quad
        \inlinepic{\begin{scope}[scale=0.3] \RTwoMinusCups      
        \end{scope}} 

    \end{longtable}
So $f\circ g -Id= h_2\circ d_4 +d_2 \circ h_1$. Finally we need that $d_4\circ h_2=-Id$ and $ h_1\circ d_1=-Id$. This is similar to the relation $ f\circ a_1= h_2 \circ d_3$
\end{enumerate}
For the other side of $R2-$ we have the homotopy equivalence:\vspace{0.5cm}\\

\begin{tikzcd}[ampersand replacement=\&, column sep=large, row sep=2cm] \llbracket \inlinepic{
    \begin{scope}[scale=0.4]
        \draw[->](0,-1.5)--(0,1.5);
        \begin{scope}[xshift=1cm]
            \draw[<-](0,-1.5)--(0,1.5);
        \end{scope}
    \end{scope}
    } \rrbracket = \arrow[d , "F'"] \&\& 0 \arrow[r] \&  \inlinepic{
    \begin{scope}[scale=0.4]
        \draw[Blue,->](0,1.5)--(0,-1.5);
        \begin{scope}[xshift=-1cm]
            \Arrow
        \end{scope}
    \end{scope}
    } 
    \arrow[r] \& 0 \& \\ \llbracket 
    \inlinepic{
    \begin{scope}[scale=0.4]   
    \path[name path=l] (-0.5,-1.5)to[in=-90,out=45](0.5,0)to[in=-45, out=90](-0.5,1.5);
    \path[name path=r] (0.5,-1.5)to[in=-90,out=135](-0.5,0)to[in=-135,out=90](0.5,1.5);
    \path[name intersections ={of=l and r,by={cOne,cTwo}}];
        \draw[->] (-0.5,-1.5)to[in=-90,out=45](0.5,0)to[in=-45, out=90](-0.5,1.5);
        \fill[white] (cOne) circle (7pt) (cTwo) circle (7pt);
        \draw[<-] (0.5,-1.5)to[in=-90,out=135](-0.5,0)to[in=-135,out=90](0.5,1.5); 
    \node at (0.7,1.2) {\tiny 2};
    \node at (0.7,-1.2) {\tiny 1};
    \end{scope}} 
    \rrbracket = 
    \arrow[u,"G'",shift left=0.3cm]
    \&
    0 \arrow[r]\& \inlinepic{
        \begin{scope}[scale=0.4]

\draw[Blue ,->] (0.5,1.5)-- (0.5,1.25) arc(0:-180:0.5cm)--(-0.5,1.5);
\draw[Blue,->] (-0.5,-1.5)--(-0.5,-0.25) arc(180:0:0.5cm)--(0.5,-1.5);
\draw[Maroon,thick] (0.5,-0.5)--(-0.5,-0.5);

        \end{scope}
    }
    \arrow[r, "{\begin{pmatrix} d_1 \\ -d_2 \end{pmatrix}}"]
    \& \inlinepic{
    \begin{scope}[scale=0.4]
            \begin{scope}[rotate=90,]
            \clip (0,0) circle (2cm);
            \begin{scope}[scale=0.75]

            \draw[draw=Blue,-<-] (-1,1)--(1,1);
            \draw[draw=Blue,->-] (-1,-1)--(1,-1);
            \draw[draw=Maroon,thick,postaction={decorate}, decoration={markings, mark=at position 0.5 with {\arrow[scale=0.7]{<<}}}] (-1,-1)--(-1,1);
            \draw[draw=Maroon,thick,postaction={decorate}, decoration={markings, mark=at position 0.5 with {\arrow[scale=0.7]{>>}}}] (1,-1)--(1,1);
            \draw[draw=Blue,postaction={decorate}, decoration={markings, mark=at position 0.3 with {\arrow{>}}}] (1,1)-- (2,2);
            \draw[draw=Blue,postaction={decorate}, decoration={markings, mark=at position 0.3 with {\arrow{<}}}] (-1,1)-- (-2,2) ;
            \draw[draw=Blue,postaction={decorate}, decoration={markings, mark=at position 0.3 with {\arrow{>}}}] (-1,-1)-- (-2,-2);
            \draw[draw=Blue,postaction={decorate}, decoration={markings, mark=at position 0.3 with {\arrow{<}}}] (1,-1)-- (2,-2);             
            \end{scope}
        \end{scope}
    \end{scope}
    }
    \oplus 
    \inlinepic{
        \begin{scope}[scale=0.4]

\draw[Blue,<-] (0.5,-1.5)-- (0.5,-1.25) arc(0:180:0.5cm)--(-0.5,-1.5);
\draw[Blue,->] (0.5,1.5)-- (0.5,1.25) arc(0:-180:0.5cm)--(-0.5,1.5);
\draw[Blue,->-] (0,0) circle (0.5cm);
        \end{scope}
    }
    \arrow[r, "{\begin{pmatrix} d_3 & d_4 \end{pmatrix}}"]
    \& 
    \inlinepic{
        \begin{scope}[scale=0.4]

\draw[Blue,<-] (0.5,-1.5)-- (0.5,-1.25) arc(0:180:0.5cm)--(-0.5,-1.5);
\draw[Blue,<-] (-0.5,1.5)--(-0.5,0.25) arc(-180:0:0.5cm)--(0.5,1.5);
\draw[Maroon,thick] (0.5,0.5)--(-0.5,0.5);
        \end{scope}
    }\arrow[r]
    \& 0 
    \arrow[from =2-4, to= 2-3, shift left=0.3cm, "{\begin{pmatrix} 0 & h_1 \end{pmatrix}}"]
    \arrow[from =2-5, to= 2-4, shift left=0.3cm, "{\begin{pmatrix} 0 \\ h_2 \end{pmatrix}}"]
    \arrow[from =1-4, to= 2-4, out=-45, in=90, "f'"
    ,
    start anchor={[xshift=0.2cm]}
    ,
    end anchor={[xshift=1.2cm]}
    ]    
    \arrow[from =2-4, to= 1-4, in=-45, out=90, "g'",
    start anchor={[xshift=1cm]}
    ,
    end anchor={[xshift=0cm]}
    ]
    \arrow[from =2-4, to= 1-4, in=-135, out=90, "a'_1"
    ,
    start anchor={[xshift=-1.2cm]}
    ,
    end anchor={[xshift=-0.2cm]}
    ]    
    \arrow[from =1-4, to= 2-4, out=-135, in=90, "a'_2"{right},
    start anchor={[xshift=0cm]}
    ,
    end anchor={[xshift=-1cm]}
    ]
\end{tikzcd}
The maps $f'$, $g'$, $a'_1$, $a'_2$ are the same as $f$, $g$, $a_1$ and $a_2$ with the opposite orientations.   
    \item R3+: We are looking for a homotopy equivalence between $\llbracket$\inlinepic{\RThreeLeft{0.3}}$\rrbracket$ \: and \: $\llbracket$\inlinepic{\RThreeRight{0.3}}$\rrbracket$. Let us order the crossings going upwards on both sides, then the cubes of resolutions of the diagrams $\llbracket$\inlinepic{\RThreeLeft{0.3}}$\rrbracket$ \: and \: $\llbracket$\inlinepic{\RThreeRight{0.3}}$\rrbracket$ 
are given respectively in the Figure \ref{fig:RThreeLeft} and \ref{fig:RThreeRight}.
\begin{figure}[h]
    \centering
    \[\begin{tikzcd}[column sep=2cm]
	& 
    \inlinepic{
        \begin{scope}[scale=0.4]
            \HT
            \begin{scope}[xshift=1cm]
                \Arrow
            \end{scope}
        \end{scope}
    }
    & 
    \inlinepic{
        \begin{scope}[scale=0.4]
            \SnakeLeft
        \end{scope}
    }
    \\
    \inlinepic{
        \begin{scope}[scale=0.4]
            \Arrow
            \begin{scope}[xshift=0.75cm]
                \Arrow
            \end{scope}
            \begin{scope}[xshift=-0.75cm]
                \Arrow
            \end{scope}
        \end{scope}
    }
    & 
    \inlinepic{
        \begin{scope}[scale=0.4]
            \HM
            \begin{scope}[xshift=-1cm]
                \Arrow
            \end{scope}
        \end{scope}
    }
    &
    \inlinepic{
        \begin{scope}[scale=0.4]
            \HOR
            \begin{scope}[xshift=1cm]
                \Arrow
            \end{scope}
        \end{scope}
    }
    &
    \inlinepic{
        \begin{scope}[scale=0.4]
            \PopPimpleLeft
        \end{scope}
    }    \\
	&
    \inlinepic{
        \begin{scope}[scale=0.4]
            \HB
            \begin{scope}[xshift=1cm]
                \Arrow
            \end{scope}
        \end{scope}
    }
    &
    \inlinepic{
        \begin{scope}[scale=0.4]
            \SnakeRight
        \end{scope}
    }
    %
    %
	\arrow[from=1-2, to=1-3,"+"]
	\arrow["{\textcolor{white}{.} }"{description}, from=1-2, to=2-3,"+"pos=0.9]
	\arrow["{\textcolor{white}{.}}"{description}, from=3-2, to=2-3,"-" pos=0.9,swap]
	\arrow[from=1-3, to=2-4,"+"]
	\arrow[from=2-1, to=1-2,"+"]
	\arrow[from=2-1, to=2-2,"+"]
	\arrow[from=2-1, to=3-2,"+"]
	\arrow[from=2-2, to=1-3,"-"pos=0.1]
	\arrow[from=2-2, to=3-3,"+"pos=0.1,swap]
	\arrow[from=2-3, to=2-4,"-"]
	\arrow[from=3-2, to=3-3,"-"]
	\arrow[from=3-3, to=2-4,"+"]
\end{tikzcd}\]
    \caption{The cube of resolution of the left side of R3.}
    \label{fig:RThreeLeft}
\end{figure}
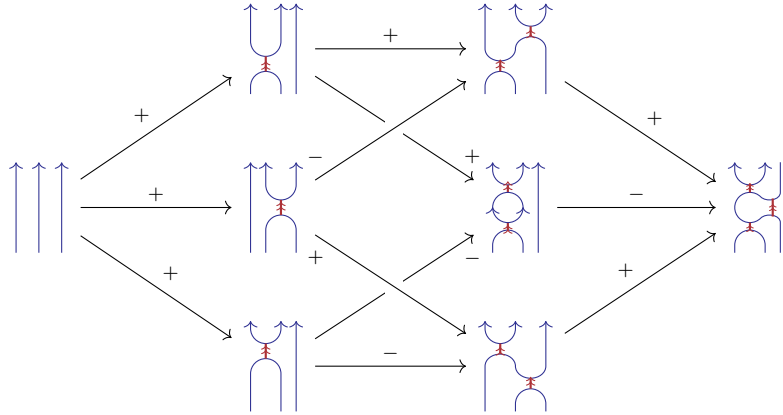

\begin{figure}[h]
    \centering 
    \[\begin{tikzcd}[column sep= 2cm]
	& 
    \inlinepic{
        \begin{scope}[scale=0.4]
            \HM
            \begin{scope}[xshift=1cm]
                \Arrow
            \end{scope}
        \end{scope}
    }
    & 
    \inlinepic{
        \begin{scope}[scale=0.4]
            \SnakeLeft
        \end{scope}
    }
    \\
    \inlinepic{
        \begin{scope}[scale=0.4]
            \Arrow
            \begin{scope}[xshift=0.75cm]
                \Arrow
            \end{scope}
            \begin{scope}[xshift=-0.75cm]
                \Arrow
            \end{scope}
        \end{scope}
    }
    & 
    \inlinepic{
        \begin{scope}[scale=0.4]
            \HB
            \begin{scope}[xshift=-1cm]
                \Arrow
            \end{scope}
        \end{scope}
    }
    &
    \inlinepic{
        \begin{scope}[scale=0.4]
            \SnakeRight
        \end{scope}
    }
    &
    \inlinepic{
        \begin{scope}[scale=0.4]
            \PopPimpleRight
        \end{scope}
    }    \\
	&
    \inlinepic{
        \begin{scope}[scale=0.4]
            \HT
            \begin{scope}[xshift=-1cm]
                \Arrow
            \end{scope}
        \end{scope}
    }
    &
    \inlinepic{
        \begin{scope}[scale=0.4]
            \HOR
            \begin{scope}[xshift=-1cm]
                \Arrow
            \end{scope}
        \end{scope}
    }
    %
    %
	\arrow[from=1-2, to=1-3,"+"]
	\arrow["{\textcolor{white}{.} }"{description}, from=1-2, to=2-3,"-" pos=0.9]
	\arrow["{\textcolor{white}{.}}"{description}, from=3-2, to=2-3,"+"pos=0.9, swap]
	\arrow[from=1-3, to=2-4,"+"]
	\arrow[from=2-1, to=1-2,"+"]
	\arrow[from=2-1, to=2-2,"+"]
	\arrow[from=2-1, to=3-2,"+"]
	\arrow[from=2-2, to=1-3,"-"pos=0.1]
	\arrow[from=2-2, to=3-3,"-"pos=0.1 , swap]
	\arrow[from=2-3, to=2-4,"+"]
	\arrow[from=3-2, to=3-3,"+"]
	\arrow[from=3-3, to=2-4,"-"]
\end{tikzcd}\]
    \caption{The cube of resolution of the right side of R3.}
    \label{fig:RThreeRight}

\end{figure}
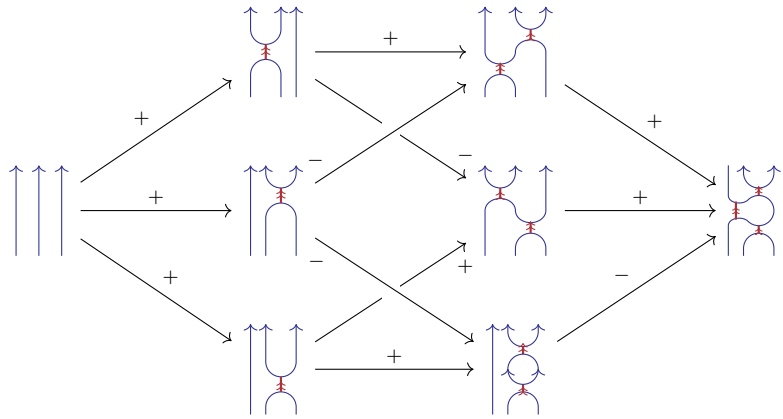

Notice that the top layers of the two complexes are the same and can be interpreted as the complex  $\llbracket$
\inlinepic{\TopLayer{0.3}}
$\rrbracket$.\\ We will prove that the two bottom layers are also homotopic. \par
If we take a trivalent planar graph $\Gamma$ where crossings between blue edges are allowed, then we can define a complex $\llbracket \Gamma \rrbracket$ in the same way as for link diagrams. 
Then we can write the bottom layer of the first complex as a bracket complex of a trivalent graph:
\begin{center}
    
$\llbracket$ \inlinepic{
    \begin{scope}[scale=0.4]
        \HB 
        \draw[white, line width=0.2cm](-1.5,-1.5)to[in=-90,out=90,looseness=1](1.5,1)--(1.5,1.5);
        \draw[->,Blue](-1.5,-1.5)to[in=-90,out=90,looseness=1](1.5,1)--(1.5,1.5);
    \end{scope}
} $\rrbracket$ \quad $=$ \quad \begin{tikzcd}
    \inlinepic{
        \begin{scope}[scale=0.4]
            \HB
            \begin{scope}[xshift=1cm]
                \Arrow
            \end{scope}
        \end{scope}
    }
    \arrow[r]
    & 
    \inlinepic{
        \begin{scope}[scale=0.4]
            \HOR
            \begin{scope}[xshift=1cm]
                \Arrow
            \end{scope}
        \end{scope}
    }
    \oplus
    \inlinepic{
        \begin{scope}[scale=0.4]
            \SnakeRight
        \end{scope}
    }
    \arrow[r]
    &
    \inlinepic{
        \begin{scope}[scale=0.4]
            \PopPimpleLeft
        \end{scope}
    } 
\end{tikzcd}
\end{center}
The Lemma \ref{TheLemma} allows us to simplify this complex. Consider the maps
\begin{center}
    $f_1=$ \quad \inlinepic{\begin{scope}[scale=0.3]\fOne\end{scope}} \hspace{1cm}
    $f_1 '=$ \quad \inlinepic{\begin{scope}[scale=0.3]\fOnePrime\end{scope}} \hspace{1cm}
    $f=$ \quad \inlinepic{\begin{scope}[scale=0.3]\Eff \end{scope}} \vspace{0.5cm}\\
    $f_2=$ \quad \inlinepic{\begin{scope}[scale=0.3]\fTwo\end{scope}} \hspace{1cm}
    $f_2 '=-$ \quad \inlinepic{\begin{scope}[scale=0.3]\fTwoPrime\end{scope}} \hspace{1cm}
    $f'=-$ \quad \inlinepic{\begin{scope}[scale=0.3]\EffPrime \end{scope}} \vspace{0.5cm}\\
\end{center}
Where the colored rectangle is to be seen as a blue sheet beside (in front of) the other component.\\
As usual, using the relations from Lemma \ref{localrelations} and \ref{localrelationsagain}, the following map is an isomorphism\vspace{0.5cm}\\
\[\begin{tikzcd}[row sep=1.5cm, column sep=3cm,ampersand replacement=\&]
    \inlinepic{
        \begin{scope}[scale=0.4]
            \HB
            \begin{scope}[xshift=1cm]
                \Arrow
            \end{scope}
        \end{scope}
    }
    \arrow[r,"{\begin{pmatrix}
        d_1\\d_2
    \end{pmatrix}}"]
    \& 
    \inlinepic{
        \begin{scope}[scale=0.4]
            \HOR
            \begin{scope}[xshift=1cm]
                \Arrow
            \end{scope}
        \end{scope}
    }
    \oplus
    \inlinepic{
        \begin{scope}[scale=0.4]
            \SnakeRight
        \end{scope}
    }
    \arrow[r,"{\begin{pmatrix}
        d_3& -d_4
    \end{pmatrix}}"]
    \&
    \inlinepic{
        \begin{scope}[scale=0.4]
            \PopPimpleLeft
        \end{scope}
    } 
\vspace{0.5cm}\\
    \inlinepic{
        \begin{scope}[scale=0.4]
            \HM
            \begin{scope}[xshift=1cm]
                \Arrow
            \end{scope}
        \end{scope}
    }
    \arrow[r,"{\begin{pmatrix}
        f_1  d_1\\f_2  d_1 \\ d_2
    \end{pmatrix}}"]
    \& 
    \inlinepic{
        \begin{scope}[scale=0.4]
            \HM
            \begin{scope}[xshift=1cm]
                \Arrow
            \end{scope}
        \end{scope}
    }
    \oplus
    \inlinepic{
        \begin{scope}[scale=0.4]
            \HM
            \begin{scope}[xshift=1cm]
                \Arrow
            \end{scope}
        \end{scope}
    }
    \oplus
    \inlinepic{
        \begin{scope}[scale=0.4]
            \SnakeRight
        \end{scope}
    }
    \arrow[r,"{\begin{pmatrix}
        f  d_3 f_1' &f d_3 f_2' & - f d_4
    \end{pmatrix}}"]
    \&
    \inlinepic{
        \begin{scope}[scale=0.4]
            \HM
            \begin{scope}[xshift=1cm]
                \Arrow
            \end{scope}
        \end{scope}}
	\arrow[from=1-1, to=2-1,"Id"]
	\arrow[shift left=9, from=1-2, to=2-2,"Id"]
	\arrow[from=1-2, to=2-2,start anchor={[xshift=-1.3cm]},end anchor={[xshift=-0.2cm]},out=-45, in=90,"f_2"]
    \arrow[,from=1-2, to=2-2,start anchor={[xshift=-0.3cm]},end anchor={[xshift=-1.4cm]},out=-135, in=90,"f_1"]
	\arrow[from=1-3, to=2-3,"f"]
\end{tikzcd}
\]
\vspace{0.5cm}\\
with inverse \vspace{0.5cm}\\
\[      \begin{tikzcd}[row sep=1.5cm, column sep=3cm,ampersand replacement=\&]
    \inlinepic{
        \begin{scope}[scale=0.4]
            \HM
            \begin{scope}[xshift=1cm]
                \Arrow
            \end{scope}
        \end{scope}
    }
    \arrow[r,]
    \& 
    \inlinepic{
        \begin{scope}[scale=0.4]
            \HM
            \begin{scope}[xshift=1cm]
                \Arrow
            \end{scope}
        \end{scope}
    }
    \oplus
    \inlinepic{
        \begin{scope}[scale=0.4]
            \HM
            \begin{scope}[xshift=1cm]
                \Arrow
            \end{scope}
        \end{scope}
    }
    \oplus
    \inlinepic{
        \begin{scope}[scale=0.4]
            \SnakeRight
        \end{scope}
    }
    \arrow[r,]
    \&
    \inlinepic{
        \begin{scope}[scale=0.4]
            \HM
            \begin{scope}[xshift=1cm]
                \Arrow
            \end{scope}
        \end{scope}}
        \vspace{0.5cm}\\
    \inlinepic{
        \begin{scope}[scale=0.4]
            \HB
            \begin{scope}[xshift=1cm]
                \Arrow
            \end{scope}
        \end{scope}
    }
    \arrow[r,]
    \& 
    \inlinepic{
        \begin{scope}[scale=0.4]
            \HOR
            \begin{scope}[xshift=1cm]
                \Arrow
            \end{scope}
        \end{scope}
    }
    \oplus
    \inlinepic{
        \begin{scope}[scale=0.4]
            \SnakeRight
        \end{scope}
    }
    \arrow[r,]
    \&
    \inlinepic{
        \begin{scope}[scale=0.4]
            \PopPimpleLeft
        \end{scope}
    } 
	\arrow[from=1-1, to=2-1,"Id"]
	\arrow[shift left=12, from=1-2, to=2-2,"Id"]
	\arrow[from=1-2, to=2-2,end anchor={[xshift=-1.3cm]},start anchor={[xshift=-0.2cm]},in=45, out=-90,"f_2'"]
    \arrow[,from=1-2, to=2-2,end anchor={[xshift=-0.3cm]},start anchor={[xshift=-1.4cm]},in=135, out=-90,"f_1'"]
	\arrow[from=1-3, to=2-3,"f'"]
\end{tikzcd}
\]
\vspace{0.5cm}\\

By isotopy $f_2 d_1=Id$, and by the tube relations $f d_3 f_1'=-Id$. We can now use the Gaussian elimination algorithm \cite{FastComputations} on the new complex. 

\[      \begin{tikzcd}[row sep=2cm, column sep=3cm,ampersand replacement=\&]
    \inlinepic{
        \begin{scope}[scale=0.4]
            \HM
            \begin{scope}[xshift=1cm]
                \Arrow
            \end{scope}
        \end{scope}
    }
    \arrow[r,"{
        \begin{pmatrix}
            f_1 d_1 \\ Id \\ d_2
        \end{pmatrix}
    }"]
    \& 
    \inlinepic{
        \begin{scope}[scale=0.4]
            \HM
            \begin{scope}[xshift=1cm]
                \Arrow
            \end{scope}
        \end{scope}
    }
    \oplus
    \inlinepic{
        \begin{scope}[scale=0.4]
            \HM
            \begin{scope}[xshift=1cm]
                \Arrow
            \end{scope}
        \end{scope}
    }
    \oplus
    \inlinepic{
        \begin{scope}[scale=0.4]
            \SnakeRight
        \end{scope}
    }
    \arrow[r,"{
        \begin{pmatrix}
            -Id & f d_3 f_2' & -f d_4
        \end{pmatrix}
    }"]
    \&
    \inlinepic{
        \begin{scope}[scale=0.4]
            \HM
            \begin{scope}[xshift=1cm]
                \Arrow
            \end{scope}
        \end{scope}}
        \vspace{0.5cm}\\
 \inlinepic{
        \begin{scope}[scale=0.4]
            \HM
            \begin{scope}[xshift=1cm]
                \Arrow
            \end{scope}
        \end{scope}
    }
    \arrow[r,"{
        \begin{pmatrix}
            0 \\ Id \\ d_2
        \end{pmatrix}
    }"]
    \& 
    \inlinepic{
        \begin{scope}[scale=0.4]
            \HM
            \begin{scope}[xshift=1cm]
                \Arrow
            \end{scope}
        \end{scope}
    }
    \oplus
    \inlinepic{
        \begin{scope}[scale=0.4]
            \HM
            \begin{scope}[xshift=1cm]
                \Arrow
            \end{scope}
        \end{scope}
    }
    \oplus
    \inlinepic{
        \begin{scope}[scale=0.4]
            \SnakeRight
        \end{scope}
    }
    \arrow[r,"{
        \begin{pmatrix}
            -Id & 0 & 0
        \end{pmatrix}
    }"]
    \&
    \inlinepic{
        \begin{scope}[scale=0.4]
            \HM
            \begin{scope}[xshift=1cm]
                \Arrow
            \end{scope}
        \end{scope}}
	\arrow[from=1-1, to=2-1,"Id"]
	\arrow[from=1-2, to=2-2,"{
        \begin{pmatrix}
            Id&-f d_3 f_2'&f d_4\\ 0 & Id &0  \\ 0 & 0& Id
        \end{pmatrix}
    }"]
	\arrow[from=1-3, to=2-3,"Id"]   \vspace{0.5cm}\\
 \inlinepic{
        \begin{scope}[scale=0.4]
            \HM
            \begin{scope}[xshift=1cm]
                \Arrow
            \end{scope}
        \end{scope}
    }
    \arrow[r,"{
        \begin{pmatrix}
            Id \\ 0 \\ 0
        \end{pmatrix}
    }"]
    \& 
    \inlinepic{
        \begin{scope}[scale=0.4]
            \HM
            \begin{scope}[xshift=1cm]
                \Arrow
            \end{scope}
        \end{scope}
    }
    \oplus
    \inlinepic{
        \begin{scope}[scale=0.4]
            \HM
            \begin{scope}[xshift=1cm]
                \Arrow
            \end{scope}
        \end{scope}
    }
    \oplus
    \inlinepic{
        \begin{scope}[scale=0.4]
            \SnakeRight
        \end{scope}
    }
    \arrow[r,"{
        \begin{pmatrix}
            0 & -Id & 0
        \end{pmatrix}
    }"]
    \&
    \inlinepic{
        \begin{scope}[scale=0.4]
            \HM
            \begin{scope}[xshift=1cm]
                \Arrow
            \end{scope}
        \end{scope}}
	\arrow[from=1-1, to=2-1,"Id"]
	\arrow[from=1-3, to=2-3,"Id"]    
    \arrow[from=2-1, to=3-1,"Id"]
	\arrow[from=2-2, to=3-2, "{
        \begin{pmatrix}
            0 & Id & 0 \\ Id & 0 & 0 \\ 0& -d_2 &Id
        \end{pmatrix}
    }"]
	\arrow[from=2-3, to=3-3,"Id"]        
\end{tikzcd}
\]
\vspace{0.5cm}\\
Then we clearly have a homotopy equivalence between the resulting complex and \begin{tikzcd}
    0 \arrow[r]& 
    \inlinepic{
        \begin{scope}[scale=0.25]
            \SnakeRight
        \end{scope}
    }
    \arrow[r]& 0
\end{tikzcd}and therefore between $\llbracket$ \inlinepic{
    \begin{scope}[scale=0.25]
        \HB 
        \draw[white, line width=0.2cm](-1.5,-1.5)to[in=-90,out=90,looseness=1](1.5,1)--(1.5,1.5);
        \draw[->,Blue](-1.5,-1.5)to[in=-90,out=90,looseness=1](1.5,1)--(1.5,1.5);
    \end{scope}
} $\rrbracket$ and    \begin{tikzcd} 0 \arrow[r]& 
    \inlinepic{
        \begin{scope}[scale=0.25]
            \SnakeRight
        \end{scope}
    }
    \arrow[r]& 0
\end{tikzcd} given by the maps $F$ and $F'$ below.
\vspace{0.5cm}\\
\[      
\begin{tikzcd}[row sep=1.5cm, column sep=1.5 cm,ampersand replacement=\&] 
    \llbracket \inlinepic{
    \begin{scope}[scale=0.4]
        \HB 
        \draw[white, line width=0.2cm](-1.5,-1.5)to[in=-90,out=90,looseness=1](1.5,1)--(1.5,1.5);
        \draw[->,Blue](-1.5,-1.5)to[in=-90,out=90,looseness=1](1.5,1)--(1.5,1.5);
    \end{scope}
} \rrbracket:
    \&
    \inlinepic{
        \begin{scope}[scale=0.4]
            \HB
            \begin{scope}[xshift=1cm]
                \Arrow
            \end{scope}
        \end{scope}
    }
    \arrow[r,]
    \& 
    \inlinepic{
        \begin{scope}[scale=0.4]
            \HOR
            \begin{scope}[xshift=1cm]
                \Arrow
            \end{scope}
        \end{scope}
    }\text{\quad}
    \oplus\text{\quad}
    \inlinepic{
        \begin{scope}[scale=0.4]
            \SnakeRight
        \end{scope}
    }
    \arrow[r,]
    \arrow[l,yshift=6,swap,"{\begin{pmatrix}
        f_2 & 0
    \end{pmatrix}}"]
    \&
    \inlinepic{
        \begin{scope}[scale=0.4]
            \PopPimpleLeft
        \end{scope} 
    \arrow[l,yshift=6, swap,"{\begin{pmatrix}
        -f_1' f \\ 0
    \end{pmatrix}}"]
    } \vspace{0.5cm}\\
    \llbracket \inlinepic{
    \begin{scope}[scale=0.4]
        \SnakeRight
    \end{scope}
} \rrbracket[1]:
    \&
    0 \arrow[r]\& \inlinepic{
        \begin{scope}[scale=0.4]
            \SnakeRight
        \end{scope}
    }
    \arrow[r,] \& 0 
	\arrow[equal,from=1-3, to=2-3,end anchor={[xshift=0cm]},start anchor={[xshift=1.1cm]},in=45, out=-90,"Id"]
    \arrow[,from=1-3, to=2-3,end anchor={[xshift=0cm]},start anchor={[xshift=-1.1cm]},in=135, out=-90,"-d_2 f_2"]
    \arrow[,to=1-3, from=2-3,start anchor={[xshift=-0.2cm]},end anchor={[xshift=-1.3cm]},out=135, in=-90,"-f_1 ' f d_4"]
    \arrow[from=1-1, to=2-1,"F",xshift=3]
    \arrow[from=2-1, to=1-1,"F'",xshift=-3]
\end{tikzcd}
\]
\vspace{0.5cm}\\ 

A similar process gives a homotopy equivalence between the complexes $\llbracket$ \inlinepic{
    \begin{scope}[scale=0.25]
        \HT
        \draw[white, line width=0.2cm](-1.5,-1.5)--(-1.5,-1)to[in=-90,out=90,looseness=1](1.5,1.5);
        \draw[->,Blue](-1.5,-1.5)--(-1.5,-1)to[in=-90,out=90,looseness=1](1.5,1.5);
    \end{scope}
} $\rrbracket$ and $\llbracket$ \inlinepic{
    \begin{scope}[scale=0.25]
        \SnakeRight
    \end{scope}
} $\rrbracket$[1]. To make this precise, let us examine the details.
    Let
\begin{center}
    $g_1=$ \quad \inlinepic{\begin{scope}[scale=0.3]\gOne\end{scope}} \hspace{1cm}
    $g_1 '=$ \quad \inlinepic{\begin{scope}[scale=0.3]\gOnePrime\end{scope}} \hspace{1cm}
    $g=$ \quad \inlinepic{\begin{scope}[scale=0.3]\Gi \end{scope}} \vspace{0.5cm}\\
    $g_2=$ \quad \inlinepic{\begin{scope}[scale=0.3]\gTwo\end{scope}} \hspace{1cm}
    $g_2 '=-$ \quad \inlinepic{\begin{scope}[scale=0.3]\gTwoPrime\end{scope}} \hspace{1cm}
    $g'=$ \quad \inlinepic{\begin{scope}[scale=0.3]\GiPrime \end{scope}} \vspace{0.5cm}\\
\end{center}
Then we have an isomorphism
\[\begin{tikzcd}[row sep=1.5cm, column sep=3cm,ampersand replacement=\&]
    \inlinepic{
        \begin{scope}[scale=0.4]
            \HT
            \begin{scope}[xshift=-1cm]
                \Arrow
            \end{scope}
        \end{scope}
    }
    \arrow[r,"{\begin{pmatrix}
        d_1\\d_2
    \end{pmatrix}}"]
    \& 
    \inlinepic{
        \begin{scope}[scale=0.4]
            \HOR
            \begin{scope}[xshift=-1cm]
                \Arrow
            \end{scope}
        \end{scope}
    }
    \oplus
    \inlinepic{
        \begin{scope}[scale=0.4]
            \SnakeRight
        \end{scope}
    }
    \arrow[r,"{\begin{pmatrix}
        -d_3& d_4
    \end{pmatrix}}"]
    \&
    \inlinepic{
        \begin{scope}[scale=0.4]
            \PopPimpleRight
        \end{scope}
    } 
\vspace{0.5cm}\\
    \inlinepic{
        \begin{scope}[scale=0.4]
            \HM
            \begin{scope}[xshift=-1cm]
                \Arrow
            \end{scope}
        \end{scope}
    }
    \arrow[r,"{\begin{pmatrix}
        g_1  d_1\\g_2  d_1 \\ d_2
    \end{pmatrix}}"]
    \& 
    \inlinepic{
        \begin{scope}[scale=0.4]
            \HM
            \begin{scope}[xshift=-1cm]
                \Arrow
            \end{scope}
        \end{scope}
    }
    \oplus
    \inlinepic{
        \begin{scope}[scale=0.4]
            \HM
            \begin{scope}[xshift=-1cm]
                \Arrow
            \end{scope}
        \end{scope}
    }
    \oplus
    \inlinepic{
        \begin{scope}[scale=0.4]
            \SnakeRight
        \end{scope}
    }
    \arrow[r,"{\begin{pmatrix}
        -g  d_3 g_1' &-g d_3 g_2' &  g d_4
    \end{pmatrix}}"]
    \&
    \inlinepic{
        \begin{scope}[scale=0.4]
            \HM
            \begin{scope}[xshift=-1cm]
                \Arrow
            \end{scope}
        \end{scope}}
	\arrow[from=1-1, to=2-1,"Id"]
	\arrow[shift left=9, from=1-2, to=2-2,"Id"]
	\arrow[from=1-2, to=2-2,start anchor={[xshift=-1.3cm]},end anchor={[xshift=-0.2cm]},out=-45, in=90,"g_2"]
    \arrow[,from=1-2, to=2-2,start anchor={[xshift=-0.3cm]},end anchor={[xshift=-1.4cm]},out=-135, in=90,"g_1"]
	\arrow[from=1-3, to=2-3,"g"]
\end{tikzcd}
\]

Notice that to find the homotopy equivalence between $\llbracket$ \inlinepic{
    \begin{scope}[scale=0.25]
        \HB 
        \draw[white, line width=0.2cm](-1.5,-1.5)to[in=-90,out=90,looseness=1](1.5,1)--(1.5,1.5);
        \draw[->,Blue](-1.5,-1.5)to[in=-90,out=90,looseness=1](1.5,1)--(1.5,1.5);
    \end{scope}
} $\rrbracket$ and   $\llbracket$\inlinepic{
        \begin{scope}[scale=0.25]
            \SnakeRight
        \end{scope}
    } $\rrbracket$
earlier we only needed the relations $f_2 d_1=Id$ and $f d_3 f_1'=-Id$. The maps $g_i$ also verify $g_2 d_1= Id$ and $-g d_3 g_1 ' =-Id$ (the sign of the differential $d_3$ is changed).
Therefore we have a homotopy equivalence:
\[      
\begin{tikzcd}[row sep=1.5cm, column sep=1.5 cm,ampersand replacement=\&] 
    \llbracket \inlinepic{
    \begin{scope}[scale=0.4]
        \HT
        \draw[white, line width=0.2cm](-1.5,-1.5)--(-1.5,-1)to[in=-90,out=90,looseness=1](1.5,1.5);
        \draw[->,Blue](-1.5,-1.5)--(-1.5,-1)to[in=-90,out=90,looseness=1](1.5,1.5);
    \end{scope}}
\rrbracket:
    \&
    \inlinepic{
        \begin{scope}[scale=0.4]
            \HT
            \begin{scope}[xshift=-1cm]
                \Arrow
            \end{scope}
        \end{scope}
    }
    \arrow[r,]
    \& 
    \inlinepic{
        \begin{scope}[scale=0.4]
            \HOR
            \begin{scope}[xshift=-1cm]
                \Arrow
            \end{scope}
        \end{scope}
    }\text{\quad}
    \oplus\text{\quad}
    \inlinepic{
        \begin{scope}[scale=0.4]
            \SnakeRight
        \end{scope}
    }
    \arrow[r,]
    \arrow[l,yshift=6,swap,"{\begin{pmatrix}
        g_2 & 0
    \end{pmatrix}}"]
    \&
    \inlinepic{
        \begin{scope}[scale=0.4]
            \PopPimpleRight
        \end{scope} 
    \arrow[l,yshift=6, swap,"{\begin{pmatrix}
        g_1' g \\ 0
    \end{pmatrix}}"]
    } \vspace{0.5cm}\\
    \llbracket \inlinepic{
    \begin{scope}[scale=0.4]
        \SnakeRight
    \end{scope}
} \rrbracket[1]:
    \&
    0 \arrow[r]\& \inlinepic{
        \begin{scope}[scale=0.4]
            \SnakeRight
        \end{scope}
    }
    \arrow[r,] \& 0 
	\arrow[equal,from=1-3, to=2-3,end anchor={[xshift=0cm]},start anchor={[xshift=1.1cm]},in=45, out=-90,"Id"]
    \arrow[,from=1-3, to=2-3,end anchor={[xshift=0cm]},start anchor={[xshift=-1.1cm]},in=135, out=-90,"-d_2 g_2"]
    \arrow[,to=1-3, from=2-3,start anchor={[xshift=-0.2cm]},end anchor={[xshift=-1.3cm]},out=135, in=-90,"-g_1 ' g d_4"]
    \arrow[from=1-1, to=2-1,"G",xshift=3]
    \arrow[from=2-1, to=1-1,"G'",xshift=-3]
\end{tikzcd}
\]
\vspace{0.5cm}\\and the bottom layers of the two complexes $\llbracket$\inlinepic{\RThreeLeft{0.2}}$\rrbracket$  and $\llbracket$\inlinepic{\RThreeRight{0.2}}$\rrbracket$ are homotopic, as they are each homotopic to the same complex   $\llbracket$\inlinepic{
        \begin{scope}[scale=0.2]
            \SnakeRight
        \end{scope}
    } $\rrbracket$[1]. However, this is not enough to prove that  $\llbracket$\inlinepic{\RThreeLeft{0.2}}$\rrbracket$ \quad is homotopic to  \quad $\llbracket$\inlinepic{\RThreeRight{0.2}}$\rrbracket$. 
We will need a few more tools to get there. 
\begin{deff}
Let $\Psi:\left(\mathcal{O}_0^r, d_0\right) \rightarrow\left(\mathcal{O}_1^r, d_1\right)$ be a morphism of complexes. The \textbf{cone} $\Gamma(\Psi)$ of $\Psi$ is the complex with chain spaces $\Gamma^r(\Psi)=\mathcal{O}_0^{r+1} \oplus \mathcal{O}_1^r$ and with differentials $\tilde{d}^r=\left(\begin{array}{cc}-d_0^{r+1} & 0 \\ \Psi^{r+1} & d_1^r\end{array}\right)$. 
\end{deff}
\begin{lemma}(Admitted)\\
     $\llbracket$ 
    \inlinepic{
        \begin{scope}[scale=0.8]
            \draw[->] (-60:0.4)--(120:0.4);
            \fill[white] (0,0) circle (0.1cm);
            \draw[->] (-120:0.4)--(60:0.4);
  
        \end{scope}
    }
    $\rrbracket$ $=$ $\Gamma($
     \inlinepic{
        \begin{scope}[scale=0.8]
        \draw[->,Blue,bend right] (-120:0.4)to(120:0.4);
        \draw[->,Blue, bend left] (-60:0.4)to(60:0.4);
        \end{scope}
    }   
    $\xrightarrow{\text{
     $d=$ \inlinepic{
        \begin{scope}[scale=0.25]
        \saddleUp
        \end{scope}
    }           
    }}$ \inlinepic{ 
    \begin{scope}[scale=0.8]
        \draw[Blue,->] (0,0.2) to[in=-120,out=0] (60:0.4);
        \draw[Blue,->] (0,0.2) to[in=-60,out=180] (120:0.4);
        \draw[Maroon,thick,decoration={ markings, mark=at position 0.7 with {\arrow[scale=0.5]{>>}}},postaction={decorate}] (0,-0.2)--(0,0.2);
        \draw[Blue] (-60:0.4) to[in=0, out=120] (0,-0.2);
        \draw[Blue] (-120:0.4) to[in=180, out=60] (0,-0.2);
    \end{scope}
    }$)[-1]$,\hspace{1cm}
    $\llbracket$ 
    \inlinepic{
        \begin{scope}[scale=0.8]
            \draw[->] (-120:0.4)--(60:0.4);
            \fill[white] (0,0) circle (0.1cm);
            \draw[->] (-60:0.4)--(120:0.4);
  
        \end{scope}
    }
    $\rrbracket$ $=$
     \inlinepic{ 
    \begin{scope}[scale=0.8]
        \draw[Blue,->] (0,0.2) to[in=-120,out=0] (60:0.4);
        \draw[Blue,->] (0,0.2) to[in=-60,out=180] (120:0.4);
        \draw[Maroon,thick,decoration={ markings, mark=at position 0.7 with {\arrow[scale=0.5]{>>}}},postaction={decorate}] (0,-0.2)--(0,0.2);
        \draw[Blue] (-60:0.4) to[in=0, out=120] (0,-0.2);
        \draw[Blue] (-120:0.4) to[in=180, out=60] (0,-0.2);
    \end{scope}
    }
    $\xrightarrow{\text{
     $d=$ \inlinepic{
        \begin{scope}[scale=0.25]
        \saddleDown
        \end{scope}
    }           
    }}$
     \inlinepic{
        \begin{scope}[scale=0.8]
        \draw[->,Blue,bend right] (-120:0.4)to(120:0.4);
        \draw[->,Blue, bend left] (-60:0.4)to(60:0.4);
        \end{scope}
    }   .\\
\end{lemma}
A result of this lemma is that we can write $\llbracket \inlinepic{\RThreeLeft{0.25}} \rrbracket = \Gamma ( \llbracket \inlinepic{\TopLayer{0.3}} \rrbracket \xrightarrow{\text{\hspace{0.2cm}}d\text{\hspace{0.2cm}}}
\llbracket \inlinepic{
    \begin{scope}[scale=0.25]
        \HB 
        \draw[white, line width=0.2cm](-1.5,-1.5)to[in=-90,out=90,looseness=1](1.5,1)--(1.5,1.5);
        \draw[->,Blue](-1.5,-1.5)to[in=-90,out=90,looseness=1](1.5,1)--(1.5,1.5);
    \end{scope}
} \rrbracket )[-1]$ \\and
$\llbracket \inlinepic{\RThreeRight{0.25}} \rrbracket = \Gamma ( \llbracket \inlinepic{\TopLayer{0.3}} \rrbracket \xrightarrow{\text{\hspace{0.2cm}}d\text{\hspace{0.2cm}}}
\llbracket \inlinepic{
    \begin{scope}[scale=0.25]
        \HT
        \draw[white, line width=0.2cm](-1.5,-1.5)--(-1.5,-1)to[in=-90,out=90,looseness=1](1.5,1.5);
        \draw[->,Blue](-1.5,-1.5)--(-1.5,-1)to[in=-90,out=90,looseness=1](1.5,1.5);
    \end{scope}
} \rrbracket )[-1]$.\\
\begin{deff}
    A morphism of complexes $F: \Omega_a \rightarrow \Omega_b$ is said to be a strong deformation retract if there is a morphism $G: \Omega_b \rightarrow \Omega_a$ and homotopy maps $h$ from $\Omega_a$ to itself so that $F G=I, I-G F=d h+h d$ and $h G=0$. In this case we say that $G$ is the inclusion in a strong deformation retract. 
\end{deff}
The use of strong homotopy retracts is that they behave well under cone constructions. 
\begin{lemma}\label{Cone}
    Let $(\mathcal{O}_0,d_0)$, $(\mathcal{O},d)$ and $(\mathcal{O}',d')$ be three chain complexes. Suppose that there is a deformation retract $F:\mathcal{O}\longrightarrow \mathcal{O}'$ with inclusion $G:\mathcal{O}'\longrightarrow \mathcal{O}$. 
    \begin{enumerate}
        \item Let $\Psi:\mathcal{O}_0 \longrightarrow \mathcal{O}$ be a morphism of chain complexes, then $\Gamma(F\circ \Psi )\simeq \Gamma(\Psi)$.\\
        \item  Similarly, let $\Psi':\mathcal{O} \longrightarrow \mathcal{O}_0$ be a morphism of chain complexes, then $\Gamma( \Psi \circ G )\simeq \Gamma(\Psi)$.\\
    \end{enumerate}

\end{lemma}
    \begin{proof}
            Let $h:\mathcal{O}_*\rightarrow \mathcal{O}_{*-1}$ be the homotopy equivalence such that $Id-GF=dh+hd$, $FG=Id$ and $hG=0$. We describe the homotopy equivalence in each case.
        \begin{enumerate}
        \item    
        \[
        \begin{tikzcd}[ampersand replacement=\&, column sep= 2.3cm, row sep= 2.3cm]
            \Gamma(F\Psi): 
            \& 
            \begin{pmatrix}
                \mathcal{O}_{0,i+1}\\
                \mathcal{O}_i'
            \end{pmatrix} 
            \&
            \begin{pmatrix}
                \mathcal{O}_{0,i+2}\\
                \mathcal{O}_{i+1}'
            \end{pmatrix}             
            \\
            \Gamma(\Psi): 
            \& 
            \begin{pmatrix}
                \mathcal{O}_{0,i+1}\\
                \mathcal{O}_i
            \end{pmatrix} 
            \&
            \begin{pmatrix}
                \mathcal{O}_{0,i+2}\\
                \mathcal{O}_{i+1}
            \end{pmatrix}   
            \arrow[from=1-2,to=1-3,"
            {
                \begin{pmatrix}
                    -d_0 & 0\\ F\Psi & d'
                \end{pmatrix}
            }
            "]    
            \arrow[from=2-2,to=2-3,"
            {
                \begin{pmatrix}
                    -d_0 & 0\\ \Psi & d
                \end{pmatrix}
            }
            "]         
            \arrow[yshift=-0.2cm,from=2-3,to=2-2,"
            {
                \begin{pmatrix}
                    0 & 0\\ 0 & -h + GFh
                \end{pmatrix}
            }
            "]           
            \arrow[swap,xshift=-0.2cm,from=1-2,to=2-2,"
            {
                \tilde{G}=
                \begin{pmatrix}
                    Id & 0\\ -h\Psi & G
                \end{pmatrix}
            }
            "]    
            \arrow[swap,from=2-2,to=1-2,"
            {
                \tilde{F}=
                \begin{pmatrix}
                    Id & 0\\ Fh\Psi & F
                \end{pmatrix}
            }
            "] 
            \arrow[swap,xshift=-0.2cm,from=1-3,to=2-3,"
            {
                \tilde{G}
            }
            "]    
            \arrow[swap,from=2-3,to=1-3,"
            {
                \tilde{F}
            }
            "] 
        \end{tikzcd}
        \]
        \item  
        \[
        \begin{tikzcd}[ampersand replacement=\&, column sep= 2.3cm, row sep= 2.3cm]
            \Gamma(\Psi'G): 
            \& 
            \begin{pmatrix}
                \mathcal{O}'_{i+1}\\
                \mathcal{O}_{0,i}
            \end{pmatrix} 
            \&
            \begin{pmatrix}
                \mathcal{O}'_{i+2}\\
                \mathcal{O}_{0,i+1}'
            \end{pmatrix}             
            \\
            \Gamma(\Psi'): 
            \& 
            \begin{pmatrix}
                \mathcal{O}_{i+1}\\
                \mathcal{O}_{0,i}
            \end{pmatrix} 
            \&
            \begin{pmatrix}
                \mathcal{O}_{i+2}\\
                \mathcal{O}_{0,i+1}
            \end{pmatrix}   
            \arrow[from=1-2,to=1-3,"
            {
                \begin{pmatrix}
                    -d' & 0\\ \Psi' G& d_0
                \end{pmatrix}
            }
            "]    
            \arrow[from=2-2,to=2-3,"
            {
                \begin{pmatrix}
                    -d & 0\\ \Psi' & d_0
                \end{pmatrix}
            }
            "]         
            \arrow[yshift=-0.2cm,from=2-3,to=2-2,"
            {
                \begin{pmatrix}
                    -h & 0\\ 0 & 0
                \end{pmatrix}
            }
            "]           
            \arrow[swap,xshift=-0.2cm,from=1-2,to=2-2,"
            {
                \tilde{G}=
                \begin{pmatrix}
                    Id & 0\\ 0 & G
                \end{pmatrix}
            }
            "]    
            \arrow[swap,from=2-2,to=1-2,"
            {
                \tilde{F}=
                \begin{pmatrix}
                    Id & 0\\ \Psi' h & F
                \end{pmatrix}
            }
            "] 
            \arrow[swap,xshift=-0.2cm,from=1-3,to=2-3,"
            {
                \tilde{G}
            }
            "]    
            \arrow[swap,from=2-3,to=1-3,"
            {
                \tilde{F}
            }
            "] 
        \end{tikzcd}
        \]
        \end{enumerate}

    \end{proof}

Now, we know $\llbracket \inlinepic{\RThreeLeft{0.25}} \rrbracket = \Gamma ( \llbracket \inlinepic{\TopLayer{0.3}} \rrbracket \xrightarrow{\text{\hspace{0.2cm}}d_L\text{\hspace{0.2cm}}}
\llbracket \inlinepic{
    \begin{scope}[scale=0.25]
        \HB 
        \draw[white, line width=0.2cm](-1.5,-1.5)to[in=-90,out=90,looseness=1](1.5,1)--(1.5,1.5);
        \draw[->,Blue](-1.5,-1.5)to[in=-90,out=90,looseness=1](1.5,1)--(1.5,1.5);
    \end{scope}
} \rrbracket )[-1]$ and
$\llbracket \inlinepic{\RThreeRight{0.25}} \rrbracket = \Gamma ( \llbracket \inlinepic{\TopLayer{0.3}} \rrbracket \xrightarrow{\text{\hspace{0.2cm}}d_R\text{\hspace{0.2cm}}}
\llbracket \inlinepic{
    \begin{scope}[scale=0.25]
        \HT
        \draw[white, line width=0.2cm](-1.5,-1.5)--(-1.5,-1)to[in=-90,out=90,looseness=1](1.5,1.5);
        \draw[->,Blue](-1.5,-1.5)--(-1.5,-1)to[in=-90,out=90,looseness=1](1.5,1.5);
    \end{scope}
} \rrbracket )[-1]$.\\  We have homotopy equivalences  $F:\llbracket \inlinepic{
    \begin{scope}[scale=0.25]
        \HB 
        \draw[white, line width=0.2cm](-1.5,-1.5)to[in=-90,out=90,looseness=1](1.5,1)--(1.5,1.5);
        \draw[->,Blue](-1.5,-1.5)to[in=-90,out=90,looseness=1](1.5,1)--(1.5,1.5);
    \end{scope}
} \rrbracket \longrightarrow \llbracket$\inlinepic{
        \begin{scope}[scale=0.3]
            \SnakeRight
        \end{scope}
    } $\rrbracket$[1] and  $G:\llbracket \inlinepic{
    \begin{scope}[scale=0.25]
        \HT
        \draw[white, line width=0.2cm](-1.5,-1.5)--(-1.5,-1)to[in=-90,out=90,looseness=1](1.5,1.5);
        \draw[->,Blue](-1.5,-1.5)--(-1.5,-1)to[in=-90,out=90,looseness=1](1.5,1.5);
    \end{scope}
} \rrbracket \longrightarrow \llbracket$\inlinepic{
        \begin{scope}[scale=0.3]
            \SnakeRight
        \end{scope}
    } $\rrbracket$[1]. If $F$ and $G$ are strong deformation retracts, then we can hope the maps $F\circ d_L$ and $G\circ d_R$ are homotopic. That way $\Gamma(d_L)\simeq \Gamma(F\circ d_L) \simeq \Gamma(G \circ d_L) \simeq \Gamma(d_L)$ and we have the invariance by the third Reidemeister move. \\
    Checking that $F$ and $G$ are strong deformation retracts is straight forward. The map $F\circ d_L$ is
  \[      
\begin{tikzcd}[row sep=1.5cm, column sep=1.5 cm,ampersand replacement=\&] 
    \inlinepic{
        \begin{scope}[scale=0.4]
            \Arrow
            \begin{scope}[xshift=0.75cm]
                \Arrow
            \end{scope}
            \begin{scope}[xshift=-0.75cm]
                \Arrow
            \end{scope}
        \end{scope}
    } \arrow[d,"{d_{L,1}}"]
    \arrow[r]
    \&
    \inlinepic{
        \begin{scope}[scale=0.4]
            \HM
            \begin{scope}[xshift=1cm]
                \Arrow
            \end{scope}
        \end{scope}
    }\text{\quad}
    \oplus
    \text{\quad}
    \inlinepic{
        \begin{scope}[scale=0.4]
            \HM
            \begin{scope}[xshift=-1cm]
                \Arrow
            \end{scope}
        \end{scope}
    }
    \arrow[d,"{d_{L,2}}",xshift=-1.1cm] 
    \arrow[d,"{d_{L,3}}",xshift=1.1cm]
    \arrow[r]
    \&
    \inlinepic{
        \begin{scope}[scale=0.4]
            \SnakeLeft
        \end{scope}
    }
    \arrow[d,"{d_{L,4}}"]
    \\
    \inlinepic{
        \begin{scope}[scale=0.4]
            \HB
            \begin{scope}[xshift=1cm]
                \Arrow
            \end{scope}
        \end{scope}
    }
    \arrow[r,"
    {
        \begin{pmatrix}
            d_1\\d_2
        \end{pmatrix}
    }
    "]
    \& 
    \inlinepic{
        \begin{scope}[scale=0.4]
            \HOR
            \begin{scope}[xshift=1cm]
                \Arrow
            \end{scope}
        \end{scope}
    }\text{\quad}
    \oplus\text{\quad}
    \inlinepic{
        \begin{scope}[scale=0.4]
            \SnakeRight
        \end{scope}
    }
    \arrow[r,"
    {
        \begin{pmatrix}
            d_3\\-d_4
        \end{pmatrix}
    }
    "]
    \&
    \inlinepic{
        \begin{scope}[scale=0.4]
            \PopPimpleLeft
        \end{scope} 

    } \vspace{0.5cm}\\
    0 \arrow[r]\& \inlinepic{
        \begin{scope}[scale=0.4]
            \SnakeRight
        \end{scope}
    }
    \arrow[r,] \& 0 
	\arrow[equal,from=2-2, to=3-2,end anchor={[xshift=0cm]},start anchor={[xshift=1.1cm]},in=45, out=-90,"Id"]
    \arrow[,from=2-2, to=3-2,end anchor={[xshift=0cm]},start anchor={[xshift=-1.1cm]},in=135, out=-90,"-d_2 f_2"]
\end{tikzcd}
\]
Here $d_{L,2}=d_1$, and we know that $ f_2 d_1=Id$. Therefore $F\circ d_L$   is
\[\begin{tikzcd}[row sep=1.5cm, column sep=1.5 cm,ampersand replacement=\&] 
        \inlinepic{
        \begin{scope}[scale=0.4]
            \Arrow
            \begin{scope}[xshift=0.75cm]
                \Arrow
            \end{scope}
            \begin{scope}[xshift=-0.75cm]
                \Arrow
            \end{scope}
        \end{scope}
    } 
    \arrow[r]
    \&
    \inlinepic{
        \begin{scope}[scale=0.4]
            \HM
            \begin{scope}[xshift=1cm]
                \Arrow
            \end{scope}
        \end{scope}
    }\text{\quad}
    \oplus
    \text{\quad}
    \inlinepic{
        \begin{scope}[scale=0.4]
            \HM
            \begin{scope}[xshift=-1cm]
                \Arrow
            \end{scope}
        \end{scope}
    }
    \arrow[r]
    \&
    \inlinepic{
        \begin{scope}[scale=0.4]
            \SnakeLeft
        \end{scope}
    }\\
    0 \arrow[r]\& \inlinepic{
        \begin{scope}[scale=0.4]
            \SnakeRight
        \end{scope}
    }
    \arrow[r,] \& 0 
	\arrow[from=1-2, to=2-2,end anchor={[xshift=0cm]},start anchor={[xshift=1.1cm]},in=45, out=-90,"d_{L,3}"]
    \arrow[,from=1-2, to=2-2,end anchor={[xshift=0cm]},start anchor={[xshift=-1.1cm]},in=135, out=-90,"-d_2"]
\end{tikzcd}
\]

For $G\circ d_R$ we have

  \[      
\begin{tikzcd}[row sep=1.5cm, column sep=1.5 cm,ampersand replacement=\&] 
    \inlinepic{
        \begin{scope}[scale=0.4]
            \Arrow
            \begin{scope}[xshift=0.75cm]
                \Arrow
            \end{scope}
            \begin{scope}[xshift=-0.75cm]
                \Arrow
            \end{scope}
        \end{scope}
    } \arrow[d,"{d_{R,1}}"]
    \arrow[r]
    \&
    \inlinepic{
        \begin{scope}[scale=0.4]
            \HM
            \begin{scope}[xshift=1cm]
                \Arrow
            \end{scope}
        \end{scope}
    }\text{\quad}
    \oplus
    \text{\quad}
    \inlinepic{
        \begin{scope}[scale=0.4]
            \HM
            \begin{scope}[xshift=-1cm]
                \Arrow
            \end{scope}
        \end{scope}
    }
    \arrow[d,"{d_{R,2}}",xshift=-1.1cm] 
    \arrow[d,"{d_{R,3}}",xshift=1.1cm]
    \arrow[r]
    \&
    \inlinepic{
        \begin{scope}[scale=0.4]
            \SnakeLeft
        \end{scope}
    }
    \arrow[d,"{d_{R,4}}"]
    \\
    \inlinepic{
        \begin{scope}[scale=0.4]
            \HB
            \begin{scope}[xshift=-1cm]
                \Arrow
            \end{scope}
        \end{scope}
    }
    \arrow[r,"
    {
        \begin{pmatrix}
            d_1'\\d_2'
        \end{pmatrix}
    }
    "]
    \&   
    \inlinepic{
        \begin{scope}[scale=0.4]
            \SnakeRight
        \end{scope}
    }
    \text{\quad}
    \oplus\text{\quad}
    \inlinepic{
        \begin{scope}[scale=0.4]
            \HOR
            \begin{scope}[xshift=-1cm]
                \Arrow
            \end{scope}
        \end{scope}
    }
    \arrow[r,"
    {
        \begin{pmatrix}
            d_4'\\-d_3'
        \end{pmatrix}
    }
    "]
    \&
    \inlinepic{
        \begin{scope}[scale=0.4]
            \PopPimpleRight
        \end{scope} 

    } \vspace{0.5cm}\\
    0 \arrow[r]\& \inlinepic{
        \begin{scope}[scale=0.4]
            \SnakeRight
        \end{scope}
    }
    \arrow[r,] \& 0 
	\arrow[from=2-2, to=3-2,end anchor={[xshift=0cm]},start anchor={[xshift=1.1cm]},in=45, out=-90,"-d_2'g_2"]
    \arrow[equal,from=2-2, to=3-2,end anchor={[xshift=0cm]},start anchor={[xshift=-1.1cm]},in=135, out=-90,"Id"]
\end{tikzcd}
\]
and $d_{R,3}=d_1'$, $g_2d_1'=Id$. Then $d_2=d_{L,3}$ and $d_{R,2}=d_2$
\[\begin{tikzcd}[row sep=1.5cm, column sep=1.5 cm,ampersand replacement=\&] 
        \inlinepic{
        \begin{scope}[scale=0.4]
            \Arrow
            \begin{scope}[xshift=0.75cm]
                \Arrow
            \end{scope}
            \begin{scope}[xshift=-0.75cm]
                \Arrow
            \end{scope}
        \end{scope}
    } 
    \arrow[r]
    \&
    \inlinepic{
        \begin{scope}[scale=0.4]
            \HM
            \begin{scope}[xshift=1cm]
                \Arrow
            \end{scope}
        \end{scope}
    }\text{\quad}
    \oplus
    \text{\quad}
    \inlinepic{
        \begin{scope}[scale=0.4]
            \HM
            \begin{scope}[xshift=-1cm]
                \Arrow
            \end{scope}
        \end{scope}
    }
    \arrow[r]
    \&
    \inlinepic{
        \begin{scope}[scale=0.4]
            \SnakeLeft
        \end{scope}
    }\\
    0 \arrow[r]\& \inlinepic{
        \begin{scope}[scale=0.4]
            \SnakeRight
        \end{scope}
    }
    \arrow[r,] \& 0 
	\arrow[from=1-2, to=2-2,end anchor={[xshift=0cm]},start anchor={[xshift=1.1cm]},in=45, out=-90,"-d_{L,3}"]
    \arrow[,from=1-2, to=2-2,end anchor={[xshift=0cm]},start anchor={[xshift=-1.1cm]},in=135, out=-90,"d_2"]
\end{tikzcd}
\]
Therefore $F\circ d_L=-G\circ d_R$, so $\Gamma(F\circ d_L)=\Gamma(-G\circ d_R)$. The maps $F$ and $-G$ are strong deformation retracts and that gives us $\Gamma( d_L)=\Gamma( d_R)$, i.e. invariance by the third Reidemeister move. 
\\ Using the maps defined in the proof of Lemma \ref{Cone} we can write explicitely the homotopy equivalence corresponding to $R3$. The signs of differentials are taken as in Figure \ref{fig:RThreeLeft} and \ref{fig:RThreeRight}.
\[
\begin{tikzcd}[column sep=large,row sep=3cm,ampersand replacement=\&]
    \llbracket \inlinepic{\RThreeLeft{0.25}}\rrbracket: \text{\quad}
    \inlinepic{
        \begin{scope}[scale=0.3]
            \Arrow
            \begin{scope}[xshift=0.75cm]
                \Arrow
            \end{scope}
            \begin{scope}[xshift=-0.75cm]
                \Arrow
            \end{scope}
        \end{scope}
    }
    \arrow[r]
    \&
    \inlinepic{
        \begin{scope}[scale=0.3]
            \HT
            \begin{scope}[xshift=1cm]
                \Arrow
            \end{scope}
        \end{scope}
    }
    \oplus
    \inlinepic{
        \begin{scope}[scale=0.3]
            \HM
            \begin{scope}[xshift=-1cm]
                \Arrow
            \end{scope}
        \end{scope}
    }    
    \oplus
    \inlinepic{
        \begin{scope}[scale=0.3]
            \HB
            \begin{scope}[xshift=1cm]
                \Arrow
            \end{scope}
        \end{scope}
    }
    \arrow[r]
    \&
    \inlinepic{
        \begin{scope}[scale=0.3]
            \SnakeLeft
        \end{scope}
    }\text{\:}
    \oplus\text{\:}
    \inlinepic{
        \begin{scope}[scale=0.3]
            \HOR
            \begin{scope}[xshift=1cm]
                \Arrow
            \end{scope}
        \end{scope}
    }   \text{\:} 
    \oplus\text{\:}
    \inlinepic{
        \begin{scope}[scale=0.3]
            \SnakeRight
        \end{scope}
    }
    \arrow[r]
    \&
    \inlinepic{
        \begin{scope}[scale=0.3]
            \PopPimpleLeft
        \end{scope}
    } 
    \\
    \llbracket \inlinepic{\RThreeRight{0.25}}\rrbracket: \text{\quad}
    \inlinepic{
        \begin{scope}[scale=0.3]
            \Arrow
            \begin{scope}[xshift=0.75cm]
                \Arrow
            \end{scope}
            \begin{scope}[xshift=-0.75cm]
                \Arrow
            \end{scope}
        \end{scope}
    }
    \arrow[r]
    \&
    \inlinepic{
        \begin{scope}[scale=0.3]
            \HM
            \begin{scope}[xshift=1cm]
                \Arrow
            \end{scope}
        \end{scope}
    }\text{\:}
    \oplus\text{\:}
    \inlinepic{
        \begin{scope}[scale=0.3]
            \HB
            \begin{scope}[xshift=-1cm]
                \Arrow
            \end{scope}
        \end{scope}
    }    \text{\:}
    \oplus\text{\:}
    \inlinepic{
        \begin{scope}[scale=0.3]
            \HT
            \begin{scope}[xshift=-1cm]
                \Arrow
            \end{scope}
        \end{scope}
    }
    \arrow[r]
    \&
    \inlinepic{
        \begin{scope}[scale=0.3]
            \SnakeLeft
        \end{scope}
    }    
    \oplus
    \inlinepic{
        \begin{scope}[scale=0.3]
            \SnakeRight
        \end{scope}
    }\text{\:}
    \oplus\text{\:}
    \inlinepic{
        \begin{scope}[scale=0.3]
            \HOR
            \begin{scope}[xshift=-1cm]
                \Arrow
            \end{scope}
        \end{scope}
    }
    \arrow[r]
    \&
    \inlinepic{
        \begin{scope}[scale=0.3]
            \PopPimpleRight
        \end{scope}
    } 
    \arrow[dash,from=1-4,to=2-4,"0"]
    \arrow[from=1-1,to=2-1,"-Id "{scale=0.9},xshift=0.75cm, swap]
    \arrow[to=1-1,from=2-1,"-Id"{scale=0.9}, xshift=0.85cm,swap]
    \arrow[from=1-4, to=1-3,out=120, in=60,looseness=0.5, end anchor={[xshift=-0.4cm]},start anchor={[xshift=0.4cm]},"f_1'f",swap]
    \arrow[from=1-3, to=1-2,out=120, in=60,looseness=0.5, end anchor={[xshift=0.8cm]},start anchor={[xshift=0cm]},"-f_2",swap]
    \arrow[from=2-4, to=2-3,out=-120, in=-60,looseness=0.5, end anchor={[xshift=1.2cm]},start anchor={[xshift=0.4cm]},"-g_1'g"]
    \arrow[from=2-3, to=2-2,out=-120, in=-60,looseness=0.5, end anchor={[xshift=0.8cm]},start anchor={[xshift=1.4cm]},"-g_2"]
    \arrow[from=1-2,to=2-2,xshift=-1.3cm,"-Id"{scale=0.9, pos=0.9},swap]
    \arrow[from=2-2,to=1-2,xshift=-1.2cm,"-Id"{scale=0.9, pos=0.9},swap]
    \arrow[from=1-2,to=2-2,xshift=-0cm,"{Id}"{scale=0.9, pos=0.9},swap]
    \arrow[from=2-2,to=1-2,xshift=0.1cm,"-Id"{scale=0.9, pos=0.2},swap]
    \arrow[dash,white, line width = 0.2cm,from=2-2,to=1-2,start anchor={[xshift=-1cm]},end anchor={[xshift=1.2cm]},out=90,in=-90]
    \arrow[from=2-2,to=1-2,start anchor={[xshift=-1cm]},end anchor={[xshift=1.2cm]},out=90,in=-90,"{Id}"{scale=0.9, pos=0.9},swap]
    \arrow[dash,white, line width = 0.2cm,shorten <=0.3cm,from=1-2,to=2-2,start anchor={[xshift=0.2cm]},end anchor={[xshift=1.2cm]},out=-90,in=90,"{Id}"{scale=0.9, pos=0.9}]
    \arrow[from=1-2,to=2-2,start anchor={[xshift=0.2cm]},end anchor={[xshift=1.2cm]},out=-90,in=90,"{Id}"{scale=0.9, pos=0.9}]
    \arrow[from=1-3, to=2-3,"A",xshift=-0.25cm,swap]
    \arrow[from=2-3, to=1-3,"A'",xshift=0.25cm,swap]
\end{tikzcd}
\]\textcolor{white}{.}\vspace{0.5cm}\\
\begin{center}
    
$A=\begin{pmatrix}
    -Id & 0 & 0 \vspace{0.2cm}\\
    0 & 
    \inlinepic{
        \begin{scope}[scale=0.3]
            \MapThree
        \end{scope}
    }
    &-Id \vspace{0.2cm}\\-\:
    \inlinepic{
        \begin{scope}[scale=0.3]
            \MapOne
        \end{scope}
    }
    &\quad-\:
    \inlinepic{
        \begin{scope}[scale=0.3]
            \MapFour
        \end{scope}
    }
    &\quad
    \inlinepic{
        \begin{scope}[scale=0.3]
            \MapTwo
        \end{scope}
    }

\end{pmatrix}$\vspace{1cm}
$A'=\begin{pmatrix}
    -Id & 0 & 0 \vspace{0.2cm}\\
    \inlinepic{
        \begin{scope}[scale=0.3]
            \MapFive
        \end{scope}
    }\quad &
    \inlinepic{
        \begin{scope}[scale=0.3]
            \MapSix
        \end{scope}
    }\quad &-\:
    \inlinepic{
        \begin{scope}[scale=0.3]
            \MapSeven
        \end{scope}
    }\vspace{0.2cm}\\0&
    -Id&\quad \inlinepic{
        \begin{scope}[scale=0.3]
            \MapEight
        \end{scope}
    }

\end{pmatrix}$\vspace{0.5cm}\\

\end{center}

$f_2=
\inlinepic{
    \begin{scope}[scale=0.25]
        \fTwo
    \end{scope}
}
$\: ,\qquad 
$g_2=
\inlinepic{
    \begin{scope}[scale=0.25]
        \gTwo
    \end{scope}
}
$\: ,\qquad 
$f_1'f=
\inlinepic{
    \begin{scope}[scale=0.25]
        \MapNine
    \end{scope}
}
$\: and \qquad 
$g_1'g=
\inlinepic{
    \begin{scope}[scale=0.25]
        \MapTen
    \end{scope}
}
$
\vspace{0.5cm}\\

    \item $R3-$: We have:\vspace{0.5cm}\\
 $\llbracket$\inlinepic{
\begin{tikzpicture}[
    every braid/.style={->},
    braid/.cd,
    height=1cm,
    gap=.2,
]
    \pic[scale=0.25] at (0,0) {braid={s_1 s_2 s_1^{-1}}};
\end{tikzpicture}
 }$\: \rrbracket$ $\simeq$ 
 $\Gamma ( \: \llbracket \inlinepic{
    \begin{scope}[scale=0.25]
        \HB 
        \draw[white, line width=0.2cm](-1.5,-1.5)to[in=-90,out=90,looseness=1](1.5,1)--(1.5,1.5);
        \draw[->,Blue](-1.5,-1.5)to[in=-90,out=90,looseness=1](1.5,1)--(1.5,1.5);
    \end{scope}} \rrbracket \xrightarrow{\hspace{0.5cm}d\hspace{0.5cm}} \llbracket \inlinepic{\TopLayer{0.3}} \rrbracket \: )$
\vspace{0.5cm}\\
 $\llbracket$\inlinepic{
\begin{tikzpicture}[
    every braid/.style={->},
    braid/.cd,
    height=1cm,
    gap=.2,
]
    \pic[scale=0.25] at (0,0) {braid={s_2^{-1} s_1 s_2}};
\end{tikzpicture}
 }$\: \rrbracket$ $\simeq$ $\Gamma (\:  \llbracket \inlinepic{
    \begin{scope}[scale=0.25]
        \HT
        \draw[white, line width=0.2cm](-1.5,-1.5)--(-1.5,-1)to[in=-90,out=90,looseness=1](1.5,1.5);
        \draw[->,Blue](-1.5,-1.5)--(-1.5,-1)to[in=-90,out=90,looseness=1](1.5,1.5);
    \end{scope}} \rrbracket \xrightarrow{\hspace{0.5cm}d\hspace{0.5cm}} \llbracket \inlinepic{\TopLayer{0.3}} \rrbracket \: )$. \vspace{0.5cm}\\
The inclusions of deformations retracts  $F':\llbracket
\inlinepic{
    \begin{scope}[scale=0.25]
        \SnakeRight
    \end{scope}
} \rrbracket \longrightarrow
\llbracket \inlinepic{
    \begin{scope}[scale=0.25]
        \HB 
        \draw[white, line width=0.2cm](-1.5,-1.5)to[in=-90,out=90,looseness=1](1.5,1)--(1.5,1.5);
        \draw[->,Blue](-1.5,-1.5)to[in=-90,out=90,looseness=1](1.5,1)--(1.5,1.5);
    \end{scope}} \rrbracket$  and $G':\llbracket
\inlinepic{
    \begin{scope}[scale=0.25]
        \SnakeRight
    \end{scope}
} \rrbracket \longrightarrow
\llbracket \inlinepic{
    \begin{scope}[scale=0.25]
         \HT
        \draw[white, line width=0.2cm](-1.5,-1.5)--(-1.5,-1)to[in=-90,out=90,looseness=1](1.5,1.5);
        \draw[->,Blue](-1.5,-1.5)--(-1.5,-1)to[in=-90,out=90,looseness=1](1.5,1.5);
    \end{scope}} \rrbracket$
    satisfy \\ $dF'=-dG'$. Invariance then follows from Lemma \ref{Cone} as before. \\
    The homotopy equivalence maps for $R3-$ are detailed in the diagram below. 

\[
\begin{tikzcd}[column sep=1.2cm,row sep=2cm,ampersand replacement=\&]
    \llbracket \inlinepic{\begin{tikzpicture}[
    every braid/.style={->},
    braid/.cd,
    height=1cm,
    gap=.2,
]
    \pic[scale=0.25] at (0,0) {braid={s_1 s_2 s_1^{-1}}};
\end{tikzpicture}}\rrbracket: \text{\quad}
    \inlinepic{
        \begin{scope}[scale=0.3]
            \HB
            \begin{scope}[xshift=1cm]
                \Arrow
            \end{scope}
        \end{scope}
    }
    \arrow[r,"{\tiny \begin{pmatrix}
        -\\-\\+
    \end{pmatrix}}"]
    \&
    \inlinepic{
        \begin{scope}[scale=0.3]
            \HOR
            \begin{scope}[xshift=1cm]
                \Arrow
            \end{scope}
        \end{scope}
    }   \text{\:} 
    \oplus\text{\:}
    \inlinepic{
        \begin{scope}[scale=0.3]
            \SnakeRight
        \end{scope}
    }
    \: \oplus \:
    \inlinepic{
        \begin{scope}[scale=0.3]
            \Arrow
            \begin{scope}[xshift=0.75cm]
                \Arrow
            \end{scope}
            \begin{scope}[xshift=-0.75cm]
                \Arrow
            \end{scope}
        \end{scope}
    }
    \arrow[rr,"{\tiny \begin{pmatrix}
        -&+&0\\+&0&+\\0&-&+
    \end{pmatrix}}"]
    \& \&
    \inlinepic{
        \begin{scope}[scale=0.3]
            \PopPimpleLeft
        \end{scope}
    } 
    \oplus\text{\:}
    \inlinepic{
        \begin{scope}[scale=0.3]
            \HT
            \begin{scope}[xshift=1cm]
                \Arrow
            \end{scope}
        \end{scope}
    }
    \oplus
    \inlinepic{
        \begin{scope}[scale=0.3]
            \HM
            \begin{scope}[xshift=-1cm]
                \Arrow
            \end{scope}
        \end{scope}
    }    
    \arrow[r,"{\tiny \begin{pmatrix}
        +&+&-
    \end{pmatrix}}"]
    \&
    \inlinepic{
        \begin{scope}[scale=0.3]
            \SnakeLeft
        \end{scope}
    }\text{\:}
    \\
    \llbracket \inlinepic{\begin{tikzpicture}[
    every braid/.style={->},
    braid/.cd,
    height=1cm,
    gap=.2,
]
    \pic[scale=0.25] at (0,0) {braid={s_2^{-1} s_1 s_2}};
\end{tikzpicture}}\rrbracket: \text{\quad}    \inlinepic{
        \begin{scope}[scale=0.3]
            \HT
            \begin{scope}[xshift=-1cm]
                \Arrow
            \end{scope}
        \end{scope}
    }
    \arrow[r,"{\tiny \begin{pmatrix}
        -\\-\\+
    \end{pmatrix}}"]
    \&
    \inlinepic{
        \begin{scope}[scale=0.3]
            \SnakeRight
        \end{scope}
    }\:
    \oplus\text{\:}
    \inlinepic{
        \begin{scope}[scale=0.3]
            \HOR
            \begin{scope}[xshift=-1cm]
                \Arrow
            \end{scope}
        \end{scope}
    } 
    \: \oplus \:
    \inlinepic{
        \begin{scope}[scale=0.3]
            \Arrow
            \begin{scope}[xshift=0.75cm]
                \Arrow
            \end{scope}
            \begin{scope}[xshift=-0.75cm]
                \Arrow
            \end{scope}
        \end{scope}
    }
    \arrow[rr,"{\tiny \begin{pmatrix}
        -&+&0\\+&0&+\\0&-&+
    \end{pmatrix}}"]\&
    \&
    \inlinepic{
        \begin{scope}[scale=0.3]
            \PopPimpleRight
        \end{scope}
    } 
    \oplus\text{\:}
    \inlinepic{
        \begin{scope}[scale=0.3]
            \HM
            \begin{scope}[xshift=1cm]
                \Arrow
            \end{scope}
        \end{scope}
    }
    \oplus
    \inlinepic{
        \begin{scope}[scale=0.3]
            \HB
            \begin{scope}[xshift=-1cm]
                \Arrow
            \end{scope}
        \end{scope}
    }    
    \arrow[r,"{\tiny \begin{pmatrix}
        +&+&-
    \end{pmatrix}}"]
    \&
    \inlinepic{
        \begin{scope}[scale=0.3]
            \SnakeLeft
        \end{scope}
    }\text{\:}
    \arrow[dash,from=1-1,to=2-1,"0"description,xshift=0.8cm]
    \arrow[equal,from=1-2, to=2-2, xshift=1.3cm,"Id"description]
    \arrow[dash,from=1-4,to=2-4,xshift=-1cm,"0"description]
    \arrow[equal,from=1-4, to=2-4, xshift=1cm,"-Id"description]
    \arrow[equal,from=1-4, to=2-4,"Id"description]
    \arrow[equal,from=1-5, to=2-5,"Id"description]  
    \arrow[from=1-2, to=2-2,"A",xshift=-0.4cm]  
    \arrow[from=2-2, to=1-2,"A'",xshift=-0.8cm]  
    \arrow[from=1-4, to=1-2,out=120, in=60,looseness=0.5, end anchor={[xshift=-1.5cm]},start anchor={[xshift=-1cm]},"f_1f",swap]
    \arrow[from=1-2, to=1-1,out=120, in=60,looseness=0.8, end anchor={[xshift=0.4cm]},start anchor={[xshift=-1cm]},"-f_2",swap]
    \arrow[from=2-4, to=2-2,out=-120, in=-60,looseness=0.5, end anchor={[xshift=-1.5cm]},start anchor={[xshift=-1cm]},"-g_1g",swap]
    \arrow[from=2-2, to=2-1,out=-120, in=-60,looseness=0.8, end anchor={[xshift=0.4cm]},start anchor={[xshift=-1cm]},"-g_2",swap]
\end{tikzcd}
\]\vspace{0.5cm}\\
Where $A$ and $A'$ are the matrices:
\[A=
\begin{pmatrix}
   \inlinepic{
    \begin{scope}[scale=0.2]
        \MapThree
    \end{scope}
   } \hspace{0.5cm} & -Id \vspace{0.5cm} \\-\:
   \inlinepic{
    \begin{scope}[scale=0.2]
        \MapFour
    \end{scope}
   }   &
   \inlinepic{
    \begin{scope}[scale=0.2]
        \MapTwo
    \end{scope}
   }  
\end{pmatrix}\hspace{1cm}
A'=
\begin{pmatrix}
    \inlinepic{
    \begin{scope}[scale=0.2]
        \MapTwo
    \end{scope}
   }  \hspace{0.5cm}&
   -\: 
   \inlinepic{
    \begin{scope}[scale=0.2]
        \MapSeven
    \end{scope}
   } \vspace{0.5cm} \\
   - Id &
         \inlinepic{
    \begin{scope}[scale=0.2]
        \MapEight
    \end{scope}
   }  
\end{pmatrix}
\]

\end{itemize}
    These 5 moves generate all oriented Reidemeister moves \cite{GeneratingSets}. This proves invariance up to homotopy of the bracket complex, and therefore the Khovanov complex. What remains to check is that the homotopy equivalence in question is of degree $0$. As we already have explicit formulas for the maps associated with each Reidemeister move, it is just a matter of verifying that these maps are of degree $0$.\par
    A few last remarks about invariance:
\begin{itemize}
    \item There are a few more arbitrary choices in the definition of $CKh$. In particular, we need to check that different choices of the ordering of crossings, or of the ordering of states at the same height, give homotopic chain complexes.
    \item We proved invariance under Reidemeister moves by working locally. To see why this is enough, one can use the language of canopolies (see \cite{BarNatan2005}).

\end{itemize}
\end{proof}
\subsection{The trivalent TQFT}
\label{TrivalentTQFT}
\begin{deff}
    The category $Cob^3$ has for objects closed curves in the plane and for morphisms cobordisms between them (embedded in $\mathbb{R}^3$). \\
    A (1+1)-Topological Quantum Field Theory (TQFT) $\mathcal{F}$ is a monoidal functor from the category $Cob^3$ to the category of $\mathcal{A}b$ of abelian groups. Meaning that for any two curves $\gamma$ and $\gamma'$, $\mathcal{F}(\gamma\sqcup \gamma')=\mathcal{F}(\gamma)\otimes\mathcal{F}(\gamma')$. So for two cobordisms $C$ and $C'$, $\mathcal{F}(C\sqcup C')=\mathcal{F}(C)\otimes\mathcal{F}(C')$.

\end{deff}
\begin{deff}
Let $R$ be a ring. 
A Frobenius algebra is a set $A$ equipped with an $R$-module equipped with maps $m: A \otimes A \rightarrow A, \Delta: A \rightarrow A \otimes A, 1: \mathbb{Z} \rightarrow A$, and $\varepsilon: A \rightarrow \mathbb{Z}$  such that: 
\begin{enumerate}
    \item $\mathfrak{m}(a \otimes b)=m(b \otimes a)$ and $m(m(a \otimes b) \otimes c)=m(a \otimes m(b \otimes c))$
    \item $ \iota\Delta=\Delta.$ and $\left(1_A \otimes \Delta\right) \circ \Delta=\left(\Delta \otimes 1_A\right) \circ \Delta$
    \item $m(\mathbf{1} \otimes a)=a$ and $\left(\varepsilon \otimes 1_A\right) \circ \Delta=1_A$
    \item  $\Delta \circ m=\left(m \otimes 1_A\right) \circ\left(1_A \otimes \Delta\right)$.
\end{enumerate}
Where $\iota: A \otimes A \rightarrow A \otimes A$ is the involution given by $\mathrm{\iota}(\mathrm{a} \otimes \mathrm{b})=\mathrm{b} \otimes \mathrm{a}$.

\end{deff}
Set $R=\mathbb{Z}$. It can be shown then that a Frobenius algebra $A$ determines a (1+1)-TQFT. Conversely, a (1+1)-TQFT determines a Frobenius algebra. The correspondence between the two is the following. \vspace{0.5cm}
\begin{center}
    
$ 
\inlinepic{
    \draw (0,0) circle (0.3cm);
}
\longleftrightarrow
A$ \quad , \quad
$ 
\inlinepic{
    \FrobOne{}{0.3}
}
\longleftrightarrow
1 $ \quad , \quad
$ 
\inlinepic{
    \FrobEpsilon{}{0.3}
}
\longleftrightarrow
\epsilon$ \quad , \quad
$ 
\inlinepic{
    \Frobm{}{0.2}
}
\longleftrightarrow
m$ \quad , \quad
$ 
\inlinepic{
    \FrobCom{}{0.2}
}
\longleftrightarrow
\Delta$. \vspace{0.5cm}
\end{center}

\par

We have seen that $S(\inlinepic{
    \draw[Blue] (0,0) circle (0.2cm);
})$ is a free $\mathbb{Z}[X_1,X_2]$-module of rank 2. It is generated by \inlinepic{\FrobOne{Blue}{0.3}} and \inlinepic{\FrobOne{Blue}{0.3}
\begin{scope}[yshift=0.2cm,scale=0.3]
    \BDot    
\end{scope}
}.  \\  From the relation 
    \inlinepic{\draw[Blue](-0.3,-0.3)rectangle(0.3,0.3);
    \fill[Blue](0,0.1) circle (1pt);
    \fill[Blue](0,-0.1) circle (1pt);
    } \: $=$ \: $ (X_1+X_2)$\:
    \inlinepic{\draw[Blue](-0.3,-0.3)rectangle(0.3,0.3);
    \fill[Blue](0,0) circle (1pt);
    }\: $-$ \: $ X_1X_2$\:
    \inlinepic{\draw[Blue](-0.3,-0.3)rectangle(0.3,0.3);}\: , it follows that $S(\inlinepic{
    \draw[Blue] (0,0) circle (0.2cm);
})=\mathbb{Z}[X_1,X_2][X]_{/(X-X_1)(X-X_2)}$ as a Frobenius algebra. The trivalent TQFT arises once we set $X_1=X_2=0$.

\begin{deff}
    The trivalent TQFT  $\mathcal{T}:\mathcal{F}\rightarrow\mathcal{A}b$ is the functor: 
    $$\mathcal{T}: \mathcal{F}\xrightarrow{\mathcal{G}}\mathbb{Z}[X_1,X_2]Mod\xrightarrow {X_1=X_2=0}\mathcal{A}b$$ 
\end{deff}

$\mathcal{T}$ is monoidal and $\mathcal{T}(\bigcirc)$ is the Frobenius algebra $\mathbb{Z}[X]_{/X^2}$. The relation $X^2=0$ is homogenous.  $\mathcal{T}(\bigcirc)$ is still graded of degree $(q+q^{-1})$.\par
Let us take a closer look at the relations in $\mathcal{T}(\inlinepic{
    \draw[Blue] (0,0) circle (0.2cm);
})$.
\begin{itemize}
    \item Adding decorations on red surfaces multiplies the cobordism by $X_1+X_2$ or $X_1X_2$, i.e. $0$.
    \item The neck cutting relation is \vspace{0.5cm}\\
        \inlinepic{\begin{scope}[scale=0.2]\BCyl{}{+}{+}\end{scope}} \: $=$ \: $-$ \:\inlinepic{    \begin{scope}[yshift=0.4cm,scale=0.2]
        \BCup{1}{+}
    \end{scope}
    \begin{scope}[yshift=-0.4cm,scale=0.2]
        \BCap{0}{-}
        
    \end{scope}} \: 
    $-$
    \: \inlinepic{    \begin{scope}[yshift=0.4cm,scale=0.2]
        \BCup{0}{+}
    \end{scope}
    \begin{scope}[yshift=-0.4cm,scale=0.2]
        \BCap{1}{-}
    \end{scope}}\:
    $+$ \: $(X_1+X_2)$
    \quad \inlinepic{    \begin{scope}[yshift=0.4cm,scale=0.2]
        \BCup{0}{+}
    \end{scope}
    \begin{scope}[yshift=-0.4cm,scale=0.2]
        \BCap{}{-}  \vspace{0.5cm} 

    \end{scope}}\quad $=$ \: $-$ \:\inlinepic{    \begin{scope}[yshift=0.4cm,scale=0.2]
        \BCup{1}{+}
    \end{scope}
    \begin{scope}[yshift=-0.4cm,scale=0.2]
        \BCap{0}{-}
        
    \end{scope}} \: 
    $-$
    \: \inlinepic{    \begin{scope}[yshift=0.4cm,scale=0.2]
        \BCup{0}{+}
    \end{scope}
    \begin{scope}[yshift=-0.4cm,scale=0.2]
        \BCap{1}{-}
    \end{scope}}\: .
    \item Moving a blue dot across a component of the binding multiplies thee initial cobordism by $-1$. \vspace{0.5cm}\\
    \inlinepic{\begin{scope}[scale=0.3]
        \Migration{B2}
    \end{scope}} \: $=$ \: $-$ \:
    \inlinepic{\begin{scope}[scale=0.3]
        \Migration{B1}
    \end{scope}}\: .
    \item Two blue points on the same connected component give $0$.\vspace{0.5cm}\\
    \inlinepic{\begin{scope}[scale=0.15]
        \BId{2}{}{}
    \end{scope}} \: $=$ \: $ (X_1+X_2)$\quad
    \inlinepic{\begin{scope}[scale=0.15]
        \BId{1}{}{}
    \end{scope}}\: $-$ \: $ X_1X_2$\quad
    \inlinepic{\begin{scope}[scale=0.15]
        \BId{}{}{}
    \end{scope}} \: $=0$ \quad.

\end{itemize}

\subsection{Categorification of the Jones polynomial}\label{Categorification}

\begin{deff}

Let $D$ be a diagram of a link $L$. The Jones polynomial $V(L) \in \mathbb{Z}[t^{\pm\frac{1}{2} }]$ is defined by the following rules:\begin{itemize}
    \item $V(\bigcirc)=1$
    \item $V(D\sqcup\bigcirc)=(-t^{\frac{1}{2}}-t^{-\frac{1}{2}})V(D)$
    \item $t^{-1}V(\: \inlinepic{
        \begin{scope}[scale=0.8]
            \draw[->] (-60:0.4)--(120:0.4);
            \fill[white] (0,0) circle (0.1cm);
            \draw[->] (-120:0.4)--(60:0.4);
  
        \end{scope}
    }\:
    )-tV(\:
        \inlinepic{
        \begin{scope}[scale=0.8]
            \draw[->] (-120:0.4)--(60:0.4);
            \fill[white] (0,0) circle (0.1cm);
            \draw[->] (-60:0.4)--(120:0.4);
  
        \end{scope}
    }\:
    )=(t^{\frac{1}{2}}-t^{-\frac{1}{2}})V(\:
        \inlinepic{
        \begin{scope}[scale=0.8]
        \draw[->,bend right] (-120:0.4)to(120:0.4);
        \draw[->, bend left] (-60:0.4)to(60:0.4);
        \end{scope}
    }   
    \: ).$
\end{itemize}

\end{deff}

\begin{thh}\label{categorification}
    Let $L$ be some link, and $\chi_q(L)$ be the graded Euler characteristic of the chain complex $ CKh(L)$:
    $$
    \chi_q(L)=\sum_{i \in \mathbb{Z}}(-1)^i V(CKh_i(L)) 
    $$
    Then up to normalisation, $\chi_q(L)$ is equal to $V(L)$ with the change of variable $q=-t^{-\frac{1}{2}}$.
\end{thh}
\begin{proof}
    We have the equalities:
    \[ CKh(\: \inlinepic{
        \begin{scope}[scale=0.8]
            \draw[->] (-60:0.4)--(120:0.4);
            \fill[white] (0,0) circle (0.1cm);
            \draw[->] (-120:0.4)--(60:0.4);
  
        \end{scope}}\:)
        \: = \: CKh(\:  \inlinepic{
        \begin{scope}[scale=0.8]
        \draw[->,Blue,bend right] (-120:0.4)to(120:0.4);
        \draw[->,Blue, bend left] (-60:0.4)to(60:0.4);
        \end{scope}
    }    \:) \{ -1\} \oplus
        CKh(\: \inlinepic{ 
    \begin{scope}[scale=0.8]
        \draw[Blue,->] (0,0.2) to[in=-120,out=0] (60:0.4);
        \draw[Blue,->] (0,0.2) to[in=-60,out=180] (120:0.4);
        \draw[Maroon,thick,decoration={ markings, mark=at position 0.7 with {\arrow[scale=0.5]{>>}}},postaction={decorate}] (0,-0.2)--(0,0.2);
        \draw[Blue] (-60:0.4) to[in=0, out=120] (0,-0.2);
        \draw[Blue] (-120:0.4) to[in=180, out=60] (0,-0.2);
    \end{scope}
    }\: )[1]\{-2\}
        \] and 
    \[ CKh(\:         \inlinepic{
        \begin{scope}[scale=0.8]
            \draw[->] (-120:0.4)--(60:0.4);
            \fill[white] (0,0) circle (0.1cm);
            \draw[->] (-60:0.4)--(120:0.4);

        \end{scope}}\:)
        \: = \: CKh(\:  \inlinepic{
        \begin{scope}[scale=0.8]
        \draw[->,Blue,bend right] (-120:0.4)to(120:0.4);
        \draw[->,Blue, bend left] (-60:0.4)to(60:0.4);
        \end{scope}
    }    \:) \{ 1\} \oplus
        CKh(\: \inlinepic{ 
    \begin{scope}[scale=0.8]
        \draw[Blue,->] (0,0.2) to[in=-120,out=0] (60:0.4);
        \draw[Blue,->] (0,0.2) to[in=-60,out=180] (120:0.4);
        \draw[Maroon,thick,decoration={ markings, mark=at position 0.7 with {\arrow[scale=0.5]{>>}}},postaction={decorate}] (0,-0.2)--(0,0.2);
        \draw[Blue] (-60:0.4) to[in=0, out=120] (0,-0.2);
        \draw[Blue] (-120:0.4) to[in=180, out=60] (0,-0.2);
    \end{scope}
    }\: )[-1]\{2\}
        \]
    Shifting the degree up by $2$ in the first equation and down by $2$ in the second, and then taking the difference of the graded Euler characteristics, we obtain
    \[q^2 \chi(CKh(\:
    \inlinepic{
        \begin{scope}[scale=0.8]
            \draw[->] (-60:0.4)--(120:0.4);
            \fill[white] (0,0) circle (0.1cm);
            \draw[->] (-120:0.4)--(60:0.4);
  
        \end{scope}
    }
     \:))-q^{-2}\chi(CKh(\: 
        \inlinepic{
        \begin{scope}[scale=0.8]
            \draw[->] (-120:0.4)--(60:0.4);
            \fill[white] (0,0) circle (0.1cm);
            \draw[->] (-60:0.4)--(120:0.4);
  
        \end{scope}}
     \:))=(q-q^{-1})\chi(CKh(\:
        \inlinepic{
        \begin{scope}[scale=0.8]
        \draw[->,bend right] (-120:0.4)to(120:0.4);
        \draw[->, bend left] (-60:0.4)to(60:0.4);
        \end{scope}
    }  
      \:)).\]
    \end{proof}
\section{Recovering the original Khovanov homology}
\label{sec:RecoveringKh}
\subsection{The original definition}
Let $D$ be a diagram of an oriented link $L$. Let $n_+$ (resp. $n_-$) be the number of positive (resp. negative) crossings of $D$.\\
As in section \ref{DefPic} we start by constructing a cube of resolution.  A state $s$ of $D$ associates $0$ or $1$ to each crossing.\\
This time the smoothings do not depend of the sign of the crossing, and they look like this:\[
\inlinepic{
    \begin{scope}
        \Smoothing
    \end{scope}
}
\]
Then the vertices of the cube of resolution are objects of the category $Cob^3$. The height of a smoothing $D_s$ is  $\sum_{c} s(c)-n_-$.  The edges of cube are saddles where $s(c)$ changes and the identity everywhere else (the edges are morphisms in $Cob^3$). The signs are chosen exactly as in the first cube of resolution. The result is a formal chain complex in $Kob(Mat(Cob^3))$, that is an invariant of links(\cite{BarNatan2005}).\\
Taking the cube of resolution we defined in section \ref{DefPic} and forgetting the red edges, we get exactly the same picture. To see this, let $D_s$ be an oriented smoothing where we forget the red edges, and $D'_s$ the corresponding non-oriented smoothing. Let $c$ be a positive crossing. If $s(c)=1$ in $D_s$ then $s'(c)=1$ in $D'_s$ and if $s(c) =0$ then $s'(c)=0$ as well:
\begin{center}
    In $D_s$ \: \inlinepic{\Positive}\quad, \quad in $D'_s$ \: \inlinepic{\Smoothing}\: . \vspace{0.5cm}
\end{center}
If $c$ is a negative crossing, then if $s(c)=-1$ in $D_s$ $s'(c)=0$ in $D'_s$, and if $s(c)=0$ in $D_s$ $s'(c)=1$ in $D'_s$. \begin{center}
    In $D_s$ \: \inlinepic{\Negative}\quad, \quad in $D'_s$ \: \inlinepic{\SmoothingN}\: . \vspace{0.5cm}
\end{center}
Let $i$ be the height of $D_s$ and $i'$ the height of $D'_s$ , then:
\begin{equation*}
    \begin{split}
    i & =\sum_{c \: s(c)=1}1 - \sum_{c \: s(c)=-1}1
    \\ & =\sum_{c \text{ positive, }s'(c)=1}1 - \sum_{c \text{ negative, }s'(c)=0}1
    \\ & =\sum_{c \text{ positive, }s'(c)=1}1 - \sum_{c \text{ negative, }}1 + \sum_{c \text{ negative, }s'(c)=1}1
    \\ & =(\sum_{s'(c)=1}1 )- n_-
    \\ &= i'.
    \end{split}
\end{equation*}
\\
The next step is to apply the TQFT $\mathcal{T}'=\mathbb{Z}[X]_{/X^2}$ to the new cube of resolution. Let $C_2$ be the resulting chain complex. Let $n_s$ be the number of blue components in $D_s$ and $n'_s$ the number of components in $D'_s$. Then $n_s=n'_s$ and $\mathcal{T}(D_s)=(\mathbb{Z}[X]_{/X^2})^{\otimes n_s}=\mathcal{T}'(D'_s)$ (Theorem \ref{gradedDim}). Let $deg(1)=1$ and $deg(X)=-1$, and $Kh(D)_i=C_{2,i}$\{$i+n_+-n-$\}.\\
The complex $Kh(D)$ is the original Khovanov complex (Section 7 in \cite{BarNatan2005}). The discussion above shows that $Kh(D)$ and $CKh(D)$ have the same chain groups. This suggests that CKh(D) and Kh(D) could be isomorphic, and we can extract the original Khovanov homology from $CKh$, justifying it's name. Although, a priori, the differentials do not have to be the same.\\
The goal of this section is to prove that $Kh(D)$ is indeed isomorphic to $\mathcal{T}(CKh(D))$. We start with defining the tools that will be needed, adapted from \cite{CupFoams}. 
\subsection{Shadings}
\begin{deff}
A shading of a manifold $M^n$ consists of two codimension 1 submanifolds, an oriented one, $U_r$, and a non-oriented one, $U_b$, which are transverse to each other and to $\partial M^n$, as well as a checkerboard coloring of $M^n$, a checkerboard coloring being a choice of color, white or black, for each connected component of the complement of $U_r \cup U_b$. We refer to $U_r$ and $U_b$ as red and blue respectively. The components of the intersection $U_r \cap U_b$ are called bindings, they decompose both $U_r$ and $U_b$ into facets. The components of the complement of $U_r \cup U_b$ in $M^n$ are called regions.
\end{deff}
A checkerboard coloring of a manifold $M$ is entirely determined by a choice of coloring of one region in it \cite{CupFoams}. \par
A trivalent manifold embedded in $M^n$ is generalization of trivalent graphs and trivalent surfaces. It is made of red and blue facets, oriented codimension 1 submanifolds, with boundary components attached transversally to the boundary of $M^n$ or glued together along \textbf(bindings) in a way that respects a certain flow condition. We are interested in the case $n=1,2$ or $3$. If $n=1$ than a trivalent manifold in $M^1$ is just a descrete set of red and blue points with a checkerboard coloring of their complement. If $n=2$ we have trivalent  graphs and if $n=3$ we have trivalent surfaces.\par
There is a correspondance between shadings of a manifold $M^n$ and trivalent manifolds embedded in it. Let $(U_r,U_b)$ be a shading of $M^n$, then $\Gamma(U_r,U_b)$ is the trivalent manifold constructed as follows:
\begin{itemize}
    \item Each blue facet inherits the orientation induced from the white region it bounds.
    \item A red facet is only kept if its orientation respects that of the one induced by its neighboring white region. Otherwise it is annihilated.
\end{itemize}
An example of this construction is shown in Figure \ref{fig:TrivFromShade}.
\begin{figure}[h]
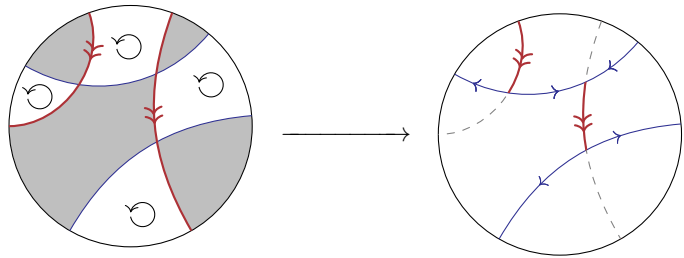

    \[ 
    \inlinepic{
        \begin{scope}[scale=0.8]
            \MFromShading
        \end{scope}
    }\hspace{0.25cm}\xrightarrow{\hspace{1.5cm}}\hspace{0.25cm}
    \inlinepic{
        \begin{scope}[scale=0.8]
            \MFromShadingTwo
        \end{scope}
    }
    \]
    \caption{The construction of a trivalent  graph from shading of a disk.}
    \label{fig:TrivFromShade}
\end{figure}\par
Conversely, we have the following lemma. 
\begin{lemma}
    \label{CompletionDef}
    Let $V$ be a trivalent manifold embedded in $M^n$, where $n=2$ or $3$. Suppose that $\partial V=\Gamma(\tilde{U}_r,\tilde{U}_b)$ for some shading $(\tilde{U}_r,\tilde{U}_b)$ of $\partial M^n$, then there exists a shading $(U_r,U_b)$ of $M^n$ that restricts to $(\tilde{U}_r,\tilde{U}_b)$ on  $\partial M^n$ and such that $\Gamma(U_r,U_b)=V$.
\end{lemma}
\begin{proof}
    Let $V$ be a trivalent manifold as in the lemma. We construct $(U_r,U_b)$ as follows. 
    \begin{itemize}
        \item \textbf{Step 1:} We orient the boundary. The regions of $\partial M^n$ painted white keep the orientation induced by $M$ and the orientation of the regions painted black is reversed.
        \item \textbf{Step 2:} Let $R$ be a region of $M^n$. Then the boundary of $R$ is the union of regions of $\partial M^n$ and facets of $V$. Two oppositely oriented components of the boundary of $R$ meet in two cases: either they are both in the boundary, then they meet in $\partial V$, or they are both facets of $V$, and they meet in a trivalent vertex or a binding. \\
        Working in the region $R$, consider $U$ the set of points of $\partial R$ where the orientation is opposite to that induced by the region. $U$ is a codimention one submanifold of $R$. We take a red colored copy of each component of $U$, and push it inside of $R$. The new formed region is painted black.\\
        We do this for each region to obtain the desired shading.
    \end{itemize}
    An example of this construction is shown in Figure \ref{fig:ShadeFromTriv}
    \begin{figure}[h]
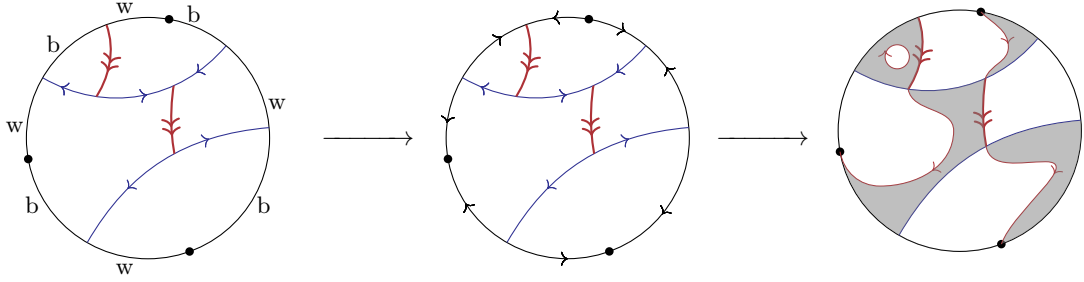

        \centering
        \[
        \inlinepic{
            \begin{scope}[scale=0.8]
                \ShadingFromM
            \end{scope}
        }
        \hspace{0.25cm}
        \xrightarrow{\hspace{1cm}}
        \hspace{0.25cm}
        \inlinepic{
            \begin{scope}[scale=0.8]
                \ShadingFromMTwo
            \end{scope}
        }
        \hspace{0.25cm}
        \xrightarrow{\hspace{1cm}}
        \hspace{0.25cm}
        \inlinepic{
            \begin{scope}[scale=0.8]
                \PushAndShade
            \end{scope}
        }
        \]
        \caption{The construction of a shading of $\mathbb{D}^2$ from a trivalent graph.}
        \label{fig:ShadeFromTriv}
    \end{figure}

\end{proof}
A competion of a trivalent manifold $V\subset M^n$ is a shading $(U_r,U_b)$ such that $\Gamma(U_r,U_b)=V$. Lemma \ref{CompletionDef} says that completions often exist. In particular if $G,G'$ are trivalent graphs embedded in $\mathbb{D}^2$ then:
\begin{itemize}
    \item The boundary of $G$ is a shading of $\partial \mathbb{D}^2$ and $G$ admits a completion.
    \item A $G$-Foam-$G'$ has for boundary a trivalent graph embedded in the boundary of the cyclinder $\mathbb{D}^2\times[0,1]$. Therefore any cobordism between trivalent graphs admits a completion.
\end{itemize}
\subsection{A basis of $\mathcal{T}(G)$}
Let $G$ be a closed trivalent graph embedded in the disk $\mathbb{D}^2$. Let $(G_r,G_b)$ be its completion. We paint boundary of $\mathbb{D}^2$ white, which fully determines the checkerboard coloring of the shading $(G_r,G_b)$. We can then simplify the diagram of the completion by dropping the checkerboard coloring. To make $G$ more visible, the red edges that do not survive after applying $\Gamma$ are drawn dashed. \\
We also allow to mark the blue loops in $G_b$ with any number of points.
    
 A shading $S$ of $\mathbb{D}^2\times[0,1]$ can be constructed from the (dotted) completion $(G_r,G_b)$ as follows:
\begin{enumerate}
        \item we attach to each red loop in $G_r$ a disk in $\mathbb{D}^2 \times [\frac{1}{2},1]$, such that the collection of disks $D_r$ project injectively onto $\mathbb{D}^2 \times\{1\}$. 
        \item We attach to each blue loop in $G_{b}$ a disk in $\mathbb{D}^2 \times [0,1]$, such that  the collection of disks $D_b$ project injectively onto $\mathbb{D}^2 \times\{1\}$, and such that $D_b \cap (\mathbb{D}^2 \times [0,\frac{1}{2}])=G_{b}\times [0,\frac{1}{2}]$.
        \item The blue disks are decorated with the same number of blue points as their boundaries, placed at height less than $\frac{1}{2}$.
\end{enumerate}
The foam $\Gamma(S)$ of the resulting shading is called a cup foam. It is uniquely determined by the dotted completion $(G_r,G_b)$ used in its definition. 
\begin{lemma}\label{Basis}
    Let $G$ be a closed trivalent graph. Let $(G_r,G_b)$ be a completion of $G$. The family of all dotted completions obtained from $(G_r,G_b)$, where each blue loop contains at most one dot, gives rise to a basis of $\mathcal{T}(G)$, that we denote $Foam(G)$.
\end{lemma}
\textbf{Sketch of proof:} Let $\Sigma \subset \mathbb{D}^2\times[0,1]$ be an $\emptyset$-foam-$G$. The relations from Lemma \ref{localrelations} allow us to write $\Sigma$ as a linear combination of cobordisms where all the facets have genus $0$. \\
Suppose then that the facets of $\Sigma$ have genus $0$, and let $(\Sigma_r,\Sigma_b)$ be a completion of $\Sigma$, such that $(\Sigma_r,\Sigma_b)$ restricts to $(G_r,G_b)$ in $\mathbb{D}^2\times\{1\}$. Then $\Sigma_r$  (resp. $\Sigma_b$) is a collection of disks attached to the components of $G_r$ (resp $G_b$), and up to isotopy the shading $\Sigma_r,\Sigma_b$ verifies the conditions $(1)$ and $(2)$ in the definition of cup foams. 
\\
If we have decorations on the red disks then we remove them by multiplying $\Sigma$ by the corresponding polynomial. For blue points, Cobordisms that contain more than one blue point in the same component are just zero in $\mathcal{T}$ (refer to section \ref{TrivalentTQFT}). If we have a single point in one component, we can push it to the height we want using the migration relation repeatedly, verifying the condition (3). \\
This proves that $Foam(G)$ generates $\mathcal{T}(G)$. \\
Let $F$ be an element of  $Foam(G)$, and let $\bar{F}\in Foam(G)$ be the cup foam where a blue loop $\gamma$ is marked with a point in $\bar{F}$ if and only if it has no points in $F$. Then if we attach $F$ to $\bar{F}$ along $G$, the evaluation of the resulting closed surface $F\bar{F}$ is $\pm 1$.  This is because the bubble relations allow us to push all the red spheres out, obtaining a set of disjoint Blue and Red spheres, where the blue spheres carry each exactly one point, coming from either $F$ or $\bar{F}$. \\
The evaluation of a red sphere is $-1$, Same for a blue sphere containing one point, which proves that depending on their number, the evaluation of $F\bar{F}$ is $\pm 1$.\\
For any foam $F'$ in $Foam(G)$ distinct from $F$, the closed surface  $F'\bar{F}$ has either an empty blue sphere or a blue sphere with two points, in either case its evaluation is $0$.\\
Let $\Sigma \in \mathcal{T}(G)$ be an arbitrary cobordism, then it can be expressed in a unique way in terms of the generating family $Foam(G)$ $\Sigma= \sum_{F\in Foam(G)}<\Sigma \bar{F}><F \bar{F}>F$, proving that $Foam(G)$ is a basis of $\mathcal{T}(G)$.\hfill $\square$

\begin{example}

    Let $G$ be a circle oriented clockwise. Then the orientation of $G$ agrees with that of the outermost region and the completion of $G$ is a blue circle.
    \[
    \begin{tikzcd}
        Trivalent \text{ } graph & Completion & basis\\
        \inlinepic{
            \draw[Blue,-<-](0,0) circle (0.5cm);
        }\hspace{0.5cm}
        &
        \inlinepic{
            \draw[Blue](0,0) circle (0.5cm);
        }     \hspace{0.5cm}   
        & 
        \inlinepic{
            \draw[Blue](0,0) circle (0.5cm);            
        }
        \arrow[r,<->]  &
        \inlinepic{\begin{scope}[scale=0.6]
            \BCapE{}{-}            
        \end{scope}
        }     
        \\
        & & 
        \inlinepic{
            \draw[Blue](0,0) circle (0.5cm); 
            \fill[Blue](0.5,0) circle (1.5pt);     
        }
        \arrow[r,<->]  &
        \inlinepic{\begin{scope}[scale=0.6]
            \BCapE{1}{-}            
        \end{scope}
        }     
    \end{tikzcd}
    \]
    Let $G$ be a circle oriented counter-clockwise. The orientation of $G$ is opposite to that induced by the outermost region. Then its completion is a blue circle surrounded by a dashed red circle. 
    \[
    \begin{tikzcd}
        Trivalent \text{ } graph & Completion & basis\\
        \inlinepic{
            \draw[Blue,->-](0,0) circle (0.5cm);
        }\hspace{0.5cm}
        &
        \inlinepic{
            \draw[Blue](0,0) circle (0.5cm);
            \draw[Maroon,dashed,->-](0,0) circle (0.7cm);
        }     \hspace{0.5cm}   
        & 
        \inlinepic{
            \draw[Blue](0,0) circle (0.5cm);  
            \draw[Maroon,dashed,->-](0,0) circle (0.7cm);          
        }
        \arrow[r,<->]  &
        \inlinepic{\begin{scope}[scale=0.6]
            \BRCapE{}{+}{}{}          
        \end{scope}
        }  \quad = \quad
        \inlinepic{\begin{scope}[scale=0.6]
            \BCapE{}{+}        
        \end{scope}
        }        
        \\
        & & 
        \inlinepic{
            \draw[Blue](0,0) circle (0.5cm); 
            \fill[Blue](0.5,0) circle (1.5pt);  
            \draw[Maroon,dashed,->-](0,0) circle (0.7cm);             
        }
        \arrow[r,<->]  &
        \inlinepic{\begin{scope}[scale=0.6]
            \BRCapE{1}{+}{}{}            
        \end{scope}
        }     \quad = \: -\quad
        \inlinepic{\begin{scope}[scale=0.6]
            \BCapE{1}{+}        
        \end{scope}
        }    
    \end{tikzcd}
    \]
\end{example}
\subsection{The main result}
Let $L$ be an oriented link, and $D$ a diagram of $L$.
Let $D_{s_1}$ be the smoothing of $D$ such that $|s_1(c)|=1$ $\forall c$. \\
A completion of $D_{s_1}$ determines a compatible completion of $D_s$ for any state $s$. It is enough to replace 
\inlinepic{   
    \begin{scope}[]
        \draw[Blue,->] (0,0.2) to[in=-120,out=0] (60:0.4);
        \draw[Blue,->] (0,0.2) to[in=-60,out=180] (120:0.4);
        \draw[dashed,Maroon](0,-0.4)--(0,0.4);
        \draw[Maroon,thick,decoration={ markings, mark=at position 0.7 with {\arrow[scale=0.5]{>>}}},postaction={decorate}] (0,-0.2)--(0,0.2);
        \draw[Blue] (-60:0.4) to[in=0, out=120] (0,-0.2);
        \draw[Blue] (-120:0.4) to[in=180, out=60] (0,-0.2);
    \end{scope}}
with \inlinepic{
        \begin{scope}[]
            \draw[->,Blue,bend right] (-120:0.4)to(120:0.4);
            \draw[->,Blue, bend left] (-60:0.4)to(60:0.4);
            \draw[dashed,Maroon](0,-0.4)--(0,0.4);
        \end{scope}
}   whenever $s(c)=0$.\\
Let $\{D_s^+\}_s$ be such a set of completions and $\{Foam(D_s)\}_s$ be the corresponding cup foam bases. Fix an order on the blue loops of $D_s$. Let $\Psi_s: \mathcal{T}(D_s) \rightarrow (\mathbb{Z}[X]_{/X^2})^{\otimes n_s}$ be the map sending a cup foam  $F \in Foam(D_s)$ to the tensor product $x_F=x_{F,1} \otimes x_{F,2} \otimes \ldots \otimes x_{F,n_s}$ where $x_{F,i}=X$ if the blue loop $i$ contains a dot, and $x_{F,i}=1$ otherwise. Then $\Psi_s$ is an isomorphism and it is the identification we use between  $\mathcal{T}(D_s)$ and $(\mathbb{Z}[X]_{/X^2})^{\otimes n_s}$ from now on.
\begin{claim}\label{claim}
    Let $c$ be a negative crossing. If the map $d_c: \inlinepic{   
    \begin{scope}[]
        \draw[Blue,->] (0,0.2) to[in=-120,out=0] (60:0.4);
        \draw[Blue,->] (0,0.2) to[in=-60,out=180] (120:0.4);
        \draw[dashed,Maroon](0,-0.4)--(0,0.4);
        \draw[Maroon,thick,decoration={ markings, mark=at position 0.7 with {\arrow[scale=0.5]{>>}}},postaction={decorate}] (0,-0.2)--(0,0.2);
        \draw[Blue] (-60:0.4) to[in=0, out=120] (0,-0.2);
        \draw[Blue] (-120:0.4) to[in=180, out=60] (0,-0.2);
    \end{scope}} 
    \longrightarrow
    \inlinepic{
        \begin{scope}[]
            \draw[->,Blue,bend right] (-120:0.4)to(120:0.4);
            \draw[->,Blue, bend left] (-60:0.4)to(60:0.4);
            \draw[dashed,Maroon](0,-0.4)--(0,0.4);
        \end{scope}
    } $
    merges two blue loops then $d_c=\pm  m$. If it splits one blue loop into two then $d_c=\pm \Delta$.\\
    The same holds if $c$ is a positive crossing so $d_c$ goes from
    \inlinepic{
        \begin{scope}[]
            \draw[->,Blue,bend right] (-120:0.4)to(120:0.4);
            \draw[->,Blue, bend left] (-60:0.4)to(60:0.4);
            \draw[dashed,Maroon](0,-0.4)--(0,0.4);
        \end{scope}
    } to 
    \inlinepic{   
    \begin{scope}[]
        \draw[Blue,->] (0,0.2) to[in=-120,out=0] (60:0.4);
        \draw[Blue,->] (0,0.2) to[in=-60,out=180] (120:0.4);
        \draw[dashed,Maroon](0,-0.4)--(0,0.4);
        \draw[Maroon,thick,decoration={ markings, mark=at position 0.7 with {\arrow[scale=0.5]{>>}}},postaction={decorate}] (0,-0.2)--(0,0.2);
        \draw[Blue] (-60:0.4) to[in=0, out=120] (0,-0.2);
        \draw[Blue] (-120:0.4) to[in=180, out=60] (0,-0.2);
    \end{scope}} 
    instead.
\end{claim}
The sign in the claim only depends on the neighborhood of the crossing $c$, so there is a change of basis of $CKh(D)$ where we can ignore these signs. This proves the main result of this section:
\begin{prop}
    The complexes $\mathcal{T}(CKh(D))$ and $Kh(D)$ are isomorphic.
\end{prop}
\begin{proof}
    \textbf{of the claim:}\\
    Ignoring the orientation, there are four cases to consider, depending on how the strands near the crossing are connected. \vspace{0.5cm}\\
    \begin{tikzcd}[column sep=2cm]
        & D_{s_1} \arrow[r,yshift=3,"{d_S\text{ for split}}"]  & \arrow[l,yshift=-3,"{d_M\text{ for merge}}"] D_{s_2}\\
        1.&
            \inlinepic{\MergeOne{1.5}} \arrow[r,yshift=3]  & \arrow[l,yshift=-3]\inlinepic{\SplitOne{1.5}}\\
        2.&
            \inlinepic{\MergeTwo{1.3}} \arrow[r,yshift=3]  & \arrow[l,yshift=-3]\inlinepic{\SplitTwo{1.3}}\\
        3.&
            \inlinepic{\MergeThree{1.5}} \arrow[r,yshift=3]  & \arrow[l,yshift=-3]\inlinepic{\SplitThree{1.5}}\\
        4.&
            \inlinepic{\MergeFour{1.3}} \arrow[r,yshift=3]  & \arrow[l,yshift=-3]\inlinepic{\SplitFour{1.3}}
    \end{tikzcd}\vspace{0.5cm}\\
    The goal is to compute $S$ and $M$ in the foam bases of $D_{s_1}$ and $D_{s_2}$. 

In the first case, the result of applying $d_S$ to $1 \in Foam(D_{s,1})$ is \quad
\inlinepic{
    \begin{scope}[scale=0.3]
    \draw[Blue,dotted] (0,0) ellipse (2cm and 0.5cm);
    \path[name path=E] (0,0) ellipse (2cm and 0.5cm);
    \path[name path=D] (-1,-1)--(1,1);
    \path[name intersections={of=E and D, by={A,B}}];
    \draw[Maroon,thick,dotted] (A)--(B);
    \draw[Maroon,thick,dotted] (A)to[out=-90,in=20](0,-1);
    \draw[Maroon,thick] (0,-1)to[out=180,in=-90](B);
    \draw[Blue](2,0)arc(0:180:2cm);
    \draw[Maroon,thick](B)to[out=90,in=-135](135:2);
    \draw[Maroon,thick,dotted](A)to[out=90,in=45](135:2);
    \fill[Red,fill opacity=0.6](A)to[out=90,in=45](135:2)to[in=90,out=-135](B)to[in=180,out=-90](0,-1)to[in=-90,out=20](A);
    \draw[Blue](2,0)--(2,-2);
    \draw[Blue](-2,0)--(-2,-2);
    \draw[Blue](2,-2)arc(0:-180:0.75cm and 0.5cm);
    \draw[Blue](-2,-2)arc(-180:0:0.75cm and 0.5cm);
    \draw[Blue,dotted](2,-2)arc(0:180:0.75cm and 0.5cm); 
    \draw[Blue,dotted](-2,-2)arc(180:0:0.75cm and 0.5cm);
    \draw[Blue](0.5,-2)to[in=0,out=90,](0,-1)to[out=180,in=90](-0.5,-2);
    \end{scope}
}.\vspace{0.3cm}\\We can simplify the drawing and add information about the shading behind it: \vspace{0.5cm} \\
\[\inlinepic{
    \begin{scope}[scale=0.5]
        \draw[Blue](2,0)arc(0:-180:0.75cm and 0.5cm);
        \draw[Blue](-2,0)arc(-180:0:0.75cm and 0.5cm);
        \draw[Blue,dotted](2,0)arc(0:180:0.75cm and 0.5cm); 
        \draw[Blue,dotted](-2,0)arc(180:0:0.75cm and 0.5cm);
        \draw[Blue] (0.5,0)arc(0:180:0.5cm);
        \draw[Blue] (2,0)arc(0:180:2cm);
        \draw[Maroon,dashed](0,0)arc(0:90:2cm and 1.5cm);
        \begin{scope}
            \clip (2,0)arc(0:180:2cm)--(-0.5,0)arc(180:0:0.5cm);
            \draw[Maroon,ultra thick](0,0)arc(0:90:2cm and 1.5cm);   
        \end{scope}
    \end{scope}
}\] \vspace{0.5cm} \\
There could potentially be more red surfaces, for example:
\[\inlinepic{
    \begin{scope}[scale=0.5]
        \draw[Blue](2,0)arc(0:-180:0.75cm and 0.5cm);
        \draw[Blue](-2,0)arc(-180:0:0.75cm and 0.5cm);
        \draw[Blue,dotted](2,0)arc(0:180:0.75cm and 0.5cm); 
        \draw[Blue,dotted](-2,0)arc(180:0:0.75cm and 0.5cm);
        \draw[Blue] (0.5,0)arc(0:180:0.5cm);
        \draw[Blue] (2,0)arc(0:180:2cm);
        \draw[Maroon,dashed](0,0)arc(0:90:2cm and 1.25cm);
        \draw[Maroon,dashed,bend right,looseness=1.3](-2,1.9)to(2,1.9);
        \begin{scope}
            \clip (2,0)arc(0:180:2cm)--(-0.5,0)arc(180:0:0.5cm);
            \draw[Maroon,ultra thick](0,0)arc(0:90:2cm and 1.25cm);   
            \draw[Maroon, ultra thick,bend right,looseness=1.3](-2,1.9)to(2,1.9);
            \draw[Maroon,dashed,bend right](-1.8,2.1)to(1.8,2.1);
        \end{scope}
        \begin{scope}[even odd rule]
            \clip (-2,2) rectangle (2,1)  (2,0)arc(0:180:2cm)--(-0.5,0)arc(180:0:0.5cm);
            \draw[Maroon,ultra thick,bend right](-1.8,2.1)to(1.8,2.1);           
        \end{scope}
    \end{scope}
}\] \vspace{0.5cm} \\
Using the bubble relations,lifting these surfaces up only multiplies the foam by a sign $\epsilon_R$ . So we can ignore them for now. \\
Using the neck cutting relation, dependng on the orientation
\[d_S(1)= \pm \begin{pmatrix}\quad
    \inlinepic{
    \begin{scope}[scale=0.5]
        \draw[Blue](2,0)arc(0:-180:0.75cm and 0.5cm);
        \draw[Blue](-2,0)arc(-180:0:0.75cm and 0.5cm);
        \draw[Blue,dotted](2,0)arc(0:180:0.75cm and 0.5cm); 
        \draw[Blue,dotted](-2,0)arc(180:0:0.75cm and 0.5cm);
        \draw[Blue] (-2,0)arc(180:0:0.75cm);
        \draw[Blue] (2,0)arc(0:180:0.75cm);
        \fill[Blue] (-1.25,0.6) circle (2pt);
        \draw[Maroon,dashed](0,0)arc(0:90:2cm and 1.5cm);
    \end{scope}}\quad - \quad \inlinepic{
    \begin{scope}[scale=0.5]
        \draw[Blue](2,0)arc(0:-180:0.75cm and 0.5cm);
        \draw[Blue](-2,0)arc(-180:0:0.75cm and 0.5cm);
        \draw[Blue,dotted](2,0)arc(0:180:0.75cm and 0.5cm); 
        \draw[Blue,dotted](-2,0)arc(180:0:0.75cm and 0.5cm);
        \draw[Blue] (-2,0)arc(180:0:0.75cm);
        \draw[Blue] (2,0)arc(0:180:0.75cm);
        \fill[Blue] (1.25,0.6) circle (2pt);
        \draw[Maroon,dashed](0,0)arc(0:90:2cm and 1.5cm);
    \end{scope}  }  \quad
\end{pmatrix}
\]
In cup foams the red disks lie lower than the blue disks. Let $\epsilon_1$ be the sign in the equation above, and $\epsilon_2$ be the sign of the relation:
\[
\inlinepic{
    \begin{scope}[scale=0.5]
        \draw[Blue](-2,0)arc(-180:0:0.75cm and 0.5cm);
        \draw[Blue,dotted](-2,0)arc(180:0:0.75cm and 0.5cm);
        \draw[Blue] (-2,0)arc(180:0:0.75cm);
        \draw[Maroon,dashed](0,0)arc(0:90:2cm and 1.5cm);
    \end{scope}  } \quad = \: \epsilon_2 \quad
\inlinepic{
    \begin{scope}[scale=0.5]
        \draw[Blue](-2,0)arc(-180:0:0.75cm and 0.5cm);
        \draw[Blue,dotted](-2,0)arc(180:0:0.75cm and 0.5cm);
        \draw[Blue] (-2,0.5)arc(180:0:0.75cm);
        \draw[Blue] (-0.5,0)--(-0.5,0.5) (-2,0)--(-2,0.5);
        \draw[Maroon,dashed](0,0)arc(0:90:2.5cm and 1cm);
        \begin{scope}
            \clip (-1.25,0.5)circle (0.75cm);
            \draw[Maroon,ultra thick](0,0)arc(0:90:2.5cm and 1cm);
        \end{scope}
    \end{scope}    
}
\]
Then 
\[d_S(1)= \epsilon_1 \epsilon_2 \begin{pmatrix}\quad
    \inlinepic{
    \begin{scope}[scale=0.5]
        \draw[Blue](2,0)arc(0:-180:0.75cm and 0.5cm);
        \draw[Blue,dotted](2,0)arc(0:180:0.75cm and 0.5cm); 
        \draw[Blue] (2,0)arc(0:180:0.75cm);
        \fill[Blue] (-1.25,0) circle (2pt);
    \draw[Blue](-2,0)arc(-180:0:0.75cm and 0.5cm);
        \draw[Blue,dotted](-2,0)arc(180:0:0.75cm and 0.5cm);
        \draw[Blue] (-2,0.5)arc(180:0:0.75cm);
        \draw[Blue] (-0.5,0)--(-0.5,0.5) (-2,0)--(-2,0.5);
        \draw[Maroon,dashed](0,0)arc(0:90:2.5cm and 1cm);
        \begin{scope}
            \clip (-1.25,0.5)circle (0.75cm);
            \draw[Maroon,ultra thick](0,0)arc(0:90:2.5cm and 1cm);
        \end{scope}
    \end{scope}}\quad - \quad \inlinepic{
    \begin{scope}[scale=0.5]
        \draw[Blue](2,0)arc(0:-180:0.75cm and 0.5cm);
        \draw[Blue,dotted](2,0)arc(0:180:0.75cm and 0.5cm); 
        \draw[Blue] (2,0)arc(0:180:0.75cm);
    \draw[Blue](-2,0)arc(-180:0:0.75cm and 0.5cm);
        \draw[Blue,dotted](-2,0)arc(180:0:0.75cm and 0.5cm);
        \draw[Blue] (-2,0.5)arc(180:0:0.75cm);
        \draw[Blue] (-0.5,0)--(-0.5,0.5) (-2,0)--(-2,0.5);
        \draw[Maroon,dashed](0,0)arc(0:90:2.5cm and 1cm);
        \begin{scope}
            \clip (-1.25,0.5)circle (0.75cm);
            \draw[Maroon,ultra thick](0,0)arc(0:90:2.5cm and 1cm);
        \end{scope}
        \fill[Blue] (1.25,0.6) circle(2pt);
    \end{scope}  }  \quad
\end{pmatrix}
\]\\
We also need to move the points to the top using the migration relation
\[d_S(1)= \epsilon_1 \epsilon_2 \begin{pmatrix}\quad - \quad
    \inlinepic{
    \begin{scope}[scale=0.5]
        \draw[Blue](2,0)arc(0:-180:0.75cm and 0.5cm);
        \draw[Blue,dotted](2,0)arc(0:180:0.75cm and 0.5cm); 
        \draw[Blue] (2,0)arc(0:180:0.75cm);
        \fill[Blue] (-1.25,1.1) circle (2pt);
    \draw[Blue](-2,0)arc(-180:0:0.75cm and 0.5cm);
        \draw[Blue,dotted](-2,0)arc(180:0:0.75cm and 0.5cm);
        \draw[Blue] (-2,0.5)arc(180:0:0.75cm);
        \draw[Blue] (-0.5,0)--(-0.5,0.5) (-2,0)--(-2,0.5);
        \draw[Maroon,dashed](0,0)arc(0:90:2.5cm and 1cm);
        \begin{scope}
            \clip (-1.25,0.5)circle (0.75cm);
            \draw[Maroon,ultra thick](0,0)arc(0:90:2.5cm and 1cm);
        \end{scope}
    \end{scope}}\quad - \quad \inlinepic{
    \begin{scope}[scale=0.5]
        \draw[Blue](2,0)arc(0:-180:0.75cm and 0.5cm);
        \draw[Blue,dotted](2,0)arc(0:180:0.75cm and 0.5cm); 
        \draw[Blue] (2,0)arc(0:180:0.75cm);
    \draw[Blue](-2,0)arc(-180:0:0.75cm and 0.5cm);
        \draw[Blue,dotted](-2,0)arc(180:0:0.75cm and 0.5cm);
        \draw[Blue] (-2,0.5)arc(180:0:0.75cm);
        \draw[Blue] (-0.5,0)--(-0.5,0.5) (-2,0)--(-2,0.5);
        \draw[Maroon,dashed](0,0)arc(0:90:2.5cm and 1cm);
        \begin{scope}
            \clip (-1.25,0.5)circle (0.75cm);
            \draw[Maroon,ultra thick](0,0)arc(0:90:2.5cm and 1cm);
        \end{scope}
        \fill[Blue] (1.25,0.6) circle(2pt);
    \end{scope}  }  \quad
\end{pmatrix}
\]
The red surfaces we lifted up at the beginning should also go back under the blue disks, this multiplies both foams in the equation by the same sign $\epsilon'_R$, and we have what we need which is 
\[d_S(1)=-\epsilon_1 \epsilon_2 \epsilon_R \epsilon'_R( X\otimes 1 + 1 \otimes X)\]
\\
Now we compute $d_S(X)$ in the same way. \\
\begin{equation*}  
\begin{split}   
\inlinepic{
    \begin{scope}[scale=0.3]
        \draw[Blue](2,0)arc(0:-180:0.75cm and 0.5cm);
        \draw[Blue](-2,0)arc(-180:0:0.75cm and 0.5cm);
        \draw[Blue,dotted](2,0)arc(0:180:0.75cm and 0.5cm); 
        \draw[Blue,dotted](-2,0)arc(180:0:0.75cm and 0.5cm);
        \draw[Blue] (0.5,0)arc(0:180:0.5cm);
        \draw[Blue] (2,0)arc(0:180:2cm);
        \draw[Maroon,dashed](0,0)arc(0:90:2cm and 1.5cm);
        \begin{scope}
            \clip (2,0)arc(0:180:2cm)--(-0.5,0)arc(180:0:0.5cm);
            \draw[Maroon,ultra thick](0,0)arc(0:90:2cm and 1.5cm);   
        \end{scope}
        \fill[Blue] (0,1.25) circle (2pt);
    \end{scope}
}
\quad & = \quad \epsilon_1 \begin{pmatrix}\quad
    \inlinepic{
    \begin{scope}[scale=0.5]
        \draw[Blue](2,0)arc(0:-180:0.75cm and 0.5cm);
        \draw[Blue](-2,0)arc(-180:0:0.75cm and 0.5cm);
        \draw[Blue,dotted](2,0)arc(0:180:0.75cm and 0.5cm); 
        \draw[Blue,dotted](-2,0)arc(180:0:0.75cm and 0.5cm);
        \draw[Blue] (-2,0)arc(180:0:0.75cm);
        \draw[Blue] (2,0)arc(0:180:0.75cm);
        \fill[Blue] (-1.25,0.6) circle (2pt);
        \fill[Blue] (1.25,0.6) circle (2pt);
        \draw[Maroon,dashed](0,0)arc(0:90:2cm and 1.5cm);
    \end{scope}}\quad - \quad \inlinepic{
    \begin{scope}[scale=0.5]
        \draw[Blue](2,0)arc(0:-180:0.75cm and 0.5cm);
        \draw[Blue](-2,0)arc(-180:0:0.75cm and 0.5cm);
        \draw[Blue,dotted](2,0)arc(0:180:0.75cm and 0.5cm); 
        \draw[Blue,dotted](-2,0)arc(180:0:0.75cm and 0.5cm);
        \draw[Blue] (-2,0)arc(180:0:0.75cm);
        \draw[Blue] (2,0)arc(0:180:0.75cm);
        \fill[Blue] (1.15,0.5) circle (2pt);
        \fill[Blue] (1.35,0.5) circle (2pt);
        \draw[Maroon,dashed](0,0)arc(0:90:2cm and 1.5cm);
    \end{scope}  }  \quad
\end{pmatrix}\\
 &= \quad \epsilon_1\begin{pmatrix}\quad
 \inlinepic{
    \begin{scope}[scale=0.5]
        \draw[Blue](2,0)arc(0:-180:0.75cm and 0.5cm);
        \draw[Blue](-2,0)arc(-180:0:0.75cm and 0.5cm);
        \draw[Blue,dotted](2,0)arc(0:180:0.75cm and 0.5cm); 
        \draw[Blue,dotted](-2,0)arc(180:0:0.75cm and 0.5cm);
        \draw[Blue] (-2,0)arc(180:0:0.75cm);
        \draw[Blue] (2,0)arc(0:180:0.75cm);
        \fill[Blue] (-1.25,0.6) circle (2pt);
        \fill[Blue] (1.25,0.6) circle (2pt);
        \draw[Maroon,dashed](0,0)arc(0:90:2cm and 1.5cm);
    \end{scope}}    \quad
 \end{pmatrix}\\
&= -\epsilon_1 \epsilon_2  \quad \begin{pmatrix} \quad \inlinepic{
    \begin{scope}[scale=0.5]
        \draw[Blue](2,0)arc(0:-180:0.75cm and 0.5cm);
        \draw[Blue,dotted](2,0)arc(0:180:0.75cm and 0.5cm); 
        \draw[Blue] (2,0)arc(0:180:0.75cm);
        \fill[Blue] (-1.25,1.1) circle (2pt);
    \draw[Blue](-2,0)arc(-180:0:0.75cm and 0.5cm);
        \draw[Blue,dotted](-2,0)arc(180:0:0.75cm and 0.5cm);
        \draw[Blue] (-2,0.5)arc(180:0:0.75cm);
        \draw[Blue] (-0.5,0)--(-0.5,0.5) (-2,0)--(-2,0.5);
        \draw[Maroon,dashed](0,0)arc(0:90:2.5cm and 1cm);
        \begin{scope}
            \clip (-1.25,0.5)circle (0.75cm);
            \draw[Maroon,ultra thick](0,0)arc(0:90:2.5cm and 1cm);
        \end{scope}
        \fill[Blue](1.25,0.6) circle (2pt);
    \end{scope}} \quad
 \end{pmatrix}
\end{split}
\end{equation*}
Therefore \[d_S(X)=-\epsilon_1 \epsilon_2 \epsilon_R \epsilon'_R( X\otimes X)\]
\\ the map $d_S=\pm \Delta$ as needed.

Similarly, there is a certain sign $\epsilon_R$ such that:\vspace{0.5cm}\\
\[d_M(1\otimes 1)=\epsilon_R 
\inlinepic{
    \begin{scope}[scale=0.3]
        \draw[Blue](-2,0) arc(-180:0:2cm and 0.5cm);
        \draw[Blue,dotted](-2,0) arc(180:0:2cm and 0.5cm);
        \draw[Blue] (0.4,2)arc(180:0:0.8cm);
        \draw[Blue,xscale=-1] (0.4,2)arc(180:0:0.8cm);
        \draw[Blue](2,0)--(2,2)
        (-2,0)--(-2,2);
        \draw[Blue,looseness=2](0.4,2)--(0.4,1.5)to[out=-90,in=-90](-0.4,1.5)--(-0.4,2);
        \draw[Maroon,dashed](0,-0.5)to[looseness=2,out=90,in=180](2,2);
        \begin{scope}
            \clip (-2,0) arc(-180:0:2cm and 0.5cm)--(2,2)arc(0:180:0.8cm)--(0.4,2)--(0.4,1.5)to[out=-90,in=-90](-0.4,1.5)--(-0.4,2)--cycle;
            \draw[Maroon,thick](0,-0.5)to[looseness=2,out=90,in=180](2,2);
        \end{scope}
    \end{scope}
}
\]\vspace{0.5cm}\\
Then pushing the red surface downwards multiplies by a sign $\epsilon_1$\vspace{0.5cm}\\
\begin{equation*}
\centering
\begin{split}
d_M(1\otimes 1)&= \epsilon_R \epsilon_1
\inlinepic{
    \begin{scope}[scale=0.3]
        \draw[Blue](-2,0) arc(-180:0:2cm and 0.5cm);
        \draw[Blue,dotted](-2,0) arc(180:0:2cm and 0.5cm);
        \draw[Blue] (0.4,2)arc(180:0:0.8cm);
        \draw[Blue,xscale=-1] (0.4,2)arc(180:0:0.8cm);
        \draw[Blue](2,0)--(2,2)
        (-2,0)--(-2,2);
        \draw[Blue,looseness=2](0.4,2)--(0.4,1.5)to[out=-90,in=-90](-0.4,1.5)--(-0.4,2);
        \draw[Maroon,dashed](0,-0.5)to[looseness=1,out=90,in=180](2,2);
        \begin{scope}
            \clip (-2,0) arc(-180:0:2cm and 0.5cm)--(2,2)arc(0:180:0.8cm)--(0.4,2)--(0.4,1.5)to[out=-90,in=-90](-0.4,1.5)--(-0.4,2)--cycle;
            \draw[Maroon,thick](0,-0.5)to[looseness=1,out=90,in=180](2,2);
        \end{scope}
    \end{scope}
}\vspace{0.5cm}\\
&= \epsilon_R \epsilon_1 1.
\end{split}
\end{equation*}
and
\begin{equation*}
\centering
\begin{split}
d_M(X\otimes 1)&= \epsilon_R \epsilon_1
\inlinepic{
    \begin{scope}[scale=0.3]
        \draw[Blue](-2,0) arc(-180:0:2cm and 0.5cm);
        \draw[Blue,dotted](-2,0) arc(180:0:2cm and 0.5cm);
        \draw[Blue] (0.4,2)arc(180:0:0.8cm);
        \draw[Blue,xscale=-1] (0.4,2)arc(180:0:0.8cm);
        \fill[Blue](-1.2,2.2) circle(2pt);
        \draw[Blue](2,0)--(2,2)
        (-2,0)--(-2,2);
        \draw[Blue,looseness=2](0.4,2)--(0.4,1.5)to[out=-90,in=-90](-0.4,1.5)--(-0.4,2);
        \draw[Maroon,dashed](0,-0.5)to[looseness=1,out=90,in=180](2,2);
        \begin{scope}
            \clip (-2,0) arc(-180:0:2cm and 0.5cm)--(2,2)arc(0:180:0.8cm)--(0.4,2)--(0.4,1.5)to[out=-90,in=-90](-0.4,1.5)--(-0.4,2)--cycle;
            \draw[Maroon,thick](0,-0.5)to[looseness=1,out=90,in=180](2,2);
        \end{scope}
    \end{scope}
}\vspace{0.5cm}\\
&= \epsilon_R \epsilon_1 X. 
\end{split}
\end{equation*}

\begin{equation*}
\centering
\begin{split}
d_M(1\otimes X)&= \epsilon_R \epsilon_1
\inlinepic{
    \begin{scope}[scale=0.3]
        \draw[Blue](-2,0) arc(-180:0:2cm and 0.5cm);
        \draw[Blue,dotted](-2,0) arc(180:0:2cm and 0.5cm);
        \draw[Blue] (0.4,2)arc(180:0:0.8cm);
        \draw[Blue,xscale=-1] (0.4,2)arc(180:0:0.8cm);
        \fill[Blue](1.2,2.2) circle(2pt);
        \draw[Blue](2,0)--(2,2)
        (-2,0)--(-2,2);
        \draw[Blue,looseness=2](0.4,2)--(0.4,1.5)to[out=-90,in=-90](-0.4,1.5)--(-0.4,2);
        \draw[Maroon,dashed](0,-0.5)to[looseness=1,out=90,in=180](2,2);
        \begin{scope}
            \clip (-2,0) arc(-180:0:2cm and 0.5cm)--(2,2)arc(0:180:0.8cm)--(0.4,2)--(0.4,1.5)to[out=-90,in=-90](-0.4,1.5)--(-0.4,2)--cycle;
            \draw[Maroon,thick](0,-0.5)to[looseness=1,out=90,in=180](2,2);
        \end{scope}
    \end{scope}
}\vspace{0.5cm}\\
&= \epsilon_R \epsilon_1 X.
\end{split}
\end{equation*}
\begin{equation*}
\centering
\begin{split}
d_M(X\otimes X)&= \epsilon_R \epsilon_1
\inlinepic{
    \begin{scope}[scale=0.3]
        \draw[Blue](-2,0) arc(-180:0:2cm and 0.5cm);
        \draw[Blue,dotted](-2,0) arc(180:0:2cm and 0.5cm);
        \draw[Blue] (0.4,2)arc(180:0:0.8cm);
        \draw[Blue,xscale=-1] (0.4,2)arc(180:0:0.8cm);
        \fill[Blue](-1.2,2.2) circle(2pt);
        \fill[Blue](1.2,2.2) circle(2pt);
        \draw[Blue](2,0)--(2,2)
        (-2,0)--(-2,2);
        \draw[Blue,looseness=2](0.4,2)--(0.4,1.5)to[out=-90,in=-90](-0.4,1.5)--(-0.4,2);
        \draw[Maroon,dashed](0,-0.5)to[looseness=1,out=90,in=180](2,2);
        \begin{scope}
            \clip (-2,0) arc(-180:0:2cm and 0.5cm)--(2,2)arc(0:180:0.8cm)--(0.4,2)--(0.4,1.5)to[out=-90,in=-90](-0.4,1.5)--(-0.4,2)--cycle;
            \draw[Maroon,thick](0,-0.5)to[looseness=1,out=90,in=180](2,2);
        \end{scope}
    \end{scope}
}\vspace{0.5cm}\\
&= 0. 
\end{split}
\end{equation*}
So $d_M=\pm m$.
\par
The other cases can be proven in the same way, using different relations each time. The point is to level the red disks and points correctly to find back the cup foams, and to make sure that the signs are coherent.
\end{proof}

\section{Lee Khovanov homology}\label{sec: Lee}

\begin{deff}
Let $D$ be a diagram of a tangle $T$. The Lee Khovanov complex $CKh'(D)$ is the quotient of the complex $CKh(D)\otimes \mathbb{Q}$ by the relations $X_1=-X_2=1$ . \\ If $D$ is the diagram of a link $L$, then the Lee Khovanov homology of the link $\mathcal{L}Kh(L)$ is the homology of the complex $\mathcal{G}(CKh'(D))$.

\end{deff}
The complex $CKh'(D)$ is no longer graded: $deg(X^2)=-2 \neq 1 =deg(1)$.
\\
An interesting fact about this invariant is that it is rather "weak", it can be fully determined by it's linking matrix. We follow \cite{BarNatan2006} for a proof of this fact.
\subsection{Lee degeneration}
\begin{deff}
    Let $D$ be a diagram of a tangle, and $\Gamma$ a smoothing of $D$. A coloring of $\Gamma$ is a choice of a label $a$ or $b$ for each of:
    \begin{itemize}
        \item the two segments in a neighborhood of each 0-smoothing.
        \item two of the four segments at each side of the binding in a 1-smoothings, chosen as in the figure below 
        \begin{figure}[h]
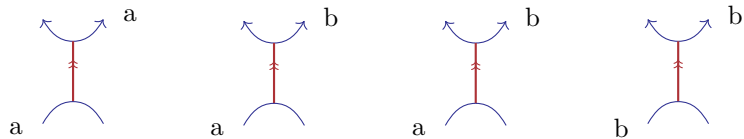

            \centering
            \inlinepic{
                \begin{scope}[scale=2]
                    \draw[Blue,->] (0,0.2) to[in=-120,out=0] (60:0.4);
                    \draw[Blue,->] (0,0.2) to[in=-60,out=180] (120:0.4);
                    \draw[Maroon,thick,decoration={ markings, mark=at position 0.7 with {\arrow[scale=0.5]{>>}}},postaction={decorate}] (0,-0.2)--(0,0.2);
                    \draw[Blue] (-60:0.4) to[in=0, out=120] (0,-0.2);
                    \draw[Blue] (-120:0.4) to[in=180, out=60] (0,-0.2);
                \end{scope}
                \node at (0.75,0.75) {a};                
                \node at (-0.75,-0.75) {a};
            }
            \hspace{0.5cm}
            \inlinepic{
                \begin{scope}[scale=2]
                    \draw[Blue,->] (0,0.2) to[in=-120,out=0] (60:0.4);
                    \draw[Blue,->] (0,0.2) to[in=-60,out=180] (120:0.4);
                    \draw[Maroon,thick,decoration={ markings, mark=at position 0.7 with {\arrow[scale=0.5]{>>}}},postaction={decorate}] (0,-0.2)--(0,0.2);
                    \draw[Blue] (-60:0.4) to[in=0, out=120] (0,-0.2);
                    \draw[Blue] (-120:0.4) to[in=180, out=60] (0,-0.2);
                \end{scope}
                \node at (0.75,0.75) {b};
                \node at (-0.75,-0.75) {a};
            }
            \hspace{0.5cm}
            \inlinepic{
                \begin{scope}[scale=2]
                    \draw[Blue,->] (0,0.2) to[in=-120,out=0] (60:0.4);
                    \draw[Blue,->] (0,0.2) to[in=-60,out=180] (120:0.4);
                    \draw[Maroon,thick,decoration={ markings, mark=at position 0.7 with {\arrow[scale=0.5]{>>}}},postaction={decorate}] (0,-0.2)--(0,0.2);
                    \draw[Blue] (-60:0.4) to[in=0, out=120] (0,-0.2);
                    \draw[Blue] (-120:0.4) to[in=180, out=60] (0,-0.2);
                \end{scope}
                \node at (0.75,0.75) {b};
                \node at (-0.75,-0.75) {a};
            }
            \hspace{0.5cm}
            \inlinepic{
                \begin{scope}[scale=2]
                    \draw[Blue,->] (0,0.2) to[in=-120,out=0] (60:0.4);
                    \draw[Blue,->] (0,0.2) to[in=-60,out=180] (120:0.4);
                    \draw[Maroon,thick,decoration={ markings, mark=at position 0.7 with {\arrow[scale=0.5]{>>}}},postaction={decorate}] (0,-0.2)--(0,0.2);
                    \draw[Blue] (-60:0.4) to[in=0, out=120] (0,-0.2);
                    \draw[Blue] (-120:0.4) to[in=180, out=60] (0,-0.2);
                \end{scope}
                \node at (0.75,0.75) {b};
                \node at (-0.75,-0.75) {b};
            }
            \caption{The four possible colorings near a 1-smoothing.}
        \end{figure}
    \end{itemize}
    
\end{deff}
\begin{thh} \label{leeDegeneration}
     Let $D $ be the diagram of a tangle $T$, and $n$ the number of its components.\\
     Within the appropriate category (that will be defined later), $CKh'(D)$ is homotopy equivalent to a complex with $2^n$ generators and with vanishing differentials.
\end{thh}

\subsubsection{The Karoubi envelope}
\begin{deff}
    Let $\mathcal{C}$ be a category. An endomorphism $p: O\rightarrow O$ of some object$ O$ in $\mathcal{C}$ is called a projection if $p\circ p = p$. The Karoubi envelope $\mathcal{K}ar(\mathcal{C})$ of $\mathcal{C}$ is the category whose objects are ordered pairs $(O,p)$ where $O$ is an object in $\mathcal{C}$ and $ p: O\rightarrow O$ is a projection. If $(O_1, p_1)$ and $(O_2, p_2)$ are two such pairs, the set of morphisms in  $\mathcal{K}ar(\mathcal{C})$ from $(O_1, p_1)$ to $(O_2, p_2)$ is the collection of all morphisms $f : O_1 \rightarrow O_2$ in $\mathcal{C}$ for which $f = f \circ p_1 = p_2 \circ f$. An object $(O,p)$ in $\mathcal{K}ar(\mathcal{C})$ may also be denoted by $im(p)$.\\The composition of morphisms in $\mathcal{K}ar(\mathcal{C})$ is defined in the obvious way (by composing the corresponding morphisms in $\mathcal{C}$). The identity automorphism of an object $(O,p)$ in $\mathcal{K}ar(\mathcal{C})$ is $p$ itself.
\end{deff}
There is an embedding functor $\mathcal{I}$ from a category $\mathcal{C}$ into its Karoubi envelope $\mathcal{K}ar(\mathcal{C})$, defined by sending an object $O\in \mathcal{C}$ to the pair $(O,Id)=O'$, and we have  $mor_{\mathcal{K}ar(\mathcal{C})}(O'_1 ,O'_2) = mor_{\mathcal{C}}(O_1,O_2)$. We can identify the objects of $C$ with their image by $\mathcal{I}$.

\par 
Suppose, from now on, that the category $\mathcal{C}$ is pre-additive and that the direct sum of objects makes sense in $\mathcal{C}$.
\begin{prop}\label{projections}
     Let $p: O \rightarrow O$ be an endomorphism in $\mathcal{C}$.
     \begin{itemize}
         \item If $p$ is a projection then so is $1- p$.
         \item In this case $O = im(p)\oplus im(1- p)$ in ${\mathcal{K}ar(\mathcal{C})}$.
     \end{itemize}

\end{prop}

\subsubsection{Application to Lee's theory}

Let $\Gamma$ be a smoothing of a diagram $D$. Place a point somewhere on a blue component of $\Gamma\times[0,1]$. This gives an automorphism $P$ that is idempotent $P\circ P=Id$. \vspace{0.5cm}\\
$P\circ P=$\quad \inlinepic{\begin{scope}[scale=0.25]
        \BId{2}{}{}
    \end{scope}} \: $=$ \: $ (X_1+X_2)$\quad
    \inlinepic{\begin{scope}[scale=0.25]
        \BId{1}{}{}
    \end{scope}}\: $-$ \: $ X_1X_2$\quad
    \inlinepic{\begin{scope}[scale=0.25]
        \BId{}{}{}
    \end{scope}}\: $=$ \: \inlinepic{\begin{scope}[scale=0.25]
        \BId{}{}{}\end{scope}}
    \:$= Id$.
\vspace{0.5cm}\\
 Let $a=\frac{1}{2}(Id+P)$ and $b=\frac{1}{2}(Id-P)$. Then:\begin{itemize}
    \item $a$ and $b$ are projections: $a\circ a=a$ and $b\circ b=b$;
    \item $a$ and $b$ are complementary $a+b=Id$;
    \item $a$ and $b$ are disjoint $ab=ba=0$;
    \item $a$ and $b$ are eigenprojections of $P$: $P\circ a = a$ and $P \circ b = -b$.
\end{itemize}

Then $(\Gamma,a)$ and $(\Gamma,b)$ are elements of $\mathcal{K}ar(Kob)$, and $\Gamma $ is isomorphic to $(\Gamma,a) \oplus (\Gamma,b)$.\\
Given a coloring of $\Gamma$ and applying $a$ and $b$ where they appear, we get a projection of $\Gamma$, and thus an element of $\mathcal{K}ar(Kob)$, and a decomposition of $\Gamma$. \par

\begin{example}We have the following decompositions:\vspace{0.5cm}\\
\[        \inlinepic{
            \begin{scope}[scale=0.3]
                \Arrow
            \end{scope}
        }
        \: = \:
        \inlinepic{
            \begin{scope}[scale=0.3]
                \Arrow
            \end{scope}
            \node at (0.2,-0.3) {\small a};

        } \: \oplus \: 
        \inlinepic{
            \begin{scope}[scale=0.3]
                \Arrow
            \end{scope}
            \node at (0.2,-0.3) {\small b};
        }    \] \\ 
        \[     
        \inlinepic{
            \begin{scope}[scale=0.3]
                \Arrow
            \end{scope}
        }\:
        \inlinepic{
            \begin{scope}[scale=0.3]
                \Arrow
            \end{scope}
        }
        \: = \:
        \inlinepic{
            \begin{scope}[scale=0.3]
                \Arrow
            \end{scope}
            \node at (0.2,-0.3) {\small a};
        }\hspace{-0.1cm}    \inlinepic{
            \begin{scope}[scale=0.3]
                \Arrow
            \end{scope}
            \node at (0.2,-0.3) {\small a};
        }\: \oplus \:
        \inlinepic{
            \begin{scope}[scale=0.3]
                \Arrow
            \end{scope}
            \node at (0.2,-0.3) {\small a};
        }\hspace{-0.1cm}    \inlinepic{
            \begin{scope}[scale=0.3]
                \Arrow
            \end{scope}
            \node at (0.2,-0.3) {\small b};
        }\: \oplus \:
            \inlinepic{
            \begin{scope}[scale=0.3]
                \Arrow
            \end{scope}
            \node at (0.2,-0.3) {\small b};
        }\hspace{-0.1cm}    \inlinepic{
            \begin{scope}[scale=0.3]
                \Arrow
            \end{scope}
            \node at (0.2,-0.3) {\small a};
        }\: \oplus \:
            \inlinepic{
            \begin{scope}[scale=0.3]
                \Arrow
            \end{scope}
            \node at (0.2,-0.3) {\small b};
        }\hspace{-0.1cm}    \inlinepic{
            \begin{scope}[scale=0.3]
                \Arrow
            \end{scope}
            \node at (0.2,-0.3) {\small b};
        }\: 
\]
\[
\inlinepic{
                \begin{scope}[scale=1.3]
                    \draw[Blue,->] (0,0.2) to[in=-120,out=0] (60:0.4);
                    \draw[Blue,->] (0,0.2) to[in=-60,out=180] (120:0.4);
                    \draw[Maroon,thick,decoration={ markings, mark=at position 0.7 with {\arrow[scale=0.5]{>>}}},postaction={decorate}] (0,-0.2)--(0,0.2);
                    \draw[Blue] (-60:0.4) to[in=0, out=120] (0,-0.2);
                    \draw[Blue] (-120:0.4) to[in=180, out=60] (0,-0.2);
                \end{scope}
            }\: = \:
            \inlinepic{
                \begin{scope}[scale=1.3]
                    \draw[Blue,->] (0,0.2) to[in=-120,out=0] (60:0.4);
                    \draw[Blue,->] (0,0.2) to[in=-60,out=180] (120:0.4);
                    \draw[Maroon,thick,decoration={ markings, mark=at position 0.7 with {\arrow[scale=0.5]{>>}}},postaction={decorate}] (0,-0.2)--(0,0.2);
                    \draw[Blue] (-60:0.4) to[in=0, out=120] (0,-0.2);
                    \draw[Blue] (-120:0.4) to[in=180, out=60] (0,-0.2);
                \end{scope}
                \node at (0.5,0.5) {a};                
                \node at (-0.5,-0.5) {a};
            }
            \: \oplus \:
            \inlinepic{
                \begin{scope}[scale=1.3]
                    \draw[Blue,->] (0,0.2) to[in=-120,out=0] (60:0.4);
                    \draw[Blue,->] (0,0.2) to[in=-60,out=180] (120:0.4);
                    \draw[Maroon,thick,decoration={ markings, mark=at position 0.7 with {\arrow[scale=0.5]{>>}}},postaction={decorate}] (0,-0.2)--(0,0.2);
                    \draw[Blue] (-60:0.4) to[in=0, out=120] (0,-0.2);
                    \draw[Blue] (-120:0.4) to[in=180, out=60] (0,-0.2);
                \end{scope}
                \node at (0.5,0.5) {b};
                \node at (-0.5,-0.5) {a};
            }
            \: \oplus \: 
            \inlinepic{
                \begin{scope}[scale=1.3]
                    \draw[Blue,->] (0,0.2) to[in=-120,out=0] (60:0.4);
                    \draw[Blue,->] (0,0.2) to[in=-60,out=180] (120:0.4);
                    \draw[Maroon,thick,decoration={ markings, mark=at position 0.7 with {\arrow[scale=0.5]{>>}}},postaction={decorate}] (0,-0.2)--(0,0.2);
                    \draw[Blue] (-60:0.4) to[in=0, out=120] (0,-0.2);
                    \draw[Blue] (-120:0.4) to[in=180, out=60] (0,-0.2);
                \end{scope}
                \node at (0.5,0.5) {b};
                \node at (-0.5,-0.5) {a};
            }
            \: \oplus \:
            \inlinepic{
                \begin{scope}[scale=1.3]
                    \draw[Blue,->] (0,0.2) to[in=-120,out=0] (60:0.4);
                    \draw[Blue,->] (0,0.2) to[in=-60,out=180] (120:0.4);
                    \draw[Maroon,thick,decoration={ markings, mark=at position 0.7 with {\arrow[scale=0.5]{>>}}},postaction={decorate}] (0,-0.2)--(0,0.2);
                    \draw[Blue] (-60:0.4) to[in=0, out=120] (0,-0.2);
                    \draw[Blue] (-120:0.4) to[in=180, out=60] (0,-0.2);
                \end{scope}
                \node at (0.5,0.5) {b};
                \node at (-0.5,-0.5) {b};
            }
\]

\end{example}

Now let us look at the Lee Khovanov complex of a simple crossing. 
\[\begin{tikzcd}[row sep=large, column sep=large,ampersand replacement=\&]
    CKh(\inlinepic{
        \begin{scope}[scale=0.8]
            \draw[->] (-60:0.4)--(120:0.4);
            \fill[white] (0,0) circle (0.1cm);
            \draw[->] (-120:0.4)--(60:0.4);
  
        \end{scope}
    }): \&
    0 \arrow[r]
    \& \inlinepic{
            \begin{scope}[scale=0.3]
                \Arrow
            \end{scope}
        }\:
        \inlinepic{
            \begin{scope}[scale=0.3]
                \Arrow
            \end{scope}
        }
    \arrow[r,"d"]
    \arrow[d,"{\sim}"{rotate=90,yshift=-0.1cm,xshift=-0.1cm}]
    \& \inlinepic{
    \begin{scope}[scale=1.3]
                    \draw[Blue,->] (0,0.2) to[in=-120,out=0] (60:0.4);
                    \draw[Blue,->] (0,0.2) to[in=-60,out=180] (120:0.4);
                    \draw[Maroon,thick,decoration={ markings, mark=at position 0.7 with {\arrow[scale=0.5]{>>}}},postaction={decorate}] (0,-0.2)--(0,0.2);
                    \draw[Blue] (-60:0.4) to[in=0, out=120] (0,-0.2);
                    \draw[Blue] (-120:0.4) to[in=180, out=60] (0,-0.2);
    \end{scope}}
    \arrow[r]
    \arrow[d,"{\sim}"{rotate=90,yshift=-0.1cm,xshift=-0.1cm}]
    \& 0\\
    \&
    \&
    {
        \begin{pmatrix}
            \:
        \inlinepic{
            \begin{scope}[scale=0.2]
                \Arrow
            \end{scope}
            \node at (0.1,-0.2) {\tiny a};
        }\hspace{-0.1cm}    \inlinepic{
            \begin{scope}[scale=0.2]
                \Arrow
            \end{scope}
            \node at (0.1,-0.2) {\tiny a};
        }\: \vspace{0.3cm} \\ \:
        \inlinepic{
            \begin{scope}[scale=0.2]
                \Arrow
            \end{scope}
            \node at (0.1,-0.2) {\tiny a};
        }\hspace{-0.1cm}    \inlinepic{
            \begin{scope}[scale=0.2]
                \Arrow
            \end{scope}
            \node at (0.1,-0.2) {\tiny b};
        }\vspace{0.3cm}\\
            \inlinepic{
            \begin{scope}[scale=0.2]
                \Arrow
            \end{scope}
            \node at (0.1,-0.2) {\tiny b};
        }\hspace{-0.1cm}    \inlinepic{
            \begin{scope}[scale=0.2]
                \Arrow
            \end{scope}
            \node at (0.1,-0.2) {\tiny a};
        }\vspace{0.3cm}\\
            \inlinepic{
            \begin{scope}[scale=0.2]
                \Arrow
            \end{scope}
            \node at (0.1,-0.2) {\tiny b};
        }\hspace{-0.1cm}    \inlinepic{
            \begin{scope}[scale=0.2]
                \Arrow
            \end{scope}
            \node at (0.1,-0.2) {\tiny b};
        }
        \end{pmatrix}
    }\arrow[r,"M"]
    \&
        \begin{pmatrix}
            
\            \inlinepic{
                \begin{scope}[scale=0.8]
                    \draw[Blue,->] (0,0.2) to[in=-120,out=0] (60:0.4);
                    \draw[Blue,->] (0,0.2) to[in=-60,out=180] (120:0.4);
                    \draw[Maroon,thick,decoration={ markings, mark=at position 0.7 with {\arrow[scale=0.5]{>>}}},postaction={decorate}] (0,-0.2)--(0,0.2);
                    \draw[Blue] (-60:0.4) to[in=0, out=120] (0,-0.2);
                    \draw[Blue] (-120:0.4) to[in=180, out=60] (0,-0.2);
                \end{scope}
                \node at (0.3,0.3) {\tiny a};                
                \node at (-0.3,-0.3) {\tiny a};
            }
            \\
            \inlinepic{
                \begin{scope}[scale=0.8]
                    \draw[Blue,->] (0,0.2) to[in=-120,out=0] (60:0.4);
                    \draw[Blue,->] (0,0.2) to[in=-60,out=180] (120:0.4);
                    \draw[Maroon,thick,decoration={ markings, mark=at position 0.7 with {\arrow[scale=0.5]{>>}}},postaction={decorate}] (0,-0.2)--(0,0.2);
                    \draw[Blue] (-60:0.4) to[in=0, out=120] (0,-0.2);
                    \draw[Blue] (-120:0.4) to[in=180, out=60] (0,-0.2);
                \end{scope}
                \node at (0.3,0.3) {\tiny b};
                \node at (-0.3,-0.3) {\tiny a};
            }
            \\
            \inlinepic{
                \begin{scope}[scale=0.8]
                    \draw[Blue,->] (0,0.2) to[in=-120,out=0] (60:0.4);
                    \draw[Blue,->] (0,0.2) to[in=-60,out=180] (120:0.4);
                    \draw[Maroon,thick,decoration={ markings, mark=at position 0.7 with {\arrow[scale=0.5]{>>}}},postaction={decorate}] (0,-0.2)--(0,0.2);
                    \draw[Blue] (-60:0.4) to[in=0, out=120] (0,-0.2);
                    \draw[Blue] (-120:0.4) to[in=180, out=60] (0,-0.2);
                \end{scope}
                \node at (0.3,0.3) {\tiny b};
                \node at (-0.3,-0.3) {\tiny a};
            }
            \\
            \inlinepic{
                \begin{scope}[scale=0.8]
                    \draw[Blue,->] (0,0.2) to[in=-120,out=0] (60:0.4);
                    \draw[Blue,->] (0,0.2) to[in=-60,out=180] (120:0.4);
                    \draw[Maroon,thick,decoration={ markings, mark=at position 0.7 with {\arrow[scale=0.5]{>>}}},postaction={decorate}] (0,-0.2)--(0,0.2);
                    \draw[Blue] (-60:0.4) to[in=0, out=120] (0,-0.2);
                    \draw[Blue] (-120:0.4) to[in=180, out=60] (0,-0.2);
                \end{scope}
                \node at (0.3,0.3) {\tiny b};
                \node at (-0.3,-0.3) {\tiny b};
            }
        \end{pmatrix}  
\end{tikzcd}\]
$d=\quad
\inlinepic{
    \begin{scope}[scale=0.4]
        \saddleDown
    \end{scope}
}
$
\\
Note that in this setting $(X_1=-X_2=1)$ we have again the relation \inlinepic{ 
    \draw[Maroon,thick](0,0)--(0,-0.3);
    \draw[Blue](0,0)--(45:0.3) (0,0)--(135:0.3);
    \fill[Blue](45:0.15) circle(1pt);
} $=$ $-$ \inlinepic{ 
    \draw[Maroon,thick](0,0)--(0,-0.3);
    \draw[Blue](0,0)--(45:0.3) (0,0)--(135:0.3);
    \fill[Blue](135:0.15) circle(1pt);}. So $2$ \inlinepic{ 
    \draw[Maroon,thick](0,0)--(0,-0.3);
    \draw[Blue](0,0)--(45:0.3) (0,0)--(135:0.3);
    \node at ($(45:0.3)+(0.1,-0.1)$) {\tiny a};} $=$ \inlinepic{ 
    \draw[Maroon,thick](0,0)--(0,-0.3);
    \draw[Blue](0,0)--(45:0.3) (0,0)--(135:0.3);
    } $+$ \inlinepic{ 
    \draw[Maroon,thick](0,0)--(0,-0.3);
    \draw[Blue](0,0)--(45:0.3) (0,0)--(135:0.3);
    \fill[Blue](45:0.15) circle(1pt);
    } $=$ \inlinepic{ 
    \draw[Maroon,thick](0,0)--(0,-0.3);
    \draw[Blue](0,0)--(45:0.3) (0,0)--(135:0.3);
    } $-$ \inlinepic{ 
    \draw[Maroon,thick](0,0)--(0,-0.3);
    \draw[Blue](0,0)--(45:0.3) (0,0)--(135:0.3);
    \fill[Blue](135:0.15) circle(1pt);
    } $=$ $2$ \inlinepic{ 
    \draw[Maroon,thick](0,0)--(0,-0.3);
    \draw[Blue](0,0)--(45:0.3) (0,0)--(135:0.3);
    \node at ($(135:0.3)+(-0.1,-0.1)$) {\tiny b};}. \vspace{0.3cm}\\
    Let $(l_0,l_0',l_1,l_1')\in \{a,b\}^4$. Let $\bar{l}=Id-l$ so that $l \bar{l}=0$.\vspace{0.5cm} \\
    In the component $d: \: \inlinepic{
            \begin{scope}[scale=0.3]
                \Arrow
            \end{scope}
            \node at (0.2,-0.3) {\tiny $l_0$};
        }       \inlinepic{
            \begin{scope}[scale=0.3]
                \Arrow
            \end{scope}
            \node at (0.2,-0.3) {\tiny $l_1$};}
            \xrightarrow{\inlinepic{
                \begin{scope}[scale=0.3]
                    \saddleDown
                \end{scope}}
            }\inlinepic{
                \begin{scope}[scale=1.3]
                    \draw[Blue,->] (0,0.2) to[in=-120,out=0] (60:0.4);
                    \draw[Blue,->] (0,0.2) to[in=-60,out=180] (120:0.4);
                    \draw[Maroon,thick,decoration={ markings, mark=at position 0.7 with {\arrow[scale=0.5]{>>}}},postaction={decorate}] (0,-0.2)--(0,0.2);
                    \draw[Blue] (-60:0.4) to[in=0, out=120] (0,-0.2);
                    \draw[Blue] (-120:0.4) to[in=180, out=60] (0,-0.2);
                \end{scope}
                \node at (0.5,0.5) {\tiny $l_0'$};
                \node at (-0.5,-0.5) {\tiny $l_1'$};
            }$
 $l_0$ and $l_0'$ are in the same component. So if $l_0$ is different from $l_0'$ we have $d=0$. The same holds for $l_1$ and $l_1'$.\vspace{0.5cm}\\
 If $l_0=l_0'$ and $l_1=l_1'$ then $d=\inlinepic{
                \begin{scope}[scale=0.3]
                    \saddleDown
                    \node at (2,2) {\tiny $l_1$};
                    \node at (-1.5,1.5) {\tiny $l_0$};
                \end{scope}}=\inlinepic{
                \begin{scope}[scale=0.3]
                    \saddleDown
                    \node at (-1.5,2.1) {\tiny $l_0\bar{l_1}$};
                \end{scope}}.$ 
So $d=0$ unless $l_0$ is different from $l_1$. We have \vspace{0.5cm}\\ $M=
\begin{pmatrix}
    \: 0\:&0&0&0\vspace{0.2cm}\\
    0&
    \inlinepic{
                
                \begin{scope}[scale=0.3]
                    \saddleDown
                    \node at (-1.5,2.1) {\tiny a};
                \end{scope}}
    &0&0\\
    0&0&
                \inlinepic{
                \begin{scope}[scale=0.3]
                    \saddleDown
                    \node at (-1.5,2.1) {\tiny b};
                \end{scope}} \vspace{0.2cm}
    &0\\
    0&0&0&0
\end{pmatrix}
$
\vspace{0.5cm}\\
Therefore the complex $CKh(\:\inlinepic{
        \begin{scope}[scale=0.8]
            \draw[->] (-60:0.4)--(120:0.4);
            \fill[white] (0,0) circle (0.1cm);
            \draw[->] (-120:0.4)--(60:0.4);
  
        \end{scope}
    }\:)$ is isomorphic to the direct sum of 4 chain complexes, where the only ones with non vanishing differentials are \vspace{0.5cm}\\ \begin{center}
        
        $  \inlinepic{
            \begin{scope}[scale=0.3]
                \Arrow
            \end{scope}
            \node at (0.2,-0.3) {\tiny $a$};
        }       \inlinepic{
            \begin{scope}[scale=0.3]
                \Arrow
            \end{scope}
            \node at (0.2,-0.3) {\tiny $b$};}
            \xrightarrow{\inlinepic{
                \begin{scope}[scale=0.3]
                    \saddleDown
                    \node at (-1.5,2.1) {\tiny a};
                \end{scope}}
            }\inlinepic{
                \begin{scope}[scale=1.3]
                    \draw[Blue,->] (0,0.2) to[in=-120,out=0] (60:0.4);
                    \draw[Blue,->] (0,0.2) to[in=-60,out=180] (120:0.4);
                    \draw[Maroon,thick,decoration={ markings, mark=at position 0.7 with {\arrow[scale=0.5]{>>}}},postaction={decorate}] (0,-0.2)--(0,0.2);
                    \draw[Blue] (-60:0.4) to[in=0, out=120] (0,-0.2);
                    \draw[Blue] (-120:0.4) to[in=180, out=60] (0,-0.2);
                \end{scope}
                \node at (0.5,0.5) {\tiny $a$};
                \node at (-0.5,-0.5) {\tiny $b$};
            }$ \vspace{0.5cm}\quad and \quad
            $  \inlinepic{
            \begin{scope}[scale=0.3]
                \Arrow
            \end{scope}
            \node at (0.2,-0.3) {\tiny $b$};
        }       \inlinepic{
            \begin{scope}[scale=0.3]
                \Arrow
            \end{scope}
            \node at (0.2,-0.3) {\tiny $a$};}
            \xrightarrow{\inlinepic{
                \begin{scope}[scale=0.3]
                    \saddleDown
                    \node at (-1.5,2.1) {\tiny b};
                \end{scope}}
            }\inlinepic{
                \begin{scope}[scale=1.3]
                    \draw[Blue,->] (0,0.2) to[in=-120,out=0] (60:0.4);
                    \draw[Blue,->] (0,0.2) to[in=-60,out=180] (120:0.4);
                    \draw[Maroon,thick,decoration={ markings, mark=at position 0.7 with {\arrow[scale=0.5]{>>}}},postaction={decorate}] (0,-0.2)--(0,0.2);
                    \draw[Blue] (-60:0.4) to[in=0, out=120] (0,-0.2);
                    \draw[Blue] (-120:0.4) to[in=180, out=60] (0,-0.2);
                \end{scope}
                \node at (0.5,0.5) {\tiny $b$};
                \node at (-0.5,-0.5) {\tiny $a$};
            }$    \end{center}
        These two complexes are acyclic. In fact  the maps \inlinepic{
                \begin{scope}[scale=0.3]
                    \saddleDown
                    \node at (-1.5,2.1) {\tiny a};
                \end{scope}} and \inlinepic{
                \begin{scope}[scale=0.3]
                    \saddleDown
                    \node at (-1.5,2.1) {\tiny b};
                \end{scope}} are isomorphisms, with inverses respectively \inlinepic{
                \begin{scope}[scale=0.3]
                    \saddleUp
                    \node at (-1.5,-2.1) {\tiny a};
                \end{scope}} and \inlinepic{
                \begin{scope}[scale=0.3]
                    \saddleUp
                    \node at (-1.5,-2.1) {\tiny b};
                \end{scope}} (Neck cutting and bigon relations).
\vspace{0.5cm}\\

\[CKh(\:\inlinepic{
        \begin{scope}[scale=0.8]
            \draw[->] (-60:0.4)--(120:0.4);
            \fill[white] (0,0) circle (0.1cm);
            \draw[->] (-120:0.4)--(60:0.4);
  
        \end{scope}
    }\:) \simeq \inlinepic{
            \begin{scope}[scale=0.3]
                \Arrow
            \end{scope}
            \node at (0.2,-0.3) {\small a};
        }\hspace{-0.1cm}    \inlinepic{
            \begin{scope}[scale=0.3]
                \Arrow
            \end{scope}
            \node at (0.2,-0.3) {\small a};
        }\: \oplus \:
        \inlinepic{
            \begin{scope}[scale=0.3]
                \Arrow
            \end{scope}
            \node at (0.2,-0.3) {\small b};
        }\hspace{-0.1cm}    \inlinepic{
            \begin{scope}[scale=0.3]
                \Arrow
            \end{scope}
            \node at (0.2,-0.3) {\small b};
        }
            \xrightarrow{\quad0\quad}
        \inlinepic{
                \begin{scope}[scale=1]
                    \draw[Blue,->] (0,0.2) to[in=-120,out=0] (60:0.4);
                    \draw[Blue,->] (0,0.2) to[in=-60,out=180] (120:0.4);
                    \draw[Maroon,thick,decoration={ markings, mark=at position 0.7 with {\arrow[scale=0.5]{>>}}},postaction={decorate}] (0,-0.2)--(0,0.2);
                    \draw[Blue] (-60:0.4) to[in=0, out=120] (0,-0.2);
                    \draw[Blue] (-120:0.4) to[in=180, out=60] (0,-0.2);
                \end{scope}
                \node at (0.4,0.4) {a};
                \node at (-0.4,-0.4) {a};
            }
            \: \oplus \:
            \inlinepic{
                \begin{scope}[scale=1]
                    \draw[Blue,->] (0,0.2) to[in=-120,out=0] (60:0.4);
                    \draw[Blue,->] (0,0.2) to[in=-60,out=180] (120:0.4);
                    \draw[Maroon,thick,decoration={ markings, mark=at position 0.7 with {\arrow[scale=0.5]{>>}}},postaction={decorate}] (0,-0.2)--(0,0.2);
                    \draw[Blue] (-60:0.4) to[in=0, out=120] (0,-0.2);
                    \draw[Blue] (-120:0.4) to[in=180, out=60] (0,-0.2);
                \end{scope}
                \node at (0.4,0.4) {b};
                \node at (-0.4,-0.4) {b};
            }   
    \]

\par

The same works for a negative crossing, and constructing tangle diagrams by attaching one crossing at a time, we see that it also works in general. So $CKh'(D)$ is homotopy equivalent to a complex $C'$ with vanishing differentials. The complex $C'$ is generated by the smoothings of $D$ where the colors agree around each crossings. \\
Let $D_0$ be a colored smoothing of $D$, so an element of $\mathcal{K}ar(Kob)$. Suppose we have both $a$ and $b$ on the same blue edge of $D_0$. We know $ab=0$, therefore $D_0=0$ in $\mathcal{K}ar(Kob)$. \\ If two blue edges meeting at a trivalent vertex have the same label $a$ or $b$, then moving one to the other side inverts it, and we get $a$ and $b$ on the same edge. Again, $D_0=0$. \\ If we don't have either of those situations, then $D_0$ is nonzero (this can be proven by composing the associated projection by a well chosen cobordism, refer to Lemma \ref{Basis}). In this last case we say that the coloring of $D_0$ is cohesive. We proved that $CKh'(D)$ is homotopy equivalent to a complex $C'$ with generators the cohesively colored smoothings and with no differentials. What is left to do is to count these generators.\par
Let $L$ be a link with $n$ components, $D$ a diagram of $L$.\\
There is a one to one correspondance between cohesively colored smoothings of $D$ and $\{a,b\}^n$ as follows. 
We see elements of $\{a,b\}^n$ as a choice of labels $a$ or $b$ for each component of $L$.  
\[
\inlinepic{
        \begin{scope}[scale=1.5]
            \draw[->] (-60:0.4)--(120:0.4);
            \draw[->] (-120:0.4)--(60:0.4);
        \end{scope}
    \node at($(-60:0.4)+(0.2,0)$) {\small $l$};
    \node at($(-120:0.4)+(-0.2,0)$) {\small $l$};
        } \longleftrightarrow 
    \inlinepic{
            \begin{scope}[scale=0.4]
                \Arrow
            \end{scope}
            \node at (0.2,-0.3) {\small $l$};
        }    \inlinepic{
            \begin{scope}[scale=0.4]
                \Arrow
            \end{scope}
            \node at (0.2,-0.3) {\small $l$};
        }\hspace{2cm}
    \inlinepic{
        \begin{scope}[scale=1.5]
            \draw[->] (-60:0.4)--(120:0.4);
            \draw[->] (-120:0.4)--(60:0.4);
        \end{scope}
    \node at($(-60:0.4)+(0.2,0)$) {\small $l$};
    \node at($(-120:0.4)+(-0.2,0)$) {\small $\bar{l}$};
        } \longleftrightarrow
        \inlinepic{
                \begin{scope}[scale=1.5]
                    \draw[Blue,->] (0,0.2) to[in=-120,out=0] (60:0.4);
                    \draw[Blue,->] (0,0.2) to[in=-60,out=180] (120:0.4);
                    \draw[Maroon,thick,decoration={ markings, mark=at position 0.7 with {\arrow[scale=0.5]{>>}}},postaction={decorate}] (0,-0.2)--(0,0.2);
                    \draw[Blue] (-60:0.4) to[in=0, out=120] (0,-0.2);
                    \draw[Blue] (-120:0.4) to[in=180, out=60] (0,-0.2);
                \end{scope}
                \node at (0.4,0.4) {\small $\bar{l}$};
                \node at (-0.4,-0.4) {\small $\bar{l}$};
            }   \: \sim \: 
            \inlinepic{
                \begin{scope}[scale=1.5]
                    \draw[Blue,->] (0,0.2) to[in=-120,out=0] (60:0.4);
                    \draw[Blue,->] (0,0.2) to[in=-60,out=180] (120:0.4);
                    \draw[Maroon,thick,decoration={ markings, mark=at position 0.7 with {\arrow[scale=0.5]{>>}}},postaction={decorate}] (0,-0.2)--(0,0.2);
                    \draw[Blue] (-60:0.4) to[in=0, out=120] (0,-0.2);
                    \draw[Blue] (-120:0.4) to[in=180, out=60] (0,-0.2);
                \end{scope}
                \node at (0.4,-0.4) {\small $l$};
                \node at (-0.4,0.4) {\small $l$};
            }
\]

This completes the proof of the Theorem \ref{leeDegeneration}.

\section{Functoriality}\label{Funct}

\subsection{Embedded cobordisms}
\begin{deff}
    Let B be a set of points in $\mathbb{S}^2$, possibly empty.\\ Let $T$ and $T'$ be two tangles embedded in $\mathbb{D}^3$ with boundary $B$.\\A cobordism from $T$ to $T'$ is a surface $C$ embedded in $\mathbb{D}^3\times[0,1]\subset\mathbb{R}^4$ such that:\\
        \begin{itemize}
        \item There exists an $\epsilon_0 > 0$ such that $ C \cap \mathbb{D}^3\times[0,\epsilon_0]=T\times[0,\epsilon_0]$.
        \item There exists an $\epsilon_1 < 1$ such that $C\cap \mathbb{D}^3\times[\epsilon_1,1]=(-T')\times[\epsilon_1,1]$
        \item $C \cap (\mathbb{S}^2\times[0,1])$ is the boundary $B$ times $[0,1]$
        \item The boundary of $C$ is $C\cap (\mathbb{S}^2\times [0,1])$.
    \end{itemize}
\end{deff}
\begin{deff}
    Let B be a set of points in $\mathbb{S}^3$, possibly empty. We define
    $Cob^4(B)$ as the category whose objects are tangles in $\mathbb{D}^3$ with boundary $B$ and whose morphisms are cobordisms between tangles up to boundary preserving isotopies.
    
\end{deff}
\begin{deff}
    A movie is a finite subset $S=\{t_1,t_2,\ldots,t_n\}\subset (0,1)$ together with a
 link diagram $D_t$ for each $t\in[0,1]-S$, such that for $t \in (t_i,t_{i+1})$, $D_t$ is an isotopy
 of planar diagrams, and for small $\epsilon$, the diagrams $D_{t_i-\epsilon}
 ,D_{t_{i}+\epsilon}$ are related by
 either a Reidemeister move or one of the Morse moves represented below
 \begin{figure}[h]
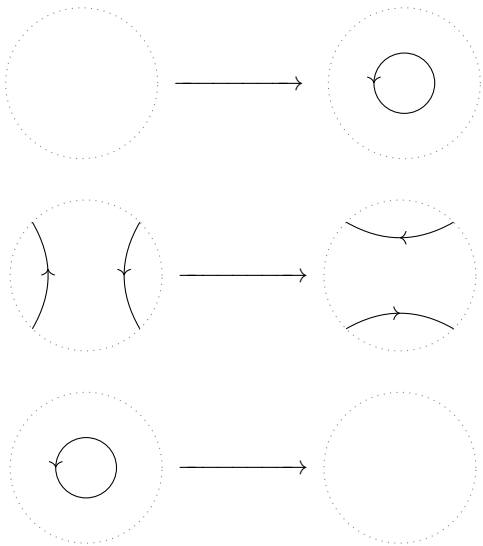

    \centering
    \inlinepic{
        \draw[gray,dotted](0,0)circle(1cm);

    }\: $\xrightarrow{\hspace{1.5cm}}$ \: \inlinepic{
        \draw[gray,dotted](0,0)circle(1cm);
        \draw[->-](0,0)circle(0.4cm);
    }\vspace{0.5cm}\\
    \inlinepic{
        \draw[gray,dotted](0,0)circle(1cm);
        \draw[bend right,->-](45:1)to(-45:1);
        \draw[bend left,-<-](135:1)to(-135:1);

    }\: $\xrightarrow{\hspace{1.5cm}}$\: \inlinepic{
        \draw[gray,dotted](0,0)circle(1cm);
        \draw[bend left,->-](45:1)to(135:1);
        \draw[bend left,->-](-135:1)to(-45:1);

    }\vspace{0.5cm}\\
    \inlinepic{
        \draw[gray,dotted](0,0)circle(1cm);
        \draw[->-](0,0)circle(0.4cm);

    }\: $\xrightarrow{\hspace{1.5cm}}$\: \inlinepic{

        \draw[gray,dotted](0,0)circle(1cm);

    }\vspace{0.5cm}
    \caption{Morse moves}
 \end{figure}
\end{deff}
\begin{example}
    Let $C$ be a cobordism in $\mathbb{D}^3\times [0,1]$ from $T$ to $T'$, then up to small isotopy we can suppose that $C\cap (\mathbb{D}^3\times{t})$ is a link for all $t$ other than a finite set $S\subset (0,1)$. around the singularities we have one of the elementary cobordisms represented in Figure \ref{elementaryCob}. Therefore any cobordism can be represented by a movie. 

\end{example}

\begin{figure}[h]
    \centering
    \includegraphics[width=0.5\textwidth]{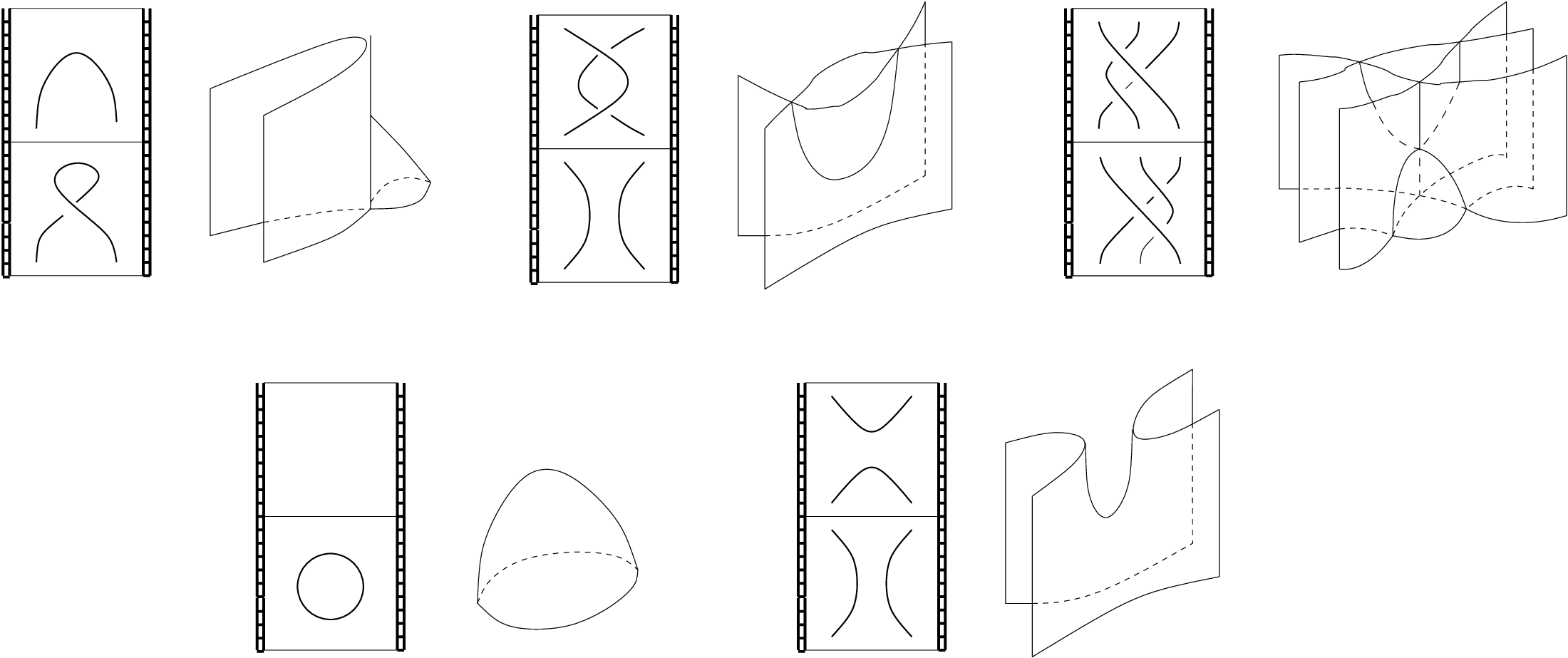}
    \caption{Elementary cobordisms taken from \cite{Carter1997}, all clips are reversible.}
    \label{elementaryCob}
\end{figure}

\begin{thh} \label{MM}
    Two movies represent the same cobordism up to boundary preserving isotopy if and only if they are related by a finite sequence of movies moves, shown in the Figure \ref{fig:MM}.
\end{thh}

We have defined $CKh(T)$ for objects of the category $Cob^4(B)$, now we want to define the image of morphisms by $CKh$, to hopefully get a functor from the category $Cob^4(B)$ into the category $Kob_{/h}(B)$.
Naturally, we start by defining $CKh(C)$ for $C$ an elementary cobordism. Then an arbitrary cobordism is a sequence of movies of elementary cobordisms $C_1, \ldots C_n$. We take $CKh(C)$ to be the composition   $CKh(C_1)\circ \ldots\circ CKh(C_n)$.\\
We have two types of elementary cobordisms:
\begin{itemize}
    \item Morse moves: a cup, a cap or a saddle. In this case the crossings of $T$ and those of $T'$ can be naturally identified, and consequently the vertices of their cubes of resolution too. For each vertex $v$ of the cube of resolution, the Morse move gives rise to a natural cobordism $D_v \rightarrow D'_v$. Consequently, we get a map from $CKh(D)$ to $CKh(D')$. This is a chain map. For example if $C$ is a saddle we have \\
    \[
    \inlinepic{
    \begin{scope}[scale=0.4]
    \CommFive
    \end{scope}
    }=\inlinepic{
    \begin{scope}[scale=0.4]  
    \CommSix
    \end{scope}
    }
    \]
    
    \item If $C$ corresponds to a Reidemeister move:\\
    \begin{itemize}
        \item If the oriented Reidemeister move is $R_1+$, $ R_1-$, $R_2+$, $ R_2-$, $R_3+$ or $R3-$, as in the proof of Theorem \ref{invariance}, then we associate to it the homotopy equivalence map we used in that proof (except for $R3+$ where we take the opposite map). 
        \item All other oriented Reidemeister moves are generated by the ones mentioned above \cite{GeneratingSets}. Let $R$ be a Reidemeister move, Then we choose an expression $R=R_{i_1}\circ R_{i_2}\circ \ldots \circ R_{i_r} $ in terms of the generators $R_1+$, $ R_1-$, $R_2+$, $ R_2-$ and $R_3$, and we set $CKh(R)=CKh(R_{i_1}) \circ CKh(R_{i_2})\circ \ldots \circ CKh(R_{i_r})$. We also choose the representation of the inverse of $R^{-1}=R_{i_r}^{-1}\circ R_{i_{r-1}}^{-1}\circ \ldots \circ R_{i_1}^{-1}$. An example is provided in Figure \ref{fig:CyclicRThree} and \ref{fig:NRThree}.
    \end{itemize}
\begin{figure}[h]
    \centering
    \[
    \begin{tikzcd}[column sep=huge,row sep=1cm, nodes={inner sep=0.5cm}]
        \vcenter{\hbox{\includegraphics[width=2.25cm]{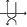}}} \arrow[leftrightarrow,r,"R2-"]&
        \vcenter{\hbox{\includegraphics[width=2.5cm]{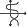}}} \arrow[leftrightarrow,r,"R2+"]\arrow[d, phantom,""{coordinate,name=Z}]&
        \vcenter{\hbox{\includegraphics[width=2.5cm]{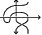}}}  \arrow[leftrightarrow,dll, "R3+"swap ,rounded corners,  to path={--([xshift=2ex]\tikztostart.east)|-(Z) [near end]\tikztonodes-|([xshift=-2ex]\tikztotarget.west)--(\tikztotarget)}] \\
        \vcenter{\hbox{\includegraphics[width=2.5cm]{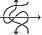}}} \arrow[leftrightarrow,r,"R2+"]&
        \vcenter{\hbox{\includegraphics[width=2.5cm]{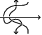}}} \arrow[leftrightarrow,r,"R2-"]&
        \vcenter{\hbox{\includegraphics[width=2.5cm]{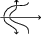}}}
    \end{tikzcd}\]
    \caption{An example of the expression of the $R3c$ move in terms of the moves in the generating set (R3c for cyclic).}
    \label{fig:CyclicRThree}
\end{figure}
\begin{figure}[h]
    \centering    
    \[
    \begin{tikzcd}[column sep=huge,row sep=1cm, nodes={inner sep=0.5cm}]
        \vcenter{\hbox{\includegraphics[width=2cm]{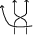}}} \arrow[leftrightarrow,r,"R2+"]&
        \vcenter{\hbox{\includegraphics[width=2cm]{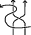}}} \arrow[leftrightarrow,r,"R2+"]\arrow[d, phantom,""{coordinate,name=Z}]&
        \vcenter{\hbox{\includegraphics[width=2cm]{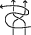}}}  \arrow[leftrightarrow,dll, "R3-"swap ,rounded corners,  to path={--([xshift=2ex]\tikztostart.east)|-(Z) [near end]\tikztonodes-|([xshift=-2ex]\tikztotarget.west)--(\tikztotarget)}] \\
        \vcenter{\hbox{\includegraphics[width=2cm]{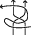}}} \arrow[leftrightarrow,r,"R2+"]&
        \vcenter{\hbox{\includegraphics[width=2cm]{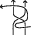}}} \arrow[leftrightarrow,r,"R2+"]&
        \vcenter{\hbox{\includegraphics[width=2cm]{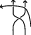}}}
    \end{tikzcd}\]
    \caption{Example of the expression the $R3n$ move in terms of the moves of the generating set (All the crossings here are negative).}
    \label{fig:NRThree}
\end{figure}
\end{itemize}

\begin{thh}\label{functoriality}
    The map $CKh$ defined above is a functor from the category $Cob^4$ to $Kob_{/h}$.
\end{thh}
\subsection{Proof of functoriality}
To prove Theorem \ref{functoriality}, we check that $CKh$ is invariant with respect to movie moves. In theory one could compute the morphisms associated to each movie move and prove the result manually, but this would be tedious given the high  number of crossings in some of the movie moves for, and that we have to consider all orientations, significantly increasing the number of morphisms we have to check. So instead, we first use an argument of $Kh$-simplicity to prove functoriality up to sign, as in \cite{BarNatan2005}.
\begin{deff}
    We say that a tangle diagram $D$ is Kh–simple if every degree 0 homotopy equivalence of $CKh(D)$ with itself is homotopic to a $\pm 1$ multiple of the identity. 
\end{deff}
\begin{lemma}\label{pairings}
Tangle diagrams with no closed components and no crossings are Kh-simple. Such diagrams are called pairings (See Figure \ref{fig:pairing}). 
\begin{figure}[h]
    \centering
    \begin{tikzpicture}[scale=0.6]
        \draw[gray,dotted] (0,0) circle (2cm);
        \draw[->-](45:2)to[out=-135,in=150](-30:2);
        \draw[->-](-80:2)to[out=100,in=-20](160:2);
        \draw[->-](130:2)to[out=-50,in=-110,looseness=0.7](70:2);
    \end{tikzpicture}
    \caption{An example of a pairing diagram.}
    \label{fig:pairing}
\end{figure}
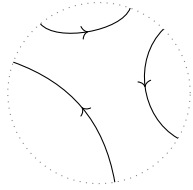
\end{lemma}

\begin{figure}
    \centering
    
    \begin{subfigure}{1\textwidth}
        \centering
        \includegraphics[width=0.7\textwidth]{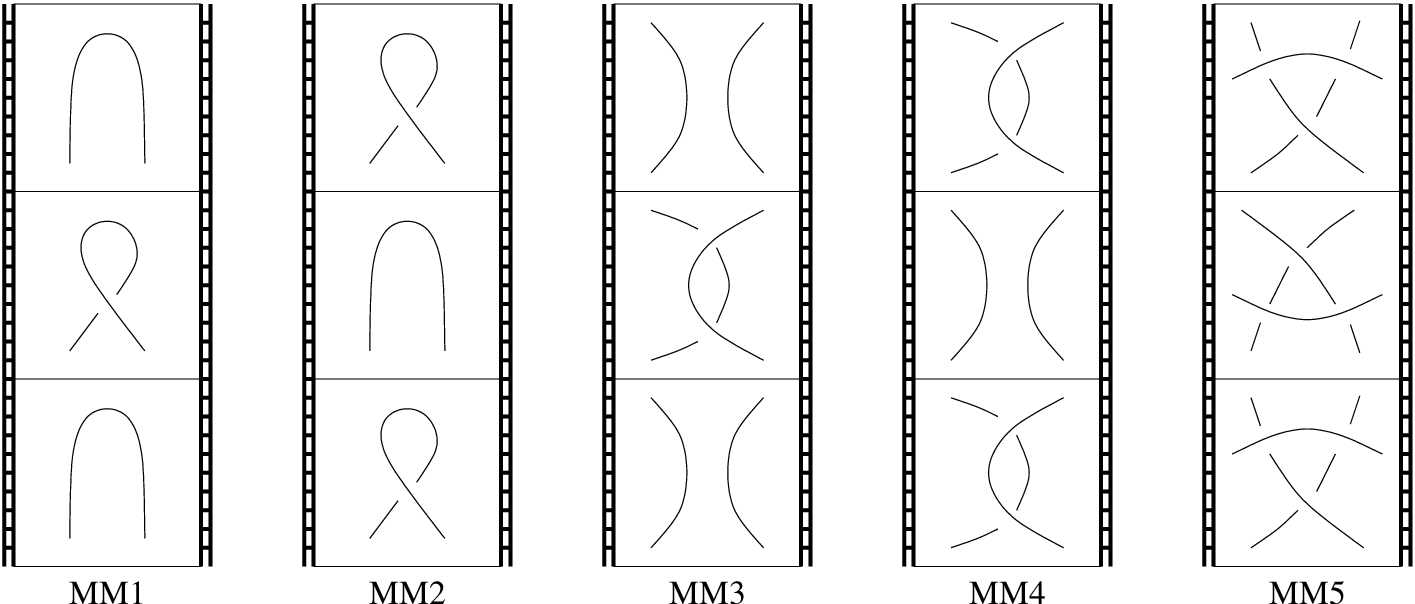}
        \caption{Movie moves 1–5, equivalent to “do nothing” identity clips.}
        \label{fig:MM1to5}
    \end{subfigure}

    \vspace{1.2cm} 
    \begin{subfigure}{1\textwidth}
        \centering
        \includegraphics[width=1\textwidth]{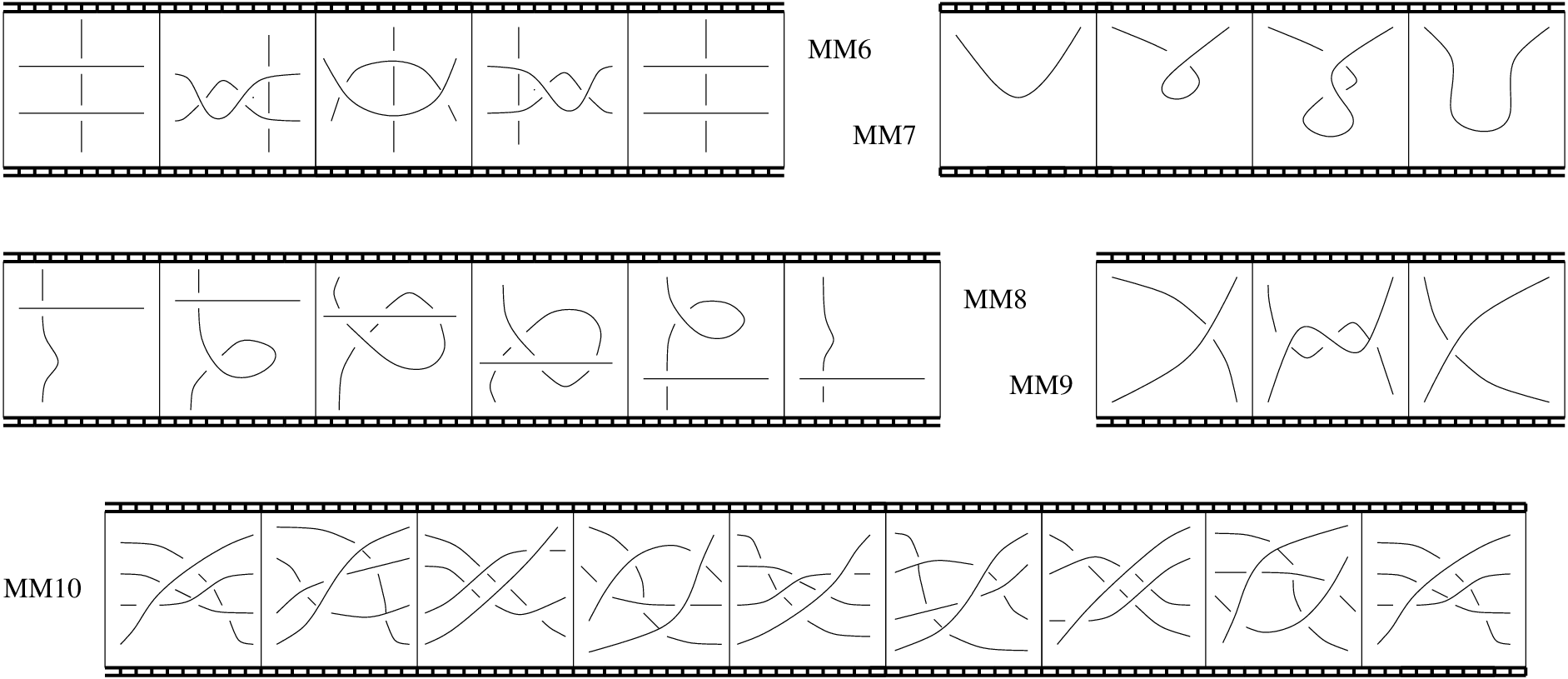}
        \caption{Movie moves 6–10, reversible circular clips (equivalent to identity clips).}
        \label{fig:MM6to10}
    \end{subfigure}

    \vspace{1.2cm} 
    \begin{subfigure}{1\textwidth}
        \centering
        \includegraphics[width=1\textwidth]{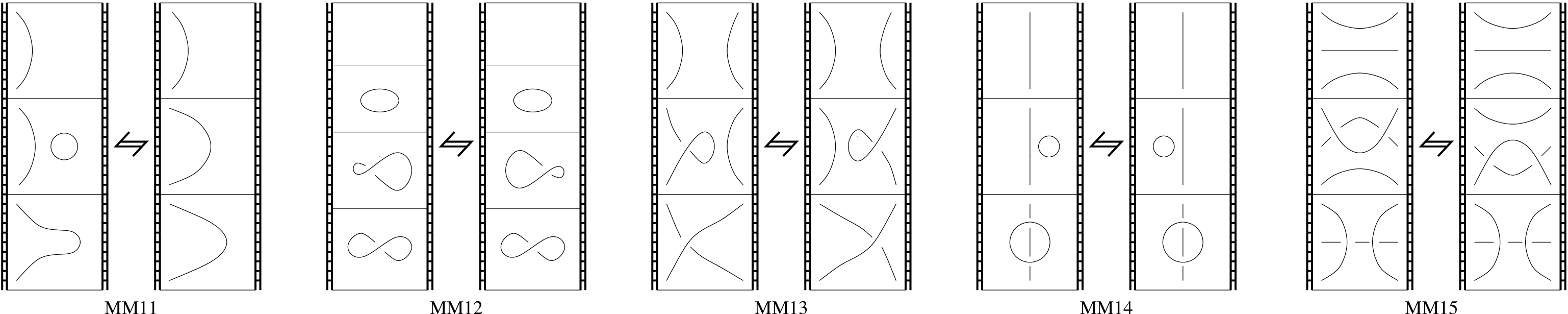}
        \caption{Movie moves 11–15, non-reversible clips (can be read both top-down and bottom-up).}
        \label{fig:MM11to15}
    \end{subfigure}
    \vspace{0.5cm} 
    \caption{Movie moves as in \cite{Carter1997}.}
    \label{fig:MM}
\end{figure}


\begin{proof}
    Let $D$ be a pairing. Then $CKh(D)$ is just the one step complex with $D$ at height 0 and no differentials. A degree 0 homotopy equivalence $C$ of $Ckh(D)$ with itself is a formal $\mathbb{Z}[X_1,X_2]$-linear combination of degree 0 cobordisms with top and bottom equal to $D$. \\
    Let us take one such cobordism and call it $\Sigma$. let $\Sigma_b$ its blue subsurface, $p$ the number of points and $s$ the number of squares in $\Sigma$. Let $P$ be a monomial in $\mathbb{Z}[X_1,X_2]$ of degree $d$.
The degree of $P\Sigma$ is $deg(P\Sigma)=\chi(\Sigma_b)-2(p+2s+d) - \frac{1}{2}|B|\leq \chi(\Sigma_b)  - \frac{1}{2}|B|$.\begin{enumerate}
    \item If $\Sigma$ is the identity cobordism and $P=1$ then $deg(P\Sigma)=0$

    \item If we add closed components, red or blue and  without points, we just multiply  the initial cobordism by an integer. This does not change the degree. We can ignore closed components.
    \item If we add handles to $\Sigma_b$, $\chi(\Sigma_b)$ decreases  $\chi(\Sigma_b)< \frac{1}{2}|B|$ and $deg(P\Sigma)$ is negative. $\Sigma_b$ has to be the identity map.
    \item Adding points, squares or taking $deg(P)>0$ decreases the degree. Thus $\Sigma$ does not have any decorations on it, and $P$ is an integer.
    \item We can have red surfaces in the cobordism since it doesn't change the degree. Using the neck cutting relation we can break the red surface into disks (without decorations). Then we can remove the disks using the bubble relations.
    \item Finally, we find that $\Sigma$ has to be a multiple of the identity. Therefore so does $C$ the homotopy equivalence. Since $C$ has to be invertible, it has to be a $\pm1$ multiple of the identity.
\end{enumerate}

\end{proof}
\begin{lemma}\label{isotopic}
     If a tangle diagram $D$ is Kh-simple and a tangle diagram $D^{\prime}$ represents an isotopic tangle, then $D^{\prime}$ is also Kh–simple.
\end{lemma}
\begin{proof}

By the invariance of $CKh$, the complexes $CKh(D')$ and $CKh(D)$ are homotopy equivalent. Choose a homotopy equivalence $F: KCKh(D') \rightarrow$ $CKh(D)$ between the two (with homotopy inverse $G: CKh(D) \rightarrow CKh(D')$), and assume $\alpha'$ is a degree 0 automorphism of $CK h(D')$. Since $D$ is $Kh$-simple, we know that $\alpha:=F \alpha' G$ is homotopic to $\pm I$. Therefore $\alpha' \sim G F \alpha' G F=G \alpha F \sim \pm G F \sim$ $\pm I$, and $D'$ is also $K h$-simple. 

\end{proof}

Let $D$ be a tangle diagram and let $DX$ be a tangle diagram obtained by adding an extra crossing somewhere along the boundary of $D$, like in Figure \ref{fig:DX}.
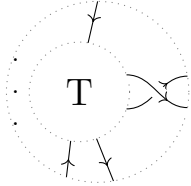
\begin{figure}[h]
    \centering
    \begin{tikzpicture}[scale=0.6]
        \draw[gray,dotted](0,0)circle(2cm);
        \draw[gray,dotted](-0.4,0)circle(1.1cm);
        \node at (-0.4,0) {\Large T};
        \draw[->-](90:2)--($(-0.4,0)+(80:1.1)$);
        \draw[-<-](-80:2)--($(-0.4,0)+(-70:1.1)$);
        \draw[->-](-110:2)--($(-0.4,0)+(-100:1.1)$);
        \node[align=center] at (-1.8,0) {.\\.\\.};
        \draw[postaction={decorate},
  decoration={markings,
    mark=at position 0.7 with {\arrow{>}}
  }]($(-0.4,0)+(-20:1.1)$)to[out=0,in=180](10:2);
        \draw[line width=0.2cm,white]($(-0.4,0)+(20:1.1)$)to[out=0,in=180](-10:2);
        \draw[postaction={decorate},
  decoration={markings,
    mark=at position 0.7 with {\arrow{>}}
  }]($(-0.4,0)+(20:1.1)$)to[out=0,in=180](-10:2);
    \end{tikzpicture}
    \caption{The diagram $DX$ from a diagram $D$}
    \label{fig:DX}
\end{figure}

 We can add the inverse crossing $X^{-1}$ alongside $X$. Then $D$ and $DXX^{-1}$ are related by an $R_2$ move, and represent isotopic tangles. We will use this fact later. 
\begin{lemma}\label{DXD}
    If $DX$ is $Kh$-simple then so is $D$.
\end{lemma}
\begin{proof}

Let $\alpha: CKh(D) \rightarrow CKh(D)$ be a degree 0 automorphism. We adjoin a crossing $X$ on the right to $D$ and to $\alpha$ to get an automorphism $\alpha I_X: CK h(D X) \rightarrow CK h(D X)$ (where $I_X$ denotes the identity automorphism of $CK h(X))$. \\As $D X$ is $K h$-simple, $\alpha I_X$ is homotopy equivalent to $ \pm I$. We now adjoin the inverse $X^{-1}$ of $X$ and find that $\alpha I_{X X^{-1}}: CK h\left(DX X^{-1}\right) \rightarrow CK h\left(D X X^{-1}\right) $ is homotopy equivalent to $\pm I$. But $X X^{-1}$ is differs from $\asymp$ by merely a Reidemeister move. Let $F: CK h(\asymp) \rightarrow CK h\left(X X^{-1}\right)$ be the homotopy equivalence between $CK h(\asymp)$ and $CK h\left(X X^{-1}\right)$ (with up-to-homotopy inverse $G: CK h\left(X X^{-1}\right) \rightarrow CK h(\asymp)$ ). Now $\alpha=\alpha I_{CK h(\asymp)} \sim \alpha(G \circ F)=\left(I_{CK h(D)} G\right) \circ\left(\alpha I_{X X^{-1}}\right) \circ\left(I_{CK h(D)} F\right) \sim \pm\left(I_{CK h(D)} G\right) \circ I \circ$ $\left(I_{CK h(D)} F\right)= \pm I_{CK h(D)}(G \circ F) \sim \pm I$, and $D$ is $CK h-$ simple.

\end{proof}

\begin{lemma}\label{DDX}
    If $D$ is $Kh$-simple then so is $DX$.
\end{lemma}
\begin{proof}
    Assume $D$ is $Kh$-simple. Using Lemma \ref{isotopic} so is $DXX^{-1}$. Using Lemma \ref{DXD}, we can drop the  crossing  $X^{-1}$. We find that $DX$ is $Kh$-simple.
\end{proof}

\begin{proof} (of Theorem \ref{functoriality})\par
    Again, we want to check that $CKh$ respects movie moves $MM_1$ to $MM15$ from Figure \ref{fig:MM}. We break this into a few steps.
\par 
\textbf{Step 1: }$MM1$ to $MM5$.\\  The movie moves corresponding to $R1+$, $R1-$, $R2+$, $R2-$, $R3+$ and $R3-$ verify $MM1$ through $MM5$. This is precisely the content of the proof of Theorem \ref{invariance}. The way we choose the maps for other oriented Reidemeister moves ensures that the result remains true for them as well. 
\par
\textbf{Step 2:} $MM6$ to $MM10$ up to sign.\\The reversible circular movie moves $MM6$ to $MM10$ represent degree 0 homotopy equivalences with the same start and end frame. The diagrams in these movies have no closed components, so peeling off the crossings one by one and using Lemma \ref{DDX} we prove that they are $Kh$-simple. Therefore the homotopy equivalence corresponding to each of these movie moves is $\pm Id$. \par

\textbf{Step 3:} $MM11$ to $MM15$. \\ Since these movies have a small number of crossings we can compute the maps explicitely and compare them.\\
\begin{itemize}
    \item \textbf{MM11:} Going down the left side of $MM11$ we have the composition
        \[
        {\inlinepic{\begin{scope}[scale=0.2]
            \SaddleKinda
        \end{scope}
            }}\quad
        \circ \quad
        {\inlinepic{
            \begin{scope}[scale=0.2]
                \draw[Blue,->-](-1.5,0.5)--(1.5,3.5);
                \draw[Blue,-<-](1.5,-0.5)--(-1.5,-3.5);
                \draw[Blue](-1.5,0.5)--(-1.5,-3.5) (1.5,-0.5)--(1.5,3.5);
                \begin{scope}[xshift=2cm,yshift=-2cm]
                    \draw[draw=Blue]
                    (1,0) arc (0:180:1cm);     
                    \draw[Blue,->-] (-1,0) arc  (-180:0 :1cm and 0.5cm); 
                    \draw[draw=Blue, dotted] (1,0) arc (0:180:1cm and 0.5cm) ;
                \end{scope}
            \end{scope}
            }
        }\quad=\quad
        {\inlinepic{\begin{scope}[scale=0.2]
            \MMElevenId
        \end{scope}
            }}
        \] which is isotopic to the identity.\\
        Going up MM11 is symmetrical.
    \item  \textbf{MM12:}
    \\
We will need the homotopy maps associated to the Reidemeister moves $R1'+$. We use the representations in Figure \ref{fig:ROnePlusP} and \ref{fig:ROneMinusP}. 
\begin{figure}[h]
    \centering
    \includegraphics[width=6cm]{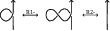}
    \caption{$R1'+$ as a sequence of $R1-$ and $R2-$.}
    \label{fig:ROnePlusP}
\end{figure}
\begin{figure}[h]
    \centering
    \includegraphics[width=6cm]{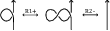}
    \caption{$R1'-$ as a sequence of $R1+$ and $R2-$.}
    \label{fig:ROneMinusP}
\end{figure}
\vspace{0.5cm}\\
Note that the choice of order of crossings in the definition of the Khovanov complex of a diagram affects the signs of the resulting morphisms. For example looking at the morphism associated to $R2+$, swapping the order of two crossings gives minus the identity on the oriented resolution. \\
Two choices of numbering of crossing of a diagram $D$ result in two complexes $CKh(D)_1$ and $CKh(D)_2$ which are related by a diagonal map made of $\pm Id$. We always choose this map so that it is minus the identity on the oriented resolution.\\

\[\begin{tikzcd}[row sep=2cm,column sep=1.5cm]
    \llbracket
    \inlinepic{
        \begin{scope}[scale=0.4]
            \ROneLeft{+}
        \end{scope}
    }
    \rrbracket:&
    \inlinepic{
        \begin{scope}[scale=0.4]
            \Circle{+}
            \begin{scope}[xshift=1cm]
                \Arrow
            \end{scope}
        \end{scope}
    } \arrow[r]
    \arrow[d,xshift=0.1cm]
    \arrow[d,phantom,dash,xshift=2.5cm, "
    {
        -\:\inlinepic{
            \begin{scope}[scale=0.3]
                \BRCupE{1}{+}{}{}
                \begin{scope}[xshift=3cm,yshift=-2cm]
                    \BIdE{+}{+}{-}
                \end{scope}
            \end{scope}
        } \: + \: 
                \inlinepic{
            \begin{scope}[scale=0.3]
                \BRCupE{}{+}{}{}
                \begin{scope}[xshift=3cm,yshift=-2cm]
                    \BIdE{1}{+}{-}
                \end{scope}
            \end{scope}}
    }"]
    & 
    \inlinepic{
        \begin{scope}[scale=0.4]
            \PLeft
        \end{scope}
    }
    \\
    &
    \inlinepic{
        \begin{scope}[scale=0.4]
            \Arrow
        \end{scope}
    }
    \arrow[u,xshift=-0.1cm]
    \arrow[u,phantom,dash,xshift=-1.5cm, "
    {
        \inlinepic{
            \begin{scope}[scale=0.3]
                \BCapE{0}{+}
                \begin{scope}[xshift=3cm,yshift=2cm]
                    \BIdE{+}{-}{}
                \end{scope}
            \end{scope}
        }
    }"]
\end{tikzcd}\]\\
and
\[\begin{tikzcd}[row sep=2cm,column sep=1.5cm]
    \llbracket
    \inlinepic{
        \begin{scope}[scale=0.4]
            \ROneLeft{-}
        \end{scope}
    }
    \rrbracket:
    & 
    \inlinepic{
        \begin{scope}[scale=0.4]
            \PLeft
        \end{scope}
    }\arrow[r]&
    \inlinepic{
        \begin{scope}[scale=0.4]
            \Circle{+}
            \begin{scope}[xshift=1cm]
                \Arrow
            \end{scope}
        \end{scope}
    } 
    \arrow[d,xshift=0.1cm]
   \arrow[d,phantom,dash,xshift=2.5cm, "
    {
        \inlinepic{
            \begin{scope}[scale=0.3]
                \BRCupE{1}{+}{}{}
                \begin{scope}[xshift=3cm,yshift=-2cm]
                    \BIdE{+}{+}{-}
                \end{scope}
            \end{scope}
        } \: - \: 
                \inlinepic{
            \begin{scope}[scale=0.3]
                \BRCupE{}{+}{}{}
                \begin{scope}[xshift=3cm,yshift=-2cm]
                    \BIdE{1}{+}{-}
                \end{scope}
            \end{scope}}
    }"]&
    \\
    & &
    \inlinepic{
        \begin{scope}[scale=0.4]
            \Arrow
        \end{scope}
    }
    \arrow[u,xshift=-0.1cm]
    \arrow[u,phantom,dash,xshift=-1.5cm, "
    {
         \inlinepic{
            \begin{scope}[scale=0.3]
                \BCapE{0}{+}
                \begin{scope}[xshift=3cm,yshift=2cm]
                    \BIdE{+}{-}{}
                \end{scope}
            \end{scope}
        }
    }"]
\end{tikzcd}\]\vspace{0.5cm}
\\
\begin{minipage}{\textwidth}
  \begin{wrapfigure}{r}{0.25\textwidth}
    \vspace{-20pt} 
    \centering
    \includegraphics[width=0.23\textwidth]{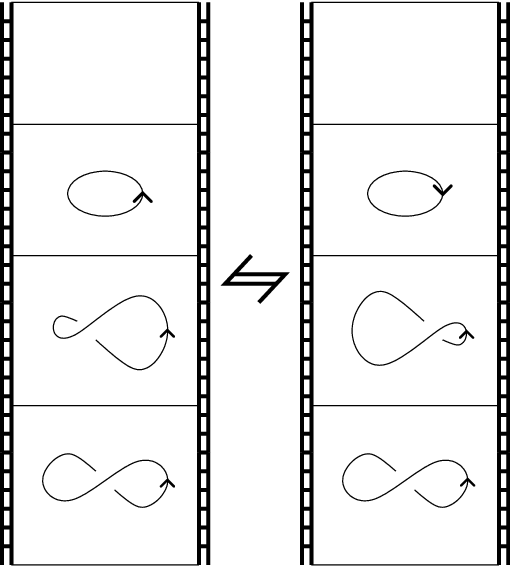}
  \end{wrapfigure}
Going down, the right side of $MM12$ is a cap followed by $R1+$. The left side is a cap followed by $R1'+$
The resulting maps are \\$F_L=
\inlinepic{
    \begin{scope}[scale=0.25]
        \BRCapE{1}{-}{}{}
        \begin{scope}[xshift=3cm]
            \BCapE{}{+}
        \end{scope}
    \end{scope}
}\: - \: \inlinepic{
    \begin{scope}[scale=0.25]
        \BRCapE{}{-}{}{}
        \begin{scope}[xshift=3cm]
            \BCapE{1}{+}
        \end{scope}
    \end{scope}
}
$ and $F_R=-\:\inlinepic{
    \begin{scope}[scale=0.25]
        \BRCapE{1}{+}{}{}
        \begin{scope}[xshift=-3cm]
            \BCapE{}{-}
        \end{scope}
    \end{scope}
}\: + \: \inlinepic{
    \begin{scope}[scale=0.25]
        \BRCapE{}{+}{}{}
        \begin{scope}[xshift=-3cm]
            \BCapE{1}{-}
        \end{scope}
    \end{scope}
}$. \\
Using the neck cutting relation we find that $F_R=F_L=-\inlinepic{\FrobArc{Blue}{0.2}}$.
\\
Going up $G_R=G_L=\:\inlinepic{
    \begin{scope}[scale=0.25]
        \BCapE{}{+}
        \begin{scope}[xshift=-3cm]
            \BCapE{}{-}
        \end{scope}
    \end{scope}
}\: + \: \inlinepic{
    \begin{scope}[scale=0.25]
        \BCapE{}{+}
        \begin{scope}[xshift=-3cm]
            \BCapE{}{-}
        \end{scope}
    \end{scope}
}
$
\end{minipage}
\vspace{1cm}
    \item \textbf{MM13:}\vspace{1cm}
 \\
\begin{minipage}{\textwidth}
  \begin{wrapfigure}{l}{0.25\textwidth}
    \vspace{-25pt} 
    \centering
    \includegraphics[width=0.23\textwidth]{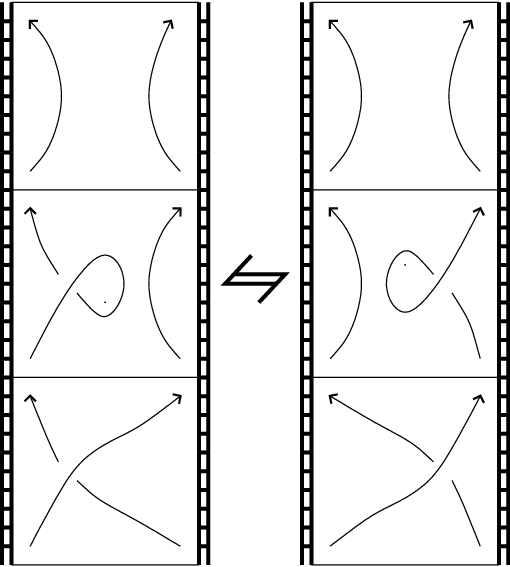}
  \end{wrapfigure}
Going down, the right side of $MM13$ is $R1+$ followed by a saddle, the left side is $R1-$ followed by a saddle. Using the neck cutting relations 
$F_L=F_R=
\inlinepic{
    \begin{scope}[scale=0.2]
        \begin{scope}[even odd rule]
            \clip(-3,4)rectangle (3,-4) (0,1)rectangle (4,-1);
            \BIdE{}{+}{-}             
        \end{scope}
        \draw[Blue,dotted](1,-1)--(1.5,-0.5)--(1.5,1);
        \draw[Blue](0,1)--(3,1);
        \draw[Blue](0,-1)--(2.5,-1);
        \draw[Blue,dotted](3,1)--(4,1) (2.5,-1)--(4,-1);
        \draw[Blue] (0,1)arc(90:270:0.5cm and 1cm);
        \draw[Blue,dotted] (0,1)arc(90:-90:0.5cm and 1cm);       
        \draw[Blue] (4,0)ellipse(0.5cm and 1cm);
        \begin{scope}[xshift=4cm]
            \BIdE{}{+}{-}            
        \end{scope}
    \end{scope}
}
$\\
Going up $G_L=G_R=Id$.
\end{minipage}
\vspace{2cm}\\

    \item \textbf{MM14:}\\
There are multiple possible orientations for $MM14$, but they can be treated in almost the same way. \\

\begin{minipage}{\textwidth}
  \begin{wrapfigure}{l}{0.25\textwidth}
    \vspace{-15pt} 
    \centering
    \includegraphics[width=0.23\textwidth]{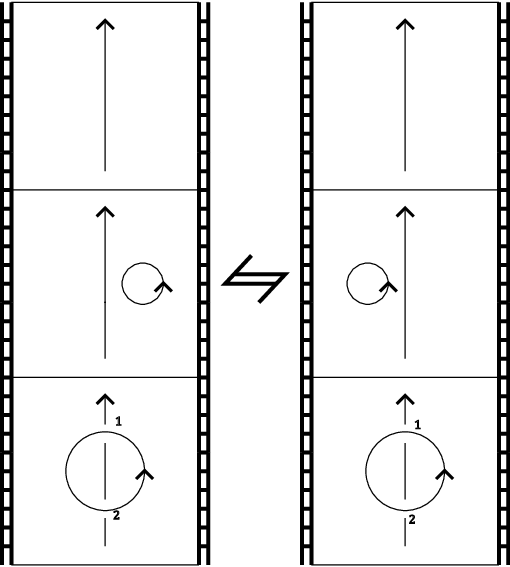}
  \end{wrapfigure}
Going down on the right we have a cap followed by $R2-$ and on the left a cap followed by $R2-$.
We obtain 
\[
\begin{tikzcd}[row sep=1cm]
    \inlinepic{
        \begin{scope}[scale=0.4]
            \Arrow
        \end{scope}
    }\\
    \inlinepic{
        \begin{scope}[scale=0.5]
            \Arrow
            \begin{scope}[xshift=-1cm]
                \Circle{+}
            \end{scope}
        \end{scope}
    }\hspace{0.5cm}\oplus \hspace{0.5cm}
    \inlinepic{
        \begin{scope}[scale=0.6]
            \draw[Blue,->-] (0,-0.5)--(0,0.5);
            \draw[Maroon,thick,decoration={
  markings,
  mark=at position 0.5 with {\arrow[scale=0.4]{>>}}},postaction={decorate}] (0,0.5)--(0,0.75);
            \draw[Maroon,thick,decoration={
  markings,
  mark=at position 0.5 with {\arrow[scale=0.4]{>>}}},postaction={decorate}] (0,-0.5)--(0,-0.75);
            \draw[Blue,->-](0,-0.5) arc(-90:90:0.5cm);
            \draw[Blue,->-](0,0.75)arc(90:270:0.75cm);
            \draw[Blue,->](0,0.75)--(0,1);
            \draw[Blue,->-](0,-1)--(0,-0.75);
        \end{scope}
    }
    \arrow[from =1-1, to= 2-1, out=-45, in=90, 
    ,
    start anchor={[xshift=0.5cm]}
    ,
    end anchor={[xshift=1.5cm]},"F"
    ]    
    \arrow[from =2-1, to= 1-1, in=-45, out=90, ,
    start anchor={[xshift=1.3cm]}
    ,
    end anchor={[xshift=0.3cm]},"G"
    ]
    \arrow[equal,from =1-1, to= 2-1, out=-135, in=90, ,
    start anchor={[xshift=-0.3cm]}
    ,
    end anchor={[xshift=-1.3cm]},"Id"  ]  
\end{tikzcd} \]

\vspace{0.5cm}
\end{minipage}
Where \[
F=-\: 
\inlinepic{\DoubleBubbleB{0.4}}
\hspace{1.5cm} G=
\inlinepic{\DoubleBubbleT{0.4}}
\]

\vspace{0.5cm}
\end{itemize}
The movie move $MM15$ can be checked similarly.\vspace{0.5cm}
\par
\textbf{Step 4:} Checking the signs of $MM6$ to $MM10$.\\
Let $D$ be the tangle diagram in the first frame in $MM6$, for example, and $\varphi$ be the corresponding homotopy equivalence of $CKh(D)$. We already proved that $\varphi= \pm Id$. To make computations simpler let $\varphi_0$ be the induced map on the Lee Khovanov complex $CKh'(D)$. \\
We proved in section \ref{sec: Lee}, that $CKh'(D)$ is equivalent, in $\mathcal{K}ar(Kob)$, to a complex $\Omega$ with vanishing differentials. We also described the non zero objects of $\Omega$. In particular, the oriented resoluation $D_0$ where all strands have the same label, "a" for example, is always non zero. \\
Since $\varphi=\pm Id$, so is $\varphi_0$, and it is enough to check the morphism at the oriented resolution $D_0$ to determine the sign.\\
As a proof of concept, we start with the simpler move $MM9+$.\vspace{0.5cm}\par
\textbf{MM9+:} \vspace{1cm}
 \\
\begin{minipage}{\textwidth}
  \begin{wrapfigure}{r}{0.4\textwidth}
    \vspace{-35pt} 
    \centering
    \includegraphics[width=0.38\textwidth]{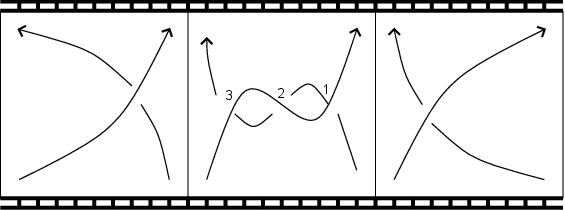}
  \end{wrapfigure}
   As mentioned above, we only check the morphism at the colored oriented resolution. Here we get:
   \[\begin{tikzcd}[column sep=huge]
       \inlinepic{
        \begin{scope}[scale=0.4]
            \Arrow 
            \begin{scope}[xshift=1cm]
                \Arrow
            \end{scope}
        \node at (-0.2,0) {\tiny a};
        \node at (1.2,0) {\tiny a};
        \end{scope}    
   }\arrow[r,"f_1"]
 &
   \inlinepic{
        \begin{scope}[scale=0.4]
            \Arrow 
        \node at (-0.2,0) {\tiny a};
        \begin{scope}[xshift=1.5cm]
                \Circle{-}
        \node at (-0.2,0) {\tiny a};

            \end{scope}
        \begin{scope}[xshift=3.5cm]
                \Circle{+}
            \node at (-0.2,0) {\tiny a};
            \end{scope}
        \begin{scope}[xshift=5cm]
                \Arrow
        \node at (-0.2,0) {\tiny a};
            \end{scope}
        \end{scope}    
   }\arrow[r,"f_2"]
   &\inlinepic{
        \begin{scope}[scale=0.4]
            \Arrow 
            \begin{scope}[xshift=1cm]
                \Arrow
            \end{scope}
        \node at (-0.2,0) {\tiny a};
        \node at (1.2,0) {\tiny a};
        \end{scope}    
   }
    \end{tikzcd}
   \]
   Where $f_1=$ \inlinepic{\BridgeBubbleOne{0.2}} and $f_2 =$ \inlinepic{\BridgeBubbleTwo{0.2}}\quad . So $f_2\circ f_1= Id.$\\
\end{minipage}
\vspace{1cm}

\par
We finally look into the movie move $MM10$. Which is probably the most intimidating. \vspace{0.5cm}\par
\textbf{MM10:}\vspace{0.2cm}
\\
    \begin{center}\includegraphics[width=0.9\textwidth]{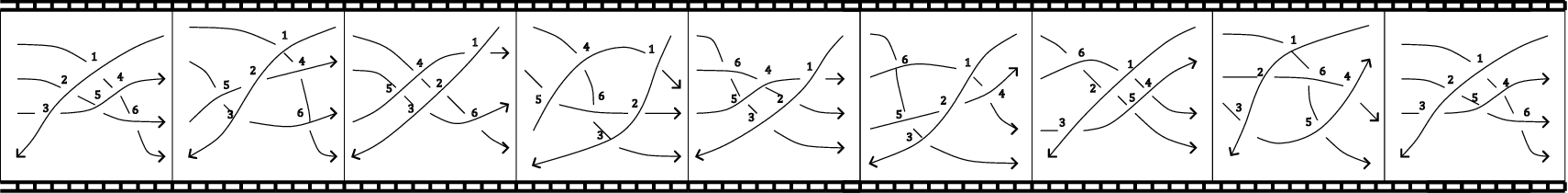}\vspace{0.5cm}\end{center}
$MM10$ corresponds to the sequence $R3-$, $R3-$, $R3-$, $R3n$ (from Figure \ref{fig:NRThree}), $R3-$, $R3-$, $R3-$, $R3n$. \\
The map associated to $R3-$ is the identity on the oriented resolution. Using the Figure \ref{fig:NRThree} $R3n$ is also the identity on the oriented resolution.\\
The other movie moves can be checked similarly. Note that proving invariance by $MM6$ to $MM10$ also proves that different choices of expressions of oriented Reidemeiser moves in terms of those of the generating set yield the same homotopy equivalence.\\
This concludes the proof of functoriality of Khovanov homology in the category $Cob^4$. 
\end{proof}

\section{Immersed cobordisms}\label{Imm}
In \cite{ImmersedCobordisms}, the construction of Khovanov homology is extended to immersed cobordisms with double point singularities. Double point singularities are locally modeled by a transverse intersection of two planes in $\mathbb{R}^4$:
$$X=\mathbb{C}\times\{0\}\cap \{0\}\times \mathbb{C}\subset \mathbb{C}^2.$$
The category of embedded cobordisms $Cob^4(B)$ is extended to a category $Cob_{\times}^4(B)$  with the same objects, and with morphisms immersed cobordisms with double point singularities up to boundary preserving isotopy.\\
The structure of $X$ near a singular point can be described by a movie of the form:\\
\begin{figure}[h]
    \centering 
    \includegraphics[width=6cm]{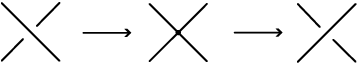}
\end{figure}\\
A cobordism in $Cob_{\times}^4(B)$ can then be represented with a sequence of the movies we used to represent embedded cobordisms, together with this additional movie representing a double point singularity.\\
We have the following theorem:
\begin{thh}
     Two movies respresent isotopic immersed surfaces if and only if they are related by a finite sequence of:
     \begin{itemize}
        \item the movie movies MM1-MM15 (Figure \ref{fig:MM}),
        \item moves that correspond to sliding a node along a horizontal segment of a double point curve (See for example Figures \ref{fig: FC23} and \ref{fig: FC1}),
        \item and the movie moves MM16-MM18 shown Figure \ref{fig:IMM}.
     \end{itemize} 
\end{thh}

A proof of this theorem can be found in \cite{ImmersedCobordisms}.\\
\begin{figure}
    \centering
    \begin{subfigure}[b]{0.28\textwidth}
        \centering
        \includegraphics[width=\textwidth]{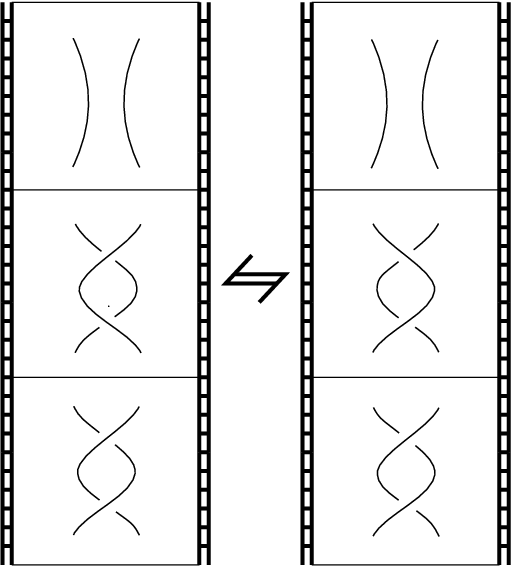}
        \subcaption{MM16}
        \label{fig:MM16}
    \end{subfigure}
    \hfill
    \begin{subfigure}[b]{0.28\textwidth}
        \centering
        \includegraphics[width=\textwidth]{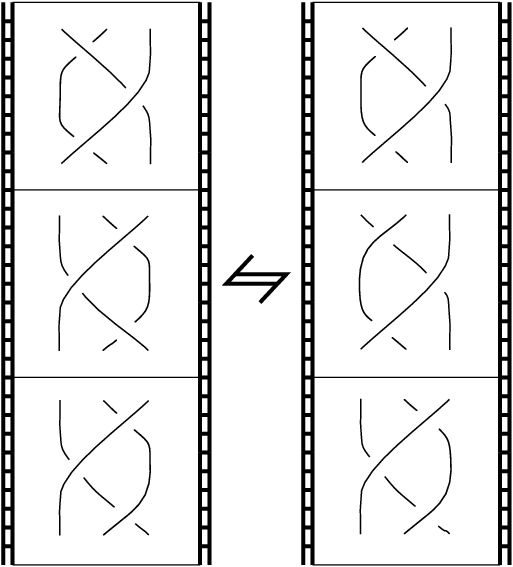}
        \subcaption{MM17}
        \label{fig:MM17}
    \end{subfigure}
    \hfill
    \begin{subfigure}[b]{0.28\textwidth}
        \centering
        \includegraphics[width=\textwidth]{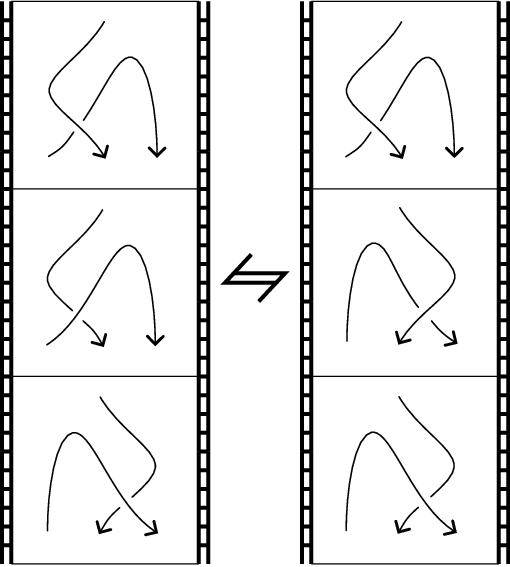}
        \subcaption{MM18}
        \label{fig:MM18}
    \end{subfigure}
    \caption{Movie moves with double point singularities.}
    \label{fig:IMM}
\end{figure}
\begin{figure}
    \centering
    \begin{subfigure}[b]{0.3\textwidth}
        \centering
        \includegraphics[width=\textwidth]{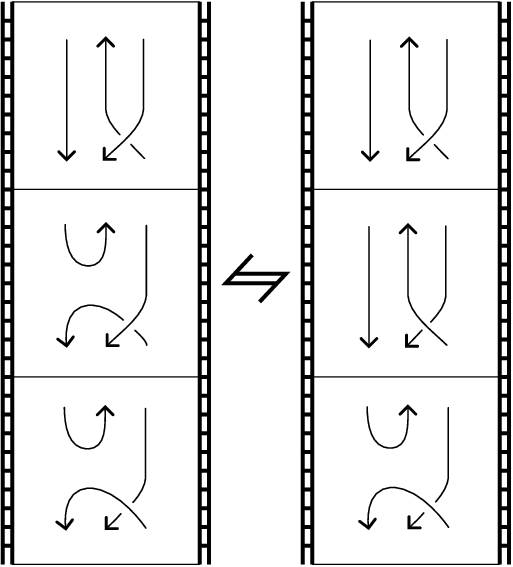}
    \end{subfigure}
    \hspace{0.06\textwidth}
    \begin{subfigure}[b]{0.3\textwidth}
        \centering
        \includegraphics[width=\textwidth]{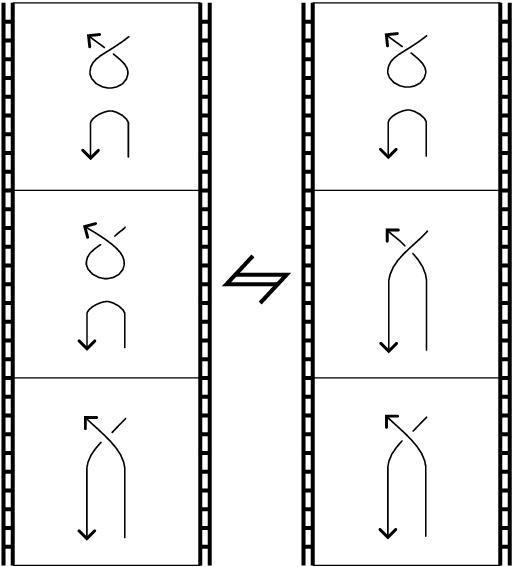}
    \end{subfigure}
    \caption{}
    \label{fig: FC23}
\end{figure}

\begin{figure} 
    \centering
    \includegraphics[width=0.9\textwidth]{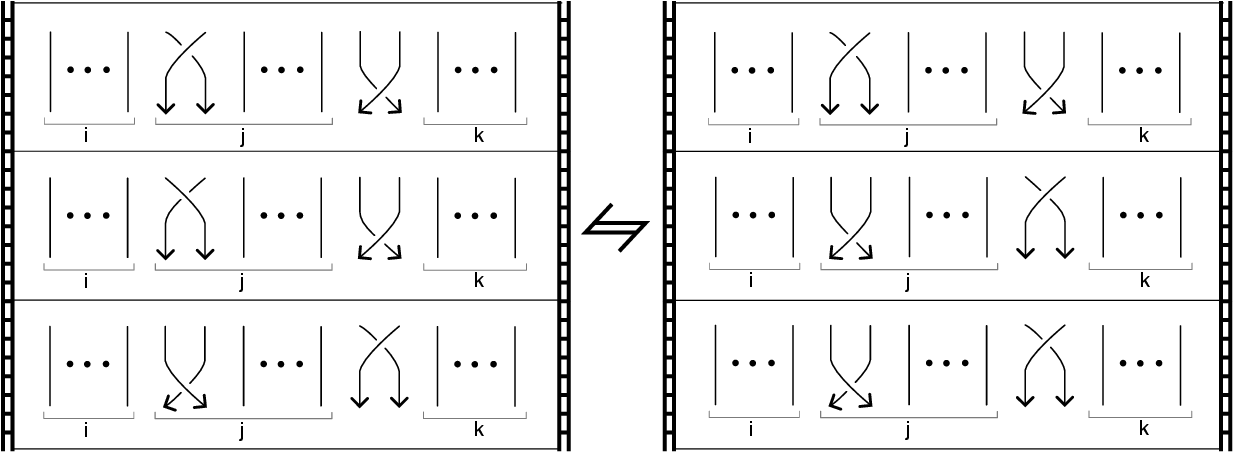}
    \caption{}
    \label{fig: FC1}
\end{figure}
To extend the functor $Kh$ to the category of singular cobordisms $Cob_{\times}^4$, the theorem above shows that it is enough to define $Kh$ on the movie representing a double point singularity. Then we just have to check that $Kh$ respect the new movie moves.\\
Let $A$ and $B$ be the maps shown in Figure \ref{fig: A} and \ref{fig: B}. They are of respective degrees $-2$ and $2$.

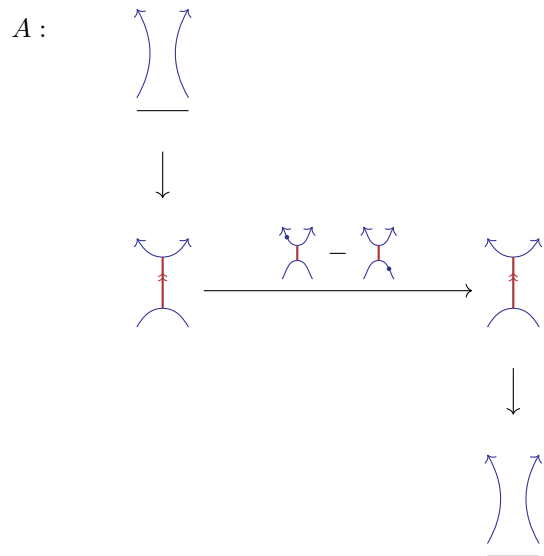
\begin{figure}
    \centering
\[\begin{tikzcd}
	{A:} &  \inlinepic{
        \begin{scope}[scale=1.7]
            \draw[->,Blue,bend right] (-120:0.4)to(120:0.4);
            \draw[->,Blue, bend left] (-60:0.4)to(60:0.4);
            \draw[white](0,0.8) circle(0.01cm);
            \draw[yshift=-0.1cm] (-120:0.4)--(-60:0.4);
        \end{scope}
} \arrow[d,yshift=-0.4cm] \\
	& \inlinepic{\begin{scope}[scale=1.7]
                    \draw[Blue,->] (0,0.2) to[in=-120,out=0] (60:0.4);
                    \draw[Blue,->] (0,0.2) to[in=-60,out=180] (120:0.4);
                    \draw[Maroon,thick,decoration={ markings, mark=at position 0.7 with {\arrow[scale=0.5]{>>}}},postaction={decorate}] (0,-0.2)--(0,0.2);
                    \draw[Blue] (-60:0.4) to[in=0, out=120] (0,-0.2);
                    \draw[Blue] (-120:0.4) to[in=180, out=60] (0,-0.2);
                    \draw[white](0,0.8) circle(0.01cm);
                \end{scope}} & 
    \inlinepic{\begin{scope}[scale=1]
                    \draw[Blue,->] (0,0.1) to[in=-120,out=0] (60:0.4);
                    \fill[Blue] (123:0.25) circle(0.03cm);
                    \draw[Blue,->] (0,0.1) to[in=-60,out=180] (120:0.4);
                    \draw[Maroon,thick] (0,-0.1)--(0,0.1);
                    \draw[Blue] (-60:0.4) to[in=0, out=120] (0,-0.1);
                    \draw[Blue] (-120:0.4) to[in=180, out=60] (0,-0.1);
                \end{scope}}\: -\: \inlinepic{\begin{scope}[scale=1]
                    \draw[Blue,->] (0,0.1) to[in=-120,out=0] (60:0.4);
                    \draw[Blue,->] (0,0.1) to[in=-60,out=180] (120:0.4);
                    \draw[Maroon,thick] (0,-0.1)--(0,0.1);
                    \fill[Blue] (-57:0.25) circle(0.03cm);
                    \draw[Blue] (-60:0.4) to[in=0, out=120] (0,-0.1);
                    \draw[Blue] (-120:0.4) to[in=180, out=60] (0,-0.1);
                \end{scope}}
    & \inlinepic{\begin{scope}[scale=1.7]
                    \draw[Blue,->] (0,0.2) to[in=-120,out=0] (60:0.4);
                    \draw[Blue,->] (0,0.2) to[in=-60,out=180] (120:0.4);
                    \draw[Maroon,thick,decoration={ markings, mark=at position 0.7 with {\arrow[scale=0.5]{>>}}},postaction={decorate}] (0,-0.2)--(0,0.2);
                    \draw[Blue] (-60:0.4) to[in=0, out=120] (0,-0.2);
                    \draw[Blue] (-120:0.4) to[in=180, out=60] (0,-0.2);
                    \draw[white](0,0.8) circle(0.01cm);
                \end{scope}} \arrow[d,yshift=-0.4cm]\\
	&&&  \inlinepic{
        \begin{scope}[scale=1.7]
            \draw[->,Blue,bend right] (-120:0.4)to(120:0.4);
            \draw[->,Blue, bend left] (-60:0.4)to(60:0.4);
            \draw[white](0,0.8) circle(0.01cm);
            \draw[yshift=-0.1cm] (-120:0.4)--(-60:0.4);
        \end{scope}
}
	\arrow[from=2-2, to=2-4,yshift=-0.5cm]
\end{tikzcd}\]    \caption{The map $A$ corresponding to passing from a positive crossing to a negative crossing through a double point singularity. It is of degree $-2$.}
    \label{fig: A}
\end{figure}
\begin{figure}
    \centering
\[\begin{tikzcd}[column sep=2.5cm,row sep=1.5cm]
	&& \inlinepic{
        \begin{scope}[scale=1.7]
            \draw[->,Blue,bend right] (-120:0.4)to(120:0.4);
            \draw[->,Blue, bend left] (-60:0.4)to(60:0.4);
            \draw[yshift=-0.1cm] (-120:0.4)--(-60:0.4);
        \end{scope}} \arrow[d,shorten <=0.2cm, shorten >=0.2cm]\\
	{B:} & \inlinepic{\begin{scope}[scale=1.7]
                    \draw[Blue,->] (0,0.2) to[in=-120,out=0] (60:0.4);
                    \draw[Blue,->] (0,0.2) to[in=-60,out=180] (120:0.4);
                    \draw[Maroon,thick,decoration={ markings, mark=at position 0.7 with {\arrow[scale=0.5]{>>}}},postaction={decorate}] (0,-0.2)--(0,0.2);
                    \draw[Blue] (-60:0.4) to[in=0, out=120] (0,-0.2);
                    \draw[Blue] (-120:0.4) to[in=180, out=60] (0,-0.2);
                \end{scope}} \arrow[d,shorten <=0.2cm, shorten >=0.2cm] &
                \inlinepic{\begin{scope}[scale=1.7]
                    \draw[Blue,->] (0,0.2) to[in=-120,out=0] (60:0.4);
                    \draw[Blue,->] (0,0.2) to[in=-60,out=180] (120:0.4);
                    \draw[Maroon,thick,decoration={ markings, mark=at position 0.7 with {\arrow[scale=0.5]{>>}}},postaction={decorate}] (0,-0.2)--(0,0.2);
                    \draw[Blue] (-60:0.4) to[in=0, out=120] (0,-0.2);
                    \draw[Blue] (-120:0.4) to[in=180, out=60] (0,-0.2);
                \end{scope}} \\
	& \inlinepic{
        \begin{scope}[scale=1.7]
            \draw[->,Blue,bend right] (-120:0.4)to(120:0.4);
            \draw[->,Blue, bend left] (-60:0.4)to(60:0.4);
            \draw[yshift=-0.1cm] (-120:0.4)--(-60:0.4);
        \end{scope}}
	\arrow[from=2-2, to=2-3, "Id"]
\end{tikzcd}\]
    \caption{The map $B$ corresponding to passing from a negative crossing to a positive crossing through a double point singularity. It is of degree $2$.}
    \label{fig: B}
\end{figure}
Then we have the following theorem
\begin{thh}
    The map $Kh$ is a well defined functor from $Cob_{\times}^4$ to $Kob_{/h}$.
\end{thh}
\begin{proof}
    The extended functor $Kh$ still satisfies the movie moves 1 to 15. \\
    The movie moves involving far commutativity commute because the maps A and B are applied to the same diagrams after planar isotopy.\\
    What is left to check is the movie moves $MM16$ to $MM18$.\par
    \textbf{MM16:}
\begin{itemize}
    \item MM16+:\\
\begin{minipage}{\textwidth}
  \begin{wrapfigure}{l}{0.25\textwidth}
    \vspace{-5pt} 
    \centering
    \includegraphics[width=0.23\textwidth]{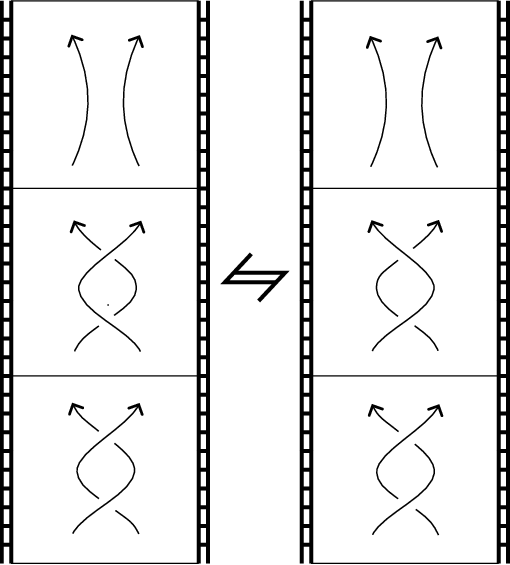}
  \end{wrapfigure}
  Going down we apply $R2+$ followed by the map $B$ on both sides. This gives the degree $2$ morphism $F:Kh(\:\raisebox{-0.3em}{\includegraphics[height=1.3em]{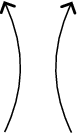}}\:)\rightarrow Kh(\:\raisebox{-0.3em}{\includegraphics[height=1.3em]{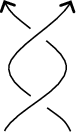}}\:)$ given by\\ 
  $\inlinepic{
            \begin{scope}[scale=0.3]
                \Arrow
            \end{scope}
        }       \inlinepic{
            \begin{scope}[scale=0.3]
                \Arrow
            \end{scope}
            }\xrightarrow{\quad
                \inlinepic{
                    \begin{scope}[scale=0.15]
                        \RTwoPlusf
                    \end{scope}
                }\quad
            }
    \inlinepic{
            \begin{scope}[scale=0.3]
                \HOR  
            \end{scope}
    }  $.\vspace{0.5cm}\\ Going up we apply $A$ then $R2+$ on both sides.The resulting morphisms $G_L$ and $G_R:Kh(\:\raisebox{-0.3em}{\includegraphics[height=1.3em]{\imagesfolder G2.eps}}\:)\rightarrow Kh(\:\raisebox{-0.3em}{\includegraphics[height=1.3em]{\imagesfolder G1.eps}}\:)$ 
    are of degree $-2$ and are given by 0 everywhere except:\vspace{1cm}
    
\end{minipage}
$G_L: \inlinepic{
            \begin{scope}[scale=0.3]
                \HOR  
            \end{scope}
    }\xrightarrow{\quad
                \inlinepic{
                    \begin{scope}[scale=0.15]
                        \RTwoPlusg
                        \begin{scope}[scale=1.5]
                            \BDot
                        \end{scope}
                    \end{scope}
                }\quad - \quad \inlinepic{
                    \begin{scope}[scale=0.15]
                        \RTwoPlusg
                        \begin{scope}[shift={(-2,-1)},scale=1.5]
                            \BDot
                        \end{scope}
                    \end{scope}
                }\quad
            }\inlinepic{
            \begin{scope}[scale=0.3]
                \Arrow
            \end{scope}
        }       \inlinepic{
            \begin{scope}[scale=0.3]
                \Arrow
            \end{scope}
            }
   $\quad
 and\quad
    $G_R: \inlinepic{
            \begin{scope}[scale=0.3]
                \HOR  
            \end{scope}
    }\xrightarrow{\quad
                \inlinepic{
                    \begin{scope}[scale=0.15]
                        \RTwoPlusg
                        \begin{scope}[shift={(2.5,0.25)},scale=1.5]
                            \BDot
                        \end{scope}
                    \end{scope}
                }\quad - \quad \inlinepic{
                    \begin{scope}[scale=0.15]
                        \RTwoPlusg
                        \begin{scope}[shift={(0,-0.75)},scale=1.5]
                            \BDot
                        \end{scope}
                    \end{scope}
                }\quad
            }\inlinepic{
            \begin{scope}[scale=0.3]
                \Arrow
            \end{scope}
        }       \inlinepic{
            \begin{scope}[scale=0.3]
                \Arrow
            \end{scope}
            }
   $.
\vspace{0.5cm}\\Moving the dot in $G_R$ we get\quad
    $G_R: \inlinepic{
            \begin{scope}[scale=0.3]
                \HOR  
            \end{scope}
    }\xrightarrow{\quad
                \inlinepic{
                    \begin{scope}[scale=0.15]
                        \RTwoPlusg
                        \begin{scope}[shift={(-3,0)},scale=1.5]
                            \BDot
                        \end{scope}
                    \end{scope}
                }\quad - \quad \inlinepic{
                    \begin{scope}[scale=0.15]
                        \RTwoPlusg
                        \begin{scope}[shift={(0,-0.75)},scale=1.5]
                            \BDot
                        \end{scope}
                    \end{scope}
                }\quad
            }\inlinepic{
            \begin{scope}[scale=0.3]
                \Arrow
            \end{scope}
        }       \inlinepic{
            \begin{scope}[scale=0.3]
                \Arrow
            \end{scope}
            }
   $.\vspace{0.5cm}\\
And using the migration relation it is clear that $G_L=G_R$. \vspace{1cm}

\item MM16-:\\
\begin{minipage}{\textwidth}
  \begin{wrapfigure}{l}{0.25\textwidth}
    \vspace{-5pt} 
    \centering
    \includegraphics[width=0.23\textwidth]{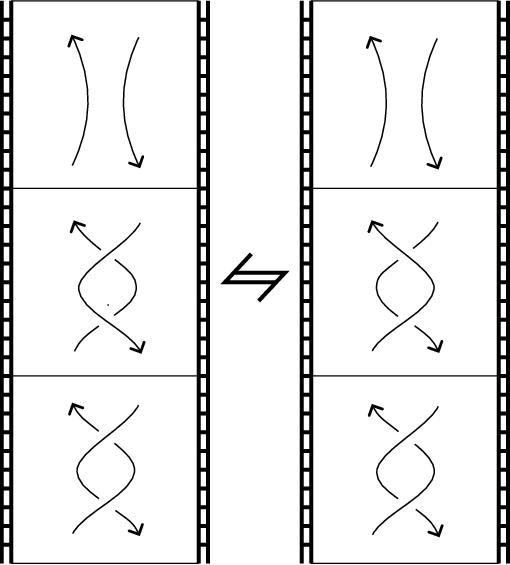}
  \end{wrapfigure}\raggedright
  Going up $MM16-$ we have $B$ followed by $R2-$ on both sides. So as before we get the same morphism. Going down we obtain the maps:\vspace{0.5cm}
  \\$G_L: \inlinepic{
            \begin{scope}[scale=0.3]
                \Arrow
            \end{scope}
        }       \inlinepic{
            \begin{scope}[scale=0.3]
                \Arrow
            \end{scope}
            }\xrightarrow{\quad
                \inlinepic{
                    \begin{scope}[scale=0.15]
                        \aTwoSmall
                        \begin{scope}[shift={(-3.5,-3)},scale=3]
                            \BDot
                        \end{scope}
                    \end{scope}
                }\quad - \quad \inlinepic{
                    \begin{scope}[scale=0.15]
                        \aTwoSmall
                        \begin{scope}[shift={(0.5,-4)},scale=3]
                            \BDot
                        \end{scope}
                    \end{scope}
                }\quad
            }\inlinepic{
            \begin{scope}[xscale=0.3,yscale=-0.3]
                \Sq
            \end{scope}
        }       
   $\quad
 and\quad
    $G_R: \inlinepic{
            \begin{scope}[scale=0.3]
                \Arrow
            \end{scope}
        }       \inlinepic{
            \begin{scope}[scale=0.3]
                \Arrow
            \end{scope}
            }\xrightarrow{\quad
                \inlinepic{
                    \begin{scope}[scale=0.15]
                        \aTwoSmall
                        \begin{scope}[shift={(2.5,-4)},scale=3]
                            \BDot
                        \end{scope}
                    \end{scope}
                }\quad - \quad \inlinepic{
                    \begin{scope}[scale=0.15]
                        \aTwoSmall
                        \begin{scope}[shift={(0,-3.5)},scale=3]
                            \BDot
                        \end{scope}
                    \end{scope}
                }\quad
            }\inlinepic{
            \begin{scope}[xscale=0.3,yscale=-0.3]
                \Sq
            \end{scope}
        } 
   $.
  \vspace{0.5cm}\\
    
\end{minipage}
\\
The migration relation again shows that $G_L=G_R$.
\end{itemize}\par
    \textbf{MM17:}\\
There are 4 orientations of the diagrams in $MM17$ to consider. Let us start with the move involving $R3+$ and $R3-$.
\begin{itemize}
    \item $MM17+$:\vspace{0.25cm}\\
\begin{minipage}{\textwidth}
  \begin{wrapfigure}{l}{0.25\textwidth}
    \vspace{-15pt} 
    \centering
    \includegraphics[width=0.23\textwidth]{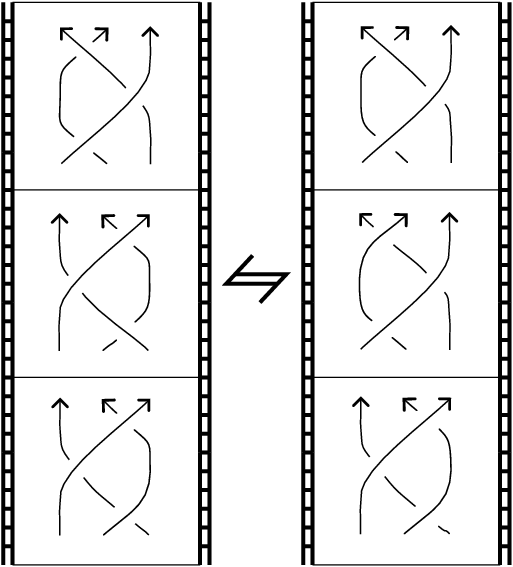}
  \end{wrapfigure}
  Going down, on the right side we have $R3-$ followed by $B$, on the left side $B$ followed by $R3+$. \\
  Going up on the right is $A$ followed by $R3-$ and on the left is $R3+$ followed by $A$.\\
  We work around $MM17+$ without doing the full computations by looking at the similarity between the construction of the maps for $R3+$ and $R3-$.\vspace{2cm}
\end{minipage}
\\
The Khovanov homology two sides of the move $R3+$ can be regarded as the cones of the maps 
$\inlinepic{
    \TopLayer{0.3}
}
\xrightarrow{d} 
\inlinepic{    \begin{scope}[scale=0.25]
        \HB 
        \draw[white, line width=0.2cm](-1.5,-1.5)to[in=-90,out=90,looseness=1](1.5,1)--(1.5,1.5);
        \draw[->,Blue](-1.5,-1.5)to[in=-90,out=90,looseness=1](1.5,1)--(1.5,1.5);
    \end{scope}
    }
$
$\inlinepic{
    \TopLayer{0.3}
}
\xrightarrow{d} \inlinepic{    \begin{scope}[scale=0.25]
        \HT
        \draw[white, line width=0.2cm](-1.5,-1.5)--(-1.5,-1)to[in=-90,out=90,looseness=1](1.5,1.5);
        \draw[->,Blue](-1.5,-1.5)--(-1.5,-1)to[in=-90,out=90,looseness=1](1.5,1.5);
    \end{scope}
    }
$
(with homological degree shifted down by 1) (Abuse of notation: we write here a map from a diagram $D$ to another diagram $D'$ for a map between their respective Khovanov homologies). We saw that there are deformation retractions 
$\begin{tikzcd}
\inlinepic{    \begin{scope}[scale=0.25]
        \HB 
        \draw[white, line width=0.2cm](-1.5,-1.5)to[in=-90,out=90,looseness=1](1.5,1)--(1.5,1.5);
        \draw[->,Blue](-1.5,-1.5)to[in=-90,out=90,looseness=1](1.5,1)--(1.5,1.5);
    \end{scope}
    }    
    \arrow[r,"F"]
&\inlinepic{
    \begin{scope}[scale=0.25]
        \SnakeRight
    \end{scope}
}[1]    \arrow[l,yshift=-3,"F'"]
\end{tikzcd}$
and 
$\begin{tikzcd}
\inlinepic{    \begin{scope}[scale=0.25]
        \HT
        \draw[white, line width=0.2cm](-1.5,-1.5)--(-1.5,-1)to[in=-90,out=90,looseness=1](1.5,1.5);
        \draw[->,Blue](-1.5,-1.5)--(-1.5,-1)to[in=-90,out=90,looseness=1](1.5,1.5);
    \end{scope}
    }    
    \arrow[r,"G"]
&\inlinepic{
    \begin{scope}[scale=0.25]
        \SnakeRight
    \end{scope}
}[1]    \arrow[l,yshift=-3,"G'"]
\end{tikzcd}$. Using Lemma \ref{Cone}, we defined homotopy equivalences:
\[
\begin{tikzcd}[ampersand replacement=\&]
    \begin{pmatrix}
    \inlinepic{\TopLayer{0.3}}\vspace{0.5cm}
    \\
    \inlinepic{    \begin{scope}[scale=0.25]
        \HB 
        \draw[white, line width=0.2cm](-1.5,-1.5)to[in=-90,out=90,looseness=1](1.5,1)--(1.5,1.5);
        \draw[->,Blue](-1.5,-1.5)to[in=-90,out=90,looseness=1](1.5,1)--(1.5,1.5);
    \end{scope}
    }    [1]
    \end{pmatrix}
    \arrow[r,yshift=0.7cm,"Id"]
    \arrow[r,yshift=-0.7cm,"F"]
    \arrow[r,end anchor={[yshift=-0.5cm]}, start anchor={[yshift=0.5cm]}]\&
    \begin{pmatrix}
    \inlinepic{\TopLayer{0.3}}\vspace{0.5cm}
    \\
    \inlinepic{
    \begin{scope}[scale=0.25]
        \SnakeRight
    \end{scope}
    }[2] 
    \end{pmatrix}
    \arrow[r,yshift=0.6cm,"Id"]
    \arrow[r,yshift=-0.6cm,"G'"]
    \arrow[r,end anchor={[yshift=-0.5cm]}, start anchor={[yshift=0.5cm]}]\&
    \begin{pmatrix}
    \inlinepic{\TopLayer{0.3}}\vspace{0.5cm}
    \\
    \inlinepic{    \begin{scope}[scale=0.25]
        \HT
        \draw[white, line width=0.2cm](-1.5,-1.5)--(-1.5,-1)to[in=-90,out=90,looseness=1](1.5,1.5);
        \draw[->,Blue](-1.5,-1.5)--(-1.5,-1)to[in=-90,out=90,looseness=1](1.5,1.5);
    \end{scope}
    }[1]
    \end{pmatrix}
\end{tikzcd}
\]
We then composed by diagonal matrices with entries $\pm Id$ on the left and of the right to recover the Khovanov homology of \inlinepic{\RThreeLeft{0.2}} and \inlinepic{\RThreeRight{0.2}}: 
\[
\begin{tikzcd}[ampersand replacement=\&]
\inlinepic{\RThreeLeft{0.25}}\arrow[r,"\mathcal{E}_L"]\&
    \begin{pmatrix}
    \inlinepic{\TopLayer{0.3}}\vspace{0.5cm}
    \\
    \inlinepic{    \begin{scope}[scale=0.25]
        \HB 
        \draw[white, line width=0.2cm](-1.5,-1.5)to[in=-90,out=90,looseness=1](1.5,1)--(1.5,1.5);
        \draw[->,Blue](-1.5,-1.5)to[in=-90,out=90,looseness=1](1.5,1)--(1.5,1.5);
    \end{scope}
    }    [1]
    \end{pmatrix}
    \arrow[r,yshift=0.7cm,"Id"]
    \arrow[r,yshift=-0.7cm,"G'F"]
    \arrow[r,end anchor={[yshift=-0.5cm]}, start anchor={[yshift=0.5cm]}]\&
    \begin{pmatrix}
    \inlinepic{\TopLayer{0.3}}\vspace{0.5cm}
    \\
    \inlinepic{    \begin{scope}[scale=0.25]
        \HT
        \draw[white, line width=0.2cm](-1.5,-1.5)--(-1.5,-1)to[in=-90,out=90,looseness=1](1.5,1.5);
        \draw[->,Blue](-1.5,-1.5)--(-1.5,-1)to[in=-90,out=90,looseness=1](1.5,1.5);
    \end{scope}
    }[1]
    \end{pmatrix}\arrow[r,"\mathcal{E}_R"]\&
    \inlinepic{\RThreeRight{0.25}}
\end{tikzcd}
\]
Then for $R3-$, using Lemma \ref{Cone} again, we have the homotopy equivalence:
\[
\begin{tikzcd}[ampersand replacement=\&]
\inlinepic{\begin{tikzpicture}[
    every braid/.style={->},
    braid/.cd,
    height=1cm,
    gap=.2,
]
    \pic[scale=0.25] at (0,0) {braid={s_1 s_2 s_1^{-1}}};
\end{tikzpicture}}
\arrow[r,"{\mathcal{E}_L}"]\&
    \begin{pmatrix}
    \inlinepic{\TopLayer{0.3}}\vspace{0.5cm}
    \\
    \inlinepic{    \begin{scope}[scale=0.25]
        \HB 
        \draw[white, line width=0.2cm](-1.5,-1.5)to[in=-90,out=90,looseness=1](1.5,1)--(1.5,1.5);
        \draw[->,Blue](-1.5,-1.5)to[in=-90,out=90,looseness=1](1.5,1)--(1.5,1.5);
    \end{scope}
        }    [-1]
    \end{pmatrix}
    \arrow[r,yshift=0.7cm,"Id"]
    \arrow[r,yshift=-0.7cm,"G'F"]
    \arrow[r,end anchor={[yshift=0.5cm]}, start anchor={[yshift=-0.5cm]}]\&
    \begin{pmatrix}
    \inlinepic{\TopLayer{0.3}}\vspace{0.5cm}
    \\
    \inlinepic{    \begin{scope}[scale=0.25]
        \HT
        \draw[white, line width=0.2cm](-1.5,-1.5)--(-1.5,-1)to[in=-90,out=90,looseness=1](1.5,1.5);
        \draw[->,Blue](-1.5,-1.5)--(-1.5,-1)to[in=-90,out=90,looseness=1](1.5,1.5);
    \end{scope}
    }[-1]
    \end{pmatrix}\arrow[r,"{\mathcal{E}_R}"]\&
\inlinepic{\begin{tikzpicture}[
    every braid/.style={->},
    braid/.cd,
    height=1cm,
    gap=.2,
]
    \pic[scale=0.25] at (0,0) {braid={s_2^{-1} s_1 s_2}};
\end{tikzpicture}}\end{tikzcd}
\]
with the same maps $G'$, $F$, $\mathcal{E}_L$ and $\mathcal{E}_R$. \\
We want to check that composing the first map by $B$ and composing $B$ by the second map result in homotopic morphisms. We can ovelook the maps $\mathcal{E}_{L,R}$, because $B$ (and later $A$ are) diagnonal, so they commute with diagonal maps formed with $\pm$ identities, in particular with $\mathcal{E}_{L,R}$. \\
Moreover, $B$ (and $A$) are zero on \inlinepic{\TopLayer{0.3}}. So we only need to check that the compositions $B G' F$ and $G'F B$ are homotopic. This follows from the fact that $B$ is the identity (with degree shift) on \inlinepic{
    \begin{scope}[scale=0.25]
        \HB
        \draw[white, line width=0.2cm](-1.5,-1.5)to[in=-90,out=90,looseness=1](1.5,1)--(1.5,1.5);
        \draw[->,Blue](-1.5,-1.5)to[in=-90,out=90,looseness=1](1.5,1)--(1.5,1.5);
    \end{scope}
} and on \inlinepic{
    \begin{scope}[scale=0.25]
        \HT
        \draw[white, line width=0.2cm](-1.5,-1.5)--(-1.5,-1)to[in=-90,out=90,looseness=1](1.5,1.5);
        \draw[->,Blue](-1.5,-1.5)--(-1.5,-1)to[in=-90,out=90,looseness=1](1.5,1.5);
    \end{scope}
}. The part "going down" of $MM17+$ is then proven. \vspace{0.5cm}\\
MM17+ Right:
\[
\begin{tikzcd}[ampersand replacement=\&]
    \inlinepic{    \begin{scope}[scale=0.25]
        \HB 
        \draw[white, line width=0.2cm](-1.5,-1.5)to[in=-90,out=90,looseness=1](1.5,1)--(1.5,1.5);
        \draw[->,Blue](-1.5,-1.5)to[in=-90,out=90,looseness=1](1.5,1)--(1.5,1.5);
    \end{scope}
    \node at (0.75,-0.2) {[-1]};
    \draw[white](0,1.6)circle(0.01cm);
    }
    \arrow[r,"B",yshift=-0.7cm]\&
    \begin{pmatrix}
    \inlinepic{\TopLayer{0.3}}\vspace{0.5cm}
    \\
    \inlinepic{    \begin{scope}[scale=0.25]
        \HB 
        \draw[white, line width=0.2cm](-1.5,-1.5)to[in=-90,out=90,looseness=1](1.5,1)--(1.5,1.5);
        \draw[->,Blue](-1.5,-1.5)to[in=-90,out=90,looseness=1](1.5,1)--(1.5,1.5);
    \end{scope}
    }    [1]
    \end{pmatrix}
    \arrow[r,yshift=0.7cm,"Id"]
    \arrow[r,yshift=-0.7cm,"G'F"]
    \arrow[r,end anchor={[yshift=-0.5cm]}, start anchor={[yshift=0.5cm]}]\&
    \begin{pmatrix}
    \inlinepic{\TopLayer{0.3}}\vspace{0.5cm}
    \\
    \inlinepic{    \begin{scope}[scale=0.25]
        \HT
        \draw[white, line width=0.2cm](-1.5,-1.5)--(-1.5,-1)to[in=-90,out=90,looseness=1](1.5,1.5);
        \draw[->,Blue](-1.5,-1.5)--(-1.5,-1)to[in=-90,out=90,looseness=1](1.5,1.5);
    \end{scope}
    }[1]
    \end{pmatrix}
\end{tikzcd}\vspace{0.5cm}
\]
MM17+ Left
\[
\begin{tikzcd}[ampersand replacement=\&]
    \begin{pmatrix}
    \inlinepic{\TopLayer{0.3}}\vspace{0.5cm}
    \\
    \inlinepic{    \begin{scope}[scale=0.25]
        \HB 
        \draw[white, line width=0.2cm](-1.5,-1.5)to[in=-90,out=90,looseness=1](1.5,1)--(1.5,1.5);
        \draw[->,Blue](-1.5,-1.5)to[in=-90,out=90,looseness=1](1.5,1)--(1.5,1.5);
    \end{scope}
        }    [-1]
    \end{pmatrix}
    \arrow[r,yshift=0.7cm,"Id"]
    \arrow[r,yshift=-0.7cm,"G'F"]
    \arrow[r,end anchor={[yshift=0.5cm]}, start anchor={[yshift=-0.5cm]}]\&
    \begin{pmatrix}
    \inlinepic{\TopLayer{0.3}}\vspace{0.5cm}
    \\
    \inlinepic{    \begin{scope}[scale=0.25]
        \HT
        \draw[white, line width=0.2cm](-1.5,-1.5)--(-1.5,-1)to[in=-90,out=90,looseness=1](1.5,1.5);
        \draw[->,Blue](-1.5,-1.5)--(-1.5,-1)to[in=-90,out=90,looseness=1](1.5,1.5);
    \end{scope}
    }[-1]
    \end{pmatrix}\arrow[r,"B",yshift=-0.7cm]\&
\inlinepic{    \begin{scope}[scale=0.25]
        \HT
        \draw[white, line width=0.2cm](-1.5,-1.5)--(-1.5,-1)to[in=-90,out=90,looseness=1](1.5,1.5);
        \draw[->,Blue](-1.5,-1.5)--(-1.5,-1)to[in=-90,out=90,looseness=1](1.5,1.5);
    \end{scope}
    \node at (0.75,-0.2) {[1]};
    \draw[white](0,1.6)circle(0.01cm);
    }\end{tikzcd}\vspace{0.5cm}
\]
Similarly, going up, we want to prove that $A F' G$ and $F' G A$ are homotopic. In this case we need the explicit formula for $F'G$: \\
\[
\begin{tikzcd}[ampersand replacement=\&,column sep=1.5cm,row sep=1cm]
    \inlinepic{
        \begin{scope}[scale=0.3]
            \HT
            \begin{scope}[xshift=-1cm]
                \Arrow
            \end{scope}
        \end{scope}
    }
    \arrow[r]
    \& 
    \inlinepic{
        \begin{scope}[scale=0.3]
            \HOR
            \begin{scope}[xshift=-1cm]
                \Arrow
            \end{scope}
        \end{scope}
    }
    \oplus
    \inlinepic{
        \begin{scope}[scale=0.3]
            \SnakeRight
        \end{scope}
    }
    \arrow[r]
    \arrow[d,"{
        \begin{pmatrix}
            m_1 & m_2 \\ m_3 & Id
        \end{pmatrix}
    }"]
    \&
    \inlinepic{
        \begin{scope}[scale=0.3]
            \PopPimpleRight
        \end{scope}
    } \\
    \inlinepic{
        \begin{scope}[scale=0.3]
            \HB
            \begin{scope}[xshift=1cm]
                \Arrow
            \end{scope}
        \end{scope}
    }
    \arrow[r]
    \& 
    \inlinepic{
        \begin{scope}[scale=0.3]
            \HOR
            \begin{scope}[xshift=1cm]
                \Arrow
            \end{scope}
        \end{scope}
    }
    \oplus
    \inlinepic{
        \begin{scope}[scale=0.3]
            \SnakeRight
        \end{scope}
    }
    \arrow[r]
    \&
    \inlinepic{
        \begin{scope}[scale=0.3]
            \PopPimpleLeft
        \end{scope}
    } 
\end{tikzcd}\vspace{0.5cm}
\]
Where $m_1=\inlinepic{
                \begin{scope}[scale=0.3]
                    \MapSeven
                \end{scope}
            }$
\hspace{0.5cm} $m_2=-\: \inlinepic{
                \begin{scope}[scale=0.3]
                    \MapEight
                \end{scope}
            }$
            \hspace{0.5cm}
            $m_3=-\:
            \inlinepic{
                \begin{scope}[scale=0.3]
                    \MapSix
                \end{scope}
            }\vspace{0.3cm}$.
            \vspace{0.25cm}\\
Let $h$ be the map:
\[
\begin{tikzcd}[ampersand replacement=\&,column sep=1.5cm,row sep=1cm]
    \inlinepic{
        \begin{scope}[scale=0.3]
            \HT
            \begin{scope}[xshift=-1cm]
                \Arrow
            \end{scope}
        \end{scope}
    }
    \arrow[r, "{
        \begin{pmatrix}
            +\\+
        \end{pmatrix}
    }"]
    \& 
    \inlinepic{
        \begin{scope}[scale=0.3]
            \HOR
            \begin{scope}[xshift=-1cm]
                \Arrow
            \end{scope}
        \end{scope}
    }
    \oplus
    \inlinepic{
        \begin{scope}[scale=0.3]
            \SnakeRight
        \end{scope}
    }
    \arrow[r,"{
        \begin{pmatrix}
            +&-
        \end{pmatrix}
    }"]
    \arrow[ld,"{
        \begin{pmatrix}
            h_1&h_2
        \end{pmatrix}
    }"{yshift=0.2cm, xshift=0.1cm}]
    \&
    \inlinepic{
        \begin{scope}[scale=0.3]
            \PopPimpleRight
        \end{scope}
    }
    \arrow[ld,"{
        \begin{pmatrix}
            h_3&h_4
        \end{pmatrix}
    }"] \\
    \inlinepic{
        \begin{scope}[scale=0.3]
            \HB
            \begin{scope}[xshift=1cm]
                \Arrow
            \end{scope}
        \end{scope}
    }
    \arrow[r,swap,"{
        \begin{pmatrix}
            +\\+
        \end{pmatrix}
    }"]
    \& 
    \inlinepic{
        \begin{scope}[scale=0.3]
            \HOR
            \begin{scope}[xshift=1cm]
                \Arrow
            \end{scope}
        \end{scope}
    }
    \oplus
    \inlinepic{
        \begin{scope}[scale=0.3]
            \SnakeRight
        \end{scope}
    }
    \arrow[r,swap,"{
        \begin{pmatrix}
            +&-
        \end{pmatrix}
    }"]
    \&
    \inlinepic{
        \begin{scope}[scale=0.3]
            \PopPimpleLeft
        \end{scope}
    } 
\end{tikzcd}\vspace{0.5cm}
\]
With $h_1=-\: \inlinepic{\hOn{0.3}}$\:,\: $h_2=\inlinepic{\hTw{0.3}}$\:, \:$h_3=\inlinepic{\hTh{0.25}}$ \:and \:$h_4=\inlinepic{\hFo{0.25}}$.\vspace{0.5cm}\\
Then $AF'G-F'GA=hd+dh$.\\

    \item \textbf{The other orientations of MM17:}\\
These relations are generated by $MM6$ to $MM10$, $MM17+$ and far commutativity relations.\\ 
A quick look at \cite{GeneratingSets} shows that different $R3$ moves can be obtained from each other inductively with by putting one between two $R2$ moves at a time, as in Figure \ref{fig: R3Def}.

\begin{figure}[h]
    \[
    \begin{tikzcd}[column sep=2cm]
    \includegraphics[width=2cm]{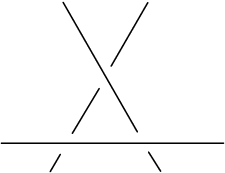}
    \arrow[r,"R2",yshift=0.5cm]&
    \includegraphics[width=2cm]{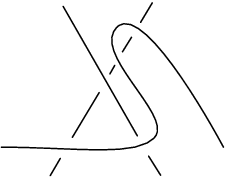}
    \arrow[r,"R3",yshift=0.5cm]&
    \includegraphics[width=2cm]{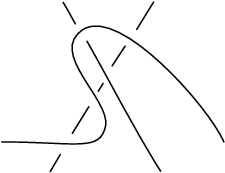}
    \arrow[r,"R2",yshift=0.5cm]&
    \includegraphics[width=2cm]{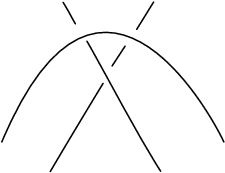}
    \end{tikzcd}\vspace{0.5cm}
    \]
    \caption{A standard way to generate $R3$ moves from one another.}
    \label{fig: R3Def}
\end{figure}
Let us now fix an orientation on the diagram in Figure \ref{fig: R3Def} such that the $R3$ move in the middle (say $R3_a$) is already proven to respect $MM17$. Let $R3_b$ be the resulting move. We want to compare the following sequences:
    \[
    \begin{tikzcd}[column sep=1.5cm]
    \includegraphics[width=1.5cm]{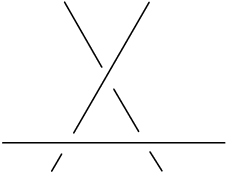}
    \arrow[r,"A/B",yshift=0.5cm]&    
    \includegraphics[width=1.5cm]{Figures/L1.eps}
    \arrow[r,"R2",yshift=0.5cm]&
    \includegraphics[width=1.5cm]{Figures/L3.eps}
    \arrow[r,"R3_a",yshift=0.5cm]&
    \includegraphics[width=1.5cm]{Figures/L5.eps}
    \arrow[r,"R2",yshift=0.5cm]&
    \includegraphics[width=1.5cm]{Figures/L7.eps}
    \end{tikzcd}\vspace{0.5cm}\vspace{0.5cm}
    \]
    and
    \[
    \begin{tikzcd}[column sep=1.5cm]
    \
    \includegraphics[width=1.5cm]{Figures/L2.eps}
    \arrow[r,"R2",yshift=0.5cm]&
    \includegraphics[width=1.5cm]{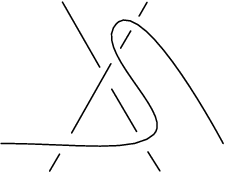}
    \arrow[r,"R3_a",yshift=0.5cm]&
    \includegraphics[width=1.5cm]{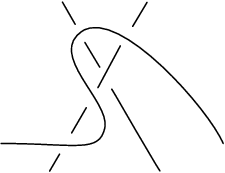}
    \arrow[r,"R2",yshift=0.5cm]&
    \includegraphics[width=1.5cm]{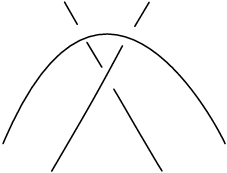}
    \arrow[r,"A/B",yshift=0.5cm]&    
    \includegraphics[width=1.5cm]{Figures/L7.eps}
    \end{tikzcd}\vspace{0.5cm}
    \]
In the first one, by far commutativity relations, we can exchange the moves $A/B$ and $R_2$. We obtain:
    \[
    \begin{tikzcd}[column sep=1.5cm]
    \includegraphics[width=1.5cm]{Figures/L2.eps}
    \arrow[r,"R2",yshift=0.5cm]&    
    \includegraphics[width=1.5cm]{Figures/L4.eps}
    \arrow[r,"A/B",yshift=0.5cm]&
    \includegraphics[width=1.5cm]{Figures/L3.eps}
    \arrow[r,"R3_a",yshift=0.5cm]&
    \includegraphics[width=1.5cm]{Figures/L5.eps}
    \arrow[r,"R2",yshift=0.5cm]&
    \includegraphics[width=1.5cm]{Figures/L7.eps}
    \end{tikzcd}\vspace{0.5cm}\vspace{0.5cm}
    \]
    Now applying $MM17$ we can switch $A/B$ and $R3_a$ to get:
    \[
    \begin{tikzcd}[column sep=1.5cm]
    \includegraphics[width=1.5cm]{Figures/L2.eps}
    \arrow[r,"R2",yshift=0.5cm]&    
    \includegraphics[width=1.5cm]{Figures/L4.eps}
    \arrow[r,"R3_a",yshift=0.5cm]&
    \includegraphics[width=1.5cm]{Figures/L6.eps}
    \arrow[r,"A/B",yshift=0.5cm]&
    \includegraphics[width=1.5cm]{Figures/L5.eps}
    \arrow[r,"R2",yshift=0.5cm]&
    \includegraphics[width=1.5cm]{Figures/L7.eps}
    \end{tikzcd}\vspace{0.5cm}\vspace{0.5cm}
    \]
    Finally, by far commutativity we can switch the last two maps to get exactly the sequence we wanted.\\
    This finishes the proof of $MM17$.

\end{itemize}\par
    \textbf{MM18:}\\
    In the context of Khovanov homology planar isotopy yields the identity map so $MM18$ is clear.
\end{proof}
\nocite{*}
\printbibliography
\end{document}